\documentclass[12pt, a4paper, parskip=half, abstracton]{scrartcl}

\usepackage{array}
\usepackage{marginnote}
\usepackage{xcolor}
\usepackage{amscd,amssymb,amsfonts,amsmath,latexsym,amsthm}
\usepackage{hyperref}
\usepackage[all,cmtip]{xy}
\usepackage{mathrsfs}
\usepackage{graphicx}
\usepackage{bm} 
\usepackage{nicefrac}

\usepackage[capitalise, nameinlink]{cleveref}

\usepackage{etoolbox}

\usepackage[nottoc]{tocbibind} 

\usepackage{slashed}
\usepackage[utf8]{inputenc}
\usepackage{microtype}
\usepackage[english]{babel}
\usepackage{mathtools}
\usepackage{bm}
\usepackage{esvect}

\title{On the Farrell--Jones conjecture for localising invariants}
\author{
Ulrich Bunke\thanks{Fakult{\"a}t f{\"u}r Mathematik,
Universit{\"a}t Regensburg,
93040 Regensburg,
ulrich.bunke@mathematik.uni-regensburg.de} 
\and
Daniel Kasprowski\thanks{{Rheinische Friedrich-Wilhelms-Universit\"at Bonn, Mathematisches Institut, Endenicher Allee 60, 53115 Bonn,
kasprowski@uni-bonn.de}}
\and
Christoph Winges\thanks{Fakult\"at f\"ur Mathematik,
Universit\"at Regensburg,
93040 Regensburg,
christoph.winges@ur.de}
}
\date{}

\numberwithin{equation}{section}
\newtheorem{theorem}{Theorem}[section] 
\newtheorem{prop}[theorem]{Proposition}
\newtheorem{lem}[theorem]{Lemma}

\newtheorem{kor}[theorem]{Corollary}

\theoremstyle{remark}
\theoremstyle{definition}
\newtheorem{construction-alt}[theorem]{Construction}
\newtheorem{ddd-alt}[theorem]{Definition}
\newtheorem{ex-alt}[theorem]{Example}
\newtheorem{rem-alt}[theorem]{Remark}
\newtheorem{standing-alt}[theorem]{Standing Assumption}

\newenvironment{ddd}    
{%
	\pushQED{\qed}\begin{ddd-alt}}
	{\popQED\end{ddd-alt}}
	
\newenvironment{construction}    
{%
	\pushQED{\qed}\begin{construction-alt}}
	{\popQED\end{construction-alt}}

\newenvironment{ex}    
{%
	\pushQED{\qed}\begin{ex-alt}}
	{\popQED\end{ex-alt}}

\newenvironment{rem}    
{%
	\pushQED{\qed}\begin{rem-alt}}
	{\popQED\end{rem-alt}}
	
{%
	\pushQED{\qed}\begin{standing-alt}}
	{\popQED\end{standing-alt}}
	
\crefname{lem}{Lemma}{Lemmas}
\crefname{prop}{Proposition}{Propositions}
\crefname{ddd-alt}{Definition}{Definitions}
\crefname{theorem}{Theorem}{Theorems}

\newcommand{\ev}{\mathrm{ev}}
\newcommand{\wt}{\widetilde}

\newcommand{\CAT}{\mathbf{CAT}}
\newcommand{\hfd}{\mathrm{fd}}
\newcommand{\CW}{\mathbf{CW}}
\newcommand{\cofib}{\mathrm{cofib}}
\newcommand{\Homol}{{H}}

\newcommand{\cJ}{\mathcal{J}}

\DeclareMathOperator{\yo}{yo}

\newcommand{\Res}{\mathrm{Res}}

\newcommand{\Orb}{\mathbf{Orb}}

\newcommand{\cR}{\mathcal{R}}

\newcommand{\BC}{\mathbf{BC}}

\newcommand{\Pro}{\mathrm{Pro}}

\newcommand{\cN}{\mathcal{N}}

\newcommand{\Fin}{\mathbf{Fin}}

\newcommand{\Sing}{\mathrm{Sing}}

\newcommand{\bM}{\mathbf{M}}

\newcommand{\cP}{\mathcal{P}}
\newcommand{\cM}{\mathcal{M}}

\newcommand{\cK}{\mathcal{K}}

\newcommand{\CAlg}{{\mathbf{CAlg}}}

\newcommand{\cL}{{\mathcal{L}}}
\newcommand{\cW}{{\mathcal{W}}}
\newcommand{\PSh}{{\mathbf{PSh}}}

\newcommand{\Add}{{\mathtt{Add}}}

\newcommand{\bA}{{\mathbf{A}}}

\newcommand{\bV}{{\mathbf{V}}}
\newcommand{\bC}{{\mathbf{C}}}
\newcommand{\bE}{{\mathbf{E}}}
\newcommand{\bI}{{\mathbf{I}}}
\newcommand{\bN}{{\mathbf{N}}}
\newcommand{\bR}{{\mathbf{R}}}
\newcommand{\beins}{{\mathbf{1}}}
\DeclareMathOperator{\const}{const}

\newcommand{\Alg}{{\mathbf{Alg}}}

\newcommand{\cO}{{\mathcal{O}}}
\newcommand{\cC}{{\mathcal{C}}}
\newcommand{\cU}{{\mathcal{U}}}
\newcommand{\cY}{{\mathcal{Y}}}
\newcommand{\cX}{{\mathcal{X}}}

\newcommand{\cD}{{\mathcal{D}}}
\newcommand{\cF}{{\mathcal{F}}}
\newcommand{\cB}{{\mathcal{B}}}

\newcommand{\cV}{{\mathcal{V}}}
\newcommand{\cT}{{\mathcal{T}}}
\newcommand{\cH}{{\mathcal{H}}}

 \newcommand{\Cat}{{\mathbf{Cat}}}

\DeclareMathOperator{\sk}{sk}

\newcommand{\bD}{\mathbf{D}}

\newcommand{\Fun}{\mathbf{Fun}}

\newcommand{\Idem}{\mathrm{Idem}}

\newcommand{\Coarse}{\mathbf{Coarse}}
\newcommand{\Spc}{\mathbf{Spc}}

\newcommand{\IZ}{\mathbb{Z}}

\newcommand{\IC}{\mathbb{C}}
\newcommand{\IQ}{\mathbb{Q}}

\renewcommand{\Add}{\mathbf{Add}}

\newcommand{\Catex}{\mathbf{Cat}_{\infty}^{\mathrm{ex}}}
\newcommand{\rex}{\mathrm{Rex}}
\newcommand{\lex}{\mathrm{Lex}}
\newcommand{\perf}{\mathrm{perf}}
\newcommand{\Cle}{{\Cat^{\lex}_{\infty,*}}}
\newcommand{\Cre}{\Cat^{\rex}_{\infty,*}}
\newcommand{\Crp}{\Cat^{\rex,\perf}_{\infty,*}}

\newcommand{\CL}{\Cat^{\mathrm{LEX}}_{\infty,*}}

\newcommand{\CLL}{\mathbf{CAT}^{\lex}_{\infty,*}}
\newcommand{\CRR}{\mathbf{CAT}^{\rex}_{\infty,*}}
\newcommand{\CLLL}{\mathbf{CAT}^{\mathrm{cplt}}_{\infty,*}}
\newcommand{\CRRR}{\mathbf{CAT}^{\mathrm{cocplt}}_{\infty,*}}

\newcommand{\bd}{\mathrm{bd}}
\newcommand{\op}{\mathrm{op}}
\newcommand{\CATi}{\mathbf{CAT}_{\infty}}
\newcommand{\Cati}{\mathbf{Cat}_{\infty}}

\newcommand{\Clep}{\Cat^{\lex,\perf}_{\infty,*}}

\newcommand{\Prl}{\mathbf{Pr}^{\mathrm{L}}_{\omega}}

\newcommand{\Prlp}{\mathbf{Pr}^{\mathrm{L}}_{\omega,*}}
\newcommand{\Prlpo}{\mathbf{Pr}^{\mathrm{L},\otimes}_{\omega,*}}

\newcommand{\Catrex}{\mathbf{Cat}^{\rex}_{\infty}}

\newcommand{\CATCocplt}{\CRRR}
\newcommand{\CATCplt}{\CLLL}

\newcommand{\smalll}{\mathrm{small}}
\newcommand{\fin}{\mathrm{fin}}

\newcommand{\mbn}{\mathrm{mb}}
\newcommand{\nat}{\mathbb{N}}
\newcommand{\Set}{\mathbf{Set}}

\newcommand{\Sp}{\mathbf{Sp}}
\newcommand{\Sh}{\mathbf{Sh}}
\newcommand{\Ch}{\mathbf{Ch}}
\newcommand{\Mod}{\mathbf{Mod}}
\newcommand{\Top}{\mathbf{Top}}
\newcommand{\RelCat}{\mathbf{RelCat}}

\newcommand{\cp}{\mathrm{cp}}

\DeclareMathOperator{\colim}{colim}

\DeclareMathOperator{\As}{A}
\DeclareMathOperator{\map}{Map}
\DeclareMathOperator{\Map}{\map}
\DeclareMathOperator{\Ind}{Ind}
\DeclareMathOperator{\ind}{ind}
\DeclareMathOperator{\res}{res}
\DeclareMathOperator{\diag}{diag}
\DeclareMathOperator{\pr}{pr}
\DeclareMathOperator{\id}{id}
\DeclareMathOperator{\Nerve}{N}
\DeclareMathOperator{\inc}{incl}
\DeclareMathOperator{\incl}{incl}
\DeclareMathOperator{\Hom}{Hom}
\DeclareMathOperator{\Mul}{Mul}
\DeclareMathOperator{\Stab}{Stab}

\newcommand{\exten}[1]{#1^+}
\newcommand{\extenp}[1]{#1^{\prod}}
\newcommand{\abs}[1]{\lvert #1 \rvert}

\newcommand{\hG}{\mathrm{\mathrm{h}G}}
\newcommand{\Grp}{\mathbf{Grp}}

\newcommand{\prodsum}{\nicefrac{\prod}{\bigoplus}}

\newcommand{\extenps}[1]{#1^{\prodsum}}
\newcommand{\Assoc}{\mathcal{A}}
\newcommand{\cElm}{\mathcal{E}lm}
\renewcommand{\emptyset}{\varnothing}

\newcommand*\cocolon{%
	\nobreak
	\mskip6mu plus1mu
	\mathpunct{}%
	\nonscript
	\mkern-\thinmuskip
	{:}%
	\mskip2mu
	\relax
}

\newcommand{\rp}{r^{\perf}}
\newcommand{\rc}{\mathrm{rc}}
\newcommand{\Equiv}{\mathrm{Equiv}}

\begin{document}

	\maketitle
	\begin{abstract}
We show the Farrell--Jones conjecture with coefficients in left-exact $\infty$-categories 
for finitely $\cF$-amenable groups and, more generally, Dress-Farrell-Hsiang-Jones groups.
Our result subsumes and unifies arguments for the K-theory of additive categories and spherical group rings and extends it for example to categories of perfect modules over $\mathbb{E}_{1}$-ring spectra.
 \end{abstract}

\tableofcontents
\setcounter{tocdepth}{5}

 \section{Introduction}\label{sec:intro}
 
 Let $G$ be a group, let $\cF$ be a family of subgroups of $G$, and let $F \colon G\Orb\to \bM$ be a functor from the orbit category of {$G$} to a cocomplete $\infty$-category.
 
\begin{ddd}\label{wriohgjrotggerergw}
 The assembly map associated to $G$, $\cF$ and $F$ is defined as the canonical morphism
\begin{equation}\label{fevuihiuqhvvsad}
 \As_{\cF,F}\colon \mathop{\colim}\limits_{G_{\cF}\Orb} F\to F(*)\ . \qedhere
\end{equation}
\end{ddd}
Here $G_{\cF}\Orb$ is the full subcategory of $G\Orb$ of $G$-orbits with stabilisers in $\cF$, and $*$ is the final object of $G\Orb$.

 The Farrell--Jones conjecture states that the assembly map $\As_{\cF,F}$ is an equivalence for the family $\cF = \cV\cC yc$ of virtually cyclic subgroups in the case where $F$ is the equivariant K- (or L-) theory functor associated to an additive category with $G$-action (and involution) \cite{FJ}.
 Since we will consider the conjecture for more general coefficients, we refer to this version of the Farrell--Jones conjecture as the linear Farrell--Jones conjecture.

 Via surgery theory, the {linear} Farrell--Jones conjecture has interesting consequences to manifold topology. For example, it implies the Borel and the Novikov conjecture.
 For more background information on the {linear} Farrell--Jones conjecture we refer to the surveys \cite{LR,surveyLueck,surveyRV} as well as L\"uck's ongoing book project \cite{ICbook}. 
 
  While the emphasis of the present paper lies in extending the class of functors and coefficients for which the conjecture holds, most of the literature on the {linear} Farrell--Jones conjecture is devoted to understanding and extending the class of groups for which the conjecture is known to be true.
 Due to extensive work of Bartels, Lück, Reich and many other authors the {linear} conjecture is known for all hyperbolic or CAT(0)-groups, for mapping class groups, and for many linear groups \cite{blr,BL-borel,wegner:cat0,BB,BLRR}.  
 
 Building on the definition of wide covers from \cite{blr}, Bartels introduced the axiomatic condition of finite $\cF$-amenability  \cite{bartels-relhyp}. A more general, but also more technical, condition is that of a {Dress--Farrell--Hsiang--Jones group relative $\cF$} \cite[Def.~2.2]{KUWW}, see also \cref{def:dfhj}, {which we subsequently abbreviate to DFHJ group.}
 The combination of {this condition} with
  the known inheritance properties of the {linear} Farrell--Jones conjecture gives the list of groups for which the conjecture is currently known. We refer to \cref{thm:fullfjc} and its proof for more details.

 In order to indicate that we study the assembly map $\As_{\cF,F}$ for more general functors $F$, we will drop the word ``{linear}''.
 In \cite{UW,Enkelmann:2018aa, KUWW} 
 the results about the {linear} Farrell--Jones conjecture mentioned above 
 where extended to equivariant versions of Waldhausen's A-theory functor.
 The goal of the present paper is to generalise the class of functors $F$ even further. For example, our results generalise the verification of the {linear} K-theoretic 
  Farrell--Jones conjecture from ordinary rings to ring spectra.
Our general setting in fact subsumes the cases of the {linear} Farrell--Jones conjecture for the K-theory of an additive category as well as for A-theory  as special cases. This fact will be explained in some detail further below in this introduction.

In \cref{weoirgjwegwergwerg9}, we will describe the class of functors
$F=\Homol \bC_{G}$ that go into 
 our main  results  \cref{thm:main} and {\cref{thm:fullfjc}} below. 
 We first introduce the notation which is necessary to state this definition.  
 
 A left-exact $\infty$-category is a   pointed $\infty$-category which admits all finite limits.
By $\Cle$ we denote 
  the large $\infty$-category of small left-exact $\infty$-categories and   finite limit-preserving functors.
   It contains the large $\infty$-category of small stable $\infty$-categories $\Catex$ as a full subcategory.
  
 Let $\bM$ be a   cocomplete stable $\infty$-category and consider a functor
 \[ \Homol \colon \Cle\to \bM\ .\]
 \begin{ddd}[{\cite[Def.~6.7]{unik}}]\label{wtoigwgreerf}
 The functor  $H$ is called a finitary localising invariant if it preserves zero objects  and filtered colimits, sends excisive squares to pushout squares, and inverts Morita equivalences.
 \end{ddd}
 
 This definition generalises \cite{MR3070515} from stable to left-exact $\infty$-categories.
 In fact, any finitary localising invariant on left-exact $\infty$-categories arises from a localising invariant on stable $\infty$-categories in the sense of \cite{MR3070515} by precomposing with the stabilisation functor \cite[Lem.~6.9]{unik}.
 In the present paper we will further {assume} that $\Homol$ is lax monoidal with respect to the symmetric monoidal structure on $\Cle$ explained in \cref{sec:prelims} and some stably monoidal structure on $\bM$, and that $\bM$   admits countable products.
   
The data for the construction of the functor $H\bC_{G} {\colon G\Orb \to \bM}$ in \cref{weoirgjwegwergwerg9}   consists of a 
  left-exact $\infty$-category with $G$-action $\bC$, i.e., an object of $\Fun(BG,\Cle)$,  and a functor $\Homol \colon \Cle\to \bM$.
   There is a canonical inclusion $j^{G} \colon BG \to G\Orb$ sending the unique object of $BG$ to the transitive $G$-set $G$. We define the functor
\begin{equation}\label{vfoijiorvefvfdsv}
\bC_{G} := j_{!}^{G}(\bC) \colon G\Orb\to \Cle
\end{equation}
 as {the} left Kan extension of $\bC$ along this inclusion.
For a subgroup $H$ of $G$ we can calculate the value of $\bC_{G}$ on the orbit $G/H$ using the  pointwise formula for the left Kan extension:
\begin{equation}\label{foij1oiwefqfeqewfef}
 \bC_{G}(G/H)\simeq \mathop{\colim}\limits_{BH}\Res^{G}_{H}(\bC)\ .
\end{equation}

\begin{ddd}\label{weoirgjwegwergwerg9}
We define the functor
$\Homol\bC_{G}:=\Homol\circ \bC_{G}:G\Orb\to \bM$.
\end{ddd}
  
 We consider a group $G$ and a family of subgroups $\cF$. 
 The notion of a DFHJ group is discussed in \cref{sec:dfhj}, see in particular  \cref{def:dfhj}.
 The notion of a phantom equivalence in $\bM$ will be introduced in \cref{ethigoewggergwgerg}.  
 It is a weakening of the notion of an equivalence,
but if $\bM$ is compactly generated, then every phantom equivalence is an equivalence.
Recall that $\bM$ is assumed to be a stably monoidal and cocomplete stable $\infty$-category which admits countable products.
The following is the main theorem of the present paper.

 \begin{theorem}\label{thm:main}
If $G$ is a DFHJ group relative $\cF$ and $H$ is a lax monoidal, finitary localising invariant, then the assembly map
 \begin{equation}\label{viijevoiqrevoqrvqvqrv}
 \As_{\cF,\Homol\bC_{G}} \colon \mathop{\colim}\limits_{G_{\cF}\Orb}\Homol\bC_{G}\to
 \Homol (\mathop{\colim}\limits_{BG} \bC)
\ . \end{equation}
is a phantom equivalence.
\end{theorem}

 Note that up to the identification of its target using \eqref{foij1oiwefqfeqewfef}, the map   
  $\As_{\cF,\Homol\bC_{G}}$ in \eqref{viijevoiqrevoqrvqvqrv} is  the same as the one in \eqref{fevuihiuqhvvsad}.
  We prefer to state the theorem in this form since it explicitly mentions the object $\Homol (\colim_{BG} \bC)$ that the assembly map tries to calculate.
 
 \begin{rem}
  As observed by Reis \cite{reis}, \cref{thm:main} still holds if we drop the assumption that $\Homol$ is lax monoidal.
 \end{rem}
 
 As explained in \cite[Sec.~6.2]{unik}, finitary localising invariants arise by precomposing finitary, localising invariants on $\Catex$ with the stabilisation functor $\Stab \colon \Cle \to \Catex$.
 Since $\Stab$ is a symmetric monoidal functor, finitary localising invariants on $\Catex$ {that} are in addition  lax monoidal provide examples for the functor $\Homol$ in  \cref{thm:main}.
 Our main example of a finitary and  lax monoidal  localising invariant is the  nonconnective algebraic K-theory functor, see \cite[Sec.~9]{MR3070515} and \cite[Prop.~5.9]{bgt-2}.
 Other possible choices include topological Hochschild homology \cite[Cor.~6.9]{bgt-2} and related functors like $TC^n$ \cite[Cor.~6.15]{bgt-2} or topological cyclic homology regarded as a pro-spectrum-valued functor.
However,  in case of the latter functors the Farrell--Jones conjecture is known to hold for all groups, see \cite[Thm.~6.1]{LRRV} and  \cite[Thm.~1.3 and its proof]{LRRV2}.

 In the following, we need the notion of a wreath product.
 Let $G,F$  be groups.  Then we consider the group $G^{F}$ of maps from $F$ to $G$ with the pointwise structure and  the action of $F$ by automorphisms induced from the left multiplication on itself. This action is used to define the wreath product
 $G \wr F: = G^{F} \rtimes F$.
 Note that $G^{F}$ is
 a subgroup of $G\wr F$ in a canonical way.
 
 Following \cite[Sec.~14.2]{ICbook}, our results can be combined into a single, comparatively concise statement as follows.  
Let $G$ be a group, and let $\Homol \colon \Cle \to \bM$ be a lax monoidal, finitary localising invariant.
\begin{ddd}\label{def:fullfjc}
 We say that the group $G$ satisfies the Full Farrell--Jones conjecture for $\Homol$ if the assembly map $\As_{\cV\cC yc,\Homol\bC_{G \wr F}}$ is {a phantom} equivalence for every finite group $F$ and left-exact $\infty$-category $\bC$ with $G \wr F$-action.
\end{ddd}

By allowing left-exact $\infty$-categories with $G$-action as coefficients, we retain all inheritance properties typically known for the linear Farrell--Jones conjecture.

Given the results of \cref{sec:famenablegroups,qrgqiorgegegergwegre,sec:dfhj,sec:inheritance}, the proof of the following  theorem is a combination of arguments scattered throughout the literature (see also \cite[Thm.~14.1]{ICbook}).
Its new aspect  is the additional  flexibility 
in the choice of the functor $\Homol$ and the coefficients $\bC$.

Let $\Homol \colon \Cle \to \bM$ be a lax monoidal, finitary localising invariant with values in a stably monoidal and cocomplete stable $\infty$-category which admits countable products.
We denote the class of groups that satisfy the full Farrell--Jones conjecture for $\Homol$ by $\cF\cJ_\Homol$.

\begin{theorem}\label{thm:fullfjc}
 The class $\cF\cJ_\Homol$
 has the following properties:
 \begin{enumerate}
   \item\label{it:fullfjc13} If $K$ is a subgroup of a group $G$ in $\cF\cJ_\Homol$, then $K$ belongs to $\cF\cJ_\Homol$.
  \item\label{it:fullfjc1} Groups that act isometrically, properly and cocompactly on a finite-dimensional $\mathrm{CAT}(0)$-space belong to $\cF\cJ_\Homol$.  
  \item\label{it:fullfjc2} Hyperbolic groups belong to $\cF\cJ_\Homol$.
  \item\label{it:fullfjc3} Virtually solvable groups belong to $\cF\cJ_\Homol$.
  \item\label{it:fullfjc4} All subgroups of $\mathrm{GL}_n(\IQ)$ and $\mathrm{GL}_n(k(t))$ belong to $\cF\cJ_\Homol$, where $k(t)$ is the function field over a finite field $k$.
  \item\label{it:fullfjc7} If $L$ is a lattice in a locally compact, second countable Hausdorff group $G$ such that $\pi_0(G)$ is discrete and belongs to $\cF\cJ_\Homol$, then $L$ belongs to $\cF\cJ_\Homol$.
  \item\label{it:fullfjc8} Fundamental groups of connected manifolds of dimension at most $3$ belong to $\cF\cJ_\Homol$.
  \item\label{it:fullfjc10} Fundamental groups of graphs of virtually cyclic or of graphs of abelian groups belong to $\cF\cJ_\Homol$.
  \item\label{it:fullfjc9} The mapping class group of any closed, orientable surface with a finite number of punctures belongs to $\cF\cJ_\Homol$. 
  \item\label{it:fullfjc14} If $G$ belongs to $\cF\cJ_\Homol$ and $G'$ is a group which contains $G$ as a subgroup of finite index, then $G'$ belongs to $\cF\cJ_\Homol$.
  \item\label{it:fullfjc15} If $G_1$ and $G_2$ belong to $\cF\cJ_\Homol$, then $1 \times G_2$ belongs to $\cF\cJ_\Homol$.
  \item\label{it:fullfjc16} If $\Gamma \colon I \to \Grp$ is a filtered diagram of groups and $\Gamma_i$ belongs to $\cF\cJ_\Homol$ for all $i$ in $I$, then $\colim_I \Gamma$ belongs to $\cF\cJ_\Homol$.
  \item\label{it:fullfjc17} If $\pi \colon G \to Q$ is an epimorphism such that $Q$ belongs to $\cF\cJ_\Homol$ and $\pi^{-1}(C)$ belongs to $\cF\cJ_\Homol$ for every cyclic subgroup $C$ of $Q$, then $G$ belongs to $\cF\cJ_\Homol$.
  \item\label{it:fullfjc18} If $(G_i)_{i \in I}$ is a family of groups in $\cF\cJ_\Homol$, then the free product $*_{i \in I} G_i$ belongs to $\cF\cJ_\Homol$.
 \end{enumerate}
\end{theorem}
The proof of this theorem will be given in \cref{hieruhqwefiwfewf}.

\begin{rem}
By \cref{thm:fullfjc}.\ref{it:fullfjc14}, a group $G$ satisfies the Full Farrell--Jones conjecture if and only if the assembly map $\As_{\cV\cC yc,\Homol\bC_{G'}}$ is a phantom equivalence for every group $G'$ which contains $G$ as a subgroup of finite index and left-exact $\infty$-category $\bC$ with $G$-action.
\end{rem}

In the following we explain why the functors $\Homol \bC_{G}$ introduced in \cref{weoirgjwegwergwerg9} generalise {and unify} the functors arising from algebraic K-theory previously considered by other authors.

First assume that $\bA$ is an additive category with a strict $G$-action, i.e., an object of $\Fun(BG,\Add)$. We form   a stable $\infty$-category $\Ch_{\infty}^{b}(\bA)$ with $G$-action by localising the category $\Ch^{b}(\bA)$ at the collection of chain homotopy equivalences.
 As explained in \cite[Ex.~1.2 and Ex.~1.6]{unik}, {the functor $K\Ch_{\infty}^{b}(\bA)_G$} is equivalent to the $G\Orb$-spectrum
defined in \cite[Sec.~2]{davis_lueck}, respectively \cite[Sec.~3]{BR-coefficients}.

For a ring $R$, we could take the additive category $\bA=\Mod(R)^{\mathrm{fg},\mathrm{proj}}$
of finitely generated projective $R$-modules with the trivial $G$-action. 
The generalisation from discrete rings to associative ring spectra $\cR$ is obtained by replacing $\Mod(R)^{\mathrm{fg},\mathrm{proj}}$ with the left-exact $\infty$-category $\Mod(\cR)^{\mathrm{perf},\op}$, the opposite of the $\infty$-category of perfect $R$-modules, see \cite[Ex.~1.4]{unik}.

Finally, we refer to \cite[Ex.~1.9 and Cor.~7.71]{unik} for an argument that the functors appearing in the A-theoretic Farrell--Jones conjecture \cite{UW,Enkelmann:2018aa, KUWW} are special instances of \cref{weoirgjwegwergwerg9}.

We conclude the introduction with an overview of the structure of this article.
The first part of \cref{sec:phantoms} formulates an abstract criterion to decide that an assembly map is a phantom equivalence.
\cref{sec:assembly} contains some recollections about assembly maps and the notion of phantom equivalence.
\cref{sec:borncoarse} provides some basic vocabulary concerning $G$-bornological coarse spaces which is necessary for the following sections.
\cref{rfoiqrejgoireggregwrgregwregwergregr} introduces the notion of properness which in turn is used to define the notion of a transfer class in \cref{sec:phantom-transfer}.
The key point about transfer classes is that they enable us to
show that the vertical transformation 
\[\xymatrix@C=5em{
 \mathop{\colim}\limits_{G_\cF\Orb} \Homol\bC_G\ar[r]^-{\As_{\cF,\Homol\bC_G}}\ar[d]_-{\Delta} & \Homol\bC_G(*)\ar[d]^-{\Delta} \\
 \prodsum \mathop{\colim}\limits_{G_\cF\Orb} \Homol\bC_G\ar[r]^-{\prodsum \As_{\cF,\Homol\bC_G}} & \prodsum \Homol\bC_G(*)
}\]
induced by the diagonals induces the zero map on the cofibres of the horizontal morphisms. Hence $\As_{\cF,\Homol\bC_G}$ is a phantom equivalence.
\cref{sec:squeezing} gives examples of proper objects. For the verification of properness we will use {the concept of} equivariant coarse homology theories \cite{equicoarse}.

Almost the entirety of the remaining article deals with the construction of transfer classes.
To apply \cref{ergiooegergergwergergergw}, one must in particular show that the functor $\Homol\bC_G$ arises from a functor on $G$-bornological coarse spaces; see \cref{wtgklpwergrewfwref} for a precise definition.
Such functors on $G$-bornological coarse spaces, which generalise the concepts of controlled algebra to left-exact $\infty$-categories, have been previously defined in \cite{unik}.
We recall their construction in \cref{sec:controlled-objects}.
For our purposes, we need to equip these functors with additional multiplicative structure.
The impatient reader will find a summary of the necessary multiplicative structure in \cref{thm:fix-orbit-wmodule}.
\cref{sec:prelims,sec:controlled-objects-monoidal} show that our categories of controlled objects in fact form a lax monoidal functor, which is used in \cref{sec:orbits-and-fixed-points} to prove a highly structured version of \cref{thm:fix-orbit-wmodule}.

The construction of transfer classes relies on point-set input data which need to be imported into our setting.
\cref{sec:controlledCW} recalls the notion of controlled CW-complexes over a bornological coarse space from \cite{BKW:coarseA}.
\cref{sec:realisation,sec:finiteness} then construct a natural transformation from these categories of controlled CW-complexes to the categories of controlled objects introduced in \cref{sec:control-monoidal}.
 
\cref{sec:famenablegroups} establishes the first of our main theorems, namely that the assembly map for finitely $\cF$-amenable groups is a phantom equivalence.
In \cref{qrgqiorgegegergwegre}, we prove that the condition of being a Dress--Farrell--Hsiang group also implies that the assembly map is a phantom equivalence.
\cref{sec:dfhj} combines the methods of the two preceding sections to prove that the assembly map $\As_{\cF,\Homol\bC_G}$ is a phantom equivalence if $G$ is a DFHJ-group with respect to the family $\cF$.
In all three cases, we construct appropriate transfer classes which allow us to apply \cref{ergiooegergergwergergergw}.

The final \cref{sec:inheritance} collects various inheritance properties of the assembly map in an abstract setting.
 
 {{\em Conventions.}} In order to address size issues, we fix a sequence of four increasing Grothendieck universes whose sets will be called very small, small, large, and very large. The groups, bornological coarse spaces, CW-complexes etc will be very small. The categories of these objects are small, but locally very small.
  The objects of $\Cati$ are small, locally very small $\infty$-categories, and this category itself is large, but locally small. By $\CATi$ we denote the very large, locally large $\infty$-category of large, locally small $\infty$-categories. It contains the subcategory $\Prl$ of presentable $\infty$-categories and left adjoint functors.

{{\em Acknowledgements.} We thank Arthur Bartels for various discussions concerning the proofs of the linear Farrell--Jones conjecture. In particular, the idea that the transfer should arise from an action of controlled Swan theory on controlled K-theory is due to him. We also profited from dicussions with Thomas Nikolaus and Fabian Hebestreit regarding the contents  of  \cref{sec:control-monoidal}.}
 
{U.~Bunke and C.~Winges\ were supported by CRC 1085 Higher Invariants (Universit\"at Regensburg, funded by the DFG). C.~Winges was supported by the Max Planck Society and Wolfgang L\"uck's ERC Advanced Grant ``KL2MG-interactions" (no.~662400).
D.~Kasprowski and C.~Winges were funded by the Deutsche Forschungsgemeinschaft (DFG, German Research Foundation) under Germany's Excellence Strategy - GZ 2047/1, Projekt-ID 390685813.}

\section{Phantom equivalences and transfer classes}\label{sec:phantoms}
 
\subsection{Assembly maps and phantom equivalences}\label{sec:assembly}

In this section we first explain in \cref{qgoiegjqefewfqfewfqewf} that our definition of the assembly map in \eqref{fevuihiuqhvvsad} coincides with the classical definition of  the assembly map in terms of (unreduced) homology theories \cite{davis_lueck}.
Then we introduce the notion of a phantom equivalence {and prove \cref{veiowefwwfeewfewf} which provides} a criterion for detecting phantom equivalences.

 Let $\Spc$ denote the $\infty$-category of spaces.
 We use the notation $\PSh(\bC)$ for the $\infty$-category of $\Spc$-valued presheaves on an $\infty$-category $\bC$.
  
 Let $G$ be a group and let
 \[ \yo \colon G\Orb \to \PSh(G\Orb) \]
 denote the Yoneda embedding of the orbit category of $G$ into the category of $\Spc$-valued presheaves.
 By the universal property of the Yoneda embedding \cite[Thm.~5.1.5.6]{htt}, the pullback along $\yo$ induces an equivalence of $\infty$-categories
 \begin{equation*}\label{iqhgioqrjeoiedqfqewf}
  \Fun^{\colim}(\PSh(G\Orb),\bM)\to \Fun(G\Orb,\bM)
 \end{equation*}
 for any cocomplete $\infty$-category $\bM$.
 Here the domain is the full subcategory of the functor category  $\Fun(\PSh(G\Orb),\bM)$ of colimit-preserving functors{.}
 The inverse of this equivalence is given by the left Kan extension functor along the Yoneda embedding.
 For a functor $F \colon G\Orb \to \bM$, we will also use
 the notation $F$ for its colimit-preserving extension to presheaves.  
 
 We have a functor
 $\cElm \colon G\Top \to \PSh(G\Orb)$ which sends a topological $G$-space $X$ to the presheaf on $G\Orb$ sending $S$ to the mapping space $\Map_{G\Top}(S,X)$ considered as an object in  $\Spc$.
  By Elmendorf's theorem, the functor $\cElm$  presents the $\infty$-category $\PSh(G\Orb)$  as the localisation $G\Top[W_{G}^{-1}]$ of the category $G\Top$ at the class $W_{G}$
  of equivariant weak equivalences.
  Under this equivalence, the colimit-preserving extension of a functor $F$ in $\Fun(G\Orb,\bM)$ to presheaves
    corresponds to the equivariant homology theory $H^{G}(-;F) \colon G\Top\to \bM$  associated to $F$ by Davis and L\"uck.
 
 We now consider a family $\cF$ of subgroups of $G$.
\begin{ddd}\label{wrtoihgjwtgergwerfew}
We define the classifying space of the family $\cF$ to be the  object
\[ E_\cF G:= \mathop{\colim}\limits_{G_\cF\Orb} \yo \]
in $\PSh(G\Orb)$.
\end{ddd}
Under the equivalence $\PSh(G\Orb)\simeq G\Top[W_{G}^{-1}]$ given by Elmendorf's theorem, the presheaf $E_{\cF} G$ corresponds to the homotopy type of a $G$-CW-complex  $E_{\cF}G^{\CW}$ in $G\Top$ which is also called the classifying space for the family $\cF$ in equivariant homotopy theory.

\begin{lem}\label{qgoiegjqefewfqfewfqewf}
The following maps in $\bM$ are equivalent:
\begin{enumerate} 
\item  \label{erpogwergwerferf0} the assembly map $\As_{\cF,F}\colon \mathop{\colim}\limits_{G_{\cF}\Orb}F \to F(*)$ in \eqref{fevuihiuqhvvsad};
\item  \label{erpogwergwerferf} the map
$F(E_\cF G) \to F(*)$ induced by the projection $E_\cF G \to *$;
\item \label{erpogwergwerferf1} the Davis--L\"uck assembly map $H^{G}(E_{\cF}G^{\CW};F)\to H^{G}(*;F)$ induced by the projection $E_{\cF}G^{\CW}\to *$.
  \end{enumerate}
\end{lem}
\begin{proof}
The equivalence of \ref{erpogwergwerferf} and \ref{erpogwergwerferf1} is an immediate consequence of Elmendorf's theorem and of the identification of $H^G(-;F)$ with the colimit-preserving extension of $F$ to presheaves.
We now show the equivalence of \ref{erpogwergwerferf0} and \ref{erpogwergwerferf}.
Note that $* \simeq \colim_{G\Orb} {\yo}$. Hence we have a commutative diagram
 \[\xymatrix{
  E_\cF G\ar[d] & {\mathop{\colim}\limits_{G_\cF\Orb} {\yo}}\ar[d]\ar[l]_-{\simeq} \\
  {*} & {\mathop{\colim}\limits_{G\Orb} {\yo}}\ar[l]_-{\simeq} \\
 }\]
 in $\PSh(G\Orb)$. We now apply $F$ to obtain the left half of the following diagram:
 \begin{equation}\label{fsvvascasdcdc}
\xymatrix{
  F(E_\cF G)\ar[d] & {F(\mathop{\colim}\limits_{G_\cF\Orb}  {\yo})}\ar[d]\ar[l]_-{\simeq} & {\mathop{\colim}\limits_{G_\cF\Orb} {F }}\ar[d]^{\As_{\cF,F}}\ar[l]_-{\simeq} \\
  {F(*)} & {F(\mathop{\colim}\limits_{G\Orb}  {\yo})}\ar[l]_-{\simeq} & {\mathop{\colim}\limits_{G\Orb} {F }}\ar[l]_-{\simeq} \\
 }
\end{equation} 
 The right horizontal maps are equivalences since the extension of $F$ to $\PSh(G\Orb)$ commutes with colimits by definition.
The outer square in \eqref{fsvvascasdcdc} is the desired equivalence between the maps in  \ref{erpogwergwerferf0} and \ref{erpogwergwerferf}.
\end{proof}

\begin{rem}\label{egiowgregwergwerg}
 The equivalence  between the maps in 
 \ref{qgoiegjqefewfqfewfqewf}.\ref{erpogwergwerferf0} and
 \ref{qgoiegjqefewfqfewfqewf}.\ref{erpogwergwerferf1} is also shown in 
 \cite[Sec.~5]{davis_lueck}.
\end{rem}

 If $\bM$ is stable, then the assembly map is an equivalence if and only if its cofibre is trivial. 
 While we would like to prove that the cofibre is trivial, our arguments will only show a weaker condition, namely that it is a phantom object.
 Therefore, in the following we recall the notions of phantom objects and phantom equivalences.

Let $\bM$ be a cocomplete $\infty$-category.
 Recall that an object $K$ in $\bM$ is called compact if the functor
$\Map_{\bM}(K,-) \colon \bM\to \Spc$ preserves filtered colimits. 
Let $M$ be an object of $\bM$.

\begin{ddd} $M$ is called a phantom object if $\Map_{\bM}(K,M)\simeq *$ for every compact object $K$ of $\bM$. 
\end{ddd}
Let $m:M\to M^{\prime}$ be a morphism in $\bM$.
\begin{ddd}\label{ethigoewggergwgerg} The morphism $m  $ is called a phantom  equivalence if $\Map_{\bM}(K,m)$ is an equivalence of spaces  for every compact object $K$ of $\bM$. 
\end{ddd}

\begin{rem}\label{rem:phantoms}
A final object is a phantom object, and an equivalence is a phantom equivalence.
If $\bM$ is compactly generated, the converses of these assertions are true.
Indeed, in a compactly generated $\infty$-category a morphism $m$ is an equivalence in $\bM$ if and only if $\Map_{\bM}({K},m)$ is an equivalence in $\Spc$ for all compact objects $K$ of $\bM$.
If $M$ is not compactly generated, then the class of phantom equivalences is strictly bigger than the class of equivalences.

In a stable $\infty$-category, a morphism is a phantom equivalence if and only if its cofibre is a phantom object.
\end{rem}

 In the following we describe a simple criterion to recognise that
 an object $M$ in $\bM$ is a phantom object.
 It will be used in the proof of
 \cref{thm:main}.
 Recall that an $\infty$-category $\bM$ is called semi-additive if it is pointed, admits finite coproducts and products, and the canonical comparison morphism from the coproduct to the product
 is an equivalence.
 
Let $\bM$ be a semi-additive and cocomplete $\infty$-category which in addition admits countable products. 
Then for every object $M$ in $\bM$ we have a canonical morphism $\bigoplus_\nat M \to \prod_\nat M$.

\begin{ddd}
 For an  object $M$ in $\bM$ we define the object 
 \[ \prodsum(M) := \cofib\big( \bigoplus_\nat M \to \prod_\nat M \big)\]
 of $\bM$.
\end{ddd}

\begin{lem}
\label{veiowefwwfeewfewf}
If the morphism
\[ M \xrightarrow{\diag} \prod_{\nat} M \to \prodsum(M) \]
is trivial, then $M$ is a phantom object.
\end{lem}
\begin{proof}
The assumption implies that the morphism $\diag \colon M \to \prod_\nat M$ factors through $\bigoplus_\nat M$. Let $K$ be a compact object in $\bM$ and let $A$ be a compact object in $\Spc$. Then we consider for every morphism $A \to \Map_\bM(K,M)$ and every $n$ in $\nat$ the following {commutative} diagram:
	\[\xymatrix{
A\ar[r]\ar@/_1pc/@{-->}[rrd] & \Map_\bM(K,M) \ar[r]\ar@/^2pc/[rrr]^{\id} & \Map_\bM(K,\bigoplus_{\nat}M)\ar[r] &\Map_\bM(K,\prod_{\nat}M) \ar[r]_-{\pr_{n+1,*}} & \Map_\bM(K,M)\\
& & \Map_\bM(K,\bigoplus_{i=0}^n M)\ar@/_1pc/[rru]^{0}\ar[u]&&	
}\]
 Since the sum over $\nat$ is a filtered colimit of its finite partial sums and $K$ is compact, $\Map_\bM(K,\bigoplus_\nat M)$ is also a filtered colimit of the sequence of spaces $(\Map_\bM(K,\bigoplus_{i=0}^n M))_{n\in \nat}$. As $A$ is also compact, the dashed arrow exists for a suitable choice of $n$ in $\nat$. Hence $A \to \Map_\bM(K,M)$ is zero. As $\Spc$ is compactly generated, it follows that $\Map_\bM(K,M) \simeq *$.
 \end{proof}

 \subsection{Bornological coarse spaces}\label{sec:borncoarse}
 
  Throughout this article, we will make use of the language of bornological coarse spaces introduced in \cite{equicoarse}. See \cite[Sec.~3]{injectivity} for a concise summary of the key notions.
  
  For most of our purposes, it is sufficient to be familiar with the following:
  \begin{itemize}
  	\item the definition of the category $G\BC$ of $G$-bornological coarse spaces \cite[Def.~3.6 \& 3.7]{injectivity};
  	\item the notion of a $\pi_0$-excisive functor \cite[Def.~3.33]{unik}.
  \end{itemize}
  For the complete proof, we also require the concept of an equivariant coarse homology theory \cite[Def.~3.13]{injectivity}, but this will only be used in \cref{sec:squeezing}.

 For the reader's convenience, we give a quick outline of the basic notions mentioned above.
 Let $X$ be a set. An entourage on $X$ is a subset of $X \times X$.
 Regarding entourages as relations on $X$, it makes sense to speak of the inverse $U^{-1}$ of an entourage $U$ and the composition $U \circ V$ of two entourages $U$ and $V$.
 Given an entourage $U$ on $X$ and a subset $A$ of $X$, the $U$-thickening $U[A]$ of $A$ is given by
 \begin{equation}\label{qwefoiheiuohfiqwefewfewfqeefedq}
  U[A] := \{ x \in X \mid (\exists\,a \in A \colon (x,a) \in U) \}\ .
 \end{equation}
 A $G$-bornological coarse space is a triple $(X,\cB,\cC)$, where $X$ is a $G$-set, $\cB$ is a collection of subsets of $X$ called the bornology and $\cC$ is a collection of entourages on $X$ called the coarse structure, subject to the following conditions:
 \begin{enumerate}
  \item $\cB$ is invariant under the $G$-action on the power set of $X$, contains all singleton sets, and is closed under taking finite unions and subsets;
  \item $\cC$ contains the diagonal, is closed under forming subsets, finite unions, inverses and compositions;
  \item the subposet $\cC^G$ of $G$-invariant entourages is cofinal in $\cC$;
  \item for every $B$ in $\cB$ and $U$ in $\cC$, the thickening $U[B]$ is also in $\cB$.
 \end{enumerate}
 The members of $\cB$ are called bounded, and members of $\cC$ are called coarse entourages.
 A morphism $(X,\cB,\cC) \to (X',\cB',\cC')$ of $G$-bornological coarse spaces is an equivariant map $f$ between the underlying sets which is proper and controlled, meaning that $f^{-1}(\cB') \subseteq \cB$ and $(f \times f)(\cC) \subseteq \cC'$.
 
 Two subsets $Y$, $Z$ of a $G$-bornological coarse space are called coarsely disjoint if for all $y$ in $Y$ and $z$ in $Z$ the set $\{(y,z)\}$ is not a coarse entourage.
 
 Let $E \colon G\BC \to \bM$ be a functor to a semi-additive $\infty$-category.
 \begin{ddd}\label{wetiogwergrgwr}
 The functor $E$ is called $\pi_0$-excisive if for every partition $(Y,Z)$ of an object  $X$ in $G\BC$ into coarsely disjoint invariant subsets the canonical map $E(Y) \oplus E(Z) \to E(X)$ is an equivalence.
 \end{ddd} 
  If $E$ is $\pi_0$-excisive and  $(Y,Z)$ is a partition of $X$ into  coarsely disjoint invariant subsets, then   we have canonical  projections $E(X)\to E(Y)$ and $E(X)\to E(Z)$.
 
 In the following construction we introduce the construction and the notation for  analogous projections associated to a countable partition of $X$.
 
  \begin{construction}\label{regiouweroigegergewgreg}
	Assume that $E$ is a $\pi_{0}$-excisive functor from $G\BC$ to a semi-additive $\infty$-category.
	Let $Y$ be a $G$-bornological coarse space, and let $(Y_{n})_{n\in \nat}$ be a collection of $G$-invariant, pairwise coarsely disjoint subsets of $Y$ such that $Y = \bigcup_n Y_n$.  	
	
	For every $k$ in $\nat$ we consider the partition $(Y_{k},\bigcup_{n \neq k} Y_{n})$ of $Y$.
	We then let $q^{E}_{k} \colon E(Y)\to E(Y_{k})$ denote the corresponding projection map.
 \end{construction}

 While the assembly map $\As_{\cF,F}$  in \eqref{fevuihiuqhvvsad} 
   depends on a functor $F\colon G\Orb\to \bM$, our argument that it is a phantom equivalence  requires us to extend this functor
to a functor $E\colon G\BC\to \bM$. This will be made precise in \cref{wtgklpwergrewfwref} below.
Via this extension, the techniques of coarse homotopy theory become applicable to study the assembly map.

 If $S$ is a $G$-set, the following two bornological coarse structures are primarily of interest: $S_{min,min}$ denotes the $G$-bornological space given by $S$ equipped with the minimal coarse structure {(which contains only subsets of the diagonal)} and the minimal bornology  {(which consists precisely of the finite subsets of $S$)}, while $S_{min,max}$ denotes $S$ equipped with the minimal coarse structure and the maximal bornology (which consists of all subsets of $S$).  
 The second case leads to a fully faithful functor
   \begin{equation}\label{eq:gorb-embed}
   i \colon G\Orb\to G\BC\ , \quad S\mapsto S_{min,max}\ .
  \end{equation}
  
\begin{rem}
 For the reader familiar with controlled algebra, the meaning of the bornological coarse structures $S_{min,min}$ and $S_{min,max}$ may become clearer by thinking about the associated categories of controlled objects. Controlled objects over $S_{min,max}$ correspond to the $S$-indexed direct sum of the coefficient category, while controlled objects over $S_{min,min}$ correspond to the $S$-indexed direct product \cite[Rem.~10.8]{be:coarseassembly}.
\end{rem}  
  
  Let $F \colon G\Orb\to \bM$ and $E \colon G\BC \to \bM$ be functors.
    \begin{ddd}\label{wtgklpwergrewfwref}
  	$E$ extends $F$ if there exists an equivalence $F\simeq E\circ i$.
  \end{ddd}

Our proof of \cref{thm:main} will use the idea that if  a functor $E$ extends $F$ in the sense of \cref{wtgklpwergrewfwref}, then  $F$ can be  twisted   by arbitrary $G$-bornological coarse spaces.  We can also  twist $
E$ by arbitrary objects from $\PSh(G\Orb)$. 
In the following we explain the details.  

 The category $G\BC$ has a symmetric monoidal structure  \cite[Ex.~2.17]{equicoarse} which will be denoted by $-\otimes-$.
 If $X,Y$ are in $G\BC$, then
 the underlying $G$-coarse space of $X\otimes Y$ is the cartesian product of the underlying $G$-coarse spaces of $X$ and $Y$, but the bornology  of $X\otimes Y$ differs from the cartesian one and is generated by the subsets $A\times B$ for all bounded subsets $A$ of $X$ and $B$ of $Y$.

 Let $Y$ be in $G\BC$
 and let $E \colon G\BC\to \bM$ be any functor.
 
 \begin{ddd}\label{def:twist}
 	The functor
 	\[ E_Y := E(Y \otimes -) \colon G\BC \to \bM \]
 	is called the twist of $E$ by $Y$.
 \end{ddd}
 
\begin{construction}
	\label{constr:exten}
	Combining the  embedding $i$ from \eqref{eq:gorb-embed}
	with the symmetric monoidal structure, we define the functor
 \begin{equation}\label{fdvwdfvdfsv}
 \wt i \colon G\Orb \times G\BC \to G\BC\ , \quad (S,X) \mapsto i(S) \otimes X = S_{min,max} \otimes X\ .
\end{equation} 
 For any functor $E \colon G\BC \to \bM$, define \begin{equation}\label{ervwevkjnkjwvevwevw}
 \exten{E} \colon \PSh(G\Orb) \times G\BC \to \bM
\end{equation}
   as the essentially unique functor which is colimit-preserving in the first argument 
   and fits into the commutative diagram
   \[\xymatrix{
    G\Orb \times G\BC\ar[r]^-{\wt i}\ar[d]_-{\yo \times \id_{G\BC}} & G\BC\ar[r]^-{E} & \bM \\
    \PSh(G\Orb) \times G\BC\ar[urr]_-{\exten{E}} & & }\ .\]
  Note that there is an equivalence $\exten{E}(\yo(S),-) \simeq E_{S_{min,max}}(-)$ for any transitive $G$-set $S$.  
 
  This extension process is furthermore compatible with adding twists in the sense that $\exten{E}(A,-)_{Y} \simeq \exten{E}_{Y}(A,-)$ for every $Y$ in $G\BC$ and $A$ in $\PSh(G\Orb)$.  
\end{construction} 
 
 \begin{rem}\label{eroiguowergfwefwefre91}
 Let $F \colon G\Orb\to \bM$ be a functor and assume that $E \colon {G\BC} \to \bM$ extends $F$ in the sense of \cref{wtgklpwergrewfwref}.
For $X$ in $G\BC$  the functor
 \[ \exten{E}(-,X) \colon \PSh(G\Orb) \to \bM \] 
 preserves colimits and therefore corresponds, via Elmendorf's theorem, to an equivariant homology theory on  $G\Top$.
 We consider $\exten{E}(-,X)$ as the twist by $X$ of the equivariant homology theory determined by $F$. The latter is recovered by inserting 
 $X=*$.
  \end{rem}

 \begin{rem}\label{eroiguowergfwefwefre9}
 In this remark we assume that $E\colon G\BC\to \bM$ is an equivariant coarse homology theory.
  Then for $Y$ in $G\BC$  the twist $E_{Y}\colon G\BC\to \bM$ by $Y$ is again an equivariant coarse homology theory \cite[Lem.~3.16]{injectivity}. 
 
 For $A$ in $\PSh(G\Orb)$ we have an equivalence
 \[ A \simeq \mathop{\colim}\limits_{(\yo(S)\to A) \in {\yo_{/A}}}
  \yo(S)  \]
  and therefore an equivalence 
  \[ \exten{E}(A,-) \simeq \mathop{\colim}\limits_{(\yo(S)\to A)\in {\yo_{/A}}}
  E_{S_{min,max}}(-)\ \ .\]
  Since the axioms of an equivariant coarse homology theory are compatible with forming colimits,
 the functor $\exten{E}(A,-)$ is again an equivariant coarse homology theory. 
 We consider $\exten{E}(A,-)$ as a twist  by $A$ of the equivariant coarse homology theory $E$. The latter is recovered by inserting $A=*$.
  \end{rem}

\subsection{\texorpdfstring{$(E,\cF)$-proper objects}{(E,F)-proper objects}}\label{rfoiqrejgoireggregwrgregwregwergregr}

The goal of this subsection is the formulation of the notion of $(E,\cF)$-properness
in \cref{ergioerwgwergwergregregwergwregwergwre}. This requires some preparations which we start to explain now.
 
 We need to extend \cref{constr:exten} {in a way that allows us to twist}
$\nat$-indexed families of presheaves on $G\Set$ in the first argument.
For the second argument, the side of $G$-bornological spaces, we insert $\nat$-indexed  families $(X_{n})_{n\in \nat}$ of bounded $G$-bornological coarse spaces.
  Such families can naturally be considered as objects  $p \colon X \to \nat_{min,min}$ in  the slice category
 $G\BC_{/\nat_{min,min}}$  of $G$-bornological coarse spaces over $\nat_{min,min}$.
By abuse of notation, we will denote objects $p \colon X \to \nat_{min,min}$ in  $G\BC_{/\nat_{min,min}}$ by their domain $X$.
 In the following we use the abbreviation $(-)_{n}$ for $(-)_{n\in \nat}$ to denote families indexed by the set $\nat$.

In the first step, we extend the functor in \eqref{fdvwdfvdfsv}.
 \begin{ddd}\label{def:fibrewise-tensor}\ 
We define the functor 
 \[ - \otimes_\nat - \colon \prod_{\nat} G\Set \times G\BC_{/\nat_{min,min}}  \to G\BC \] such that it sends a family $(T_n)_n$ in $\prod_\nat G\Set$ and an object $X$ in $G\BC_{/\nat_{min,min}}$ to the object $(T_n)_n \otimes_\nat X$ in $G\BC$  given as follows:  
 \begin{enumerate}
  \item its underlying set is $\coprod_n (T_n \times X_n)$;
  \item its bornology is generated by sets of the form $\coprod_{n \leq N}  (T_n \times X_n)$ for $N$ in $\nat$;
  \item its coarse structure is generated by entourages of the form
  \[ \coprod_n (\diag(T_n) \times (U\cap (X_{n}\times X_{n}))) \]
  for all entourages $U$ of $X$.
  \end{enumerate}
  The definition of  the functor $-\otimes_\nat-$ on morphisms is the obvious one.
\end{ddd}

\begin{rem}\label{rem:tensor-const}
 {Let $S$ be in $G\Set$ and $X$ be in $G\BC_{/\nat_{min,min}}$.
 Consider the constant family $(S)_{n}$ as an object in $\prod_\nat G\Set$.
 Unwinding definitions, one checks that there is a natural isomorphism
 \[ (S)_{n \in \nat} \otimes_\nat X \cong S_{min,max} \otimes X\ .\]
 of $G$-bornological coarse spaces, where the right hand side uses only the underlying $G$-bornological coarse space of $X$ and forgets the reference map to $\nat_{min,min}$.}
\end{rem}

In the second step, we construct the analogue of the functor in \eqref{ervwevkjnkjwvevwevw} for families. To this end we
consider the functor \begin{equation}\label{fwqefqwefewqdwed}
\ell \colon G\Set \xrightarrow {\yo}  \PSh(G\Set) \xrightarrow{\Res} \PSh(G\Orb)
\end{equation}
and form
\begin{equation}\label{svasdvdscdscadcacac}
\extenp{\ell} := (\prod_\nat \ell) \times \id \colon (\prod_\nat G\Set) \times G\BC_{/\nat_{min,min}} \to (\prod_\nat \PSh(G\Orb)) \times G\BC_{/\nat_{min,min}}\ .
\end{equation}   
Let $E \colon G\BC \to \bM$ be a functor to a cocomplete $\infty$-category.
\begin{ddd}\label{def:extenp}
 We define the functor
 \[ \extenp{E} \colon (\prod_{\nat} \PSh(G\Orb)) \times G\BC_{/\nat_{min,min}} \to \bM \] as the left Kan extension as indicated in the following diagram:
 \[\xymatrix{
  (\prod_\nat G\Set) \times G\BC_{/\nat_{min,min}} \ar[r]^-{\otimes_\nat}\ar[d]_-{\extenp{\ell}} & G\BC \ar[r]^-{E} & \bM \\
  (\prod_\nat \PSh(G\Orb)) \times G\BC_{/\nat_{min,min}}\ar[rru]_-{\extenp{E}}
  }\]
\end{ddd}

Since  $\ell$, and therefore also $\extenp{\ell}$, are fully faithful, the structural transformation of the Kan extension provides an equivalence
of functors $\extenp{E} \circ \extenp{\ell}\simeq E \circ \otimes_\nat$.

In the following we relate the functor $\extenp{E}$  defined above to the functor $\exten{E}$  in \eqref{ervwevkjnkjwvevwevw}. 
The precise statement will be given in \cref{lem:projections} below.
The construction of $\extenp{E}$, which is analogous to that of $\exten{E}$, involves terms like $E(W_{min,max}\otimes Y)$ for $G$-sets $W$. For bounded $Y$ such terms can be calculated using the values $E(S_{min,max}\otimes Y)$ for $G$-orbits $S$ if $E$ is hyperexcisive. We explain this notion next.

Let $E \colon G\BC \to \bM$ be a functor to a cocomplete $\infty$-category. 
\begin{ddd}[{\cite[Def.~4.16]{unik}}]\label{qiojerfgfrqefwf}
	The functor $E$ is hyperexcisive if for every $G$-set $W$ and every bounded $G$-bornological coarse space $Y$ the canonical morphism
	\[ \mathop{\colim}\limits_{(S \to W) \in G\Orb_{/W}} E_{S_{min,max}}(Y) \to E_{W_{min,max}}(Y) \]
	is an equivalence.
\end{ddd}

Hyperexcisiveness generalises $\pi_{0}$-excisiveness to certain infinite partitions.
Let $G\BC_\bd$ be the full subcategory of $G\BC$ of bounded $G$-bornological coarse spaces.
Restricting $\exten{E}$ from \eqref{ervwevkjnkjwvevwevw} in the second argument to $G\BC_\bd$, we obtain a functor $\exten{E}_\bd \colon \PSh(G\Orb) \times G\BC_\bd \to \bM$.
In the following lemma we provide another
characterisation of hyperexcisiveness of $E$ in terms of the functor $\exten{E}_\bd$.

\begin{lem}\label{lem:hyperexc}
 The functor $E$ is hyperexcisive if and only if $\exten{E}_\bd$ is a left Kan extension of
 \[ G\Set \times G\BC_\bd \xrightarrow{(-)_{min,max} \otimes -} G\BC \xrightarrow{E} \bM \]
 along $\ell \times \id \colon G\Set \times G\BC_\bd \to \PSh(G\Orb) \times G\BC_\bd$.
\end{lem}
\begin{proof}
 Consider the diagram
 \[\xymatrix@C=6em{
  G\Orb \times G\BC_\bd\ar@/^15pt/[dr]^-*+{\labelstyle E \circ ((-)_{min,max} \otimes -)}\ar[d]\ar@/_2.5cm/[dd]_-{\yo \times \id} & \\
  G\Set \times G\BC_\bd\ar[r]^-{E \circ ((-)_{min,max} \otimes -)}\ar[d]_-{\ell \times \id} & \bM \\
  \PSh(G\Orb) \times G\BC_\bd\ar@/_15pt/[ur]_-{E'} & 
 }\]
 where the lower right triangle is defined to exhibit $E'$ as a left Kan extension.
  By \cref{def:extenp}, $E$ is hyperexcisive if and only if the upper right triangle in the diagram exhibits the horizontal arrow as a left Kan extension.
 If $E$ is hyperexcisive, then $E'$ is an iterated Kan extension.
 Since $\exten{E}_\bd$ is defined to be the left Kan extension along $\yo \times \id$, we have $E' \simeq \exten{E}_\bd$.
 
 Conversely, assume $E' \simeq \exten{E}_\bd$. The lower triangle commutes since $\ell \times \id$ is fully faithful. From this we conclude that the upper triangle is also a left Kan extension. Hence $E$ is hyperexcisive.
\end{proof}

Given a $G$-set $T$, we can consider the presheaf $\ell(T)$ in $\PSh(G\Orb)$, using $\ell$ from  \eqref{fwqefqwefewqdwed}, and the $G$-bornological coarse space  $i(T)=T_{min,max}$ in $G\BC$, using  $i$ from  \eqref{eq:gorb-embed}.
The following corollary compares the twist of $\exten{E}_\bd$ by $\ell(T)$ as in \cref{eroiguowergfwefwefre9} with the twist by $T_{min,max}$ as in \cref{eroiguowergfwefwefre91}.

\begin{kor}\label{ergioergergegergergergerg}
	If $E$ is hyperexcisive, then there is a natural equivalence of functors
	\[ \exten{E}_{\bd,T_{min,max}}(-,-) \simeq \exten{E}_\bd(-\times \ell(T),-) \colon \PSh(G\Orb) \times G\BC_\bd \to \bM\ .\]
\end{kor}
\begin{proof}
	Since $\ell \times \id$ is fully faithful, \cref{lem:hyperexc} gives the marked equivalences in the following {chain}:
	\begin{align}\label{eq:juggling}
	 	\nonumber  \exten{E}_{\bd,T_{min,max}}(\ell(-),-)& \overset{!}{\simeq} E((-)_{min,max} \otimes T_{min,max} \otimes -) \simeq E((- \times T)_{min,max} \otimes -) \\
	 	&\overset{!}{\simeq} \exten{E}_\bd(\ell(- \times T), -) \simeq \exten{E}_\bd(\ell(-) \times \ell(T), -)
	 \end{align}
	 The functor $\exten{E}_{\bd,T_{min,max}}$ is the left Kan extension of the first term  in \eqref{eq:juggling} along $\ell \times \id$. 
	 In order to understand the corresponding left Kan extension of the last term in 	 \eqref{eq:juggling}, note that
 colimits in $\PSh(G\Orb)$ are universal. Therefore, $- \times \ell(T)$ is a colimit-preserving endofunctor of $\PSh(G\Orb)$. We conclude 
 that the left Kan extension of the last term in \eqref{eq:juggling} is equivalent to
$\exten{E}_\bd(- \times \ell(T),-)$. 
 So the desired equivalence is obtained by Kan extending equivalence~\eqref{eq:juggling}.
\end{proof}

After this discussion of hyperexcisiveness, we continue with the comparison of the functors   $\extenp{E}$  and $\exten{E}$.  In order to state \cref{lem:projections} below, we introduce more notation. Let $P$ be a subset of $\nat$.
\begin{ddd}\label{def:oiwjgoirfgrefwerf}
 We define the following functor associated to $P$:
 \begin{enumerate}
 \item $\pr_n \colon \prod_\nat \PSh(G\Orb) \to \PSh(G\Orb)$ denotes the projection onto the $n$-th component;
 \item $f_n \colon G\BC_{/\nat_{min,min}} \to G\BC_\bd$ denotes the functor that sends $X$ to $X_n$.
  \item \label{oiwjgoirfgrefwerf} $c_P$ denotes the endofunctor of $G\BC_{/\nat_{min,min}}$ that removes the preimage of $P$, thus sending $X$ to $X \setminus \bigcup_{n \in P} X_n$;
  \item  $u_P$ denotes the endofunctor $\id \times c_P$ {of $(\prod_\nat \PSh(G\Orb)) \times G\BC_{/\nat_{min,min}}$}.\qedhere
  \end{enumerate}
\end{ddd}

Let $E \colon G\BC \to \bM$ be a functor to a semi-additive and cocomplete $\infty$-category.

\begin{lem}\label{lem:projections}
  If $E$ is $\pi_0$-excisive and hyperexcisive, there exists a natural equivalence
  \[ \bigoplus_{n \in P} \exten{E} \circ (\pr_n \times f_n)\:\: \oplus \:\:  \extenp{E} \circ u_P \:\:\:\simeq\:\: \:\extenp{E} \]
  of functors $(\prod_\nat \PSh(G\Orb)) \times G\BC_{/\nat_{min,min}} \to   \bM$.  
  In particular, there are inclusion and projection maps
  \begin{equation}\label{eq:projections0}
   \exten{E} \circ (\pr_n \times f_n) \to \extenp{E} \quad\text{and}\quad \extenp{E} \to \exten{E} \circ (\pr_n \times f_n)
   \end{equation}
  for every $n$ in $\nat$.
\end{lem}
\begin{proof}
 The canonical inclusions induce for every sequence of $G$-sets $(T_n)_n$ a natural map
 \begin{equation}\label{eq:projections}
  \bigoplus_{n \in P} E(T_{n,min,max} \otimes X_n) \:\:\oplus \:\:E((T_n)_n \otimes_\nat c_P(X))  \:\:\: \to \:\:\: E((T_n)_n \otimes_\nat X)\ .
  \end{equation}
 This map is an equivalence since $E$ is $\pi_0$-excisive.
 As $E$ is hyperexcisive, \cref{ergioergergegergergergerg} gives a natural equivalence 
 \[ E(T_{n,min,max} \otimes X_n) \:\: \simeq  \:\:\exten{E}_{\bd,T_{n,min,max}}(*,X_n)  \:\:\simeq  \:\:\exten{E}(\ell(T_{n}), X_n)\ \]
 for every $n$ in $\nat$. 
 This allows us to identify \eqref{eq:projections} with a natural equivalence
 \[ \big( \bigoplus_{n \in P} (\exten{E} \circ (\pr_n \times f_n)) \:\: \oplus \:\: \extenp{E} \circ u_P \big) \circ \extenp{\ell}\:\: \:  \simeq\:\:\:  \extenp{E} \circ \extenp{\ell} \] 
 of functors $(\prod_\nat G\Set) \times G\BC_{/\nat_{min,min}}  \to \bM$. 
 Taking left Kan extensions along $\extenp{\ell}$ yields a natural equivalence of functors on $(\prod_\nat \PSh(G\Orb)) \times G\BC_{/\nat_{min,min}}$. Note that the left Kan extension of the right hand side is by definition $\extenp{E}$.
 
 We need to identify the left Kan extension on the left hand side.
 The functor
 \[ (\pr_n \times f_n)_{n \in P} \colon (\prod_\nat G\Set) \times G\BC_{/\nat_{min,min}} \to \prod_{n \in P} \big(G\Set \times G\BC_\bd\big) \]
 has a left adjoint $i_P$ which sends $(A_n,Y_n)_{n \in P}$ to
 \[ \Big( \big( \begin{cases} A_n & n \in P \\ \emptyset & n \notin P \end{cases} \big)_{n}, \coprod_{n \in P} Y_n \Big) \ .\]
 Since the restriction of a functor along a right adjoint is the left Kan extension along the left adjoint, the diagram
 \begin{equation}\label{eq:projections2}\xymatrix@C=6em{
  \prod_{n \in P} \big(G\Set \times G\BC_\bd\big)\ar[r]^-{\bigoplus_{n \in P} (\exten{E}_\bd \circ (\ell \times \id))}\ar[d]_-{i_P} & \bM \\
  (\prod_\nat G\Set) \times G\BC_{/\nat_{min,min}}\ar[ur]_-*+++{\labelstyle \bigoplus_{n \in P} (\exten{E} \circ (\ell \times \id) \circ (\pr_n \times f_n))}
 }\end{equation}
 exhibits $\bigoplus_{n \in P} (\exten{E} \circ (\ell \times \id) \circ (\pr_n \times f_n))$ as the left Kan extension of $\bigoplus_{n \in P} (\exten{E}_\bd \circ (\ell \times \id))$ along $i_P$.
 Moreover, $i_P$ fits into the commutative diagram
 \[\xymatrix@C=4em{
 \prod_{n \in P} \big(G\Set \times G\BC_\bd \big)\ar[d]_-{\extenp{\ell}_P}\ar[r]^-{i_P} & (\prod_\nat G\Set) \times G\BC_{/\nat_{min,min}}\ar[d]^-{\extenp{\ell}} \\
 \prod_{n \in P}  \big(\PSh(G\Orb) \times G\BC_\bd \big)\ar[r]^-{\widetilde i_P} & (\prod_\nat \PSh(G\Orb)) \times G\BC_{/\nat_{min,min}}
 }\]
 where $\widetilde i_P$ is left adjoint to $(\pr_n \times f_n)_{n \in P}$.
 Consequently, the left Kan extensions of $\bigoplus_{n \in P} (\exten{E}_\bd \circ (\ell \times \id))$ along $\extenp{\ell} \circ i_P$ and along $\widetilde i_P \circ \extenp{\ell}_P$ coincide.
 We now determine these two Kan extensions.
 
 The first iterated Kan extension is, in view of the Kan extension presented by \eqref{eq:projections2}, the left Kan extension of $\bigoplus_{n \in P} (\exten{E} \circ (\ell \times \id) \circ (\pr_n \times f_n))$ along $\extenp{\ell}$.
 
 For the second, note that the left Kan extension of $\bigoplus_{n \in P} (\exten{E}_\bd \circ (\ell \times \id))$ along $\extenp{\ell}_P$ is $\bigoplus_{n \in P} \exten{E}_\bd$. So the second iterated Kan extension is equivalent to $\bigoplus_{n \in P} (\exten{E} \circ (\pr_n \times f_n))$.
 
 A completely analogous argument applies to $\extenp{E} \circ u_P$ since $u_P$, considered as a functor to its essential image, has a left adjoint given by the inclusion of the essential image. 
\end{proof}

 We are finally ready to define
 the notion of $(E,\cF)$-properness. Denote by
\[ \diag \colon \PSh(G\Orb) \to \prod_\nat \PSh(G\Orb) \]
the diagonal functor. 
Let $E \colon G\BC \to \bM$ be a $\pi_{0}$-excisive and hyperexcisive functor with values in a cocomplete stable $\infty$-category.  
 Finally, fix    an object $X$ in $G\BC_{/\nat_{min,min}}$.
Then the inclusions in \eqref{eq:projections0} induce  a natural transformation
\[ \bigoplus_{n \in \nat} \exten{E}(-,X_n) \to \extenp{E}(\diag(-), X) ) \colon \PSh(G\Orb)\to \bM\ .\]
We are interested in its cofibre.
 \begin{ddd}\label{def:extenps}
We define the functor
 \[ \extenps{E}(-,X) := \cofib \Big(  \bigoplus_{n \in \nat} \exten{E}(-,X_n) \to \extenp{E}(\diag(-), X) \Big)\colon  \PSh(G\Orb)\to \bM\ .\qedhere\] 
\end{ddd}
Let $\cF$ be a family of subgroups of $G$ and recall \cref{wrtoihgjwtgergwerfew} of the classifying space $E_{\cF}G$.

\begin{ddd}\label{ergioerwgwergwergregregwergwregwergwre}
	The object $X$ in $G\BC_{/\nat_{min,min}}$ is called $(E,\cF)$-proper if the map
	\[ \extenps{E}(E_\cF G,X) \to \extenps{E}(*,X) \]
	induced by the projection $E_\cF G \to *$ is an equivalence.
\end{ddd}

\subsection{Phantom equivalences from transfer classes}\label{sec:phantom-transfer}

The  crucial part of our proof of \cref{thm:main} consists of the construction of a transfer class. In this subsection we introduce this  notion in an axiomatic way. In 
\cref{ergiooegergergwergergergw} we show that the existence of such a transfer class implies that the assembly map is a phantom equivalence.
Thus, \cref{ergiooegergergwergergergw} will eventually yield
the final step in the proof of \cref{thm:main}.

We start with explaining the symmetric monoidal structure on the $\infty$-category of  idempotent complete left-exact $\infty$-categories. Recall  that  a left-exact $\infty$-category $\bC$ is called idempotent complete if it is closed under retracts in its pro-completion $\Pro_{\omega}(\bC)$, or equivalently, if the canonical functor
$\bC\to \Pro_{\omega}(\bC)$ induces  an equivalence of $\bC$ with the  full subcategory $\Pro_{\omega}(\bC)^{\omega}$ of cocompact objects in $\Pro_{\omega}(\bC)$.
We let $\Clep$ be the full subcategory of $\Cle$ of idempotent complete left-exact $\infty$-categories.

We let $\Spc^{\op,\omega}_*$ in $\Clep$ denote the $\infty$-category of cocompact objects in the opposite of the category of pointed spaces. Equivalently, $\Spc^{\op,\omega}_* =(\Spc_{*}^{\cp})^{\op}$. The $\infty$-category  $\Spc^{\op,\omega}_*$ has the universal property
that the evaluation at the object $S^{0}$ of $\Spc^{\op,\omega}_*$  induces  an equivalence $\Fun^{\lex}(\Spc^{\op,\omega}_*,\bC)\to \bC$. In other words, specifying a left-exact  functor 
$\Spc^{\op,\omega}_*\to \bC$ is equivalent to specifying an object of $\bC$.  

The $\infty$-category $\Clep$ has a symmetric monoidal structure $\otimes$ with tensor unit $\Spc^{\op,\omega}_*$. 
 {This means that} for $\bC$ and $\bD$ in $\Clep$, the tensor product $\bC\otimes\bD$  is characterised by the following universal property:
 there is a functor $\bC\times\bD \to \bC \otimes\bD$ that is initial among all  functors from $\bC\times \bD$ to some idempotent complete left-exact $\infty$-category  which preserve finite limits in each variable separately. The  full details will be explained in \cref{sec:prelims}.

On the next categorial level, the functor
\[ - \otimes - \colon \Clep \times \Clep \to \Clep \]
preserves colimits in both variables separately.

The definition of a transfer class depends on a choice of pentuple
\begin{equation}\label{qwefoiwhfoiqwfqewfqewfqf}
 (U,\eta,V,H,\cF)
\end{equation} 
whose members we will describe in the following.
The first two {components} are a functor
\begin{equation}\label{eliwjbioebwebrebewbe1}
 U \colon G\BC\to \Clep
\end{equation}
and a {left-exact functor}
\begin{equation}\label{eliwjbioebwebrebewbe}
 \eta \colon \Spc^{\op,\omega}_{*} \to U(*)\ .
\end{equation}
By the universal property of $\Spc^{\op,\omega}_*$ explained above, specifying $\eta$ is equivalent to specifying an object in $U(*)$.
\begin{rem} \label{fbiowbervervewbv}
	Recall from \cref{sec:borncoarse} that $G\BC$ carries a symmetric monoidal structure which we also denote by $\otimes$.
	If $U$ has a lax monoidal structure, then $U$ preserves algebra objects.
	Since $*$ is an algebra in $G\BC$, we get an algebra $U(*)$ in $\Clep$. 
	In this case the unit of this algebra provides a canonical morphism $\Spc^{\op,\omega}_{*} \to U(*)$ which will usually be our candidate for $\eta$.
\end{rem}

The third entry in the list \eqref{qwefoiwhfoiqwfqewfqewfqf}   is another  functor  \begin{equation}\label{ejkhvnijevfvdfvsfdv}
V \colon G\BC \to \Clep\ .
\end{equation}

\begin{ddd}\label{giooergrefwerfwrevwerfv}
	A weak $U$-module structure on $V$ is given by the following data:
	\begin{enumerate}
		\item a natural transformation
		\[ \mu \colon U(-) \otimes  V(-)\to V(-\otimes-)\ ; \]
		\item\label{it:giooergrefwerfwrevwerfv} a commutative diagram
		\begin{equation}\label{trfsdbqewzfgzkwuqefwear}
		\xymatrix{
			\Spc^{\op,\omega}_* \otimes V(-)\ar[r]^-{\eta \otimes \id}\ar[d]_-{\simeq}&U(*)\otimes  V(-)\ar[d]^-{\mu} \\
			V(-)\ar[r]^-{\simeq}&V(*\otimes -)}
		\end{equation}
		of functors $G\BC \to \Clep$. \qedhere
	\end{enumerate}
\end{ddd}

\begin{rem}
	We use the word {\em weak} since \cref{giooergrefwerfwrevwerfv} requires only the minimal amount of structure to make the proof of \cref{ergiooegergergwergergergw} below work.
	In the situations we will actually consider (see \cref{sec:orbits-and-fixed-points}), the functor $U$ is a lax symmetric monoidal functor, and $V$ is a module functor {over $U$}.
\end{rem}

\begin{rem}
 In our applications below, we will only make use of weak modules which are canonically derived from a given left-exact $\infty$-category with $G$-action $\bC$.
 To provide some intuition, recall that, for any ring $R$, the tensor product over $\IZ$ induces an action of the symmetric monoidal category of finitely generated free $\IZ$-modules with $G$-action on the category of finitely generated projective $R[G]$-modules.
 Generalising to the left-exact setting, the idempotent completion  of $\colim_{BG} \bC$ canonically refines to a module over $\Fun(BG,\Spc^{\op,\omega}_*)$, and the examples for $U$ and $V$ in \cref{sec:famenablegroups,qrgqiorgegegergwegre,sec:dfhj} generalise this observation further to provide an action of ``controlled Swan theory'' on controlled K-theory. See \cref{sec:orbits-and-fixed-points} for the explicit construction.
\end{rem}

The {component} $H$ in the list \eqref{qwefoiwhfoiqwfqewfqewfqf} is a  lax monoidal functor $H \colon \Clep \to \bM$, where the target 
 $\bM$ is assumed to be a monoidal, semi-additive and cocomplete $\infty$-category which admits countable products.
  Given $\Homol$, the morphism $\eta$ in \eqref{qwefoiwhfoiqwfqewfqewfqf} induces a morphism
\begin{equation}\label{eq:etaH}
 \eta_H \colon \beins_\bM \to H(\Spc^{\op,\omega}_*) \xrightarrow{H(\eta)} HU(*)\ .
\end{equation}
The final entry $\cF$ of the list \eqref{qwefoiwhfoiqwfqewfqewfqf} is a family of subgroups of $G$.

We now fix a list {as in} \eqref{qwefoiwhfoiqwfqewfqewfqf}.
Furthermore, we assume that $U$ and $HV$ are $\pi_0$-excisive (see \cref{wetiogwergrgwr}), and that the functor $HV$ is  hyperexcisive (see \cref{qiojerfgfrqefwf}).
Recall the morphisms $q^U_n \colon U(\nat_{min,min}) \to U(\{n\}) \simeq U(*)$ {introduced in} \cref{regiouweroigegergewgreg}.
Recall \cref{ergioerwgwergwergregregwergwregwergwre} of $(HV,\cF)$-proper objects.

\begin{ddd}\label{def:transfer-class}
	A transfer class $(X,t)$ for $(U,\eta,V,H,\cF)$ consists of:
	\begin{enumerate}
		\item  \label{qwoifgoqergqfeqewfq} {an object} $X$ in $G\BC_{/\nat_{min,min}}$, {called the transfer space}, which admits a morphism to an $(HV,\cF)$-proper object;
		\item  \label{qwoifgoqergqfeqewfq1} a morphism $t \colon \beins_\bM \to HU(X)$
	\end{enumerate}
	such that there exists a commutative diagram
	\begin{equation}\label{wefiewofefwefewfewfewfwe111}
	\xymatrix@C=4em{\beins_\bM \ar[rr]^-{\eta_H}\ar[d]_-{t}&&HU(*)\ar[d]^-{\diag} \\
	HU(X)\ar[r]^-{p}& HU(\nat_{min,min}) \ar[r]^{(H(q^U_n))_n}&\prod_{n\in \nat} HU(*) } \end{equation}
	where $p \colon X\to \nat_{min,min}$ denotes the structure morphism of $X$.\qedhere
\end{ddd}

\begin{rem}\label{faklolergergwergwegwergwreg}
	Note that we only require the existence of a commutative diagram like \eqref{wefiewofefwefewfewfewfwe111}, not a preferred choice of a filler.
	In order to construct such  fillers we will encounter the situation where 
	 the morphism $t$ is obtained by specifying an object $t_0$ in $U(X)$.
	 {In more detail,} this object induces a left-exact functor
	$\widehat t_{0} \colon \Spc_{*}^{\op,\omega}\to U(X)$
	by the universal property of $\Spc_{*}^{\op,\omega}$.
	This functor in turn yields $t$ as the composition
	\[ t\colon \beins_{\bM} \to H(\Spc_{*}^{\op,\omega}) \xrightarrow{H(\widehat t_{0})} HU(X)\ .\]
	If $t$ arises in this manner {and} $\eta$ is given by the object $\eta_0$ of $U(X)$,
	a sequence of equivalences $((q^U_n \circ p)(t_0) \simeq \eta_0)_{n \in \nat}$ in $U(*)$ provides a filler for \eqref{wefiewofefwefewfewfewfwe111}.
\end{rem}

Let $\bM$ be a stably monoidal and cocomplete stable   $\infty$-category which admits countable products.
Let $H \colon \Clep \to \bM$ be a lax monoidal functor which preserves sums, and let $\cF$ be a family of subgroups of $G$. Let $\bC$ be in $\Fun(BG, \Clep)$ and recall \cref{weoirgjwegwergwerg9}  of $H\bC_G$.

\begin{prop}\label{ergiooegergergwergergergw}
	Assume that we are given:
	\begin{enumerate}
		\item a pair $(U,\eta)$ as in \eqref{eliwjbioebwebrebewbe1} and \eqref{eliwjbioebwebrebewbe} such that $U$ is $\pi_{0}$-excisive (see \cref{wetiogwergrgwr});
		\item a functor $V$ as in \eqref{ejkhvnijevfvdfvsfdv} having a weak $U$-module  structure (see \cref{giooergrefwerfwrevwerfv}) {such that}
		 \begin{enumerate}
		  \item $HV$ extends $H\bC_G$ in the sense of \cref{wtgklpwergrewfwref};
		  \item $HV$ is $\pi_0$-excisive and hyperexcisive (see \cref{qiojerfgfrqefwf});
		 \end{enumerate}
		\item a transfer class $(X,t)$ for $(U,\eta,V,H,\cF)$ (see \cref{def:transfer-class}).
	\end{enumerate}
	Then ${\As_{\cF,H\bC_G} \colon \mathop{\colim}\limits_{G_\cF\Orb} H\bC_G \to H(\mathop{\colim}\limits_{BG} \bC)}$ is a phantom equivalence.
\end{prop}
\begin{proof}
We let $HV^{+}$ denote the functor obtained from $HV$ as an instance of \eqref{ervwevkjnkjwvevwevw}.
Since $HV$ extends $H\bC_G$, we have a natural equivalence $\exten{HV}(-,*) \simeq H\bC_G(-)$ of functors $\PSh(G\Orb)\to \bM$.
By \cref{qgoiegjqefewfqfewfqewf}, the assembly map $\As_{\cF,H\bC_G}$ is equivalent to the map $\exten{HV}(E_\cF G,*) \to \exten{HV}(*,*)$ induced by the projection $E_{\cF}G\to *$. 
Let $C$ be the cofibre of the assembly map.
We want to show that $C$ is a phantom object.
By \cref{veiowefwwfeewfewf}, it is sufficient to produce a factorisation of $C \to \prodsum(C)$ over the zero object.

	 We let $\diag \colon \PSh(G\Orb) \to \prod_\nat \PSh(G\Orb)$ denote the diagonal functor.
	As a first step we will define a natural transformation
	\[ \tau \colon \exten{HV}(-,*) \to \extenp{HV}(\diag(-),X)\ ,\]
	where $\extenp{HV}$ is as in \cref{def:extenp}.
	Using the fact that the functor $\extenp{\ell}$ in \eqref{svasdvdscdscadcacac} is fully faithful {and that $\yo \simeq \ell \circ j_G$} for the first equivalence,
	and the identification $\diag(-) \otimes_\nat X \cong (-)_{min,max} \otimes X$ of functors $G\Orb\to G\BC$ {from \cref{rem:tensor-const}} for the second, we get an equivalence 	
\[ \extenp{HV}(\diag(\yo(-)),X) \simeq  HV(\diag(-) \otimes_\nat X)\simeq  HV((-)_{min,max} \otimes X)\ .\]
	By \cref{constr:exten}, the functor $\exten{HV}(-,*)$ is the left Kan extension of $HV ((-)_{min,max})$  along the Yoneda embedding $\yo \colon G\Orb\to \PSh(G\Orb)$.
	Therefore, by the universal property of the left Kan extension,
	 in order to define $\tau$ it suffices to give a natural transformation 	
	 \begin{equation}\label{svasvdcacwqcsadcadsc}
HV((-)_{min,max}) \to HV((-)_{min,max} \otimes X) \overset{\text{def}}{=}HV_X((-)_{min,max})
	\end{equation}
	of functors $G\Orb \to \bM$, where $HV_{X}$ denotes the twist of $HV$ by $X$ (see \cref{def:twist}).

	In fact, the transfer class gives rise to a natural transformation $\tau' \colon HV \to HV_{X}$ of functors $G\BC \to \bM$ that is defined using the lax monoidal structure of $H$ and the weak $U$-module structure $\mu$ on $V$ as the composition
	\begin{equation*}\label{vevwervewrvfrefwrfwerfwerf}
	 \tau' \colon HV \xrightarrow{t \otimes \id} HU(X)\otimes_{\bM}  HV \to H(U(X)\otimes V) \xrightarrow{H(\mu)} HV_{X} \ .
	\end{equation*}
	The restriction of $\tau'$ along the inclusion $ (-)_{min,max}\colon G\Orb \to G\BC$  from \eqref{eq:gorb-embed} is the desired transformation \eqref{svasvdcacwqcsadcadsc}.
	
	We consider the following diagram of functors $\PSh(G\Orb)\to \bM$, in which the top vertical transformations are induced by the canonical morphism $p \colon X \to \nat_{min,min}$ in $G\BC_{/\nat_{min,min}}$, and the bottom vertical transformations are induced by the projections \eqref{eq:projections0} from \cref{lem:projections}:
	\begin{equation}\label{eq:transfer}\xymatrix{
	 \exten{HV}(-,*)\ar[r]^-{\tau}\ar@/_15pt/[rdd]_-{\diag} & \extenp{HV}(\diag(-),X)\ar[r]\ar[d] & \extenps{HV}(-,X)\ar[d] \\
	 & \extenp{HV}(\diag(
	 -),\nat_{min,min})\ar[r]\ar[d] & \extenps{HV}(-,\nat_{min,min})\ar[d] \\
	 & \prod_\nat \exten{HV}(-,*)\ar[r] & \prodsum(\exten{HV}(-,*))
	}\end{equation}
	In the following we argue that it  commutes.
	The two squares on the right hand side obviously commute. Commutativity of the triangle on the left  can be checked after restriction  to the orbit category.
	Therefore, we must show that 
	the diagram of functors $G\Orb\to \bM$
	\[\xymatrix@C=4em{
	 HV((-)_{min,max})\ar[r]^-{\tau'}\ar@/_15pt/[rdd]_-{\diag} & HV_X((-)_{min,max})\ar[d]^-{HV_p} \\
	 & HV_{\nat_{min,min}}((-)_{min,max})\ar[d]^-{(q^V_n)_n} \\
	 & \prod_\nat HV((-)_{min,max}) \\
	 }\]
	commutes. Here $HV_p$ is induced by $p$ and the morphisms $q^V_n$ arise from $\pi_0$-excision for $HV$ by \cref{regiouweroigegergewgreg}.
	To see that this diagram commutes, consider the following diagram in which the morphisms $q_n^U$ also arise from \cref{regiouweroigegergewgreg} by $\pi_0$-excision for $U$, and in which all unlabelled arrows are induced by the lax monoidal structure of $H$:
	\begin{equation}\label{eq:transfer2}
	\xymatrix@C=4em{
	 & & HV_{\nat_{min,min}}\ar[r]^-{(q^V_n)_n} & \prod_\nat HV \\
	 HV_{X}\ar@/^15pt/[urr]^-{HV_p} & H(U(X) \otimes V)\ar[l]_-{H(\mu)}\ar[r]^-{H(U(p) \otimes \id)} & H(U(\nat_{min,min}) \otimes V)\ar[u]_-{H(\mu)}\ar[r]^-{(H(q^U_n \otimes \id))_n} & \prod_\nat H(U(*) \otimes V)\ar[u]_-{\prod_\nat H(\mu)} \\
	 HV\ar[u]^-{\tau'}\ar[r]^-{t \otimes \id} & HU(X) \otimes HV\ar[u]\ar[r]^-{HU(p) \otimes \id} & HU(\nat_{min,min}) \otimes HV\ar[u]\ar[r]^-{(H(q^U_n) \otimes \id)_n} & \prod_\nat ( HU(*) \otimes HV )\ar[u]
	}\end{equation}
	The triangle involving the curved arrow $HV_p$ commutes by naturality of $\mu$.
	The bottom left square commutes by definition of $\tau'$, and both the bottom centre square and bottom right square commute since $H$ is a lax monoidal functor.
	
	It remains to check that the top right square commutes.
	For $n$ in $\nat$, let $i_n \colon \{n\} \to \nat$ and $j_n \colon \nat \setminus \{n\} \to \nat$ denote the inclusion maps. Consider the following diagram:
	\[\xymatrix@C=5em{
	 \hspace{-5em}H(U(\nat_{min,min}) \otimes V)\ar[dd]_-{H(\mu)} & H((U(\{n\}) \oplus U(\nat_{min,min} \setminus \{n\})) \otimes V)\ar[l]^-{\sim}_-{H((U(i_n) + U(j_n)) \otimes \id)}\ar[r]^-{\pr \otimes \id} & H(U(\{n\}) \otimes V)\ar[dd]^-{H(\mu)} \\
	 & H(U(\{n\}) \otimes V) \oplus H(U(\nat_{min,min} \setminus \{n\}) \otimes V)\ar[ul]^(.6)*++{\labelstyle H(U(i_n) \otimes \id) + H(U(j_n) \otimes \id)}\ar[d]^-{H(\mu) \oplus H(\mu)}\ar[ur]_-{\pr}\ar[u]^-{\sim}
	 & \\
	 H(V_{\nat_{min,min}}) & H(V(\{n\})) \oplus H(V(\nat_{min,min} \setminus \{n\}))\ar[r]^-{\pr}\ar[l]_-{H(V(i_n)) + H(V(j_n))}^-{\sim} & H(V)
	}\]
	Note that commutativity of the upper left triangle expresses the additivity of $H(- \otimes V)$. The vertical morphism marked by $\sim$ is an equivalence by additivity of $H$. The horizontal mophisms marked by $\sim$ are equivalences by $\pi_0$-excision for $U$ and $HV$, respectively, using again that $H$ is additive to see that $HU$ is also $\pi_0$-excisive. By naturality of $\mu$, the lower part of the diagram also commutes. It follows that the large outer square commutes, which settles the commutativity of the top right square in \eqref{eq:transfer2}.
	
	Tensoring \eqref{wefiewofefwefewfewfewfwe111} from the definition of a transfer class (\cref{def:transfer-class}) with $V$ yields an equivalence between the composition of the bottom horizontal morphisms in \eqref{eq:transfer2} and the morphism
	\[ HV \xrightarrow{\diag \circ (\eta_H \otimes \id)} \prod_\nat (HU(*) \otimes HV)\ .\]
	Since $V$ is a weak $U$-module, condition \eqref{it:giooergrefwerfwrevwerfv} from \cref{giooergrefwerfwrevwerfv} implies that the entire composition
	$HV \to \prod_\nat HV$ along the bottom right corner in \eqref{eq:transfer2}
	is equivalent to the diagonal map. We finally conclude that the left part of \eqref{eq:transfer}, and therefore the entire diagram, commutes.
	
	By \cref{def:transfer-class}, $X$ admits a morphism to an $(HV,\cF)$-proper object that we will denote by $W$. 
	
	Inserting $E_{\cF}G \to *$ into \eqref{eq:transfer}, we obtain a commutative diagram 
	\[\xymatrix@C=2em@R=1.5em{
	 \exten{HV}(E_\cF G, *)\ar[r]\ar[d] & \exten{HV}(*,*)\ar[d]\ar[r] & C\ar[d] \\
	 \extenps{HV}(E_\cF G,X)\ar[r]\ar[d] & \extenps{HV}(*,X)\ar[d]\ar[r] & \cofib\big(\extenps{HV}(E_\cF G,X) \to  \extenps{HV}(*,X)\big)\ar[d] \\
	 \extenps{HV}(E_\cF G,W)\ar[r]\ar[d] & \extenps{HV}(*,W)\ar[d]\ar[r] & \cofib\big(\extenps{HV}(E_\cF G,W) \to  \extenps{HV}(*,W)\big)\ar[d] \\
	 \prodsum(\exten{HV}(E_\cF G,*))\ar[r] & \prodsum(\exten{HV}(*,*))\ar[r] & \prodsum(C)
	 }\]
	in which the composition of the right vertical maps is equivalent to the canonical map $C \to \prodsum(C)$ (recall that $C$ was the cofibre of the assembly map).
	Since $W$ is $(HV,\cF)$-proper, $\cofib\big(\extenps{HV}(E_\cF G,W) \to  \extenps{HV}(*,W)\big) \simeq 0$.
	Hence the right vertical composition $C \to \prodsum(C)$ vanishes, and $C$ is a phantom object by \cref{veiowefwwfeewfewf}.
\end{proof}

 \subsection{Examples of \texorpdfstring{$(E,\cF)$}{(E,F)}-proper objects}\label{sec:squeezing}

 In order to check that a candidate $(X,t)$  for a transfer class for $(U,\eta,V,H,\cF)$
 satisfies \cref{def:transfer-class}.\eqref{qwoifgoqergqfeqewfq} we must  find a morphism from $X$ to some $(HV,\cF)$-proper object  in $G\BC_{/\nat_{min,min}}$.
 \cref{rgiorgergegergergerg3232424}, which is the main result of this section, provides a sufficient supply of candidates for such $(HV,\cF)$-proper objects.
 It is an analogue of \cite[Thm.~7.2]{blr} in our setting.

We consider a $\pi_{0}$-excisive functor  $E \colon G\BC \to \bM$ which plays the role of $HV$ above. For certain statements in this section we will need the much stronger assumption that $E$ is an equivariant coarse homology theory \cite[Def.~3.10]{equicoarse}, \cite[Def.~3.13]{injectivity}. We  assume that $\bM$ is cocomplete.

As before, $\cF$ denotes any family of subgroups of $G$. Furthermore, 
  $\nat$ is considered as a discrete $G$-topological space with trivial $G$-action.

\begin{construction}\label{fewuifzhifweffewfwefwefwefwe} 
 	We consider a metric space $W$ with an isometric $G$-action together with a $G$-map $p \colon W \to \nat$ and set $W_{n}:=p^{-1}(\{n\})$.
 	We further assume that $W$ is equipped with a $G$-bornological coarse structure such that $p$ defines a morphism $W \to \nat_{min,min}$.
 	We call the corresponding coarse structure $\cC_{W}$ the original coarse structure on $W$ in order to distinguish it from the new coarse structure defined below. 
   We assume that the original coarse structure is compatible with the metric in the sense  that
   there exists some $r$ in $(0,\infty)$ with  $U_{r}\in \cC_{W}$, where
   \[ U_{r}:=\{ (w,w') \in W\times W  \mid d(w,w') \leq r \} \]
   is the metric entourage of width $r$.
Using the data described above, we construct
	an object in $G\BC_{/\nat_{min,min}}$ which we will denote by $W_{h}$. It has the same underlying set and bornological structure as $W$. But the new coarse structure 
	defined by \begin{equation}\label{eq:vanishing-control}
	\cC_{h}:= \left\{U\in \cC_{W} \mid \sup \big\{ d(w,w') \mid (w,w') \in U \cap (W_n \times W_n) \big\} \xrightarrow{n \to \infty} 0 \right\}
	\end{equation}
	  is in general smaller than the original coarse structure $\cC_{W}$.
	  		
	The structure map to $\nat_{min,min}$ is the original map $p$ which is still a morphism of $G$-bornological coarse spaces.
\end{construction}

\begin{rem}\label{rem:hybrid}
	In the language of \cite[Sec. 5.1]{buen} or  \cite[Sec.~9]{equicoarse}, \cref{fewuifzhifweffewfwefwefwefwe} can be phrased as follows: the metric on the original bornological coarse space $W$ induces a  compatible  uniform structure \cite[Def. 5.4]{buen}. The big family $\cW:=( p^{-1}([0,n]))_{n\in \nat}$ provides a hybrid datum $(W,\cW)$. The new coarse structure $\cC_{h}$ defined above is the hybrid structure associated to these data. This motivates the subscript $h$.
	
	In particular, the metric on $W$ may be replaced by any other metric that induces the same uniform structure on $W$ without changing the new coarse structure.
\end{rem}

In the present paper, a $G$-simplicial complex is a simplicial complex with an action of $G$ by automorphisms such that if $g$ in $G$ stabilises a point in the interior of a simplex, then it stabilises the whole simplex. If we are just given a simplicial complex with a $G$-action, then we can ensure this additional condition by taking the barycentric subdivision.

\begin{construction}\label{beforergiorgergegergergerg3232424}
Suppose that $W$ is a $G$-simplicial complex with a map of simplicial complexes $p \colon W \to \nat$. Then we define an object $W_{h}$ in $G\BC_{/\nat_{min,min}}$ as follows.

 We equip $W$ with the  spherical path metric.
 In order to compare with \cite{blr}, note that by \cref{rem:hybrid} we could also work with the $\ell^1$-metric  provided $W$ is finite-dimensional.
 We furthermore choose the original coarse structure  on $W$, in the sense of \cref{fewuifzhifweffewfwefwefwefwe}, to be the coarse structure $\cC_{\pi_{0}(W)}$  generated by the entourage
\[ U_{\pi_{0}(W)}:=\bigcup_{Z\in \pi_{0}(W)}Z\times Z\ . \] 
It is the maximal coarse structure on $W$ with the property that the connected components of $W$ are coarsely disjoint.
The original coarse structure obviously contains the metric coarse structure associated to the spherical path metric and is therefore compatible with the metric.   
Finally, we equip $W$ with the minimal bornology such that $p \colon W \to \nat_{min,min}$ is proper.

Then  let $W_{h}$ in  $G\BC_{/\nat_{min,min}}$ be obtained by applying \cref{fewuifzhifweffewfwefwefwefwe} to $W$ with the structures defined above.
\end{construction}

Let $W$ be a $G$-simplicial complex  and let $W_{h}$ in $G\BC_{/\nat_{min,min}}$ be obtained from $W$  by applying \cref{beforergiorgergegergergerg3232424}. 
 Recall \cref{ergioerwgwergwergregregwergwregwergwre} of an $(E,\cF)$-proper object, and \cref{qiojerfgfrqefwf} of hyperexcisiveness. 
\begin{theorem}\label{rgiorgergegergergerg3232424}
	Assume:
	\begin{enumerate}
		\item the simplicial complex $W$ is finite-dimensional and its stabilisers belong to $\cF$;
		\item\label{it:rgiorgergegergergerg3232424} $E$ is a hyperexcisive equivariant coarse homology theory.
	\end{enumerate}
	Then $W_{h}$ is $(E,\cF)$-proper.
\end{theorem}
The general outline of the proof of \cref{rgiorgergegergergerg3232424} is the same as in \cite[Sec.~7]{blr}.
We will argue by induction on the dimension of the simplicial complex $W$. The case of $0$-dimensional complexes will be settled in \cref{cor:EFproper-0}, where we  use  the assumption that $E$ is hyperexcisive and the assumption on the stabilisers of $W$.
For the induction step, we decompose a $d$-dimensional simplicial complex into a thickened
version of its $(d-1)$-skeleton and a complex consisting of a 
disjoint union of $d$-dimensional simplices.
For this step, we need that $\extenps{E}$ has appropriate excision properties.
The latter will be deduced from  the fact that $E$ is an equivariant  
coarse homology theory, see \cref{prop:extenp-homology}. 
For the induction step we must further deform the thickened $(d-1)$-skeleton to the actual $(d-1)$
skeleton, and the disjoint union of simplices to the set of their
barycentres.
For this step, we need a homotopy invariance property of $\extenps{E}$
which will be shown in \cref{prop:coarse-retract}. 
The main argument for  \cref{rgiorgergegergergerg3232424} starts 
\hyperref[soijoevevfsdvdvdfvsdv]{close to the end of this section}.

We start with showing  that $\extenp{E}$ and $\extenps{E}$ (see \cref{def:extenp} and  \cref{def:extenps}) are equivariant coarse homology theories on $G\BC/_{\nat_{min,min}}$ as functors in their second arguments. 
Recall that a functor $E \colon G\BC \to \bM$ is an equivariant coarse homology theory if it is coarsely invariant, coarsely excisive and $u$-continuous and annihilates flasques (see \cite[Def.~3.13]{injectivity}).
As we want to consider functors on the slice category $G\BC_{/\nat_{min,min}}$, we must explain what we mean by an equivariant coarse homology in this context.

We first  recall some basic definitions from coarse geometry.
We consider $X$ in $G\BC$ and denote its coarse structure by $\cC_{X}$.
An equivariant big family $\cY$ on $X$ is a filtered family $(Y_{i})_{i\in I}$ of invariant subsets of $X$ such that for every $i$ in $I$ and $U$ in $\cC_{X}$ there exists $j$ in $I$ such that {the thickening $U[Y_{i}]$ is contained in $Y_j$}
{(see \eqref{qwefoiheiuohfiqwefewfewfqeefedq})}. 
If $Z$ is  a subspace  of $X$, then we define the big family $Z\cap \cY:=(Z\cap  Y_{i})_{i\in I}$ on $Z$. If $F$ is any functor defined on $G\BC$, then we set
\[ F(\cY) := \mathop{\colim}\limits_{i\in I} F(Y_{i}) \]
provided the colimit exists.

Let $N$ be a $G$-bornological coarse space with a discrete coarse structure (later we will consider $N = \nat_{min,min}$).
Let $E' \colon G\BC_{/N} \to \bM$ be a functor to a cocomplete and stable $\infty$-category.
\begin{ddd}\label{weoirgjwergrwegwg9}
 The functor $E'$
 \begin{enumerate}
  \item is coarsely invariant if $E'$ sends the morphism $\{0,1\}_{max,max} \otimes X \to X$ over $N$ to an equivalence for every $X$ in $G\BC_{/N}$.
  \item is $\pi_0$-excisive if for every partition $(Y,Z)$ of an object $X$ in $G\BC_{/N}$ into coarsely disjoint, invariant subsets the inclusion maps induce an equivalence
  \[ E'(Y) \oplus E'(Z) \xrightarrow{\simeq} E'(X)\ .\]
  \item is coarsely excisive if $E'(\emptyset) \simeq 0$ and for every $X$ in $G\BC_{/N}$ and every complementary pair (\cite[Def.~3.11]{injectivity}) consisting of an equivariant big family $\cY = (Y_i)_{i \in I}$ and an invariant subset $Z$
  the induced square
  \[\xymatrix{
   E'(Z \cap \cY )\ar[r]\ar[d] & E'(Z)\ar[d] \\
   E'(\cY)\ar[r] & E'(X)
  }\] is a pushout.
  \item \label{weoijgwegfwerwf} annihilates flasques if $E'(X) \simeq 0$ for every $X$ such that there exists an endomorphism of $X$ over $N$ that implements flasqueness \cite[Def.~3.12]{injectivity}.
  \item is $u$-continuous if the canonical map
  \[ \mathop{\colim}\limits_{U \in \cC^G_X} E'(X_U) \to E'(X) \]
  is an equivalence for all $X$ in $G\BC_{/N}$, where $\cC_X^G$ denotes the collection of $G$-invariant entourages of $X$ and $X_U$ denotes the object of $G\BC_{/N}$ obtained from $X$ by replacing the coarse structure on $X$ by the coarse structure generated by $U$.
 \end{enumerate}
 The functor $E'$ is an equivariant coarse homology theory on  $G\BC_{/N}$ if it is coarsely invariant, coarsely excisive, $u$-continous, and annihilates flasques.
\end{ddd}

\begin{ex}
 The restriction of an equivariant coarse homology theory {along the forgetful functor $G\BC_{/N} \to G\BC$} is an equivariant coarse homology theory on $G\BC_{/N}$. \end{ex}

In the following we discuss the statement that a coarse homology theory sends 
coarsely excisive decompositions to pushouts. In the non-equivariant case this is shown e.g.\ in \cite[Lem. 3.41]{buen}\footnote{{The equivariant} statement has appeared in \cite[Cor.~4.14]{equicoarse}. Note that the definition of equivariantly coarsely excisive pairs in this reference is not sufficient to prove the statement and should be replaced by the conditions listed in \cref{rem:coarse-excision} below.}.

Let $Y$ be an invariant subset of $X$ and assume that $U$ is in $\cC_{X}^{G}$. 
If $\diag(X)\subseteq U$, then $Y\subseteq U[Y]$.
In contrast to the non-equivariant case, this inclusion is not a coarse equivalence in general.

The fact that the inclusion of a subset into its coarse thickening may not be a coarse equivalence has the consequence that the generalisation of the definition of a coarsely excisive decomposition from the non-equivariant to the equivariant case  is not completely straightforward. In order to formulate the conditions  in a compact way we introduce the following notion.
Let $X$ be in $G\BC$ and $Y$ be an invariant subset of $X$.
\begin{ddd}
We call the subset $Y$ thickenable\footnote{In \cite{equicoarse}, such subsets were called nice.} if there exists a cofinal subset of $U$ in $\cC^{G}_{X}$ such that $\diag(X)\subseteq U$ and the inclusion $Y \to U[Y]$ is a coarse equivalence.
\end{ddd}

For an invariant subset $Y$ 
 let
 \[ \{Y\} := \{ U[Y] \mid U \in \cC^G_X \} \]
 denote the big family generated by $Y$, where $\cC_{X}^{G}$ denotes the set of $G$-invariant subsets of $X$. If $Y$ is thickenable, then the canonical map $F(Y)\to  F(\{Y\})$ is an equivalence for any coarsely invariant functor $F$.

 Consider $X$ in $G\BC$ and a pair of invariant subsets $(Y,Z)$ such that $Y \cup Z = X$. 
 
 \begin{ddd}\label{rem:coarse-excision}
 We say that $(Y,Z)$ is a coarsely excisive pair if the following holds:
 \begin{enumerate}
  \item\label{rem:coarse-excision1} for every $U$ in $\cC_{X}$ there exists $V$ in $\cC_{X}$ such that $U[Y] \cap U[Z] \subseteq V[Y \cap Z]$;
  \item\label{rem:coarse-excision2} $Y$ is thickenable;
  \item\label{rem:coarse-excision3} $Y\cap Z$ is thickenable;
  \item\label{rem:coarse-excision4} there exists a cofinal subset of $V$ in $\cC_{X}^{G}$ such that $V[Y] \cap Z$ is thickenable. \qedhere
    \end{enumerate}
 \end{ddd}
  
Let $E\colon G\BC \to \bM$ be an equivariant coarse homology theory.
Let $X$ be in $G\BC$ and $(Y,Z)$ be a partition of $X$ into invariant subsets.
\begin{lem} \label{rem:coarse-excisionl}
 If $(Y,Z)$ is a coarsely excisive pair, then the induced square
 \begin{equation}\label{vasdvdsvscascsdc}
\xymatrix{
  E(Y \cap Z)\ar[r]\ar[d] & E(Z)\ar[d] \\
  E(Y)\ar[r] & E(X)
 }
\end{equation} is a pushout.
\end{lem}
 \begin{proof}
 Since $E$ is coarsely excisive, the square
 \begin{equation}\label{eqwdoijoiewfqwfwqefewdqwed}
\xymatrix{
  E(\{Y\} \cap Z)\ar[r]\ar[d] & E(Z)\ar[d] \\
  E(\{Y\})\ar[r] & E(X)
 }
\end{equation}
 is a pushout.
 We now argue that this square is equivalent to the square in \eqref{vasdvdsvscascsdc}.
 In fact, the canonical inclusions of spaces induce a map from the square \eqref{vasdvdsvscascsdc} to the square in \eqref{eqwdoijoiewfqwfwqefewdqwed}.
 We therefore must check that the induced maps on the left corners are equivalences.

 Condition~\ref{rem:coarse-excision}.\ref{rem:coarse-excision2} implies that $E(Y) \to E(\{Y\})$ is an equivalence.
 
 It remains to discuss the upper left corner. The canonical inclusions induce the following commutative diagram:
 \[\xymatrix{
  E(Y \cap Z)\ar[r]\ar[rrd] & E(\{Y\cap Z\})\ar[r] & \mathop{\colim}\limits_{V \in \cC^G_X} E( \{V[Y]\cap Z\}) \\
  & & E(\{Y\} \cap Z)
  \ar[u]
 }\]
 Condition~\ref{rem:coarse-excision}.\ref{rem:coarse-excision3} implies that $E(Y \cap Z) \to E(\{Y\cap Z\})$ is an equivalence.
 By Condition~\ref{rem:coarse-excision}.\ref{rem:coarse-excision4}, the morphism $E(V[Y] \cap Z) \to E(\{V[Y]\cap Z\})$ is an equivalence for {all $V$ in some} cofinal subset of $\cC_{X}^{G}$
 Writing the right vertical map in the form
 \[ \mathop{\colim}\limits_{V\in\cC^{G}_{X}} E(V[Y]\cap Z) \to \mathop{\colim}\limits_{V \in \cC^G_X} E( \{V[Y]\cap Z\})\ , \]
 we see that it is an equivalence.
 Finally, Condition~\ref{rem:coarse-excision}.\ref{rem:coarse-excision1} implies that the map $E(\{Y\cap Z\}) \to \colim_{V \in \cC^G_X} E(\{V[Y]\cap Z\})$ is an equivalence.
  So $E(Y \cap Z) \to E(\{Y\} \cap Z)$ is also an equivalence.\end{proof}
 
 The same argument also proves the analogue in the relative situation.
 Let $E'\colon G\BC_{/N}\to \bM$ be an equivariant coarse homology theory.  Let $X$ be in $G\BC_{/N}$ and $(Y,Z)$ be a partition of $X$ into invariant subsets.   
 \begin{lem}
 \label{rem:coarse-excisionl-recase}
 Assume that $(Y,Z)$ is coarsely excisive as a pair in $G\BC$. Then the induced square
 \begin{equation*}
\xymatrix{
  E'(Y \cap Z)\ar[r]\ar[d] & E'(Z)\ar[d] \\
  E'(Y)\ar[r] & E'(X)
 }
\end{equation*}
 is a pushout.
 \end{lem}

Let $E \colon G\BC \to \bM$ be a functor whose target is a cocomplete stable $\infty$-category.
Furthermore, let $A$ be in $\PSh(G\Orb)$ and let $(A_n)_n$ be in $\prod_\nat \PSh(G\Orb)$.
\begin{prop}\label{prop:extenp-homology}\label{cor:extenps-homology}
 If $E$ has any of the following properties, then $\extenp{E}((A_n)_n,-)$ and $\extenps{E}(A,-)$ inherit the same property:
\begin{enumerate}
 \item\label{it:extenp-homology1} coarse invariance;
 \item\label{it:extenp-homology2} $\pi_0$-excisiveness;
 \item\label{it:extenp-homology3} coarse excisiveness;
 \item\label{it:extenp-homology4} annihilation of flasques;
 \item\label{it:extenp-homology5} $u$-continuity.
\end{enumerate}  
 
\end{prop}
\begin{proof}
We first consider the case of $\extenp{E}$.
 Since all properties listed above are preserved by taking colimits, it suffices to check the case that $(A_n)_n$ is a sequence of $G$-sets. 
 So we must show that the functor
 \[ E((A_{n})_n \otimes_{\nat}(-)) \colon G\BC_{/\nat_{min,min}}\to \bM \]
 inherits the listed properties from $E$.
 This can be checked by a straightforward argument
 {which is similar to the} 
 proof of \cite[Lem.~4.17]{equicoarse}.
  
 The functor $\exten{E}$ inherits each property in the list from $E$.
 The sum $\bigoplus_n \exten{E}(A,(-)_n)$ then has the analogous property on $G\BC_{/\nat_{min,min}}$, where $(-)_{n} \colon G\BC_{/\nat_{min,min}}\to G\BC$ sends $X$ to its fibre $X_{n}$ over $n$.
 By \cref{def:extenps}, $\extenps{E}(A,-)$ fits in a cofibre sequence with $\bigoplus_n \exten{E}(A,(-)_n)$ and $\extenp{E}(\diag(A),-)$.
  So the assertion that these two functors have a property from the list implies that also $\extenps{E}(A,-)$  has this property.
\end{proof}

Suppose that $\cC'_{W}$ and $\cC_{W}$ are two original coarse structures on the metric space $W$  which satisfy the assumptions in \cref{fewuifzhifweffewfwefwefwefwe}. If $\cC'_W \subseteq \cC_W$, then the two new coarse structures defined by \eqref{eq:vanishing-control} satisfy $\cC'_{h} \subseteq \cC_{h}$. So the identity on the underlying sets is a morphism $W_{h}' \to W_{h}$
in $G\BC_{/\nat_{min,min}}$.

\begin{lem}\label{lem:choice-of-metrics}
  The induced map $\extenps{E}(-,W_{h}') \to \extenps{E}(-,W_{h})$ is an equivalence. In particular, $W'_{h}$ is $(E,\cF)$-proper if and only if $W_{h}$ is $(E,\cF)$-proper.
\end{lem}
\begin{proof}
By $u$-continuity (\cref{cor:extenps-homology}), we have an equivalence
\begin{equation}\label{qwefqefjfpoewqfqefewf}
 \mathop{\colim}\limits_{U\in \cC_{h}}\extenps{E}(-,W_{U}) \xrightarrow{\simeq} \extenps{E}(-,W_{h})\ .
\end{equation}
 Furthermore, using \cref{lem:projections} we check that the canonical inclusion induces an equivalence
 \begin{equation}\label{wergfrpogkpwrgwergrwegerfe}
 \extenps{E}(-,c_{[0,n]}(W_{U})) \stackrel{\simeq}{\to}\extenps{E}(-,W_{U})
\end{equation}
 for every $n$ in $\nat$, where $c_{[0,n]}$ is as in {\cref{def:oiwjgoirfgrefwerf}.\ref{oiwjgoirfgrefwerf}.} Analogous equivalences exist for $W'$ and $U'$ in $\cC'_{h}$.
 
 {Since $\cC'_{h} \subseteq \cC_{h}$, the map $\extenps{E}(-,W_{h}') \to \extenps{E}(-,W_{h})$ is identified via \eqref{qwefqefjfpoewqfqefewf} with the canonical map
 \[ \mathop{\colim}\limits_{U' \in \cC_{h}'} \extenps{E}(-,W'_{U'}) \to \mathop{\colim}\limits_{U \in \cC_{h}} \extenps{E}(-,W_{U})\ .\]
 Let $U$ be in $\cC_{h}$. By the compatibility of the original coarse structure $\cC'_{W}$ with the metric assumed in \cref{fewuifzhifweffewfwefwefwefwe}, there exists $r$ in $(0,\infty)$ such that $\{ (w,w') \mid d(w,w') \leq r \}\in \cC'$.
 Using \eqref{eq:vanishing-control}, there is $n$ such that $\sup \big\{ d(w,w') \mid (w,w') \in U \cap(W_k \times W_k) \big\}\le r$ for all integers $k$ satisfying $k \geq n$. 
 Consequently, there exists a map of posets $t \colon \cC_{h} \to \nat$ such that $c_{[0,t(u)]}(U) \in \cC_{h}'$ for all $U$ in $\cC_{h}$.}
 
 Since $c_{[0,n]}(W_U) = c_{[0,n]}(W_{U_{|c_{[0,n](W)}}})$ for every natural number $n$, the map $t$ induces the diagonal arrow in the following commutative diagram:
 \[\xymatrix{
  \mathop{\colim}\limits_{U' \in \cC_{h}'} \extenps{E}(-,W'_{U'})\ar[r] & \mathop{\colim}\limits_{U \in \cC_{h}} \extenps{E}(-,W_{U}) \\
  \mathop{\colim}\limits_{U' \in \cC'_{h}} \extenps{E}(-,c_{[0,t(U')]}(W'_{U'}))\ar[r]\ar[u] & \mathop{\colim}\limits_{U \in \cC_{h}} \extenps{E}(-,c_{[0,t(U)]}(W_{U}))\ar[u]\ar[ul]
 }\]
 Both vertical maps are equivalences since they are induced by the maps from \eqref{wergfrpogkpwrgwergrwegerfe}. It follows that all maps in the diagram are equivalences.
\end{proof}

We now start the induction argument for \cref{rgiorgergegergergerg3232424} and assume
that $W$ is $0$-dimensional.
For the following lemma we only need that  $E \colon G\BC \to \bM$ is a $\pi_0$-excisive and hyperexcisive functor to a cocomplete $\infty$-category.
Let $(T_n)_n$ be in $\prod_\nat G\Set$. 
We have a canonical projection
\[ (T_n)_n \otimes_\nat \nat_{min,min} \to \nat_{min,min} \]
in $G\BC$ which we use to  interpret  $(T_n)_n \otimes_\nat \nat_{min,min}$ as an object of $G\BC_{/\nat_{\min,min}}$.
It is furthermore useful to remember the definition of  ${\ell \colon G\Set \to \PSh(G\Orb)}$ in \eqref{fwqefqwefewqdwed} and \cref{def:extenp}
of $\extenp{E}$.
\begin{lem}\label{lem:extenp-juggling}
 There is a natural equivalence
 \begin{equation}\label{qwefwqfoijowqfwqfwefw}
\extenp{E}(- , (T_n)_n \otimes_\nat \nat_{min,min}) \simeq \extenp{E}(- \times (\ell(T_n))_n, \nat_{min,min})
\end{equation}
of functors $\prod_{\nat} \PSh(G\Orb) \to \bM$.
\end{lem}
\begin{proof}
  The left functor in \eqref{qwefwqfoijowqfwqfwefw} is the left Kan extension of $\extenp{E}(\extenp{\ell}(-,(T_n)_n \otimes_\nat \nat_{min,min}))$ along the fully faithful functor $\prod_{\nat}\ell$. Consequently, its restriction along $\prod_{\nat}\ell$  
is the left-hand term of
the equivalence
\begin{align}
  E ((-)\otimes_{\nat} ((T_n)_n \otimes_\nat \nat_{min,min}))
  \simeq
   E( (- \times T_n)_n \otimes_\nat \nat_{min,min} )\ .
\label{ewgoijwrgergw}
  \end{align}
  The equivalence is induced by an obvious natural isomorphism in $G\BC_{/\nat_{\min,min}}$.
  
It remains to identify the right-hand side of \eqref{qwefwqfoijowqfwqfwefw} 
  with the left Kan extension   (written as $(\prod_{\nat}\ell)_{!}$ in the following)  of the right-hand side of \eqref{ewgoijwrgergw}
 along $\prod_{\nat}\ell$. We first construct a natural transformation 
\begin{equation}\label{asvsjvasopvdsvadvadva}
(\prod_{\nat}\ell)_{!} E( (- \times T_n)_n \otimes_\nat \nat_{min,min} )\to \extenp{E}(- \times(\ell(T_n))_n, \nat_{min,min})\ . 
\end{equation} 
 Using that the functor $\ell$ in \eqref{fwqefqwefewqdwed} preserves cartesian products, we get an equivalence
   \begin{equation}\label{advadvasdvsvsv}
 E((-\times   T_n)_n\otimes_{\nat}  \nat_{min,min})\simeq\extenp{E}((\prod_{\nat}\ell)(-) \times (\ell(T_n))_n, \nat_{min,min}) 
\end{equation} 
of functors from $\prod_{\nat} G\Set$ to $\bM$.
The transformation \eqref{asvsjvasopvdsvadvadva} is now induced by \eqref{advadvasdvsvsv}
 and the universal property of the left Kan extension.
  
 It remains to show that \eqref{asvsjvasopvdsvadvadva} is an equivalence.
 Let $(A_n)_n$ be in $\prod_\nat \PSh(G\Orb)$. 
 By \cref{def:extenp} of $\extenp{E}$ as a left Kan extension, the canonical map
 \begin{align*}
  \mathop{\colim}\limits_{(S_n \to A_n \times \ell(T_n))_n \in \prod_n G\Set_{/(A_n)_{n} \times (\ell(T_n))_{n}}} \extenp{E}( (S_n)_n\otimes_{\nat}\nat_{min,min}) \to \extenp{E}((A_n)_{n} \times (\ell(T_n))_n,\nat_{min,min})
 \end{align*}
 is an equivalence.
 The functor $- \times (\ell(T_n))_n \colon \prod_\nat \PSh(G\Orb) \to \prod_\nat \PSh(G\Orb)$ induces a functor
 \[ \prod_{n\in \nat} G\Set_{/(A_n)_{n}} \to \prod_{n\in \nat} G\Set_{/(A_n)_{n} \times (\ell(T_n))_{n}} \]
 that is right adjoint to the functor given by composition with the projection maps $(A_n)_{n} \times (\ell(T_n))_{n} \to (A_n)_{n}$.
 Since right adjoints are cofinal \cite[Cor.~6.1.13]{Cisinski:2017}\footnote{Note that \cite{Cisinski:2017} uses the term ``final'' for what we call cofinal and refers by ``cofinal'' to the dual concept.}, it follows that the canonical map
 \[ \mathop{\colim}\limits_{(S_n \to A_n)_n \in \prod_n G\Set_{/(A_n)_{n}}} \extenp{E}( (S_n \times T_n)_n\otimes_{\nat } \nat_{min,min}) \to \extenp{E}((A_n)_{n} \times (\ell(T_n))_n, \nat_{min,min}) \]
 is an equivalence.
 In view of the pointwise formula for the left Kan extension, this is precisely the statement that 
 the evaluation of \eqref{asvsjvasopvdsvadvadva} at $(A_{n})_{n}$ is an equivalence.
\end{proof}

Let $W$ be as in \cref{rgiorgergegergergerg3232424}.
Let $E \colon G\BC \to \bM$ be a $u$-continuous, $\pi_0$-excisive and hyperexcisive functor to a cocomplete $\infty$-category.
\begin{kor}\label{cor:EFproper-0}
  If the simplicial complex $W$ is $0$-dimensional and its stabilisers belong to $\cF$, then $W_{h}$ is $(E,\cF)$-proper.
\end{kor}
\begin{proof}
Since $W$ is discrete as a topological space, the original coarse structure on $W$ defined in \cref{beforergiorgergegergergerg3232424} is the minimal one.
 By \cref{ergioergergegergergergerg}, the map
 \[ \exten{E}(E_\cF G,W_{n,min,max}) \to \exten{E}(*,W_{n,min,max}) \]
 is equivalent to the map
 \begin{equation}\label{eq:EFproper-0a}
  \exten{E}(E_\cF G \times \ell(W_n), *) \to \exten{E}(\ell(W_n),*)
 \end{equation}
 for every $n$.
 Observe that $W_{h} \cong (W_n)_n \otimes_\nat \nat_{min,min}$.
 Hence by \cref{lem:extenp-juggling} the morphism
 \[ \extenp{E}(\diag(E_\cF G), W_{h}) \to \extenp{E}(*, W_{h}) \]
 is equivalent to
 \begin{equation}\label{eq:EFproper-0b}
  \extenp{E}((E_\cF G \times \ell(W_n))_n, \nat_{min,min}) \to \extenp{E}((\ell(W_n))_n, \nat_{min,min})\ .
 \end{equation}
 Since the stabilisers of $W_n$ belong to $\cF$, the projection map $E_\cF G \times \ell(W_n) \to \ell(W_n)$ is an equivalence for every $n$.
 Applying this to \eqref{eq:EFproper-0a} and \eqref{eq:EFproper-0b}, it follows that the induced maps
 \[  \bigoplus_{n \in \nat} \exten{E}(E_\cF G,W_n) \to  \bigoplus_{n \in \nat} \exten{E}(*,W_n) \quad\text{and}\quad \extenp{E}(\diag(E_\cF G), W_{h}) \to \extenp{E}(\diag(*), W_{h}) \]
 are equivalences. It follows that the induced map on cofibres is an equivalence, so $W_{h}$ is $(E,\cF)$-proper (\cref{ergioerwgwergwergregregwergwregwergwre}).
\end{proof}

Let $Z$ be a $G$-invariant subspace of a $G$-bornological coarse space $X$ whose coarse and bornological structures we denote by $\cC_{X}$ and $\cB_{X}$.

\begin{ddd}
		\label{def:flasque-mor}
		The pair $(X,Z)$ is relatively flasque if there is a morphism $f \colon X \to X$ satisfying the following:
	\begin{enumerate}
		\item\label{def:flasque-mor-i} $f$ is close to $\id_X$;
		\item\label{def:flasque-mor-ii} $f(Z) \subseteq Z$;
		\item \label{def:flasque-mor-iii} for every $U$ in $\cC_{X}$ there exists $k$ in $\nat$ with $f^k(U[Z])\subseteq Z$;
		\item \label{def:flasque-mor-iv} for every $U$ in $\cC_{X}$ we have also $\bigcup_{m\in \nat}(f^m\times f^m)(U) \in \cC_{X}$;
		\item \label{def:flasque-mor-v} for $B$ in $\cB_{X}$ there exists $k$ in $\nat$ such that $B\cap f^k(X)\subseteq Z$;
		\item \label{def:flasque-mor-vi} for every $B'$ in $Z\cap \cB_{X}$ there exists $k'$ in $\nat$ such that $\bigcup_{m \leq k'} f^{-m}(B') = \bigcup_{m \in \nat} f^{-m}(B')$.\qedhere
	\end{enumerate}
\end{ddd}

Note that $X$ is flasque in the sense of \cite[Def.~3.8]{equicoarse} if and only if the pair $(X,\emptyset)$ is relatively flasque.

Let $E$ be an equivariant coarse homology theory.
\begin{prop}\label{prop:flasque-mor}
	If $(X,Z)$ is relatively flasque, then $E(Z) \to E(X)$ is an equivalence.
\end{prop}
\begin{proof}
Let $\nat_{can,min}$ denote the $G$-bornological coarse space with trivial $G$-action, minimal bornology and the coarse structure induced by the standard metric. Consider the following pushout \[\xymatrix{
 Z\ar[r]\ar[d]_-{i} & X\ar[d]\\
 Z\otimes\mathbb{N}_{can,min}\ar[r] & \widehat X	
}\ ,\] where $i\colon Z\to Z\otimes\nat_{can,min}$ sends $z$ to $(z,0)$.
One checks that this pushout exists, see \cite[Prop.~2.21]{equicoarse}.
 
The pair $(X,Z \otimes \nat_{can,min})$ is a coarsely excisive pair on $\widehat X$.
If we apply $E$, then by \cref{rem:coarse-excisionl} we get the pushout square \[\xymatrix{
 E(Z)\ar[r]\ar[d]_-{E(i)} & E(X)\ar[d]\\
E( Z\otimes\mathbb{N}_{can,min})\ar[r] & E(\widehat X	)
}\ .\]
The space $Z \otimes\mathbb{N}_{can,min}$ is flasque.
This fact is witnessed by the selfmap  $f'\colon Z\otimes\mathbb{N}_{can,min}\to Z\otimes\mathbb{N}_{can,min}$ given by $f'({z},n)=({z},n+1)$.
We claim that $\widehat X$ is also flasque.
Since $E$ is an equivariant coarse homology theory, it annihilates flasques.
Assuming the claim, we can therefore conclude that $E(Z\otimes\mathbb{N}_{can,min})\simeq 0 \simeq E(\widehat X)$. Hence $E(Z) \to E(X)$ is an equivalence.

To show the claim we consider
a map $f$  witnessing relative flasqueness of $(X,Z)$. We define $\widehat f\colon\widehat X\to \widehat X$ by $\widehat f(x) := f(x)$ for $x$ in $X\setminus Z$, and by $\widehat f(z,n) := (f(z),n+1)$ for $z$ in $Z$, $n$ in  $\nat$.
Since $f$ is close to $\id_X$, $\widehat f$ is close to $\id_{\widehat X}$, too.

If $B$ in $\widehat X$ is bounded, then $B \cap (Z \times \nat) \subseteq Z \times [0,l]$ for some $l$ in $\nat$.
By \eqref{def:flasque-mor-v} in \cref{def:flasque-mor}, there exists a natural number
$k$ such that $B\cap \widehat{f}^k(\widehat X)\subseteq Z\otimes \mathbb{N}_{can,min}$.
As $\widehat f$ shifts elements in $Z \otimes\mathbb{N}_{can,min}$, Condition~\ref{def:flasque-mor-vi} in \cref{def:flasque-mor} ensures that $B\cap \widehat{f}^{k+k'+l}(\widehat X)=\emptyset$ for some $k'$ in $\nat$.

It remains to show that $V := \bigcup_{n\in \nat}(\widehat f^n\times\widehat f^n)(U)$ is an entourage of $\widehat X$ for every entourage $U$. We consider separately the intersections of $V$ with $X \times X$, $(Z \times \nat) \times (Z \times \nat)$ and the remaining part of $V$. The first two cases follow from \eqref{def:flasque-mor-iv} in \cref{def:flasque-mor}. So we only need to consider the last case.
By symmetry, it suffices to consider pairs $(\widehat f^n(x),\widehat f^n(x'))$ such that $\widehat f^n(x)\in Z \otimes \nat_{can,min}$ and $\widehat f^n(x') \in X$. Let $k$ be minimal such that $\widehat f^k(x)\in Z \otimes \nat_{can,min}$.
Then by \eqref{def:flasque-mor-iii} in \cref{def:flasque-mor} there is $m$ depending only on $U$ such that $f^{k+m}(x')\in Z \otimes \nat_{can,min}$. It follows that
\[ (\widehat f^n(x),\widehat f^n(x')) \in  \bigcup_{n\in \nat}(f^n \times f^n)(\pr_X \times \pr_X)(U) \times \{(r,s) \mid \lvert r - s \rvert \leq m + m' \}\ , \]
where $m'$ is the maximal distance of $\nat$-coordinates allowed by $U$.
\end{proof}

Let $X$ be an object in $G\BC_{/\nat_{min,min}}$ and let $Z$ be an invariant subset of $X$.
We call the pair $(X,Z)$ relatively flasque over $\nat_{min,min}$ if there exists an endomorphism as in \cref{def:flasque-mor} which is also a morphism over $\nat_{min,min}$.
Consider objects $A$ in $\PSh(G\Orb)$ and $(A_n)_n$ in $\prod_\nat \PSh(G\Orb)$.
\begin{kor}\label{cor:flasque-mor-nat}
 If $(X,Z)$ is relatively flasque over $\nat_{min,min}$, then
 \begin{align*} \extenp{E}((A_n)_n,Z) \to \extenp{E}((A_n)_n,X)\quad \text{and}\quad \extenps{E}(A,Z) \to \extenps{E}(A,X) \end{align*}
are equivalences.
\end{kor}
\begin{proof}
 In the case of $\extenp{E}$ it suffices to check the statement after restriction along $\prod_{\nat}\ell $ in the first argument.
 Thus let $(T_n)_n$ be in $\prod_{\nat} G\Set$.
 In view of \cref{def:extenp}, we must then show that 
 \[ E((T_{n})_{n}\otimes_{\nat}Z)\to  E((T_{n})_{n}\otimes_{\nat}X) \]
 is an equivalence.
 Tthis follows from \cref{prop:flasque-mor} since $(T_n)_n \otimes_\nat -$ sends relatively flasque pairs in $G\BC_{/\nat_{min,min}}$   to relatively flasque pairs in $G\BC$.

 Note that the pair of fibres $(X_n,Z_n)$ is relatively flasque for every $n$.
 Since $\exten{E}(A,-)$ is an equivariant coarse homology theory, $\exten{E}(A,Z_n) \to \exten{E}(A,X_n)$ is an equivalence.
 In view of \cref{def:extenps}, this observation together with the statement for $\extenp{E}$ imply the statement for $\extenps{E}$.
\end{proof}

We now turn to the promised homotopy invariance result.

Let $X$ be a metric space with isometric $G$-action together with a $G$-map $p \colon X \to \nat$. Let $Z$ be a $G$-invariant subset of $X$. Suppose that the metric $d$ on $X$ is a path metric.

 We choose a coarse structure on $X$ which is compatible with the metric.
 We  furthermore  equip $X$ with the  bornology  which is the minimal one such that $p:X\to \nat_{min,min}$ is a morphism.  
 Let $X_{h}$ be the bornological coarse space obtained by  equipping $X$
  with the   coarse structure $\cC_{h}$ in \eqref{eq:vanishing-control} introduced in  \cref{fewuifzhifweffewfwefwefwefwe}.
 Then $Z_{h}$ denotes the subset $Z$ with the structures induced from $X_{h}$.

Let $E$ be an equivariant coarse homology theory.

\begin{prop}\label{prop:coarse-retract}
 Assume that there exists a map of sets $\Psi \colon [0,1] \times X \to X$ over $\nat$ such that the following holds:
 \begin{enumerate}
  \item\label{it:coarse-retract1} $\Psi_0 = \id_X$;
  \item\label{it:coarse-retract2} $\Psi_1(X) \subseteq Z$;
  \item\label{it:coarse-retract3} $(\Psi_t)_{|Z} = \id_Z$ for all $t$ in $[0,1]$;
  \item\label{it:coarse-retract4} $\Psi_s\circ \Psi_t = \Psi_{\min(s+t,1)}$ for all $s, t$ in $[0,1]$;
  \item\label{it:coarse-retract5} $\Psi$ is Lipschitz with respect to the sum metric on $[0,1] \times X$;
  \item\label{it:coarse-retract6} there exists $N$ in $\nat$ such that for every $\epsilon$ in $(0,\infty)$  we have $\Psi_{N\epsilon}(U_\epsilon[Z]) \subseteq Z$.
 \end{enumerate} 
 Then $\extenps{E}(-,Z_{h}) \to \extenps{E}(-,X_{h})$ is an equivalence.
\end{prop}
\begin{proof}
  By \cref{lem:choice-of-metrics}, we may replace the original coarse structure on $X$ by the coarse structure induced by $p$, i.e., the one generated by the entourage
  $\bigcup_{n\in \nat} X_{n}\times X_{n}$.
  
  By $u$-continuity of $\extenps{E}$ (see \cref{prop:extenp-homology}), it suffices to show that the inclusion induces an equivalence
 \[ \extenps{E}(\diag(-),Z_U) \to \extenps{E}(\diag(-), X_U) \]
 for every $U$  in $\cC_{h}^{G}$. Here    $X_{U}$ denotes
 the $G$-bornological coarse spaces obtained from $X_{h}$ by replacing the coarse structure $\cC_{h}$ by the coarse structure generated by $U$, and $Z_{U}$ denotes $Z$ with the structures induced from $X_{U}$.

 For a function $\phi \colon \nat \to [0,\infty)$, define
 \begin{equation}\label{eq:U-phi}
 U_{\phi} := \{ (x,x')\in X\times X \mid p(x) = n = p(x') \Rightarrow d(x,x') \leq \phi(n) \}\ .
 \end{equation}
 Using the description \eqref{eq:vanishing-control} of $\cC_{h}$, one checks that 
  the set of entourages $U_{\phi}$ indexed by such functions satisfying $\lim_{n \to \infty} \phi(n) = 0$ is cofinal in   $\cC_{h}$.    

 Fix a function $\phi$.
 Using \cref{cor:flasque-mor-nat}, it is enough to show that $(X_{U_\phi},Z_{U_\phi})$ is relatively flasque over $\nat_{min,min}$.
 So our task is to define the  witness $f$ satisfying the conditions listed in \cref{def:flasque-mor}. We define the set map
   \[ f \colon X \to X, \quad x \mapsto \Psi_{\min(1,\phi(p(x)))}(x)\ .\]
 Note that
 \begin{equation}\label{eq:coarse-retract}
  f^k(x) = \Psi_{\min(1,k\phi(p(x)))}(x)
 \end{equation}
 for all $k$ in $\nat$ and $x$ in $X$ by Condition~\eqref{it:coarse-retract4}.
 Fix an element $M$ in $\nat$ that is larger than the Lipschitz constant of $\Psi$.
 Observe that for any two points $x,x'$ in $X_n$ satisfying $d(x,x') \leq M \phi(n)$ we have $(x,x') \in U_\phi^M$ since $d$ is a path metric.
 
 We now check that $f$ satisfies the conditions of \cref{def:flasque-mor}.
 \begin{enumerate}
  \item Since $d(x,f(x)) \leq M\min(1,\phi(p(x)))$, we have $(x,f(x)) \in U_\phi^M$ for every $x$ in $X$. Hence $f$ is close to $\id_X$.
  This implies that $f$ is a morphism of bornological coarse spaces.
  \item The property $f(Z) \subseteq Z$ is immediate from Condition~\ref{it:coarse-retract3}.
  \item\label{it:coarse-retract-pf3} Since every entourage of $X_{U_\phi}$ is contained in $U_\phi^l$ for some $l$ in $\nat$, it suffices to consider such entourages. By Condition~\ref{it:coarse-retract6}, we have $f^{N \cdot l}(U_\phi^l[Z]) \subseteq Z$ for every $l$ in $\nat$.
  \item As in \ref{it:coarse-retract-pf3}, it is sufficient to consider entourages of the form $U_\phi^l$.
  Since $M$ bounds the Lipschitz constant of $\Psi$, we conclude from    \eqref{eq:coarse-retract} that  $d(f^k(x),f^k(x')) \leq M d(x,x')$.
  Hence $\bigcup_{k \in \nat} (f^k \times f^k)(U_\phi^l) \subseteq U_\phi^{M \cdot l}$.
  \item\label{it:coarse-retract-pf5} If $B$ is a bounded subset of $X_{h}$, then $p(B) \subseteq [0,l]$ for some $l$ in $\nat$. Choose $k$ in $\nat$ such that $k \cdot \phi(l') \geq 1$ for all $l'$ in $[0,l]$. Then $B \cap f^k(X) \subseteq Z$ by \eqref{eq:coarse-retract}.
  \item For $B$ and $k$ as in \ref{it:coarse-retract-pf5} we have $\bigcup_{m \leq k} f^{-m}(B) = \bigcup_{m \in \nat} f^{-m}(B)$.\qedhere
 \end{enumerate}
\end{proof}

\begin{rem}\label{rem:homotopy-theorem}
An alternative way to show \cref{prop:coarse-retract}  is to use the equivariant version of the Homotopy Theorem \cite[Thm.~5.26]{buen}. It suffices to show the equivalence after restriction along $\ell$. 
 Then \cite[Thm.~5.26]{buen} applies since the inclusion $S  \times Z \to S  \times X$ is a uniform homotopy equivalence.
 However, note that the proof given above is considerably simpler than the proof of  \cite[Thm.~5.26]{buen}.
 \end{rem}

As explained in the outline of the proof of \cref{rgiorgergegergergerg3232424},
we must show the theorem for $G$-simplicial complexes $W$ over $\nat$ which are disjoint unions of $d$-dimensional simplices.
Such a complex can be written in the form
\[ W = \coprod_{n \in \nat} T_n \times \Delta^d\ ,\]
 where $(T_n)_n$ is in $\prod_\nat G\Set$.
We then form the object $W_{h}$ in $G\BC_{/\nat_{min,min}}$ by applying \cref{beforergiorgergegergergerg3232424}.

\begin{prop}\label{prop:EFproper-d}
 If the stabilisers of $T_n$ belong to $\cF$ for all $n$, then $W_{h}$ is $(E,\cF)$-proper.
\end{prop} 
\begin{proof}
 Let $W'$ be the subspace of $W$ given by the barycentres of the simplices.
 We claim that \cref{prop:coarse-retract} applies to the inclusion of $W'$ into $W$.
 Assuming the claim, we can conclude that $W_{h}$ is $(E,\cF)$-proper if and only if $W'_{h}$ is $(E,\cF)$-proper. The proposition then follows since $W'_{h}$ is $(E,\cF)$-proper by \cref{cor:EFproper-0}.
 
 In order to show the claim we must construct the map
 $\Psi \colon [0,1]\times W\to W$. We define $\Psi$ on each simplex of $W$ separately.
 Using barycentric coordinates $x$ and the barycenter $b$, we define $\Psi$ such that it acts on a simplex by 
 \[ \Psi_{s}(x) :=   \left\{\begin{array}{cc}x- \min(1,\frac{s}{\|x-b\|})(x-b)  &x\not=b\\b&x=b \end{array} \right.\ ,\]
 where $\lVert-\rVert$ is the Euclidean distance in $\mathbb{R}^{d+1}$. 
 The map $\Psi$ moves the points of the simplices with unit speed straightly towards the barycentres  and then stops.
   One checks that $\Psi$ satisfies the conditions listed in \cref{prop:coarse-retract}. For Condition~\ref{it:coarse-retract5} and Condition~\ref{it:coarse-retract6},
  one can use, for simplicity, the Euclidean metric since it is bi-Lipschitz equivalent to the spherical  metric.  
 \end{proof}

\phantomsection{}\label{soijoevevfsdvdvdfvsdv}
\begin{proof}[Proof of \cref{rgiorgergegergergerg3232424}]
 We proceed by induction on the dimension of $W$. The $0$-dimensional case is covered by \cref{cor:EFproper-0}.
 
 Suppose that $W$ is $d$-dimensional with $d > 0$ and let $\sk_{d-1}(W)$ denote the $(d-1)$-skeleton of $W$.
 We claim that there exists a pushout
 \begin{equation}\label{eq:excision-1}
  \xymatrix{
  \extenps{E}(-, (\coprod_{n \in \nat} T_n \times \partial\Delta^d)_{h})\ar[r]\ar[d] & \extenps{E}(-,\sk_{d-1}(W)_{h})\ar[d] \\
  \extenps{E}(-, (\coprod_{n \in \nat} T_n \times \Delta^d)_{h})\ar[r] & \extenps{E}(-,W_{h}) &
 }\end{equation}
 of functors from  $\PSh(G\Orb)$ to $\bM$, where $T_n$ is the $G$-set of $d$-cells in $W_n$.
 Here we consider $\coprod_{n \in \nat} T_n \times \partial\Delta^d$, $\coprod_{n \in \nat} T_n \times \Delta^d$ and $\sk_{d-1}(W)$ as simplicial complexes over $\nat$ in their own right and equip them with the bornological coarse structures obtained by applying  \cref{beforergiorgergegergergerg3232424}.
 
 Assuming the claim, the projection $p \colon E_\cF G \to *$ induces the following commutative diagram
 \[\resizebox{\textwidth}{!}{\xymatrix@C=1em@R=1em{
  & \extenps{E}(*,(\coprod_{n \in \nat} T_n \times \partial\Delta^d)_{h})\ar[rr]\ar[dd] & & \extenps{E}(*,\sk_{d-1}(W)_{h})\ar[dd] \\
  \extenps{E}(E_\cF G,(\coprod_{n \in \nat} T_n \times \partial\Delta^d)_{h})\ar[rr]\ar[dd]\ar[ur]^-{p_\partial} & & \extenps{E}(E_\cF G,\sk_{d-1}(W)_{h})\ar[dd]\ar[ur]_-{p_{d-1}} & \\
  & \extenps{E}(*, (\coprod_{n \in \nat} T_n \times \Delta^d)_{h})\ar[rr] & & \extenps{E}(*,W_{h}) \\
  \extenps{E}(E_\cF G, (\coprod_{n \in \nat} T_n \times \Delta^d)_{h})\ar[rr]\ar[ur]^-{p_\Delta} & & \extenps{E}(E_\cF G,W_{h})\ar[ur]_-{p_d} &
 }}\]
 in which the front and back faces are pushouts.
 The maps $p_\partial$ and $p_{d-1}$ are equivalences by induction hypothesis. 
 The map $p_\Delta$ is an equivalence by \cref{prop:EFproper-d}.
 Then $p_d$ is also an equivalence, which is precisely the assertion of the theorem.

 The remainder of this proof is devoted to the construction of the pushout square in \eqref{eq:excision-1}.
 Let $Z$ be the topological subspace of $W$ consisting of the disjoint union of the $\frac{2}{3}$-rescaled top-dimensional simplices. Denote by $\partial Z$ the boundary of $Z$. We let $Z_{h}$ and $\partial Z_{h}$ denote the objects of $G\BC_{/\nat_{min,min}}$ obtained by equipping both subsets    with the induced bornological coarse structure from $W_{h}$.
  
  We have an isomorphism  $\coprod_{n \in \nat} T_n \times \Delta^d\cong Z$  of $G$-simplicial complexes over $\nat$.
  But note that the original coarse structure $\cC_{\pi_{0}(Z)}$ coming from   the left-hand side   is in general smaller than the original coarse structure on $Z$   induced from the original coarse structure $\cC_{\pi_{0}(X)}$ on $X$. So this isomorphism of simplicial sets in general only induces a morphism $( \coprod_{n \in \nat} T_n \times \Delta^d)_{h}\to Z_{h}$ in $G\BC_{/\nat_{min,min}}$.
  By \cref{lem:choice-of-metrics}, it nevertheless induces an equivalence of functors
  \begin{equation}\label{eq:excision-2}
   \extenps{E}(-,(\coprod_{n \in \nat} T_n \times \Delta^d)_{h}) \xrightarrow{\simeq} \extenps{E}(-,Z_{h})\ .
  \end{equation}
  Analogously, we have an equivalence
  \begin{equation}\label{eq:excision-3}
   \extenps{E}(-,(\coprod_{n \in \nat} T_n \times \partial\Delta^d)_{h}) \xrightarrow{\simeq} \extenps{E}(-,\partial Z_{h})\ .
  \end{equation}
  Let $Y$ be the complement of the interior of $Z$ 
  so that $W = Y \cup Z$ and $Y \cap Z = \partial Z$.
  In the following, we argue that $(Y,Z)$ is a coarsely excisive pair in $W_{h}$, see  \cref{rem:coarse-excision}.
  
  Recall that the coarse structure $\cC_{h}$ of $W_{h}$ admits the cofinal set of entourages of the form $U_\phi$ introduced in \eqref{eq:U-phi}.
  Using the fact that the metric on $W$  is a path metric, we have  
  \[ U_{\phi}[Z]\cap U_{\phi}[Y]\subseteq U_{\phi}[\partial Z]\ .\]
 This verifies Condition~\ref{rem:coarse-excision}.\ref{rem:coarse-excision1}.
 
 In order to verify Condition~\ref{rem:coarse-excision}.\ref{rem:coarse-excision2}, we  show that the inclusion $Y \to U_{\phi}[Y]$ is an equivariant coarse equivalence.
 To this end we must construct an inverse of the inclusion up to closeness.
 On $Y$ this inverse will {be the} identity.
 In order to define it on $U_{\phi}[Y]\setminus Y$, note that by our conventions about $G$-simplicial complexes, if $g$ in $G$ fixes a point in the interior of a simplex, then it fixes a whole simplex.
 For every orbit $[w]$ in $ (U_{\phi}[Y]\setminus Y)/G$, we choose a representative $w$.
 Then $w$  lies in the interior of some $d$-simplex.
 We define the inverse such that it sends $w$ to some point
 $y$ in $Y$ which belongs to the same simplex and satisfies $(w,y)\in U_{\phi}$.
 Since the stabiliser of $y$ contains the stabiliser of $w$, the map can then be extended equivariantly to the whole orbit. 
Doing this construction for every $G$-orbit separately gives the construction of the desired inverse on $ U_{\phi}[Y]\setminus Y$.

We proceed similarly in order to show that $Y \cap Z \to U_\phi[Y \cap Z]$ and
$U_{\phi'}[Y]\cap Z\to U_{\phi}[U_{\phi'}[Y]\cap Z]$ are coarse equivalences for all $\phi$ and $\phi'$. This verifies Conditions~\ref{rem:coarse-excision}.\ref{rem:coarse-excision3} and \ref{rem:coarse-excision}.\ref{rem:coarse-excision4}.
   
Since $(Y,Z)$ is a coarsely excisive pair in $W_{h}$ and $\extenps{E}(A,-)$ is a coarse homology theory for every $A$ in $\PSh(G\Orb)$ by \cref{prop:extenp-homology}, we have a pushout square
 \begin{equation}\label{dsafqefdfads}
\xymatrix{
   \extenps{E}(-,\partial Z_{h} )\ar[r]\ar[d] & \extenps{E}(-,Y_{h})\ar[d] \\
   \extenps{E}(-,Z_{h})\ar[r] & \extenps{E}(-,W_{h})
  }
\end{equation}
of functors from $\PSh(G\Orb)$ to $\bM$.

Using the equivalences \eqref{eq:excision-3} and \eqref{eq:excision-2},
we can replace the left part of this pushout square by the left part of the square in 
  \eqref{eq:excision-1}.
In order to replace the right upper corner of the square in \eqref{dsafqefdfads} by the corresponding right upper corner of the square in \eqref{eq:excision-1}, it only remains to show that the inclusion $j\colon \sk_{d-1}(W)_{h}\to Y_{h}$ induces an equivalence \begin{equation}\label{sadvsdvasdvadsv}
j_{*}\colon\extenps{E}(-,\sk_{d-1}(W)_{h}) \xrightarrow{\simeq}  \extenps{E}(-,Y_{h})\ .
\end{equation} 
 We want to apply \cref{prop:coarse-retract} to the subset $\sk_{d-1}(W)$ of $Y$.
 We must define the  {deformation} retraction $\Psi \colon [0,1]\times Y\to Y$.
 On $\sk_{d-1}(W)$, we let $\Psi$ be the constant homotopy.
 On $Y\setminus \sk_{d-1}(W)$, we define $\Psi$ on each $d$-simplex of $W$ separately using barycentric coordinates $x$ and the barycentre $b$.
 The restriction of $\Psi$ to the intersection of a simplex with $Y$ is given by
\[ \Psi_{s}(x):=b+\min(s,a(x)) \frac{(x-b)}{\lVert x - b \rVert}\ ,\]
where $a(x)$ is defined by
\[ a(x):=\min\{\frac{\lVert x  - b \rVert}{(d+1) x_{i}-1} \mid i=0,\dots,n\}\ .\]
This map moves the points with unit speed on the rays from the barycentres towards the boundary of the simplices and then stops after the hitting time $a(x)$. Observe that $\Psi$ is well-defined since it does not move points on the boundary of $d$-simplices.

One again checks the conditions  listed in \cref{prop:coarse-retract}. 
Using that the spherical metric is bi-Lipschitz equivalent to the Euclidean metric on the simplices,
one can again use the latter metric to check Conditions~\ref{it:coarse-retract5} and \ref{it:coarse-retract6}. 
In order to see that the hitting time is Lipschitz, note that there exists a neighbourhood of $b$ which  is contained in the complement of $Y$.

By \cref{prop:coarse-retract}, we get an equivalence 
\begin{equation}\label{svadscdscas}
j'_{*}\colon\extenps{E}(-,\sk_{d-1}(W)'_{h}) \xrightarrow{\simeq}  \extenps{E}(-,Y_{h})\ ,
\end{equation}  where $\sk_{d-1}(W)'_{h}$ denotes
 $\sk_{d-1}(W)$ equipped with the bornological coarse structure induced from $Y_{h}$, or equivalently, from $W_{h}$.  
 The identity of underlying sets is a morphism 
 \begin{equation}\label{qwfeqwefewdqwewdewdqde}
\sk_{d-1}(W)_{h}\to \sk_{d-1}(W)'_{h}
\end{equation} 
in $G\BC_{/\nat_{min,min}}$.
We want to show that   \begin{equation}\label{qwfeqwefewdqwewdewdqde1}
\extenps{E}(-,\sk_{d-1}(W)_{h})\to \extenps{E} (-,\sk_{d-1}(W)'_{h})
\end{equation}  is an equivalence.
  For $d>1$ we have a bijection
 $\pi_{0}(\sk_{d-1}(W)) \to \pi_{0}( W )$. This implies that the intrinsic original coarse structure on $ \sk_{d-1}(W)$ coincides with the one induced from $W$, see \cref{beforergiorgergegergergerg3232424}. In particular, for $d>1$ the morphism in  
    \eqref{qwfeqwefewdqwewdewdqde}  is an isomorphism and  \eqref{qwfeqwefewdqwewdewdqde1} is an equivalence.

It remains to consider the case $d=1$.
We fix $A$ in $\PSh(G\Orb)$.
By $u$-continuity of the functor $\extenps{E}(A,-)$ (see \cref{prop:extenp-homology}.\eqref{it:extenp-homology5}), it suffices to show that
\begin{equation}\label{qwfeqwefewdqwewdewdqde2}
 \mathop{\colim}\limits_{U\in \cC_{h}}\extenps{E}(A,\sk_{0}(W)_{U}) \to \mathop{\colim}\limits_{U\in \cC_{h}'} \extenps{E} (A,\sk_{0}(W)_{U})
\end{equation} 
is an equivalence, where $\cC_{h}$ and $\cC_{h}'$ denote the coarse structures of $ \sk_{0}(W)_{h}$ and $\sk_{0}(W)_{h}'$. 
Since the metric space $\sk_{0}(W)$ (with metric induced from $W$) is uniformly discrete,
for every $U$  in $\cC_{h}$ or $\cC_{h}'$ there exists
$n$ in $\nat$ such that $U_{|c_{[0,n](\sk_{0}(W))}}$ is diagonal (see {\cref{def:oiwjgoirfgrefwerf}.\eqref{oiwjgoirfgrefwerf}} for the notation $c_{[0,n]}$).
By \cref{lem:projections},
we have an equivalence
\[ \extenps{E}(A, c_{[0,n]}(\sk_{0}(W)_{U})) \simeq \extenps{E}(A, \sk_{0}(W)_{U})\ .\]
We conclude that the diagrams under the colimits of  both sides of \eqref{qwfeqwefewdqwewdewdqde2} are essentially constant
 with value $\extenps{E}(A,  \sk_{0}(W)_{\diag(\sk_{0}(W))} )$.
Combining the equivalences \eqref{qwfeqwefewdqwewdewdqde1} and \eqref{svadscdscas}, we get the desired equivalence in \eqref{sadvsdvasdvadsv}.
\end{proof}

 \begin{rem}
 Recognising the coarse structure $\cC_{h}$ of $W_{h}$ as a hybrid structure (see \cref{rem:hybrid}) and using the same reduction steps as in \cref{rem:homotopy-theorem}, the claim that \eqref{eq:excision-1} is a pushout square is a consequence of an equivariant version of \cite[Thm.~5.22]{buen}, see also \cite[Sec.~5.5]{buen}.
 Again, we provided a self-contained argument for the reader's convenience.
\end{rem}

\section{Controlled objects as a symmetric monoidal functor}\label{sec:control-monoidal}

\cref{ergiooegergergwergergergw} provides a road map to the proofs of our main theorems.
Given a left-exact $\infty$-category with $G$-action $\bC$ and a finitary localising invariant $\Homol$ (see \cref{wtoigwgreerf}), we need to extend the functor  $\Homol\bC_G$ (see \cref{weoirgjwegwergwerg9}) from $G\Orb$ to a functor  $\Homol V$ defined on $G$-bornological coarse spaces as required in \cref{ergiooegergergwergergergw}, and we need to define a transfer class.
Since our constructions of transfer classes, which are discussed in \cref{sec:famenablegroups,qrgqiorgegegergwegre,sec:dfhj}, rely on \cref{rgiorgergegergergerg3232424}, the   functor $\Homol V$  
 should be a hyperexcisive equivariant coarse homology theory.
Such an extension  $\Homol V$ of $\Homol\bC_G$  in terms of controlled objects in $\bC$ has been constructed in \cite{unik}. We will recall this construction in \cref{sec:controlled-objects}.

The main goal of this section is to equip   $\Homol V$ 
with a suitable weak module structure.
After collecting some auxiliary results in \cref{sec:prelims}, we will show in \cref{sec:controlled-objects-monoidal} that the categories of controlled objects from \cite{unik} admit a lax symmetric monoidal refinement.
\cref{sec:orbits-and-fixed-points} then uses this lax symmetric monoidal structure to define the desired
weak module structure.

\subsection{Controlled objects in left-exact \texorpdfstring{$\infty$}{infinity}-categories}\label{sec:controlled-objects}

We start with reviewing
various classes of $\infty$-categories
which we will use in this section.

By  $\Cati$ we denote the large $\infty$-category of small $\infty$-categories. It contains  the large $\infty$-category
 $\Cle$ of small left-exact $\infty$-categories, i.e., pointed $\infty$-categories admitting finite limits, and finite limit-preserving functors.  The latter in turn contains  the full subcategory
$\Clep$ of idempotent {complete objects}.

{We denote by $\Prlp$ the very large, essentially {large} $\infty$-category of pointed, compactly generated presentable $\infty$-categories and left adjoint functors which preserve compact objects.
We regard $\Prlp$ as a subcategory of the very large $\infty$-category $\CATi$ of large $\infty$-categories by sending $\bC$ to $\bC^\op$.
Define $\CL$ as the image of this inclusion functor.
This means that $\CL$ is the very large, essentially {large} $\infty$-category of opposites of pointed, compactly generated presentable $\infty$-categories and right adjoint functors which preserve cocompact objects.
In particular, all objects of $\CL$ are complete.
Therefore, $\CL$ is contained in the very large $\infty$-category $\CATCplt$ of large, complete $\infty$-categories and limit-preserving functors.}

The pro-completion functor $\Pro_{\omega} \colon \Cle\to \CL$  restricts to an equivalence 
\[ \Pro_{\omega} \colon \Clep \xrightarrow{\simeq} \CL\ .\]
The inverse of the latter is given by the functor
\[ (-)^{\omega} \colon  \CL\to \Clep \]
which takes the subcategory of cocompact objects.
The composition
\[ \Idem := (-)^{\omega}\circ \Pro_{\omega}\colon \Cle\to \Clep \] is the idempotent completion functor.
 
 The following diagram provides a quick overview of the various $\infty$-categories  introduced above and the functors relating them:
 \begin{equation}\label{aidshviuasdvasdcascsadca}
 \xymatrix{\Clep\ar@{-->}[rr]^{\subset}\ar@{-->}[d]_{\simeq}&&\ar@/_.75cm/[ll]_{\Idem}\ar@{..>}[d]^{\subset}\ar[lld]_{\Pro_{\omega}}\Cle\ar@{..>}[r]&\ar@{..>}[d]^{\subset}\Cati\\
 \CL\ar@/_{-1.5cm}/[u]_{\simeq}^{(-)^{\omega}}\ar@{-->}[r]&\CATCplt\ar@{-->}[r]&\CLL \ar@{..>}[r]&\CATi} \ .
\end{equation}
   The solid part and the dotted parts commute separately, and the arrows labelled  with $\subset$ are fully faithful.

  We now recall the key definitions of \cite[Sec.~2]{unik}.

For a set $X$, let $\cP_X$ denote the power set of $X$.
We regard $\cP_X$ as a poset with respect to subset inclusions.
It gives rise to a functor  \begin{equation}\label{ascjiaosdcjaoisdcasdcadsc}
\cP^{*} \colon \Set^{\op}\to \Cat
\end{equation} which 
sends $X$ in $\Set$ to the poset $\cP_{X}$ and a map $f \colon X\to Y$ to the inverse image map $f^{-1} \colon \cP_Y \to \cP_X$.
We define the presheaf functor as  the composition 
 \begin{equation}\label{eq:psh}
 \PSh \colon \Set \times \CATi \xrightarrow{\cP^{*,\op}\times \id} {\Cati^{\op}}\times \CATi {\xrightarrow{(-)^\op \times \id} \Cati^\op \times \CATi}
 \xrightarrow{\Fun} \CATi \ ,
\end{equation}
where we regard $\cP^*$ as $\Cati$-valued via the nerve functor
(which we will usually drop from notation).
This functor sends $X$ in $\Set$ and $\bC$ in $\CATi$ to
\[ \PSh_{\bC}(X):=\Fun(\cP^{\op}_{X},\bC) \]
in $\CATi$.
The category $\bC$ will be called the coefficient category.
By functoriality, we obtain an induced functor
 \begin{equation}\label{eq:Gpsh}
 \PSh \colon G\Set \times \Fun(BG,\CATi) \to \Fun(BG,\CATi)\ .
 \end{equation}
The  forgetful functor $u \colon G\Coarse \to G\Set$ and the inclusion
\[ i \colon \CLLL\to \CATi \] 
induce the functor 
\begin{equation}\label{qewfpowjfkqwepfqewfqewf}
 u \times i \colon G\Coarse\times \Fun(BG,\CLLL)\to G\Set\times \Fun(BG,\CATi)\ .
\end{equation}
Precomposing $\PSh$  from \eqref{eq:Gpsh} with this functor, {the resulting functor factors through the functor}
  \begin{equation}\label{qewfpowjfkqwepfqewfqewf112}\PSh\circ (u\times i)
 \colon G\Coarse\times \Fun(BG,\CLLL)\to   \Fun(BG,\CLLL)\ .
\end{equation}
We will again use the notation $\PSh_{\bC}(X)$ for the value of $\PSh\circ (u\times i)$ on $(X,\bC)$ in $ G\Coarse\times \Fun(BG,\CLLL)$.
 
 Let $X$ be in $G\Set$.
 Given an entourage $U$ on the set $X$, a subset $B$ of $X$ is called $U$-bounded if $B \times B \subseteq U$. We denote by $\cP^{U\bd}_X$ the subposet of $\cP_X$ consisting of all $U$-bounded subsets of $X$.
Let $\bC$ be in $\CLLL$.
A presheaf $M$ in $\PSh_{\bC}(X)$ is called a $U$-sheaf if $M(\emptyset) \simeq 0$ and the commutative diagram
\begin{equation}\label{ervervewevevwevwee}
\xymatrix@C=4em@R=1.5em{
 \cP^{U\bd,\op}_X\ar[r]^-{M_{|\cP^{U\bd,\op}_X}}\ar[d] & \bC \\
 \cP_X^\op\ar[ur]_-{M}
}
\end{equation} 
exhibits $M$ as a right Kan extension of its restriction to $\cP^{U\bd,\op}_X$ \cite[proof of Lem.~2.4]{unik}.

Let now $X$ be in $G\Coarse$. Then consider the full subcategory $\Sh_\bC(X)$ of $\PSh_\bC(X)$ spanned by the presheaves which are $U$-sheaves for some coarse entourage $U$ of $X$.
  By \cite[Cor.~2.10 \& 2.12]{unik}, the collection of  subcategories of sheaves forms a full subfunctor
\[ \Sh \colon G\Coarse \times \Fun(BG,\CATCplt) \to \Fun(BG,\CLL) \]
of $\PSh \circ (u \times i)$. 
Note that the subcategories $\Sh_{\bC}(X)$ of $\PSh_{\bC}(X)$ are only closed under finite limits in general.

In the following we recall some constructions from \cite[Sec.~2.5]{unik}.
In {{\em loc. cit.},} the coeffcients were assumed to belong to $\CL$, but    everything needed in the following extends verbatim to coefficients in $\CLLL$.
For the moment we fix $X$ in $G\Coarse$ and $\bC$ in $\CLLL$.
Let $U$ be a $G$-invariant entourage of $X$ which contains the diagonal.
The $U$-thinning functor
\begin{equation}\label{eq:thin}
 U(-) \colon \cP_X \to \cP_X,\quad Y \mapsto U(Y) := \{ x \in X \mid U[\{x\}] \subseteq Y \}
\end{equation}
is right adjoint to the $U$-thickening $U[-]$ defined in \eqref{qwefoiheiuohfiqwefewfewfqeefedq}.
For a $V$-sheaf $M$, the composite $M \circ U(-)$ is a $U^{-1}VU$-sheaf by \cite[Cor.~2.11]{unik}, so $U$-thinning
induces an endofunctor
\[ U_* \colon \Sh_\bC(X) \to \Sh_\bC(X) \ , \quad M(-)\mapsto  M\circ U(-)\ .\]
Since $U(Y) \subseteq Y$ for every subset $Y$ of $X$, there exists an induced natural transformation $\theta^U \colon \id \to U_*$.

We consider the collection of morphisms
\begin{equation}\label{eq:wxc}
 W_{X,\bC} := \{ M \xrightarrow{\theta^U_M} U_*M \mid U \in \cC^{G,\Delta}_X,\ M \in \Sh_\bC(X) \}
\end{equation}
in $ \Sh_\bC(X)$, where $\cC_{X}^{G,\Delta}:=\{U\in \cC_{X}^{G}\mid \diag(X)\subseteq U \}$.
As shown in  \cite[Prop.~2.76(1)]{unik}, the Dwyer--Kan localisation
\[ \widehat\bV_\bC(X) := \Sh_\bC(X)[W_{X,\bC}^{-1}] \]
is also the localisation of $\Sh_\bC(X)$ at $W_{X,\bC}$ in $\CLL$.
In other words, the canonical functor $\Sh_{\bC}(X)\to \widehat{\bV}_{\bC}(X)$ satisfies a universal property in $\CLL$ which is the analogue of the universal property of the Dwyer--Kan localisation in $\CATi$. 
By \cite[Lem.~2.81 \& 2.83]{unik}, this construction gives rise to a functor \begin{equation}\label{wergwregvfvfdsv}
\widehat\bV \colon G\Coarse \times \Fun(BG,\CLLL) \to \Fun(BG,\CLL)\ .
\end{equation}
Later,  we will also need a similar localisation construction for the presheaf functor in \eqref{eq:Gpsh} which we describe next.  
For $X$ in $G\Coarse$ and $\bC$ in $\Fun(BG,\CATi)$, both the functor $U_*$ and the natural transformation $\theta^U$ are also defined on $\PSh_\bC( X)$ in $\Fun(BG,\CATi)$. 
 So we may consider the collection of morphisms
\begin{equation}\label{eq:wxtilde}
 \wt W_{X,\bC} := \{ M \xrightarrow{\theta^U_M} U_*M \mid U \in \cC^{G,\Delta}_X,\ M \in \PSh_\bC(X) \}
\end{equation}
 in  $\PSh_\bC(X)$.
The Dwyer--Kan localisations
\[ \wt\bV_\bC(X) := \PSh_\bC(X)[\wt W^{-1}_{X,\bC}] \]
for all $X$ in $G\Coarse$ and $\bC$ in $\Fun(BG,\CATi)$ assemble to a functor
\begin{equation}\label{qwefijfoiwqefqwefewfq}
 \wt\bV \colon G\Coarse \times \Fun(BG,\CATi) \to \Fun(BG,\CATi)
\end{equation}
since the proof of \cite[Lem.~2.81]{unik}
extends verbatim to the present case,
replacing the left-exact localisation with the Dwyer--Kan localisation.
Using \cite[Prop.~2.76 (1)]{unik} and \cite[Lem.~2.81]{unik} with the left-exact localisation,
we see  that if we restrict the coefficients along the inclusion $i\colon\CLLL\to \CATi$, then  we get a functor \begin{equation}\label{wrfwffwefwefwefwefweff}
\wt \bV\circ (\id\times i):G\Coarse\times \Fun(BG,\CLLL)\to \Fun(BG,\CLL)\ .
\end{equation} 
The natural transformation $\Sh \to \PSh \circ (u \times i)$ induces a natural transformation $ \widehat\bV \to \wt\bV  \circ(\id\times i)$.
Let $X$ be in $G\Coarse$ and $\bC$ be in $\Fun(BG,\CATCplt)$. \begin{lem}\label{lem:sheaf-vs-presheaf}
	The canonical functor $\widehat\bV_\bC(X) \to \wt\bV_\bC(X)$ is fully faithful.
\end{lem}
\begin{proof}
Let $\wt\ell_{X} \colon \PSh_{\bC}(X)\to \wt \bV_{\bC}(X)$ denote the localisation functor.
	By the same argument as for the localisation $\ell_{X} \colon \Sh_{\bC}(X)\to \widehat \bV_{\bC}(X)$ in \cite[Prop.~2.76]{unik}, one shows that there are natural equivalences
	\[ \colim_{U\in \cC_{X}^{\Delta}} \Map_{\PSh_{\bC}(X)}(M,U_{*}N) \simeq \Map_{\wt\bV_{\bC}(X)}(\wt\ell_{X}(M),\wt\ell_{X}(N)) \]
	for all $M,N$ be in $\PSh_{\bC}(X)$.
	The mapping spaces in $\widehat{\bV}_\bC(X)$ can be computed by the same formula \cite[Prop.~2.76]{unik}.
	Since $\Sh_\bC(X)$ is a full subcategory of $\PSh_\bC(X)$, the lemma follows.
\end{proof}

\begin{kor}\label{eporgjwepogrwegregwer}
The natural transformation $ \widehat\bV \to \wt\bV \circ (\id\times i)$ is the inclusion of a full subfunctor.
\end{kor}

For the discussion of small sheaves we further restrict the category of  coefficient categories to the subcategory $\CL$  of $\CATCplt$. In fact, in order to ensure functorality of the constructions below   with respect to the coefficient categories it is crucial  that the functors in $\CL$ preserve cocompact objects.

Let $X$ be {in $G\BC$} and $\bC$ be in $\Fun(BG,\CL)$.
We can evaluate $\Sh_{\bC}$ on the underlying coarse space of $X$. We then use the bornology $\cB_{X}$ of $X$ in order to define a full subcategory of sheaves by adding the condition that the values
on bounded subsets are cocompact.
Recall that an object is called cocompact if it is compact in the opposite category.

An object $M$ in $\Sh_\bC(X)$ is called small if the evaluation $M(B)$ is cocompact 
in $\bC$ for every  $B$ in $\cB_{X}$.
The full subcategory $\Sh^\smalll_\bC(X)$ of small sheaves is an object of $\Cle$ \cite[Lem.~2.48]{unik}.
We obtain by \cite[Lem.~2.49]{unik} a full subfunctor
\[ \Sh^\smalll \colon G\BC \times \Fun(BG,\CL) \to \Fun(BG,\Cle) \]
of $\Sh\circ (v\times j)$, where $v \colon G\BC\to G\Coarse$ is the forgetful functor and
\[ j \colon \CL\to \CLLL \]
is the inclusion.
Localising at $W_{X,\bC} \cap \Sh^\smalll_\bC(X)$, we get the left-exact category with $G$-action
\[ \bV_{\bC}(X):= \Sh^\smalll_{\bC}(X)[(W_{X,\bC} \cap \Sh^\smalll_\bC(X))^{-1}] \]
in $\Fun(BG,\Cle)$.
By \cite[Cor.~2.79, 2.82 \& 2.84]{unik}, the collection of these categories for all $X$ and $\bC$ forms a full subfunctor
\begin{equation}\label{afsaojqoifgafvasdva}
\bV \colon G\BC \times \Fun(BG,\CL) \to \Fun(BG,\Cle)
\end{equation}
of $\widehat\bV \circ (v\times j)$.

Let $G\BC^\mbn$ be the full subcategory of $G\BC$ spanned by the $G$-bornological coarse spaces with minimal bornology, i.e.~with the bornology given by the collection of finite subsets.
Then the continuous version $\bV^c$ of the functor $\bV$ is defined as the
left Kan extension of $\bV$ along the vertical inclusion functor in
\begin{equation}\label{fdsvkvsdfpvfvsfdvsvfdvsv}
\xymatrix{
 G\BC^\mbn \times \Fun(BG,\CL)\ar[r]^-{\bV}\ar[d] & \Fun(BG,\Cle) \\
 G\BC \times \Fun(BG,\CL)\ar[ur]_-{\bV^c}
}
\end{equation} 

\begin{rem}\label{weigjwoerfewv}
A subset $F$ of $X$ is called locally finite if its induced bornology is the minimal one, or equivalently, if $F$ with the induced bornological coarse structure belongs to $G\BC^\mbn$.
By \cite[proof of Lem.~4.22]{unik}, $\bV^c_\bC(X)$ is the full subcategory of $\bV_\bC(X)$ on objects $i_*M$, where $i \colon F \to X$ is the inclusion of a locally finite subset and $M$ is an object in $\bV_\bC(F)$.
\end{rem}
We define $\bV^{c,\perf}$ as the composition
\begin{align}\label{eoigjeoefvdsfsvdfv}
 \bV^{c,\perf} \colon G\BC \times \Fun(BG,\CL) &\xrightarrow{\bV^c} \Fun(BG,\Cle)
 \xrightarrow{\Idem} \Fun(BG,\Clep)\ .
\end{align}

\begin{ddd}\label{def:VcperfG}
 We define the functors 
 \[ \bV^{c,\perf,G} := \lim_{BG} \circ \bV^{c,\perf}\quad \text{and}\quad \bV^{c,\perf}_G := \mathop{\colim}\limits_{BG} \circ \bV^{c,\perf} \]
 from $\Fun(BG,\CL) \times G\BC$ to $\Clep$.
\end{ddd}

If $\Homol$ is a finitary localising invariant (see \cref{wtoigwgreerf}), then we know by \cite[Cor.~6.18]{unik} that
 $\Homol\bV^{c,\perf}_{\bC,G}$ is a hyperexcisive equivariant coarse homology theory. 
 By \cite[Prop.~5.2]{unik}, it extends $\Homol\bC_G$.
More precisely, note that in \cref{def:VcperfG} of
$\bV^{c,\perf}_G$, one can switch the order of applying  the colimit functor and the idempotent completion functor.
Then one can use the cited results from $\cite{unik}$ since the natural transformation $\Homol \to \Homol\circ \Idem$ is an equivalence.  

The result we will use in \cref{sec:famenablegroups,qrgqiorgegegergwegre,sec:dfhj} is the following.
Recall the notion of a weak module structure from \cref{giooergrefwerfwrevwerfv}.
Let $\bC$ be a an object in $\Fun(BG,\CL)$.
As before, we write $\bV^{c,\perf,G}_{\Spc^\op_*}$ for the evaluation of $\bV^{c,\perf,G}$ at the left-exact $\infty$-category $\Spc^\op_*$ considered as an object in $\Fun(BG,\CL)$ with trivial $G$-action.
Similarly, we write $\bV^{c,\perf}_{\bC,G}$ for the evaluation of $\bV^{c,\perf}_{G}$ at $\bC$.

\begin{theorem}\label{thm:fix-orbit-wmodule}\ 
 \begin{enumerate}
  \item\label{it:fix-orbit-wmodule1} The functor $\bV^{c,\perf,G}_{\Spc^\op_*}$ is $\pi_0$-excisive.  
  \item\label{it:fix-orbit-wmodule2} There is an equivalence $\bV^{c,\perf,G}_{\Spc^\op_*}(*) \simeq \Fun(BG,\Spc^{\op,\omega}_*)$.
  \item\label{it:fix-orbit-wmodule3} The functor $\bV^{c,\perf}_{\bC,G}$ admits a weak $\bV^{c,\perf,G}_{\Spc^\op_*}$-module structure $(\eta,\mu)$. 
  \item\label{it:fix-orbit-wmodule4} Under the identification from Assertion~\ref{it:fix-orbit-wmodule2}, the morphism $\eta \colon \Spc^{\op,\omega}_* \to \bV^{c,\perf,G}_{\Spc^\op_*}(*)$  in Assertion~\ref{it:fix-orbit-wmodule3} corresponds to the unique left-exact functor $\Spc^{\op,\omega}_* \to \Fun(BG,\Spc^{\op,\omega}_*)$ sending $S^0$ to $\underline{S^0}$, the pointed space  $S^0$ with the trivial $G$-action.
  \end{enumerate}
\end{theorem}

Any reader who is willing to accept \cref{thm:fix-orbit-wmodule} on good faith may directly skip ahead to \cref{sec:controlledCWs}.

We will derive \cref{thm:fix-orbit-wmodule} from a highly structured result.
Both $\widehat\bV$ and $\bV^c$ refine to lax symmetric monoidal functors (\cref{prop:loc-monoid,prop:vc-monoid}).
The theorem will then follow from the general observation that the $G$-orbits of a module with $G$-action become a module over the $G$-fixed points of its coefficient algebra (see \cref{const:fix-orbit}).
Proving these assertions is the goal of the subsequent sections.

\subsection{Monoidal preliminaries}\label{sec:prelims}
  
In analogy to \eqref{aidshviuasdvasdcascsadca}, we have a diagram
    \begin{equation}\label{aidshviuasdvasdcascsadca1}
 \xymatrix{\Crp\ar@{-->}[rr]^{\subset}\ar@{-->}[d]_{\simeq}&&\ar@{..>}[d]^{\subset}\ar@/^-.75cm/[ll]_{\Idem}\ar[lld]_{\Ind_{\omega}}\Cre\ar@{..>}[r]&\ar@{..>}[d]^{\subset}\Cati\\
 \Prlp\ar@/_{-1.5cm}/[u]_{\simeq}^{(-)^{\omega}}\ar@{-->}[r]&\CATCocplt\ar@{-->}[r]&\CRR \ar@{..>}[r]&\CATi} \ .
\end{equation}
It can be obtained from \eqref{aidshviuasdvasdcascsadca} by applying $(-)^{\op}$.
In the following we describe the entries. By $\Cre$ we denote the large $\infty$-category  of right-exact $\infty$-categories (i.e., small, pointed $\infty$-categories admitting finite colimits), and finite colimit-preserving functors.
It contains the full subcategory $\Crp$ of idempotent complete right-exact $\infty$-categories.

The ind-completion functor 
$\Ind_{\omega} \colon \Cre\to \Prlp$ restricts to an equivalence
\[ \Ind_{\omega} \colon \Crp \xrightarrow{\simeq} \Prlp\ .\]
Its inverse   
\[ (-)^{\cp} \colon \Prlp \to \Crp \]
is the functor which takes the full subcategory of compact objects.
By $\CRRR$ we denote the very large $\infty$-category of large pointed cocomplete $\infty$-categories and colimit-preserving functors. 
It is contained in the $\infty$-category $\CRR$ of
large right-exact $\infty$-categories and finite colimit-preserving functors.
Finally, the right-exact version of the  idempotent completion functor is the composition
\[ \Idem:=(-)^{\cp}\circ \Ind_{\omega}\ .\]
We now recall the symmetric monoidal structures on $\Cle$, $\Clep$ and $\CL$.
We will actually first discuss the right-exact case for which there is a good supply of references. Then we  translate to the left-exact case by applying the functor $(-)^{\op}$.

In what follows, we regard $\Cati$ as a symmetric monoidal $\infty$-category with respect to the cartesian symmetric monoidal structure.  
We let $\Catrex$ denote the subcategory of $\Cati$ of small $\infty$-categories which admit all finite colimits (we will also say finitely cocomplete), and finite colimit-preserving functors.
Applying \cite[Cor.~4.8.1.4]{HA} to the collection of finite simplicial sets, we obtain a symmetric monoidal structure $\otimes$ on $\Catrex$ such that the forgetful functor $\Catrex \to \Cati$ is lax symmetric monoidal. 
It follows from the definition of the symmetric monoidal structure  on $\Catrex$  in \cite[Not.~4.8.1.2]{HA} that the structure map
\[ \cC \times \cD \to \cC \otimes \cD \]
of the lax symmetric monoidal structure on the functor $\Catrex \to \Cati$ is the initial transformation among functors which preserve finite colimits in both variables separately.

Applying \cite[Rem.~4.8.1.9]{HA}, defining a lax symmetric monoidal functor $\bM \to \Catrex$ is equivalent to giving a lax symmetric monoidal functor $\bM \to \Cati$ which takes values in $\Catrex$ and has the property that all structure maps preserve finite colimits in each variable separately. 

Let $\Spc^\fin$ denote the smallest full subcategory of $\Spc$ which contains the final object and is closed under finite colimits.
This category enjoys the universal property that evaluation at the final object induces an equivalence \begin{equation}\label{qefqoijqiowejfqwefeqwfqw}
\Fun^\rex(\Spc^\fin,\bC) \simeq \bC
\end{equation}  for every finitely cocomplete $\infty$-category $\bC$ \cite[Rem.~1.4.2.6]{HA}.
Using the exponential law
\[ \Fun(\bC\times \Spc^{\fin},\bD)\simeq \Fun(\bC,\Fun(\Spc^{\fin},\bD)) \]
and the universal property of the tensor product in $\Catrex$ for the first equivalence below,
this implies for all $\bC$, $\bD$ in $\Catrex$ that \begin{equation}\label{ergweogjoewprgerwgewgwerg}
 \Fun^\rex(\bC \otimes \Spc^\fin, \bD) \simeq \Fun^\rex(\bC, \Fun^\rex(\Spc^\fin,\bD)) \simeq \Fun^\rex(\bC,\bD)\ .
\end{equation}
In particular, $\Spc^\fin$ is a tensor  unit   in $\Catrex$.

In the following we extend the above to pointed categories. Note that 
$\Cre$ is a full subcategory  of $\Catrex$.
Following \cite[Constr.~5.1.1]{TheNineI}, we  show: 
\begin{lem}\label{roighjioerwgerwrwegwerg} There is an adjunction
\[ - \otimes \Spc^{\fin}_{*} \colon  \Catrex\leftrightarrows \Cre \cocolon \incl\ .\]  Furthermore, $\Cre$ has a symmetric monoidal structure such that the functor  $-  \otimes \Spc^{\fin}_{*}$ has a symmetric monoidal refinement, and the functor $\incl$ has a lax symmetric monoidal refinement. 
\end{lem}
\begin{proof}
  We will show that $\Spc^{\fin}_{*}\otimes \Spc^{\fin}_{*}\simeq \Spc^{\fin}_{*}$ and 
  that the essential image of the
  functor $-\otimes\Spc^{\fin}_{*} \colon \Catrex \to  \Catrex$  is  $\Cre$.
  Then  \cite[Prop. 4.8.2.7]{HA}
    implies all assertions.
    
    Restriction along the functor
    \[ \Spc^{\fin}\to \Spc^{\fin}_{*} \ , \quad X\mapsto (*\to X\sqcup *) \]
    induces   an equivalence 
\begin{equation}\label{ergregqgfwefewf}\Fun^{\rex}(\Spc^\fin_{*},\bD) \simeq \Fun^{\rex}(\Spc^{\fin} ,\bD) \stackrel{\eqref{qefqoijqiowejfqwefeqwfqw}}{\simeq} \bD\end{equation}  
for any $\bD$ in $\Cre$. 
     
 Next we show that for every $\bC$ in $ \Catrex$ the tensor product
  $\bC\otimes  \Spc^{\fin}_{*}$ is pointed.  
  We will employ the fact that the initial object $\emptyset_{\bD}$  in a finitely cocomplete $\infty$-category $\bD$ is also terminal if and only if the constant functor $\const_{\emptyset_{\bD}} \colon \bD\to \bD$ with value $\emptyset_{\bD}$ is adjoint to itself.
  Since $\Spc^{\fin}_{*}$ is pointed, its endofunctor  $\const_{\emptyset_{\Spc^{\fin}_{*}}}$ is adjoint to itself.
  Using that $\bC\otimes- $ preserves adjunctions, we see on the one hand that $\bC\otimes \const_{\emptyset_{\Spc^{\fin}_{*}}}$ is also adjoint to itself. 
  Since $\bC\otimes -$ preserves the empty colimit, we see on the other hand that 
$\bC\otimes \const_{\emptyset_{ \Spc^{\fin}_{*}}}\simeq \const_{\emptyset_{\bC\otimes \Spc^{\fin}_{*}}}$.
Hence  $\bC\otimes  \Spc^{\fin}_{*}$ is pointed.

 For every $\bD$ in $\Cre$ we have  a natural equivalence
\[ \Fun^\rex(\bC \otimes \Spc^\fin_*,\bD) \simeq \Fun^\rex(\bC, \Fun^\rex(\Spc^\fin_*,\bD)) \stackrel{\eqref{ergregqgfwefewf}}{\simeq} \Fun^\rex(\bC, \bD)\ ,\]
where the first equivalence is seen similarly as in \eqref{ergweogjoewprgerwgewgwerg}.  
If $\bC$ was already pointed, then this equivalence for arbitrary $\bD$ provides an equivalence
$\bC\otimes \Spc^{\fin}_{*}\simeq \bC$.  Applying this relation for $\bC\simeq \Spc^{\fin}_{*}$ we obtain the desired equivalence   
$\Spc^{\fin}_{*}\otimes \Spc^{\fin}_{*}\simeq \Spc^{\fin}_{*}$.
Furthermore,  the essential image of $-\otimes \Spc_{*}^{\fin}$ is $\Cre$. 
\end{proof}

\begin{lem}\label{iuzqfiqufewfqfewfqewfq}
The datum of a lax symmetric monoidal functor 
$\bM\to \Cre$ is equivalent to a lax symmetric monoidal functor $\bM\to  \Cati$ which takes values in $\Cre$ and whose structure maps preserve finite colimits in each variable separately.
\end{lem}
\begin{proof}
The composition of inclusions
\begin{equation}\label{cnajkdnckjsdcasdcasdc}
 \Cre \xrightarrow{\incl} \Catrex \to \Cati
\end{equation} 
is lax symmetric monoidal by \cref{roighjioerwgerwrwegwerg} and the preceding discussion.
In addition, we know that the structure maps 
of the lax symmetric monoidal structure  of  \eqref{cnajkdnckjsdcasdcasdc}   preserve finite colimits in each variable.
Therefore, by postcomposing with \eqref{cnajkdnckjsdcasdcasdc},
a lax symmetric monoidal functor $\bM \to \Cre$ gives a lax symmetric monoidal functor
$\bM \to \Cati$ with the desired properties.

For the converse, let $\bM\to \Cati$ be a functor  with the listed properties.
As discussed above, it gives rise to a lax symmetric monoidal functor $F \colon \bM\to \Catrex$.
Since $F$ takes values in $\Cre$, the unit of the adjunction in  \cref{roighjioerwgerwrwegwerg} provides an equivalence
\[ F(-)\simeq F(-)\otimes \Spc^{\fin}_{*}\ .\]
Since the functor  $-\otimes \Spc^{\fin}_{*}$ in  \cref{roighjioerwgerwrwegwerg} is symmetric monoidal, we see that the right-hand side of this equivalence and therefore $F$ is actually a lax symmetric monoidal functor $\bM \to \Cre$.
\end{proof}

The inclusion of the  full subcategory $\Crp$ of $\Cre$ spanned by the idempotent complete right-exact $\infty$-categories is the right adjoint of a  localisation
\[ \Idem \colon \Cre \to \Crp \cocolon \incl\ . \]
Unwinding universal properties,  the  canonical functors $\bC\to \Idem(\bC)$ and $\bD\to \Idem(\bD)$  induce an equivalence
\[ \Idem(\bC \otimes \bD) \xrightarrow{\simeq} \Idem(\Idem(\bC) \otimes \Idem(\bD)) \]
for all $\bC$ and $\bD$ in $\Cre$. 
It follows from this equivalence that $\Idem$ is compatible with the monoidal structure in the sense of \cite[Def. 2.2.1.6 and Ex. 2.2.1.7]{HA}.
By \cite[Prop.~2.2.1.9]{HA}, we conclude that  $\Crp$ inherits a symmetric monoidal structure from $\Cre$ and $\Idem$ acquires a symmetric monoidal  refinement.
Since  $\Spc^\cp_*\simeq  \Idem(\Spc^{\fin}_{*} )$, the tensor unit is   given by  the object $\Spc^\cp_*$ of $\Crp$.
   
The functor $\Ind_\omega \colon \Crp \to \Prlp$ is an equivalence of categories with inverse $(-)^\cp \colon \Prlp \to \Crp$.
By transport of structure, we obtain an induced symmetric monoidal structure on $\Prlp$ such that $\Ind_{\omega}$ and $(-)^{\cp}$ become   symmetric monoidal equivalences.   The tensor unit of the structure on $\Prlp$ is given by the object $\Spc_{*}$.
 
The functor $(-)^{\op}$ taking the opposite category identifies the categories $\Cle$, $ \Clep$ and $\CL$ with $\Cre$, $\Crp$ and $\Prlp$, respectively. 
By transport of structure, the latter categories acquire symmetric monoidal structures such that the functors $\Idem \colon \Cle \to \Clep$ and $\mathrm{Pro}_{\omega} \colon \Clep\stackrel{\simeq}{\to} \CL$ admit symmetric monoidal refinements.

 \begin{rem}\label{roigjioergerergergergwegrwerg}
 It is a consequence of the version of \cref{iuzqfiqufewfqfewfqewfq} for large categories
  that a lax symmetric monoidal functor $\bM \to \CLL$ is the same as a lax symmetric monoidal functor $F \colon \bM \to \CATi$ which takes values in $\CLL$ and has the property that all structure maps preserve finite limits in each variable separately.
 It suffices to checks this property for  the structure maps $F(M) \times F(M') \to F(M\otimes M')$ for all $M$ and $M'$ in $\bM$.  
 
The unit constraint of $F$ (as a lax symmetric monoidal  functor with values in $\CATi$) is a functor $u:*\to F(*)$ determined by an object $u(*)$ in $F(*)$.
Then  the unit constraint of the corresponding lax symmetric monoidal  functor 
$\bM\to \CLL$ is the essentially unique left exact functor
$\Sp^{\fin,\op}_*\to F(*)$ which sends $S^{0}$ to $u(*)$.
 \end{rem}

In the following we record the compatibility of Dwyer--Kan localisation with lax symmetric monoidal functors.
Let $\cC$ be a symmetric monoidal $\infty$-category, and let $F \colon \cC \to \Cati$ be a lax symmetric monoidal functor.
Assume that for every object $C$ in $\cC$ we are given a class $W_C$ of morphisms in $F(C)$.
We say that a functor $F(C)\to \cD$ inverts $W_{C}$ if it sends the morphisms in $W_{C}$ to equivalences.

\begin{lem}\label{wefkijeqwofqlewfeqwfqfqeewfq}
 Suppose:  
 \begin{enumerate}
  \item\label{it:wefkijeqwofqlewfeqwfqfqeewfq1} $F(C) \xrightarrow{F(f)} F(C') \to F(C')[W_{C'}^{-1}]$ inverts  $W_C$ for all morphisms $f \colon C \to C'$ in $\cC$;
  \item\label{it:wefkijeqwofqlewfeqwfqfqeewfq2} $F(C) \times F(C') \to F(C \otimes C') \to F(C \otimes C')[W_{C \otimes C'}^{-1}]$ inverts  $W_C \times W_{C'}$ for all $C$, $C'$ in $\cC$.
 \end{enumerate}
 Then there exists a lax symmetric monoidal functor $\overline{F} \colon \cC \to \Cati$ such that
\begin{enumerate}
 \item the value of $\overline{F}$ at $C$ is equivalent to the Dwyer--Kan localisation $F(C)[W_C^{-1}]$;
 \item for every morphism $f \colon C \to C'$, its image $\overline{F}(f)$ is essentially uniquely determined by being part of a commutative diagram
 \[\xymatrix{
  F(C)\ar[r]^-{F(f)}\ar[d] & F(C')\ar[d] \\
  \overline{F}(C)\ar[r]^-{\overline{F}(f)} & \overline{F}(C')
 }\]
 in which the vertical maps are the localisation functors.
\end{enumerate} 
\end{lem}
\begin{proof}
 Consider $F$ as a cocartesian fibration of $\infty$-operads $F^\otimes \to \cC^\otimes$.
 For $C$ in $\cC$, let $\overline{W}_C$ denote the collection of all morphisms in $F(C)$ which become invertible in $F(C)[W_C^{-1}]$.
 Regard $\cC^\otimes$ as a marked $\infty$-category $(\cC^\otimes,\iota\cC^\otimes)$ by marking all equivalences.
 Using the equivalences $F^\otimes_{(C_1,\ldots,C_n)} \simeq F(C_1) \times \ldots \times F(C_n)$, each fibre  inherits a marking $\overline{W}_{C_1} \times \ldots \times \overline{W}_{C_n}$.
 Let $W$ be the marking of $F^\otimes$ generated by these markings (see \cite[Rem.~2.1.2]{hinich} for a more explicit description).
 The assumptions ensure that this defines a marked cocartesian fibration $(F^\otimes,W) \to (\cC^\otimes, \iota\cC^\otimes)$ in the sense of \cite[Def.~2.1.1]{hinich}.
 By \cite[Prop.~2.1.4]{hinich}, we get a cocartesian fibration of $\infty$-operads $F^\otimes[W^{-1}] \to \cC^\otimes$ (see also \cite[Prop.~3.2.2]{hinich}){. S}ince $F(C)[W_C^{-1}] \simeq F(C)[\overline{W}_C^{-1}]$, this cocartesian fibration corresponds to the desired functor $\overline{F}$. 
\end{proof}

The next lemma will be used in \cref{sec:orbits-and-fixed-points}.
Let $\bC$ be a small monoidal $\infty$-category and $\bD$ be a cocomplete monoidal $\infty$-category.
Let $\bI$ be any small $\infty$-category.
Then $\Fun(\bC,\bD)$ becomes a monoidal $\infty$-category via Day convolution.
As a reminder, the tensor product of $F$ and $F'$ in $\Fun(\bC,\bD)$ is given by a left Kan extension
\begin{equation}\label{eq:dayconvo}
\xymatrix@C=6em{
\bC \times \bC\ar@{=>}[dr(.45)]_-{\tau}\ar[r]^-{F(-) \otimes_\bD F'(-)}\ar[d]_-{\otimes_\bC} & \bD \\
\bC\ar[ru]_{F \otimes F'} & } \ ,
\end{equation}
where $F(-) \otimes_\bD F'(-)$ denotes the pointwise tensor product of the functors $F$ and $F'$.
Both $\Fun(\bI, \bD)$ and $\Fun(\bI, \Fun(\bC,\bD))$ carry an induced (pointwise) monoidal structure.
The pointwise structure on $\Fun(\bI,\bD)$ in turn gives rise to a monoidal structure on $\Fun(\bC,\Fun(\bI,\bD))$, again via Day convolution.

\begin{lem}\label{lem:dayconvolution-vs-ptwise}
 The exponential law
 \[  \Fun(\bC,\Fun(\bI,\bD)) \simeq \Fun(\bI, \Fun(\bC,\bD)) \]
 refines to an equivalence of monoidal $\infty$-categories. 
\end{lem}
\begin{proof}
 We will show that the exponential law induces an equivalence between the $\infty$-categories of $\cO$-algebras for every $\infty$-operad $\cO$ over the associative $\infty$-operad $\Assoc$. 

  Let $\bM$ be a cocomplete monoidal $\infty$-category which we interpret as a cocartesian fibration $\bM^\otimes \to \Assoc$.
  As an $\infty$-operad, the pointwise monoidal structure on $\Fun(\bI,\bM)$ is given by the pullback
  \[ \Fun(\bI,\bM)^\otimes := \Fun(\bI,\bM^\otimes) \times_{\Fun(\bI,\Assoc)} \Assoc\ .\]
  Let $\cO \to \Assoc$ be a morphism of $\infty$-operads.  
  The relative case of \cite[Rem.~2.1.3.4]{HA} yields a canonical equivalence
  \begin{align}\label{eq:pointwise-modules}
   \Alg_{\cO / \Assoc}(\Fun(\bI,\bM)^\otimes) \simeq \Fun(\bI, \Alg_{\cO / \Assoc}(\bM^\otimes))\ .
  \end{align}
  Let $\bN$ be a small monoidal $\infty$-category and consider the Day convolution $\Fun(\bN,\bM)^\otimes \to \Assoc$ as defined in \cite[Ex.~2.2.6.10]{HA}.
  Then the universal property of $\Fun(\bN,\bM)^\otimes$ given in \cite[Def.~2.2.6.1]{HA} specialises to an equivalence
  \begin{align}\label{eq:day-modules}
   \Alg_{\cO / \Assoc}(\Fun(\bN,\bM)^\otimes) \simeq \Alg_{(\cO \times_{\Assoc} \bN^\otimes) / \Assoc}(\bM^\otimes)\ .
  \end{align}
  Therefore, we have equivalences
  \begin{align*}
   \Alg_{\cO / \Assoc}(\Fun(\bC,\Fun(\bI,\bD))^\otimes) &\overset{\eqref{eq:day-modules}}{\simeq} \Alg_{(\cO \times_{\Assoc} \bC^\otimes)/\Assoc}(\Fun(\bI,\bD)^\otimes) \\
   &\overset{\eqref{eq:pointwise-modules}}{\simeq} \Fun(\bI, \Alg_{(\cO \times_{\Assoc} \bC^\otimes)/\Assoc}(\bD^\otimes)) \\
   &\overset{\eqref{eq:day-modules}}{\simeq} \Fun(\bI, \Alg_{\cO / \Assoc}(\Fun(\bC,\bD)^\otimes)) \\
   &\overset{\eqref{eq:pointwise-modules}}{\simeq} \Alg_{\cO / \Assoc}(\Fun(\bI, \Fun(\bC,\bD))^\otimes)\ .\qedhere
   \end{align*}
\end{proof}

For a monoidal $\infty$-category $\cC$, denote by $\Mod(\cC)$ the $\infty$-category of (left) module objects in $\cC$ in the sense of \cite[Def.~4.2.1.13]{HA}.
Informally, an object of $\Mod(\cC)$ is a pair $(A,M)$ consisting of an associative algebra $A$ and an $A$-module $M$.
In particular, there is a forgetful functor $\Mod(\cC) \to \Alg(\cC)$, where the target is the $\infty$-category of associative algebras in $\cC$.

\begin{kor}\label{lem:modules-dayconvolution}
 There exists a commutative diagram
 \[\xymatrix{
  \Mod(\Fun(\bC,\Fun(\bI,\bD)))\ar[r]^-{\simeq}\ar[d] & \Mod(\Fun(\bI, \Fun(\bC,\bD)))\ar[d] \\
  \Alg(\Fun(\bC,\Fun(\bI,\bD)))\ar[r]^-{\simeq} & \Alg(\Fun(\bI, \Fun(\bC,\bD)))  
 }\]
 in which the horizontal functors are equivalences and the vertical arrows are the forgetful functors.
\end{kor}
\begin{proof}
 Consider the associative $\infty$-operad $\Assoc$ and the $\infty$-operad $\cL\cM$ from \cite[Def.~4.2.1.7]{HA}.
 As explained in \cite[Ex.~4.2.1.16]{HA}, the $\infty$-category of (left) module objects in a monoidal $\infty$-category $\bM$ is given by
  \[ \Mod(\bM) := \Alg_{\cL\cM / \Assoc}(\bM) \]
  and the $\infty$-category of associative algebras in $\bM$ is given by
  \[ \Alg(\bM) := \Alg_{\Assoc/\Assoc}(\bM)\ .\] 
 Since $\Assoc$ is a suboperad of $\cL\cM$, the forgetful functors arise by precomposition with the inclusion $\Assoc \to \cL\cM$.
 Applying \cref{lem:dayconvolution-vs-ptwise} in the cases $\bM = \Fun(\bC,\Fun(\bI,\bD))$ and $\bM = \Fun(\bI, \Fun(\bC,\bD))$ proves the corollary.
\end{proof}

\subsection{Controlled objects as symmetric monoidal functors}\label{sec:controlled-objects-monoidal}

In this section, we refine (most of) the functors whose constructions were recalled  in \cref{sec:controlled-objects} to lax symmetric monoidal functors.
Let
\[ k:=i\circ j\colon \CL\to \CATi \]
denote the inclusion.
  Then we are looking for a refinement of 
  the  functor
  \begin{equation}\label{avsvacsdcacdc}
\PSh\circ (\id\times k)\colon \Set\times \CL\to \CL\ .
\end{equation}
 {to a lax symmetric monoidal functor} such that the product $M \otimes M'$ of $M$ in $\PSh_{\bC}(X)$ and $M'$ in $\PSh_{\bC'}(X')$ is given by a right Kan extension
\begin{equation}\label{vwervlnlvwevewwerv}
\xymatrix@C=4em{
\cP_{X}^{\op}\times \cP^{\op}_{X'}\ar@{<=}[dr(.45)]\ar[r]^-{M\widehat\otimes M'}\ar[d] & \bC\otimes \bC'\ , \\
\cP^{\op}_{X\times X'}\ar[ru]_{M\otimes M'} & }
\end{equation}
where $M \widehat\otimes M'$ denotes the composition
\begin{equation}\label{eq:exterior-tensor}
 M \widehat\otimes M' \colon \cP_X^\op \times \cP_{X'}^\op \to \bC \times \bC' \to \bC \otimes \bC'\ .
\end{equation}
Moreover, the unit constraint of the lax symmetric monoidal structure will be given by the essentially unique functor $\Spc^{\fin,\op}_* \to \PSh_{\Spc^\op_*}(*)$ sending $S^0$ to the object
\begin{equation}\label{eq:unit-constraint}
 \cU \colon \cP_*^\op \to \Spc_*, \quad \begin{cases} * \mapsto S^0\ , \\ \emptyset \mapsto * \end{cases}
\end{equation}
of $\PSh_{\Spc^\op_*}(*)$.

\begin{rem}\label{eroigwergwgrgwefg}
We use the language of $\infty$-operads developed in \cite{HA} to model symmetric monoidal $\infty$-categories and lax symmetric monoidal functors between them. 
We will use the language of $\infty$-operads only in the present section.
In this language, a symmetric monoidal structure on an $\infty$-category $\cC$ is given by {a cocartesian fibration of $\infty$-operads} $\cC^{\otimes}\to \Fin_{*}$ together with an equivalence between $\cC$ and the fibre $\cC^{\otimes}_{\langle 1 \rangle}$ of $\cC^{\otimes}$ over $\langle 1 \rangle$ in $\Fin_{*}$.
A lax symmetric monoidal refinement of a functor $F \colon \cC \to \cD$  is then given by an operad map
$F^{\otimes} \colon \cC^{\otimes}\to \cD^{\otimes}$ together with an equivalence between $F$ and the induced map 
$F^{\otimes}_{\langle 1 \rangle} \colon \cC^{\otimes}_{\langle 1 \rangle} \to \cD^{\otimes}_{\langle 1 \rangle}$.

{Let us also remark that the datum a cocartesian fibration of $\infty$-operads $\cC^\otimes \to \cO^\otimes$ is equivalent to an operad map $\cO^\otimes \to \Cati^\times$, where the latter is equipped with the cartesian symmetric monoidal structure (combine Def.~2.1.2.13, Rem.~2.4.2.4 and Prop.~2.4.2.5 of \cite{HA}).}

Elsewhere in this paper, we will simply speak about symmetric monoidal $\infty$-categories and lax symmetric monoidal functors.
\end{rem}

\begin{prop}\label{weogjowergerregegwerg}
A lax symmetric monoidal refinement of $\PSh\circ (\id\times k)$ in \eqref{avsvacsdcacdc} with the properties described above exists.
\end{prop}
\begin{proof}
We consider the functor
\[ \cP\colon\Set\to \Cat \]
that sends a set $X$ to the poset $\cP_{X}$ and a map $f:X\to Y$ to the image map $f(-):\cP_{X}\to \cP_{Y}$ (note that the definition on morphisms is different  from the functor $\cP^{*}$ in \eqref{ascjiaosdcjaoisdcasdcadsc}).
 The lax symmetric monoidal structure on the functor $\cP$ is given by the product maps
\[ \cP_X \times \cP_Y \to \cP_{X \times Y},\quad (A,B) \mapsto A \times B \]
together with the unit constraint
\[ u_\cP \colon \{*\} \to \cP_* = \{\emptyset \to \{*\}\},\quad * \mapsto \{*\}\ .\]
We denote by $\Set^\times$ the $\infty$-operad corresponding to the symmetric monoidal category $\Set$ with  the cartesian structure.
Similarly, we let $\Prlpo$ denote the $\infty$-operad corresponding to the symmetric monoidal $\infty$-category $\Prlp$ explained in \cref{sec:prelims} with the symmetric monoidal structure inherited from $\Crp$ via the equivalence $\Ind \colon \Crp\xrightarrow{\simeq} \Prlp$.
 
 We now consider the lax symmetric monoidal functor $\cP$ as an operad map $\cP^{\times} \colon \Set^{\times}\to \Cat^{\times}$ and form the fibre product of $\infty$-operads 
\[ \cO^\otimes := \Set^\times \mathop{\times}\limits_{\Fin_*} \Prlpo \ .\] 
The projections to the individual factors of $\cO^\otimes$ induce operad maps
\[ Q \colon \cO^\otimes \to  \Prlpo 
\to \CATi^\times \]
and
\[ P \colon \cO^\otimes \to \Set^\times \xrightarrow{\cP^{\times}} \Cat^\times \to \Cati^\times\ .\]
{As recalled in \cref{eroigwergwgrgwefg}, the operad maps $Q$ and $P$ correspond to cocartesian fibrations of $\infty$-operads} $q \colon Q^\otimes \to \cO^\otimes$ and $p \colon P^\otimes \to \cO^\otimes$, respectively.
Applying \cite[Constr.~2.2.6.7]{HA}, we obtain a fibration of $\infty$-operads \begin{equation}\label{kjqwekfljfklafdfa}
r\colon \Fun^\cO(P,Q)^\otimes \to \cO^\otimes\ ,
\end{equation} where we use the notation $\Fun^\cO(P,Q)^\otimes$ as in {\em loc.~cit.}  
Since the functor $Q$ factors over $\Prlpo$ by construction, 
the assumptions of  \cite[Prop.~2.2.6.16]{HA} are satisfied.
{Hence it follows that} $r$ is a cocartesian fibration of $\infty$-operads.

The straightening of $r$ is a lax symmetric monoidal refinement of a functor
\[ \bR \colon \Set \times \Prlp \to \CATi\ .\]
Recall that the symmetric monoidal structure on
$\CL$ is induced from the symmetric monoidal structure on
$\Prlp$ via the equivalence $\CL\simeq \Prlp$.
The lax symmetric monoidal refinements of the functors
\begin{equation}\label{regwjgergoerwgwegwergerg1}
 \Set \times \CL \xrightarrow{\id \times (-)^\op} \Set \times \Prlp \xrightarrow{\bR} \CATi \xrightarrow{(-)^\op} \CATi
\end{equation}
induce a lax symmetric monoidal structure on the composition.
In the following, we verify the conditions stated in \cref{roigjioergerergergergwegrwerg}
to show that the functor in \eqref{regwjgergoerwgwegwergerg1}  refines to a lax symmetric monoidal functor with values in $\CLL$.

 For $X$ in $\Set$ and $\bC$ in $\CL$, we have
 \begin{equation}\label{oijoierjgwegrerewgw}
 \bR(X,\bC^\op)^\op \simeq \Fun(\cP_X^\op,\bC)
 \end{equation} by \cite[Rem.~2.2.6.8]{HA}.  Since the right-hand side of  \eqref{oijoierjgwegrerewgw} clearly belongs to $\CLL$    the  values of the functor in \eqref{regwjgergoerwgwegwergerg1}  on objects   belong to $\CLL$.

 The last paragraph of the proof of \cite[Cor.~2.2.6.14]{HA} provides the following description of the lax symmetric monoidal structure of $\bR$.
 A point in  $\Mul_\cO((X_i,\bC_i)_{i = 1,\ldots,n}, (Y,\bD))$  is given by a map $f \colon X_1 \times \ldots \times X_n \to Y$ and a functor $\phi \colon \bC_1 \times \ldots \times \bC_n \to \bD$ which preserves colimits in each variable separately and preserves compact objects.
 The structure map
 \begin{equation}\label{qwefeewewfewfewfffqef}
\otimes_{f,\phi} \colon \prod_{i=1}^n \bR (X_i,\bC_{i}) \to \bR(Y,\bD)
\end{equation}
sends a collection $(M_i)_{i = 1,\ldots,n}$ of functors $M_i \colon \cP_{X_i} \to \bC_i$ to a functor $M_1 \otimes \ldots {\otimes} M_n \colon \cP_Y \to \bD$ which fits into a left Kan extension diagram
 \begin{equation}\label{revwrvwrvwevervvwe}
\xymatrix{
  \prod_{i = 1}^n \cP_{X_i}\ar[r]^-{\prod_i M_i}\ar[d]_-{f(-) \circ \times}\ar@{=>}[drr(.45)] & \prod_{i=1}^n \bC_i \ar[r]^-{\phi} & \bD & \\
  \cP_Y\ar[urr]_-{\hphantom{xx}M_1 \otimes \ldots \otimes M_n} & &
 }
\end{equation}
 We need the special cases for $n=1,2$ of this observation. 
 
 We start with $n=1$.
 Consider morphisms $f \colon X \to Y$ in $\Set$ and $\phi \colon \bC \to \bD$ in $\Prlp$.
 For $M$ in $\bR(X,\bC)$, the image $\bR(f,\phi)(M)$ in $\bR(Y,\bD)$ is equivalent to the left Kan extension  
 \[\xymatrix{
  \cP_X\ar[r]^-{M}\ar[d]_-{f(-)} & \bC \ar[r]^-{\phi}  & \bD \\
  \cP_Y\ar[urr]_-{\hphantom{xx}\bR(f,\phi)(M)} & &
 }\]
 Since $f(-)$ has a right adjoint, namely the preimage functor $f^{-1}$, 
 the equivalence from \eqref{oijoierjgwegrerewgw} identifies this left Kan extension with the composition
 \begin{equation}\label{fqwefeeqfefqefqwefqewf}
\phi \circ - \circ f^{-1} \colon \Fun(\cP^{\op}_{X},\bC)\to \Fun(\cP^{\op}_{Y},\bD)\ .
\end{equation}
Since this is a morphism in $\CLL$, the functor in \eqref{regwjgergoerwgwegwergerg1} takes values in $\CLL$ on morphisms, too.
    
   In the case $n=2$, the left Kan extension from \eqref{revwrvwrvwevervvwe} corresponds to the right Kan extension in \eqref{vwervlnlvwevewwerv}.
  Since the structure maps of the lax symmetric monoidal structure of the functor in \eqref{regwjgergoerwgwegwergerg1} arise as right Kan extensions, they  preserve finite limits in each argument.
Therefore, we can conclude by \cref{roigjioergerergergergwegrwerg} that the functor in \eqref{regwjgergoerwgwegwergerg1} refines to a lax symmetric monoidal functor with values in $\CLL$. 

 Next we compute the unit constraint of the lax symmetric monoidal functor $\bR$.
Recall that the tensor unit of $\Set \times \Prlp$ is given by the pair $(\{*\},\Spc_*)$
which we regard as a functor $\mathbf{pt} \xrightarrow{(\{*\},\Spc_*)} \Set \times \Prlp$ between the fibres of the cocartesian fibration $\cO^\otimes \to \Fin_*$ over $\langle 0 \rangle$ and $\langle 1 \rangle$, respectively. Here we denote by $\mathbf{pt}$ a  point in the fibre of $\cO^{\otimes}\to \Fin_{*}$ over $\langle 0\rangle$.
{The map $r$ from \eqref{kjqwekfljfklafdfa} has fibre $\Fun(\{*\},\{*\}) \cong \{*\}$ over $\mathbf{pt}$}, while the fibre over $(*,\Spc_*)$ is $\Fun(\cP_*,\Spc_*)$.

Using the description of cocartesian lifts given at the end of the proof of \cite[Cor.~2.2.6.14]{HA}, the unit constraint $u_\bR$ of $\bR$ is given by the left Kan extension
\[\xymatrix{
 \{*\}\ar[d]_-{u_\cP}\ar[r]^-{S^0} & \Spc_* \\
\cP_* \cong \{ \emptyset \to *\}\ar[ru]_-{u_\bR} & &
}\]
of the unit of $\Spc_*$ along the unit constraint $u_\cP$ of the lax symmetric monoidal structure on $\cP$.
Since $u_\cP$ sends $*$ to $*$, the pointwise formula shows that $u_\bR$ is the essentially unique functor sending $\emptyset$ to $*$ and $*$ to $S^0$.

 The specialisation of \eqref{revwrvwrvwevervvwe} to the case $n=2$ shows that the tensor product of presheaves is given by the right Kan extension in \eqref{vwervlnlvwevewwerv}.
 Furthermore, the {preceding} calculation shows that the unit constraint is determined by the presheaf $\cU$ from \eqref{eq:unit-constraint}.

The only thing left to show is that there is an equivalence of functors
\[ \bR \simeq \Fun \circ (\cP^{*,\op} \times \inc) \colon \Set \times \Prlp \to \CATi\ , \]
where $\inc \colon \Prlp \to \CATi$ is the inclusion functor.
As explained in \cite[Rem.~2.2.6.8]{HA}, $\Fun^\cO(P,Q)^\otimes$ comes equipped with a morphism of $\infty$-operads $\alpha \colon \Fun^\cO(P,Q)^\otimes \times_{\cO^\otimes} P^\otimes \to Q^\otimes$.
Restricting to the fibres over $\langle 1 \rangle$ in $\Fin_*$, this morphism induces a morphism
\[ \alpha_{\langle 1 \rangle} \colon \Fun^\cO(P,Q)^\otimes_{\langle 1 \rangle} \times_{\cO^\otimes_{\langle 1 \rangle}} P^\otimes_{\langle 1 \rangle} \to Q^\otimes_{\langle 1 \rangle} \]
of $\infty$-categories over $\cO^{\otimes}_{\langle 1 \rangle} \simeq \Set \times \Prlp$.
The morphism $\alpha^\otimes_{\langle 1 \rangle}$ corresponds to a morphism
\[ \widehat\alpha \colon \Fun^\cO(P,Q)^\otimes_{\langle 1 \rangle} \to (Q^\otimes_{\langle 1 \rangle})^{P^\otimes_{\langle 1 \rangle}} \]
to the exponential object in ${\CATi}_{/\cO^\otimes_{\langle 1 \rangle}}$.
Since the functor $\cP \colon \Set \to \Cat$ takes values in the subcategory of left adjoint functors, the morphism $p_{\langle 1 \rangle} \colon P^\otimes_{\langle 1 \rangle} \to \cO^\otimes_{\langle 1 \rangle}$ is both a cocartesian and a cartesian fibration.
It follows from \cite[Cor.~A.3.10]{hms:coisotropic} that $(Q^\otimes_{\langle 1 \rangle})^{P^\otimes_{\langle 1 \rangle}}$ is given by the cocartesian unstraightening of the functor
\[ \cO^\otimes_{\langle 1 \rangle} \simeq \Set \times \Prlp \xrightarrow{\cP^{*,\op} \times \inc} \Cati^\op \times \CATi \xrightarrow{\Fun} \CATi\ .\]
 Since  $\bR$ is defined as the straightening of $\Fun^\cO(P,Q)^\otimes_{\langle 1 \rangle}$, it suffices to show that $\widehat\alpha$ is an equivalence of cocartesian fibrations. It is a reformulation of 
  \eqref{oijoierjgwegrerewgw} that  $\widehat\alpha$ is a fibreweise equivalence.
Hence we are left with showing that $\widehat\alpha$ preserves cocartesian morphisms.
As explained above, a cocartesian lift $g$ of a morphism $(f,\phi) \colon (X,\bC) \to (Y,\bD)$ in $\Set \times \Prlp$ to a morphism in $\Fun^\cO(P,Q)^\otimes_{\langle 1 \rangle}$  is a functor $M \colon \cP_X \to \bC$ together with the data of a left Kan extension
\[\xymatrix{
  \cP_X\ar[r]^-{M}\ar[d]_-{f(-)} & \bC \ar[r]^-{\phi}  & \bD \\
  \cP_Y\ar[urr]_-{\hphantom{xx}\bR(f,\phi)(M)} & &
 } .\]
Because of  $\bR(f,\phi)(M) \simeq \phi \circ M \circ f^{-1} \simeq \Fun(\cP^*(f),\phi)(M)$ the morphism  $\widehat\alpha(g)$ is cocartesian  
in $(Q^\otimes_{\langle 1 \rangle})^{P^\otimes_{\langle 1 \rangle}}$.
This completes the proof of \cref{weogjowergerregegwerg}.
\end{proof}

  The forgetful functor $u \colon \Coarse \to \Set$ 
  has a symmetric monoidal structure, and the functor
    $\CL\to \CLL$ has a lax symmetric monoidal  structure, so
\[  \PSh\circ (u \times k) \colon \Coarse \times \CL \to \CLL \]
also inherits a lax symmetric monoidal structure.

\begin{prop}\label{prop:sh-monoid}
 The full subfunctor 
 \[ \Sh\circ (\id\times j) \colon \Coarse \times \CL \to \CLL \]
 of $\PSh \circ (u \times k)$ inherits a lax symmetric monoidal structure.
\end{prop}
\begin{proof}
 By \cite[Prop.~2.2.1.1]{HA}, it is enough to show the following:
 \begin{enumerate}
  \item\label{it:sh-monoid1} the unit constraint $\Spc^{\fin,\op}_* \to \PSh_{\Spc^\op_*}(*)$ takes values in the full subcategory $\Sh_{\Spc^\op_*}(*)$; 
  \item\label{it:sh-monoid2} for $M_{1}$ in $\Sh_{\bC_{1}}(X_{1})$ and $M_{2}$ in $\Sh_{\bC_{2}}(X_{2} )$ the tensor product $M_{1}\otimes M_{2}$  belongs to $\Sh(X_{1}\times X_{2},\bC_{1}\otimes \bC_{2})$. 
 \end{enumerate} 
 
 For \ref{it:sh-monoid1}, first observe by inspection that the presheaf $\cU$ described by \eqref{eq:unit-constraint} is a sheaf.  
 Since $\cU$ is the image of $S^{0}$ under the unit constraint,
 $\Spc^{{\fin},\op}_*$ is generated by $S^{0}$ under finite limits, and
 $\Sh_{\Spc^\op_*}(*)$ is closed under finite limits, we conclude that   the unit constraint
 takes values in  $\Sh_{\Spc^\op_*}(*)$.

 For \ref{it:sh-monoid2},
 we first choose coarse entourages $U_1$ of $X_{1}$ and $U_2$ of $X_{2}$ such that $M_1$ is a $U_1$-sheaf and $M_2$ is a $U_2$-sheaf. 
We will show that $ M_{1}\otimes M_{2}$ is a $U_{1}\times U_{2}$-sheaf.
Recall that $M_{1}\otimes M_{2}$ is given by the right Kan extension
 \[\xymatrix@C=4em{
  \cP_{X_1}^\op \times \cP_{X_2}^\op\ar[r]^-{M_1 \widehat\otimes M_2}\ar[d]_-{\times} & \bC_1 \otimes \bC_2 \\
  \cP_{X_1 \times X_2}^\op\ar[ur]_-{\hphantom{xx}M_1 \otimes M_2}
 }\ ,\]
 where $\widehat\otimes$ was defined in \eqref{eq:exterior-tensor}.
Let $M_1'$ and $M_2'$ be the restrictions of $M_1$ and $M_2$ to $\cP_{X_1}^{U_1\bd,\op}$ and $\cP_{X_2}^{U_2\bd,\op}$, respectively.
 Since the functor $\bC_1 \times \bC_2 \to \bC_1 \otimes \bC_2$ preserves limits in each variable separately and sheaves are characterised as right Kan extensions (see \eqref{ervervewevevwevwee}), $M_1 \widehat\otimes M_2$ is a right Kan extension of the functor
 \[ M_1' \widehat\otimes M_2' \colon \cP_{X_1}^{U_1\bd,\op} \times \cP_{X_2}^{U_2\bd,\op} \xrightarrow{M_1' \times M_2'} \bC_1 \times \bC_2 \to \bC_1 \otimes \bC_2\] along the upper horizontal map in the commutative square 
\begin{equation}\label{wrgwergfwerfewrfrefw}
\xymatrix{
  \cP_{X_1}^{U_1\bd,\op} \times \cP_{X_2}^{U_2\bd,\op}\ar[r]\ar[d]_-{\times} & \cP_{X_1}^{\op} \times \cP_{X_2}^{\op}\ar[d]^-{\times} \\
  \cP_{X_1 \times X_2}^{(U_1 \times U_2)\bd,\op}\ar[r]^-{j} & \cP_{X_1 \times X_2}^{\op}
 }\ .
\end{equation}   
 By definition, $M_1 \otimes M_2$ is a right Kan extension of $M_1' \widehat\otimes M_2'$ along the {composition along the top right corner} in  \eqref{wrgwergfwerfewrfrefw}.
 By commutativity of this square,
 $M_1 \otimes M_2$ is also a right Kan extension of $M_1' \widehat\otimes M_2'$ along the composition {along the bottom left corner}  in  \eqref{wrgwergfwerfewrfrefw}.
   Since $j$ is fully faithful, $M_{1} \otimes M_{2}$ is also equivalent to the right Kan extension of $j^{*} (M_{1}\otimes M_{2})$ and therefore a $(U_1 \times U_2)$-sheaf.
 This finishes the verification of  Condition~\ref{it:sh-monoid2}.
\end{proof}

Recall that by \cref{eporgjwepogrwegregwer} the  functor $\widehat\bV $ in \eqref{wergwregvfvfdsv} 
is a full subfunctor of the functor
 $\wt\bV\circ (\id\times i)$ 
in   \eqref{wrfwffwefwefwefwefweff}.

\begin{prop}\label{prop:loc-monoid}
The functor
\[ \wt\bV \circ (\id\times k) \colon \Coarse \times \CL \to \CLL \]
has a lax symmetric monoidal structure such that the localisation $\PSh \circ (u \times \id) \to \wt\bV$ refines to a natural transformation of lax symmetric monoidal functors.
 The full subfunctor
 \[ \widehat\bV \circ (\id\times j)  \colon \Coarse \times \CL \to \CLL \]
 inherits a lax symmetric monoidal structure.
\end{prop}
\begin{proof}
 We will first construct a lax symmetric monoidal structure on the functor $\wt\bV \circ (\id\times k)$ considered as a $\CATi$-valued functor.
 To do this, we check that the assumptions of \cref{wefkijeqwofqlewfeqwfqfqeewfq} are satisfied for   $\PSh \circ (u \times k)$. 
 
 Condition \ref{it:wefkijeqwofqlewfeqwfqfqeewfq1} of \cref{wefkijeqwofqlewfeqwfqfqeewfq} is satisfied since $\wt\bV$ is a functor.
 
 In order to check Condition~\ref{it:wefkijeqwofqlewfeqwfqfqeewfq2}, let $\bC$ be in $\CL$,
 $M$ be in $\PSh_{\bC}(X)$ and $M'$ be in $\PSh_{\bC'}(X')$.
 Let $U$ and $U'$ be coarse entourages of $X$ and $X'$, respectively, and consider $\theta^U_M \colon M \to U_{*}M$ in $\widetilde{W}_{X,\bC}$ as well as $\theta^{U'}_{M'} \colon M' \to U'_*M'$ in $\widetilde{W}_{X',\bC'}$. 
  The functor $U_*M \otimes U'_*M'$ is given by the right Kan extension
  \[\xymatrix@C=4em{
   \cP_X^\op \times \cP_{X'}^\op\ar[r]^-{U(-) \times U'(-)}\ar[d]_-{\times} & \cP_X^\op \times \cP_{X'}^\op\ar[r]^-{M \widehat\otimes M'} & \bC \otimes \bC' \\
   \cP_{X \times X'}^\op\ar[urr]_-{\hphantom{xx}U_*M \otimes U'_*M'} & & 
  }\]
 Since $U(-) \times U'(-)$ is left adjoint as a functor $\cP_X^\op \times \cP_{X'}^\op \to \cP_X^\op \times \cP_{X'}^\op$, the functor $U_*M \widehat\otimes U'_*M'$ is a right Kan extension of $M \widehat\otimes M'$ along the product of thickening functors $U[-] \times U'[-] \colon \cP_X^\op \times \cP_{X'}^\op \to \cP_X^\op \times \cP_{X'}^\op$.
 Consequently, $U_*M \otimes U'_*M'$ also fits  into the following right Kan extension:
 \[\xymatrix@R=1.5em@C=4em{
  \cP_X^\op \times \cP_{X'}^\op\ar[d]_-{U[-] \times U'[-]}\ar[r]^-{M \widehat\otimes M'} & \bC \otimes \bC' \\
  \cP_X^\op \times \cP_{X'}^\op\ar[d]_-{\times} & & \\
   \cP_{X \times X'}^\op\ar[uur]_-{U_*M \otimes U'_*M'} & &
 }\]
 Since
 \[\xymatrix@C=4em{
  \cP_X^\op \times \cP_{X'}^\op\ar[r]^-{U[-] \times U'[-]}\ar[d]_-{\times} & \cP_X^\op \times \cP_{X'}^\op\ar[d]^-{\times} \\
  \cP_{X \times X'}^\op\ar[r]^-{(U \times U')[-]} & \cP_{X \times X'}^\op
 }\]
 commutes, it follows that $U_*M \otimes U'_*M' \simeq (U \times U')_*(M \otimes M')$.
 Under this identification, $\theta^U_M \otimes \theta^{U'}_{M'}$ is given by $\theta^{U \times U'}_{M \otimes M'} \colon M \otimes M' \to (U \times U')_*(M \otimes M')$.
 Hence $\theta^U_M \otimes \theta^{U'}_{M'}$ lies in $\wt W_{\bC \otimes \bC',X \otimes X}$.
  
 Now \cref{wefkijeqwofqlewfeqwfqfqeewfq} yields the desired lax symmetric monoidal refinement of $\wt\bV \circ (\id\times k)$ considered as a $\CATi$-valued functor.

In justifying \eqref{wrfwffwefwefwefwefweff}, we already observed that the underlying functor of $\wt\bV \circ (\id\times k)$ takes values in $\CLL$.
 To see that $\wt\bV \circ (\id\times k)$ defines a lax symmetric monoidal functor with values in $\CLL$, \cref{roigjioergerergergergwegrwerg} reduces
 the problem {to} showing that the structure maps
\[ \wt\bV_{\bC}(X) \times \wt\bV_{\bC'}(X') \to \wt\bV_{\bC \otimes \bC'}(X \otimes X') \]
 preserve finite limits in each variable separately.
 We fix an object $M$ in $\PSh_{\bC}(X)$ and consider the following diagram:
 \[\xymatrix@C=3em{
 \PSh_{\bC'}(X')\ar[r]^-{M \otimes -}\ar[d]_-{\ell'}& \PSh_{\bC \otimes \bC'}(X \otimes X')\ar[d]^-{\ell''}\\
 \wt\bV_{\bC'}(X')\ar[r]^-{\ell(M) \otimes -}& \wt\bV_{\bC \otimes \bC'}(X \otimes X')
 }\]
where $\ell$, $\ell'$ and $\ell''$ are the localisation maps.
The lower horizontal map and the filler of the diagram are obtained from the universal property of $\ell'$ as a localisation since the composition along the top right corner sends all morphisms in $\wt W_{X',\bC'}$ to equivalences.

By \cref{weogjowergerregegwerg}, the structure maps of $\PSh \circ (\id\times k)$ preserve finite limits in each variable separately. In particular, the upper horizontal map preserves finite limits.
Since $\ell''$ is a left-exact localisation \cite[Prop.~2.76]{unik}, it follows that $\ell''$ and the  composition $\ell''\circ (M \otimes -)$ are left-exact. Using the universal property of $\ell'$ as a localisation in left-exact $\infty$-categories, we now conclude that also the lower horizontal map is left-exact.
For reasons of symmetry, the same argument applies to the functor $- \otimes M'$ with $M'$ in $\PSh_{\bC'}(X')$.

By \cref{roigjioergerergergergwegrwerg},
the unit constraint of $\wt \bV \circ (\id\times k)$ viewed as a lax symmetric functor with values in $\CLL$ is given by the composition
\begin{equation}\label{fquiwehfiwefewdqewdqed}
\Spc_{*}^{\fin,\op} \to \PSh_{\Spc_{*}^{\op}}(*) \xrightarrow{\ell} \wt\bV_{\Spc_{*}^{\op}}(*)\ ,
\end{equation}
where the first left-exact functor is essentially uniquely determined by the fact that it sends $S^{0}$ to the sheaf $\cU$ in \eqref{eq:unit-constraint}.

The {last assertion of the proposition} follows in analogy to \cref{prop:sh-monoid} by \cite[Prop.~2.2.1.1]{HA}.
Since $\cU$ is a sheaf,
the unit constraint
given by  \eqref{fquiwehfiwefewdqewdqed} takes values in the full subcategory $\widehat\bV_{\Spc^\op_*}(*)$, and by \cref{prop:sh-monoid} the tensor product of sheaves is again a sheaf.
\end{proof}

Following the constructions outlined in \cref{sec:controlled-objects}, we now  use bornologies to impose finiteness conditions on sheaves.
Therefore, we consider the category $\BC$ of bornological coarse spaces.
Since the forgetful functor $v \colon \BC \to \Coarse$ is symmetric monoidal and the functor 
\[ \CL\to \CLLL \]
is lax symmetric monoidal, 
 we obtain from \cref{prop:loc-monoid} a lax symmetric monoidal functor
\[ \BC \times \CL \xrightarrow{v \times \id} \Coarse \times \CL \xrightarrow{\widehat\bV\circ (\id\times j)} \CLL\ .\]

\begin{prop}\label{prop:vc-monoid}
 The full subfunctor
 \[ \bV^c \colon \BC \times \CL \to \Cle \]
 of $\widehat\bV \circ (v \times j)$ inherits a lax symmetric monoidal structure.
\end{prop}
\begin{proof}
 Recall from \cref{sec:controlled-objects} that $\bV^c_\bC(X)$ is the full subcategory of $\bV_\bC(X)$ on objects $i_*M$, where $i \colon F \to X$ is the inclusion of a locally finite subset and $M$ is an object in $\bV_\bC(F)$ (hence represented by a small sheaf).
 In particular, $\bV^c$ is a full subfunctor of $\widehat\bV$.

 Again by \cite[Prop.~2.2.1.1]{HA}, it suffices to show that the unit constraint $\Spc^{\fin,\op}_* \to \bV_{\Spc^\op_*}(*)$ takes values in $\bV^c_{\Spc^\op_*}(*)$ and that for all $M_1$ in $\bV^c_{\bC_1}(X_1)$ and $M_2$ in $\bV^c_{\bC_2}(X_2)$ the tensor product $M_1 \otimes M_2$ lies in $\bV^c_{\bC_1 \otimes \bC_2}(X_1 \otimes X_2)$.
 
 Since $S^0$ is cocompact in $\Spc^\op_*$, the sheaf $\cU$ from \eqref{eq:unit-constraint} is small.
 This implies the claim about the unit constraint.
 
 Given $M_1$ in $\bV^c_{\bC_1}(X_1)$ and $M_2$ in $\bV^c_{\bC_2}(X_2)$, there exist for $i=1,2$ locally finite subsets $F_i$ of $X_i$ and small sheaves $N_i$ in $\bV_{\bC_i}(F_i)$  such that $M_i \simeq j_{i,*}N_i$, where $j_i \colon F_i \to X_i$ is the respective inclusion map.
 Since there is a commutative diagram
 \[\xymatrix{
  \widehat\bV_{\bC_1}(F_1) \times \widehat\bV_{\bC_2}(F_2)\ar[r]^-{\otimes}\ar[d]_-{j_{1,*} \otimes j_{2,*}} & \widehat\bV_{\bC_1 \otimes \bC_2}(F_1 \otimes F_2)\ar[d]^-{(j_1 \otimes j_2)_*} \\
  \widehat\bV_{\bC_1}(X_1) \times \widehat\bV_{\bC_2}(X_2)\ar[r]^-{\otimes} & \widehat\bV_{\bC_1 \otimes \bC_2}(X_1 \otimes X_2)
 }\]
 and $F_1 \otimes F_2$ is a locally finite subset of $X_1 \otimes X_2$, we only need to check that $N_1 \otimes N_2$ is a small sheaf over $F_1 \otimes F_2$.
 By the pointwise formula for the Kan extension \eqref{vwervlnlvwevewwerv} defining $N_1 \otimes N_2$,
 the value at a bounded subset $B$ of $F_1 \otimes F_2$ is given by
 \[ (N_1 \otimes N_2)(B) \simeq \lim_{(B_1 \times B_2) \to B \in ((\cP_{F_1} \times \cP_{F_2})_{/B})^\op} N_1(B_1) \otimes N_2(B_2)\ .\]
 Since $B$ is finite, the indexing category of this limit is finite.
 Hence the smallness of $N_1$ and $N_2$ implies that $(N_1 \otimes N_2)(B)$ is given by a finite limit of cocompact objects.
\end{proof}

\begin{kor}\label{cor:vperf-monoid}
 The functor
 \[ \bV^{c,\perf} \colon \BC \times \CL \xrightarrow{\bV^c} \Cle \xrightarrow{\Idem} \Clep \]
 has a lax symmetric monoidal refinement.
\end{kor}
\begin{proof}
 This follows from \cref{prop:vc-monoid} since $\Idem$ has a symmetric monoidal refinement.
\end{proof}

\subsection{Orbits as a module over fixed points}\label{sec:orbits-and-fixed-points}

Let $G$ be a group. For a symmetric monoidal $\infty$-category $\cC$, we equip the functor category $\Fun(BG,\cC)$ with the pointwise symmetric monoidal structure.
Then the lax symmetric monoidal functor $\bV^{c,\perf}$ from \cref{cor:vperf-monoid} induces a lax symmetric monoidal functor
\[ \Fun(BG,\BC) \times \Fun(BG,\CL) \to \Fun(BG,\Clep)\ .\]
Since $G\BC$ is a full symmetric monoidal subcategory of $\Fun(BG,\BC)$, we can restrict this functor to obtain the lax symmetric monoidal functor
\begin{equation}\label{eq:GV}
 G\bV^{c,\perf} \colon G\BC \times \Fun(BG,\CL) \to \Fun(BG,\Clep)
\end{equation}
which is an equivariant version of $\bV^{c,\perf}$.

\begin{ddd}\label{rguiergegwgergwgre}
 We define the functors
 \[ \bV^{c,\perf,G} := \lim_{BG} \circ G\bV^{c,\perf} \colon G\BC \times \Fun(BG,\CL) \to \Clep \]
and
 \[ \bV^{c,\perf}_{G} := \mathop{\colim}\limits_{BG} \circ G\bV^{c,\perf} \colon G\BC \times \Fun(BG,\CL) \to \Clep \ .\qedhere\]
\end{ddd}

\begin{lem}\label{lem:fixed-props}
 The functor $\bV^{c,\perf,G}_{\bC}$ is coarsely invariant and $\pi_{0}$-excisive.
\end{lem}
\begin{proof}
 The functor $\bV_{\bC}^{c,\perf}$ is coarsely invariant by \cite[Lem.~3.3 \& 4.5]{unik}.
 Since the forgetful functor $\Fun(BG,\Clep) \to \Clep$ detects equivalences, it follows that $G\bV_{\bC}^{c,\perf}$ is coarsely invariant.
 Consequently, $\bV^{c,\perf,G}_{\bC}$ is also coarsely invariant.

 Let $X$ be a $G$-bornological coarse space and $(Y,Z)$ be a partition of $X$ into invariant and coarsely disjoint subsets.
 Since $\Clep$ is semi-additive \cite[Lem.~7.21]{unik} and $\bV_{\bC}^{c}$ is $\pi_{0}$-excisive \cite[Lem.~3.34 \& 4.9(1)]{unik}, 
 the canonical inclusions of $Y$ and $Z$ into $X$ induce an equivalence
 \[  \bV_{\bC}^{c}(Y) \oplus \bV_{\bC}^{c}(Z) \xrightarrow{\simeq} \bV_{\bC}^{c}(X)   \ .\]
 Since both $\Idem$ and $\lim_{BG}$ are additive functors, it follows that
 \[  \bV_{\bC}^{c,\perf,G}(Y) \oplus \bV_{\bC}^{c,\perf,G}(Z) \xrightarrow{\simeq} \bV_{\bC}^{c,\perf,G}(X)\ ,\]
 so $\bV^{c,\perf,G}_\bC$ is $\pi_0$-excisive.
 \end{proof}

\begin{rem}\label{qeorighqoergergwregr}
 Let $X$ be a $G$-bornological coarse space and $(Y,Z)$ be a partition of $X$ into $G$-invariant and coarsely disjoint subsets.
 The canonical equivalence
 \[ \bV^{c,\perf,G}_\bC(Y) \oplus \bV^{c,\perf,G}_\bC(Z) \xrightarrow{\simeq} \bV^{c,\perf,G}_\bC(X) \]
 admits the following explicit inverse.
 The inclusions $\cP_Y \to \cP_X$ and $\cP_Z \to \cP_X$ induce restriction functors
 \[ (-)_{|Y} \colon \Sh_\bC(X) \to \Sh_\bC(Y) \quad\text{and}\quad (-)_{|Z} \colon \Sh_\bC(X) \to \Sh_\bC(Z)\ .\]
 It follows from \cite[Lem.~2.13 \& 2.14]{unik} that the sum of these functors provides an inverse to the canonical functor
 \[ \Sh_\bC(Y) \oplus \Sh_\bC(Z) \to \Sh_\bC(X)\ .\]
 Since $(\theta^U_M)_{|Y} \simeq \theta^{U_{|Y}}_{M_{|Y}}$ (and similarly for $Z$), the restriction functors send $W_{X,\bC}$ from \eqref{eq:wxc} to $W_{Y,\bC}$ and $W_{Z,\bC}$, respectively.
 Hence the restriction functors descend to functors $(-)_{|Y} \colon \widehat\bV_\bC(X) \to \widehat\bV_\bC(Y)$ and $(-)_{|Z} \colon \widehat\bV_\bC(X) \to \widehat\bV_\bC(Z)$ on the localisations.
 Moreover, they preserve small sheaves and thus restrict to functors $\bV_\bC(X) \to \bV_\bC(Y)$ and $\bV_\bC(X) \to \bV_\bC(Z)$. 
 
 The sum of these functors provides an inverse to the canonical functor
 \begin{equation}\label{wervervmdfvdfvvosfdvsvsfdv}
 \bV_\bC(Y) \oplus \bV_\bC(Z) \to \bV_\bC(X)\ .
\end{equation} 
 These mutually inverse equivalences restrict to the continuous version $\bV^c_\bC$ defined by the left Kan extension diagram \eqref{fdsvkvsdfpvfvsfdvsvfdvsv}.
 Since $Y$ and $Z$ are $G$-invariant, all functors above have canonical $G$-equivariant refinements.
 Hence we obtain the desired inverse by applying $\lim_{BG} \circ \Idem$ to the equivariant version of \eqref{wervervmdfvdfvvosfdvsvsfdv}.
\end{rem}

\begin{rem}
 One can also show that $\bV^{c,\perf,G}_\bC$ preserves flasqueness in the sense of \cite[Def.~3.7]{unik}.
 This follows from \cite[Lem.~3.15 \& 4.7]{unik} using the fact that $\Idem$ and $\lim_{BG}$ are additive functors.
 
 On the other hand, $\bV^{c,\perf,G}$ is not expected to be excisive for arbitrary complementary pairs since $\lim_{BG}$ does not preserve cofibre sequences in $\Clep$.
 In particular, {postcomposing $\bV^{c,\perf,G}$ with homological functors} will not give rise to coarse homology theories.
\end{rem}

In the remainder of this section, we employ the lax symmetric monoidal structure on $G\bV^{c,\perf}$ from \eqref{eq:GV} to equip $\bV^{c,\perf,G}$ with a lax symmetric monoidal structure and $\bV^{c,\perf}_G$ with a module structure over  $\bV^{c,\perf,G}$.
 The main ingredient lies in the following construction.\footnote{This construction was suggested to us by Thomas Nikolaus.}

\begin{construction}\label{const:fix-orbit}
Let $\cC$ be a symmetric monoidal $\infty$-category which admits all $BG$-indexed limits.
Then by \cite[Cor.~3.2.2.5]{HA}, the $\infty$-category of commutative algebras $\CAlg(\cC)$ in $\cC$ admits all $BG$-indexed limits and the forgetful functor $\CAlg(\cC) \to \cC$ preserves all $BG$-indexed limits.
For $A$ in $ \Fun(BG,\CAlg(\cC))$ we define $A^{G}:=\lim_{BG}A$ in $\CAlg(\cC)$.
The counit of the adjunction
\[ \underline{(-)} \colon \CAlg(\cC) \leftrightarrows    \Fun(BG,\CAlg(\cC)) \cocolon \lim_{BG} \]
provides a morphism $\underline{A^G} \to A$ in $\Fun(BG,\CAlg(\cC))$.

We now consider the $\infty$-category of modules $\Mod(\cC)$
whose objects are pairs $(A,M)$ of a commutative algebra $A$ and an $A$-module $M$ in $\cC$.
The functor $\Mod(\cC) \to \CAlg(\cC)$ {which forgets the module} is a cartesian fibration by \cite[Cor.~3.4.3.4(1)]{HA}.
For $A$ in $\CAlg(\cC)$, we write $\Mod_{A}(\cC)$ for the fibre over $A$.

By \cite[Rem.~2.1.3.4]{HA}, the induced functor $\Fun(BG, \Mod(\cC)) \to \Fun(BG,\CAlg(\cC))$ is identified with the cartesian fibration $\Mod(\Fun(BG,\cC)) \to \CAlg(\Fun(BG,\cC))$.
Regarding $A$ as an object in $\CAlg(\Fun(BG,\cC))$, the counit $c \colon \underline{A^G} \to A$ then induces a restriction functor between the fibres over $A$ and $\underline{A^G}$:
\[ c^* \colon \Mod_A(\Fun(BG,\cC)) \to \Fun(BG,\Mod_{A^G}(\cC)) \]
If $\cC$ also admits $BG$-indexed colimits and the tensor product on $\cC$ preserves $BG$-indexed colimits in each variable separately, then \cite[Cor.~4.3.2.5]{HA} shows that $\Mod_{A^G}(\cC)$ also admits $BG$-indexed colimits and that the forgetful functor $\Mod_{A^G}(\cC) \to \cC$ preserves $BG$-indexed colimits.  
Therefore, we have the composition \begin{equation}\label{qwefwedws}
  \Phi \colon \Mod_A(\Fun(BG,\cC)) \xrightarrow{c^*} \Fun(BG,\Mod_{A^G}(\cC)) \xrightarrow{\colim_{BG}} \Mod_{A^G}(\cC)\end{equation}
which fits into the commutative diagram
\begin{equation}\label{eq:fix-orbit}
\xymatrix{
 \Mod_A(\Fun(BG,\cC))\ar[r]^-{\Phi}\ar[d] & \Mod_{A^G}(\cC)\ar[d] \\
 \Fun(BG,\cC)\ar[r]^-{\colim_{BG}} & \cC
}\end{equation}
in which the vertical arrows take the underlying module object.
 \end{construction}

In the following, we use the exponential law to regard the functor $G\bV^{c,\perf}$ from \eqref{eq:GV} as a lax symmetric monoidal functor
\begin{equation}\label{eq:GV-expo}
 G\bV^{c,\perf} \colon \Fun(BG,\CL) \to \Fun(G\BC, \Fun(BG,\Clep))\ ,
\end{equation}
where we equip the outer functor category in the target with the Day convolution structure.

Since $\lim_{BG} \colon \Fun(BG,\Clep) \to \Clep$ is lax symmetric monoidal, postcomposition with $\lim_{BG}$ induces by \cite[Cor.~3.7]{Nikolaus:2016aa} a lax symmetric monoidal functor
\[ \bV^{c,\perf,G} \colon \Fun(BG,\CL) \to \Fun(G\BC, \Clep) \]
which is a lax symmetric monoidal refinement of the functor denoted by the same symbol in \cref{rguiergegwgergwgre}.

Since $\Spc^\op_*$ (equipped with the trivial $G$-action) is the tensor unit of $\Fun(BG,\CL)$, it refines to a commutative algebra object in $ \Fun(BG,\CL)$.
Consequently, evaluating $\bV^{c,\perf,G}$ at $\Spc^\op_*$ yields a commutative algebra object $\bV^{c,\perf,G}_{\Spc^\op_*}$ in $\Fun(G\BC,\Clep)$.

\begin{prop}\label{prop:fix-orbit-module}
 There exists a {functor $\bV\bM$ fitting into a} commutative diagram
 \begin{equation}\label{eq:fix-orbit-module}
 \xymatrix{
  \Fun(BG,\CL)\ar[r]^-{\bV\bM}\ar[rd]_-{\bV^{c,\perf}_G} & \Mod_{\bV^{c,\perf,G}_{\Spc^\op_*}}(\Fun(G\BC,\Clep))\ar[d] \\
  & \Fun(G\BC,\Clep)
  }\end{equation}
  where the vertical arrow takes the underlying module object.
\end{prop}
\begin{proof}
 We begin by constructing the functor $\bV\bM$.
 Since $G\bV^{c,\perf}$ from \eqref{eq:GV-expo} is lax symmetric monoidal, application of $\Mod$ induces a functor
 \[ \Mod(G\bV^{c,\perf}) \colon \Mod(\Fun(BG,\CL)) \to \Mod(\Fun(G\BC, \Fun(BG,\Clep)))\ .\]
 As $G\bV^{c,\perf}$ sends $\Spc^\op_*$ to $G\bV^{c,\perf}_{\Spc^\op_*}$, this functor restricts to a functor
 \begin{equation}\label{eq:fix-orbit-module1}
  \Mod_{\Spc^\op_*}(\Fun(BG,\CL)) \to \Mod_{G\bV^{c,\perf}_{\Spc^\op_*}}(\Fun(G\BC, \Fun(BG,\Clep)))\ .
 \end{equation}
 By \cite[Prop.~3.4.2.1]{HA}, there is an equivalence
 \begin{equation}\label{eq:fix-orbit-module2}
  \Fun(BG,\CL) \simeq \Mod_{\Spc^\op_*}(\Fun(BG,\CL))\ .
 \end{equation}
 Composing \eqref{eq:fix-orbit-module2} with \eqref{eq:fix-orbit-module1}, we obtain the first arrow in the composition
 \begin{align*}
  \bV\bM' \colon \Fun(BG,\CL)
  &\to \Mod_{G\bV^{c,\perf}_{\Spc^\op_*}}(\Fun(G\BC, \Fun(BG,\Clep))) \\
  &\overset{!}{\simeq} \Mod_{G\bV^{c,\perf}_{\Spc^\op_*}}(\Fun(BG, \Fun(G\BC,\Clep)))\ .
 \end{align*}
 The lower equivalence in \cref{lem:modules-dayconvolution} allows us to interpret $G\bV^{c,\perf}_{\Spc^\op_*}$ as an algebra in $\Fun(BG, \Fun(G\BC,\Clep))$ which we denote by the same symbol. 
 The equivalence marked by $!$ is then the restriction of the upper equivalence from \cref{lem:modules-dayconvolution} to the fibres over $G\bV^{c,\perf}_{\Spc^\op_*}$.
 Now use the functor $\Phi$ from
 \eqref{qwefwedws} to define the functor $\bV\bM := \Phi \circ \bV\bM'$ appearing as the upper horizontal arrow in \eqref{eq:fix-orbit-module}.
 
 The commutativity of \eqref{eq:fix-orbit-module} follows from the commutativity of \eqref{eq:fix-orbit}.
\end{proof}

\begin{proof}[Proof of \cref{thm:fix-orbit-wmodule}]
 In this proof, we distinguish notationally between the object $\Spc^\op_*$ of $\CL$ and the object $\underline{\Spc^\op_*}$ of $\Fun(BG,\CL)$ which is given by $\Spc^\op_*$ equipped with the trivial $G$-action.
 
 Let $\bC$ be an object in $\Fun(BG,\CL)$.
 We begin by defining the weak module structure of $\bV^{c,\perf}_{\bC,G}$ over $\bV^{c,\perf,G}_{\underline{\Spc^\op_*}}$ whose existence is part of Assertion~\ref{thm:fix-orbit-wmodule}.\ref{it:fix-orbit-wmodule3}.
 Let $\bV\bM$ be the functor from \cref{prop:fix-orbit-module}. 
 Then the module $\bV\bM(\bC)$ encodes an action transformation
 \[ \alpha \colon \bV^{c,\perf,G}_{\underline{\Spc^\op_*}} \otimes \bV^{c,\perf}_{\bC,G} \to \bV^{c,\perf}_{\bC,G}\ .\]
 Define the natural transformation $\mu$ between functors $G\BC \times G\BC \to \Clep$ as the composition
 \begin{align*}
  \mu \colon \bV^{c,\perf,G}_{\underline{\Spc^\op_*}}(-) \otimes_{\Clep} \bV^{c,\perf}_{\bC,G}(-) &\to (\bV^{c,\perf,G}_{\underline{\Spc^\op_*}} \otimes_{\Fun(G\BC,\Clep)} \bV^{c,\perf}_{\bC,G}) \circ (- \otimes_{G\BC} -) \\
  &\xrightarrow{\alpha \circ (- \otimes_{G\BC} -)} \bV^{c,\perf}_{\bC,G} \circ (- \otimes_{G\BC} -)\ ,
 \end{align*}
 where the first arrow is an instance of the transformation $\tau$ in \eqref{eq:dayconvo}.
 The unit object $\beins$ of $\Fun(G\BC,\Clep)$ refines to a commutative algebra object.
 Since $\{*\} \to G\BC$ is a symmetric monoidal functor, evaluation at $*$ defines a symmetric monoidal functor $\Fun(G\BC,\Clep) \to \Clep$ by \cite[Cor.~3.8]{Nikolaus:2016aa}.
 It follows that $\beins(*) \simeq \Spc^{\op,\omega}_*$.
 
 The commutative algebra object $\bV^{c,\perf,G}_{\underline{\Spc^\op_*}}$ comes equipped with a unit morphism $\epsilon \colon \beins\to \bV^{c,\perf,G}_{\underline{\Spc^\op_*}}$.
 Its evaluation at $*$ in $G\BC$ therefore provides a functor
 \[ \eta \colon \Spc^{\op,\omega}_* \simeq \beins(*) \to \bV^{c,\perf,G}_{\underline{\Spc^\op_*}}(*)\ .\]
 The module structure on $\bV\bM(\bC)$ also encodes that the composition
  \[ \beins \otimes \bV^{c,\perf}_{\bC,G} \xrightarrow{\epsilon \otimes \id} \bV^{c,\perf,G}_{\underline{\Spc^\op_*}} \otimes \bV^{c,\perf}_{\bC,G} \xrightarrow{\alpha} \bV^{c,\perf}_{\bC,G} \]
 is equivalent to the canonical identification $\beins \otimes \bV^{c,\perf}_{\bC,G} \simeq\bV^{c,\perf}_{\bC,G}$.
 This gives rise to the commutative diagram
 \[\xymatrix{
  \Spc^{\op,\omega}_* \otimes \bV^{c,\perf}_{\bC,G}(-)\ar[r]^-{\eta \otimes \id}\ar[d]\ar@/_80pt/[dd]_-{\simeq} & \bV^{c,\perf,G}_{\underline{\Spc^\op_*}}(*) \otimes \bV^{c,\perf}_{\bC,G}(-)\ar[d]\ar@/^80pt/[dd]^-{\mu} \\
  (\beins \otimes \bV^{c,\perf}_{\bC,G})(* \otimes -)\ar[r]^-{\epsilon \otimes \id}\ar[d]_-{\simeq} & (\bV^{c,\perf,G}_{\underline{\Spc^\op_*}} \otimes \bV^{c,\perf}_{\bC,G})(* \otimes -)\ar[d]^-{\alpha} \\
  \bV^{c,\perf}_{\bC,G}(-)\ar[r]^-{\simeq} & \bV^{c,\perf}_{\bC,G}(* \otimes -)
 }\]
 in $\Fun(G\BC,\Clep)$, where the unlabelled vertical arrows are again instances of $\tau$ in \eqref{eq:dayconvo}.
 This proves that $\eta$ and $\mu$ define a weak module structure.

 The following argument shows Assertion~\ref{thm:fix-orbit-wmodule}.\ref{it:fix-orbit-wmodule2} that $\bV^{c,\perf,G}_{\underline{\Spc^\op_*}}(*) \simeq \Fun(BG,\Spc^{\op,\omega}_*)$ in $\Clep$ and identifies the map $\eta$.
 
 As a first step, observe that evaluation at $*$ induces an equivalence $\Sh_{\Spc^\op_*}(*) \xrightarrow{\simeq} \Spc^\op_*$ in $\CLL$.
 Since $*$ has a unique non-empty entourage, namely $\diag(*)$, the localisation functor is an equivalence $\Sh_{\Spc^\op_*}(*) \xrightarrow{\simeq} \widehat\bV_{\Spc^\op_*}(*)$.
 In view of the definition of small sheaves, we have an equivalence $\bV^c_{\Spc^\op_*}(*) \xrightarrow{\simeq} \Spc^{\op,\omega}_*$.
 Since $\Spc^{\op,\omega}_*$ is idempotent complete, applying $\Idem$ induces an equivalence $\bV^{c,\perf}_{\Spc^\op_*}(*) \xrightarrow{\simeq} \Spc^{\op,\omega}_*$. 
 The description of the unit of the algebra $\bV^{c,\perf}_{\Spc^\op_*}(*)$ (see \eqref{eq:unit-constraint} and \eqref{fquiwehfiwefewdqewdqed}) implies that the unit morphism $\Spc^{\op,\omega}_* \to \bV^{c,\perf}_{\Spc^\op_*}(*)$ is an inverse to this equivalence.
 
 The object $\underline{\Spc^{\op,\omega}_*}$ is the tensor unit of $\Fun(BG,\Clep)$.
 By \cite[Cor.~3.2.1.9]{HA}, $\underline{\Spc^{\op,\omega}_*}$ carries an essentially unique commutative algebra structure and admits an essentially unique map of commutative algebras $\underline{\Spc^{\op,\omega}_*} \to \bV^{c,\perf}_{\underline{\Spc^\op_*}}(*)$.
 The preceding observation shows that the underlying functor $\Spc^{\op,\omega}_* \to \bV^{c,\perf}_{\Spc^\op_*}(*)$ is an equivalence, so we have an equivalence $\underline{\Spc^{\op,\omega}_*} \simeq \bV^{c,\perf}_{\underline{\Spc^\op_*}}(*)$ of commutative algebras in $\Fun(BG,\Clep)$.
 Application of $\lim_{BG}$ to this equivalence yields an equivalence $\lim_{BG} \underline{\Spc^{\op,\omega}_*} \simeq \bV^{c,\perf,G}_{\underline{\Spc^{\op,\omega}}}(*)$ of commutative algebras in $\Clep$.
 The underlying left-exact $\infty$-category of the left hand side is equivalent to $\Fun(BG,\Spc^{\op,\omega}_*)$ since $\underline{\Spc^{\op,\omega}_*}$ carries the trivial $G$-action.
 This proves Assertion~\ref{thm:fix-orbit-wmodule}.\ref{it:fix-orbit-wmodule2}.
 
 Under the identification $\lim_{BG} \underline{\Spc^{\op,\omega}_*} \simeq \Fun(BG,\Spc^{\op,\omega}_*)$, the unit map of the algebra $\lim_{BG} \underline{\Spc^{\op,\omega}_*}$ is given by the left-exact functor
 \[ \Spc^{\op,\omega}_* \to \Fun(BG,\Spc^{\op,\omega}_*) \]
 which sends $S^0$ to $\underline{S^0}$ as claimed by Assertion~\ref{thm:fix-orbit-wmodule}.\ref{it:fix-orbit-wmodule4}.
  
 Finally, $\bV^{c,\perf,G}_{\underline{\Spc^\op_*}}$ is $\pi_0$-excisive by \cref{lem:fixed-props}, which proves Assertion~\ref{thm:fix-orbit-wmodule}.\ref{it:fix-orbit-wmodule1}.
\end{proof}

\section{Controlled CW-complexes}\label{sec:controlledCWs}

The functors $\bV^{c,\perf,G}_{\Spc^\op_*}$ and $\bV^{c,\perf}_{\bC,G}$ discussed in \cref{sec:controlled-objects} will ultimately be fed into \cref{ergiooegergergwergergergw} to prove the Farrell--Jones conjecture for certain classes of groups.
The construction of transfer classes requires adequate point-set models for objects in $\bV^{c,\perf}_{\Spc^\op_*}$.
We will be able to obtain such models using the notion of controlled CW-complexes that we previously considered in \cite{BKW:coarseA} and that goes back to work of Weiss \cite{Weiss2002}. 

\cref{sec:controlledCW} recalls the definition of the category $\CW(X)$ of controlled CW-complexes over a coarse space $X$.
In \cref{sec:realisation}, we discuss the realisation transformation $r \colon \CW(X)^\op \to \wt\bV_{\Spc^\op_*}(X)$, where $\wt\bV_{\Spc^\op_*}(X)$ is an ambient $\infty$-category which contains $\bV^{c,\perf}_{\Spc^\op_*}(X)$ as a full subcategory.
In \cref{sec:finiteness}, we consider bornological coarse spaces $X$ and discuss finiteness conditions on objects in $\CW(X)$ which ensure that the corresponding images under $r$ lie in $\bV^{c,\perf}_{\Spc^\op_*}(X)$.

\subsection{Controlled CW-complexes}\label{sec:controlledCW}

A based CW-complex is a CW-complex $Q$ together with a chosen $0$-cell in $Q$.
A morphism of based CW-complexes is a cellular and basepoint-preserving map. 
For $Q$ a based CW-complex, we denote by $z_{i}(Q)$ the set of $i$-cells.
Let $z(Q) := \bigcup_{i\in \nat} z_{i}(Q)$ denote the set of all cells of $Q$.
Note that we regard $Q$ as a relative CW-complex, so the basepoint is not a member of $z_0(Q)$.

If $A$ is a subset of $Q$, then we denote by $\overline{A}^{CW}$ the minimal subcomplex of $Q$ containing $A$.
For a cell $q$ of $Q$ we write $q \leq A$ if $q \subseteq \overline{A}^{CW}$.
In particular, this defines a transitive and reflexive relation on the set of cells $z(Q)$.  
Note that $q \leq q'$ implies $\dim(q) \leq \dim(q^{\prime})$ with equality precisely if $q=q'$.

Let $X$ be a set.
\begin{ddd}\label{iuhdfiuvefdcadscadcadsc}
 An $X$-labelled CW-complex is a pair $(Q,\lambda)$ consisting of a based $CW$-complex $Q$ together with a map $\lambda \colon z(Q)\to X$.
\end{ddd}

	For a subcomplex $Q^{\prime}$ of $Q$, we define the subset
	\begin{equation}\label{uivhqiwecewcwcqcwec}
 \lambda(Q^{\prime}) := \lambda(z(Q^{\prime})) 
\end{equation}
	of $X$.

Let $(Q,\lambda)$ and $(Q^{\prime},\lambda^{\prime})$ be two $X$-labelled CW-complexes and let $U$ be an entourage of $X$.

\begin{ddd} \label{ergioogrqgegqwfewfeqfe3245}
	A $U$-controlled morphism $(Q,\lambda)\to (Q^{\prime},\lambda^{\prime})$  is a morphism $\phi \colon Q\to Q^{\prime}$ of based CW-complexes  such that for every $q$ in $z(Q)$ we have
	\[ \lambda^{\prime}(\overline{\phi(q)}^{CW}) \subseteq U[\{\lambda(q)\}]\]
	see \eqref{qwefoiheiuohfiqwefewfewfqeefedq} for {the definition of the thickening operation $U[-]$}.
	An $X$-labelled CW-complex is called $U$-controlled if its identity map is a $U$-controlled morphism.	
\end{ddd}

We now assume that $X$ is in $\Coarse$ and {denote the coarse structure on $X$ by $\cC_{X}$.}
\begin{ddd} \label{wergoijweogwegferfwrfr}
	An $X$-controlled morphism between $X$-labelled CW-complexes is a morphism of based CW-complexes which is $U$-controlled for some $U$ in $\cC_{X}$.  
\end{ddd}

\begin{ddd} \label{eiuchwichuiwecwecewerfwe}
	An $X$-controlled CW-complex is an $X$-labelled CW-complex whose identity morphism is $X$-controlled.
\end{ddd}

If the coarse space $X$ is clear from the context, we also speak about controlled morphisms and controlled CW-complexes.

For every map $f \colon X \to X'$ and every $X$-labelled CW-complex $(Q,\lambda)$, we obtain an $X'$-labelled CW-complex $(Q,f \circ \lambda)$. If $f$ is a morphism of coarse spaces and $\phi \colon (Q,\lambda) \to (Q',\lambda')$ is an $X$-controlled map, the same map $\phi$ is also an $X'$-controlled map $\phi \colon (Q, f \circ \lambda) \to (Q', f \circ \lambda')$. 

Let $\CAT$ be the very large category of large ordinary categories.

\begin{ddd}\label{rgjqegorqefefqefwfqefqefqef}
		We define a functor $\CW \colon \Coarse \to \CAT$ as follows:
		\begin{enumerate}
			\item The category $\CW(X)$ is the category of $X$-controlled CW-complexes and controlled morphisms.
			\item If $f \colon X\to X^{\prime}$ is a morphism between coarse spaces, then $f_{*} \colon \CW(X)\to \CW(X^{\prime})$ sends a controlled CW-complex $(Q,\lambda)$ to $f_{*}(Q,\lambda):=(Q,f\circ \lambda)$, and sends a morphism $\phi \colon (Q,\lambda)\to (Q',\lambda')$ in $\CW(X)$ to itself, regarded as a morphism $\phi \colon (Q,f \circ \lambda) \to (Q', f \circ \lambda')$ in $\CW(X')$.\qedhere
		\end{enumerate}
\end{ddd}

\subsection{The realisation transformation}\label{sec:realisation}

To relate localisations of $1$-categories to localisations of $\infty$-categories, the following observation is useful.
Let $\cC$ be a category and let $\cW$ be a subcategory of $\cC$.
Recall the notion of a calculus of left fractions from \cite[Sec.~2]{GabrielZisman}.
If $\cW$ satisfies a calculus of left fractions, then Gabriel and Zisman construct a category $\cW^{-1}\cC$ which has the same objects as $\cC$ and whose morphisms sets are given by the formula
\begin{equation}\label{eq:calculus-hom}
 \Hom_{\cW^{-1}\cC}(C,D) := \mathop{\colim}\limits_{(D \to D') \in \cW_{D/}} \Hom_\cC(C,D')\ .
\end{equation}
In particular, morphisms in $\cW^{-1}\cC$ are represented by zig-zags $C \xrightarrow{f} D' \xleftarrow{w} D$ with $w$ in $\cW$.
There is a canonical functor $\cC \to \cW^{-1}\cC$ sending a morphism $f \colon C \to D$ to the class of the zig-zag $C \xrightarrow{f} D \xleftarrow{\id} D$.

\begin{lem}\label{lem:calculus}
 If $\cW$ satisfies a calculus of left fractions, then the functor $\cC \to \cW^{-1}\cC$ exhibits $\cW^{-1}\cC$ as a localisation of $\cC$ at $\cW$ both in $\Cat$ and in $\Cati$.  
\end{lem}
\begin{proof} 
 \cite[Prop.~2.4]{GabrielZisman} asserts that $\cC \to \cW^{-1}\cC$ is the localisation of $\cC$ at $\cW$ in $\Cat$.
 
 For $y$ in $\cC$, denote by $\cW(y)$ the full subcategory of $\cC_{y/}$ spanned by the morphisms which lie in $\cW$.
 The calculus of left fractions implies that $\cW(y)$ is filtered.
 Since the inclusion of $\Set$ into $\Spc$ preserves filtered colimits, the colimit
 \[ \colim_{(y \to y') \in \cW(y)} \Hom_\cC(x,y') \]
 is the same in both categories.
  
 By \cite[7.2.7 \& Thm.~7.2.8]{Cisinski:2017} and \cite[Prop.~2.4]{GabrielZisman}, the above colimit computes the mapping spaces both of the localisation in $\Cat$ and in $\Cati$, which implies that these localisations are equivalent.
\end{proof}

Recall the functor $\wt\bV$ from \eqref{qwefijfoiwqefqwefewfq}.
We   proceed to construct a natural transformation
\[ r \colon \CW^{\op}\to \wt\bV_{\Spc^{\op}_{*}}\ .\]
As a first step, we construct a natural transformation
\[ r_0 \colon \CW^\op \to \wt\bV_{\Top^\op_*}\ .\]
The same argument as in the proof of \cite[Prop. 2.76]{unik} shows that for  every $X$ in $\Set$, the pair $(\PSh_{\Top^\op_*}(X),\wt W_X)$ satisfies a calculus of left fractions, where $\wt W_X$ denotes the class of morphisms introduced in \eqref{eq:wxtilde} which compare each object with its thinned out counterparts. With  \cref{lem:calculus} we conclude that  $\wt\bV_{\Top^\op_*}(X)$ is a $1$-category.
In particular, in order to construct $r_{0}$ it suffices to specify a functor $r_{0,X} \colon \CW(X)^\op \to \wt\bV_{\Top^\op_*}(X)$ for every $X$ in $\Coarse$, and to check naturality.

Every object $(Q,\lambda)$ of $\CW(X)^\op$ gives rise to a presheaf \begin{equation}\label{qrfiojeroifefveffs}
r_{0,X}(Q,\lambda) \colon \cP_X^\op \to \Top_*^\op,\quad Y \mapsto Q(Y)\ ,
\end{equation}
where $Q(Y)$ is the largest subcomplex of $Q$ such that $\lambda(Q(Y)) \subseteq Y$, see \eqref{uivhqiwecewcwcqcwec} for notation.

Let $\phi \colon (Q,\lambda) \to (Q',\lambda')$ be a $U$-controlled morphism in $\CW(X)$.
Assume that $\diag(X) \subseteq U$ {and recall the $U$-thinning functor $U(-) \colon \cP_X \to \cP_X$ from \eqref{eq:thin}}.
If $q$ is a cell in $Q(U(Y))$, then $\lambda(\overline{q}^{CW}) \subseteq U(Y)$ by definition.
Then $U[\lambda(\overline{q}^{CW})] \subseteq Y$ because $U(-)$ is right adjoint to $U[-] \colon \cP_X \to \cP_X$.
Since $\phi$ is $U$-controlled, it follows that
\[ \lambda^{\prime}(\overline{\phi(q)}^{CW}) \subseteq U[\{\lambda(q)\}] \subseteq U[\lambda(\overline{q}^{CW})] \subseteq Y\ .\]
So $\phi(Q(U(Y))) \subseteq Q^{\prime}(Y)$ for every $Y$ in $\cP_{X}$.
Hence $\phi$ induces a morphism represented by
\[ r_{0,X}(Q',\lambda') \xrightarrow{\phi^\op} U_*r_{0,X}(Q,\lambda) \leftarrow r_{0,X}(Q,\lambda) \]
in $\wt\bV_{\Top^\op_*}(X)$, where the second morphism is the canonical one. 
This finishes the construction of the functor $r_{0,X}$.
One  checks explicitly that this construction yields a well-defined transformation
\begin{equation*}\label{refuwhfqiuwehfiewfeqfew}
 r_{0} \colon \CW^{\op}\to \wt \bV_{\Top^{\op}_{*}}
\end{equation*}
between $\CAT$-valued functors.

In the following, we implicitly use the inclusion $\CAT\to \CATi$ given by the nerve
in order to view $ r_{0} \colon \CW^{\op}\to \wt \bV_{\Top^{\op}_{*}}$ as a natural transformation between
$\CATi$-valued functors.
The canonical functor \begin{equation}\label{vsdfvskjnjskdfvsdfvsf}
\psi \colon \Top^{\op}_{*}\to \Spc^{\op}_{*}
\end{equation}   induces the transformation
\[ \psi_{*} \colon \wt\bV_{\Top^{\op}_{*}}\to \wt\bV_{\Spc_{*}^{\op}}\ .\]
 
\begin{ddd}\label{def:realisation}
 We define the realisation transformation as the composition
 \[ r := \psi_* \circ r_0 \colon \CW^\op \to \wt\bV_{\Spc^\op_*}\ .\qedhere\]
\end{ddd}

We now show that the realisation of every controlled CW-complex is a
sheaf.
This will imply 
{that $r$ factors over the $\infty$-category} $\widehat \bV_{\Spc_{*}^{\op}}$ from \eqref{wergwregvfvfdsv}.

Let $X$ be a set and $U$ be a symmetric entourage of $X$ which contains the diagonal.
Furthermore, let $(Q,\lambda)$ be in $\CW(X)^\op$.
Recall the characterisation of $U$-sheaves in terms of the diagram \eqref{ervervewevevwevwee}.

\begin{prop}\label{goiwjegoergregergwreg}
	If $(Q,\lambda)$ is $U$-controlled, then $r(Q,\lambda)$ is a $U^2$-sheaf.
\end{prop}
\begin{proof} 
	We have  to show that the commutative diagram 
	\[\xymatrix@C=3.5em@R=2em{
	 \cP^{U^2\bd,\op}_X\ar[r]^-{r(Q,\lambda) \circ i}\ar[d]_-{i} & \Spc^\op_* \\
	 \cP^\op_X\ar[ur]_-{r(Q,\lambda)}
	}\]
	exhibits $r(Q,\lambda)$ as a right Kan extension of $r(Q,\lambda) \circ i$ along $i$,
	where $\cP^{U^2\bd}_X$ denotes the subposet of $\cP_X$ containing the $U^2$-bounded subsets of $X$.
	To make the argument  {easier to parse, we apply $(-)^\op$ to the above diagram.} 
	As before, denote by $\psi \colon \Top_* \to \Spc_*$ the canonical functor.
	Then we have to show that for every $Y$ in $\cP_X$ the canonical morphism
	\begin{equation}\label{eq:goiwjegoergregergwreg}
	 \mathop{\colim}\limits_{(Y' \to Y) \in i/Y} \psi(Q(Y')) \to \psi(Q(Y))
	\end{equation}
	is an equivalence.
	Note that $i/Y$ is isomorphic to the nerve of the subposet $\{ Y' \in \cP_X \mid Y' \subseteq Y,\ Y' \text{ is $U^2$-bounded}\}$ of $\cP_X$.
	
	{Since we need to calculate colimits of images of diagrams in $\Top_*$ under $\psi$,} we equip $\Top_*$ with the standard model structure.
	We claim that the underlying diagram
	\[ Q_Y \colon i/Y \to \Top_*,\quad (Y'\to Y) \mapsto Q(Y') \]
	is cofibrant in the projective model structure on $\Fun(i/Y,\Top_*)$.
	In fact, we will show that $Q_Y$ is a cell complex in the diagram category.
	There is a canonical filtration
	\begin{equation}\label{fwefafffdfasd}
* = Q_Y^{-1} \subseteq Q_Y^0 \subseteq Q_Y^1 \subseteq \ldots \subseteq Q_Y^n \subseteq \ldots \subseteq Q_Y\ ,
\end{equation}
	where $Q_Y^n$ sends $Y'$ to the $n$-skeleton of $Q(Y')$.
	The key observation is that an $n$-cell $q$ of $Q$ is contained in $Q(Y')$ if and only if $\lambda(\overline{q}^{CW}) \subseteq Y'$.
	It follows that $\colim_{n \in \nat} Q_Y^n \cong Q_Y$.	
	Moreover, there exists for every $n$ in $\nat$ a pushout
	\[\xymatrix@C=2em@R=2em{
	 \coprod\limits_{q \in z_{n+1}(Q(Y))} (S^n \times \Hom_{i/Y}(\lambda(\overline{q}^{CW}),-))_+ \ar[r]\ar[d] & Q_Y^n\ar[d] \\
	 \coprod\limits_{q \in z_{n+1}(Q(Y))} (D^{n+1} \times \Hom_{i/Y}(\lambda(\overline{q}^{CW}),-))_+\ar[r] & Q_Y^{n+1}
	}\]
	in $\Fun(i/Y,\Top_*)$.
	Thus, the filtration in \eqref{fwefafffdfasd} exhibits $Q_Y$ as a cell complex in $\Fun(i/Y,\Top_*)$.
	
	The comparison map from \eqref{eq:goiwjegoergregergwreg} factors as
	\[ \mathop{\colim}\limits_{(Y' \to Y) \in i/Y} \psi(Q(Y')) \to \psi(\mathop{\colim}\limits_{(Y' \to Y) \in i/Y} Q(Y')) \to \psi(Q(Y))\ .\]
	The first map is an equivalence since $Q_Y$ is a cofibrant diagram in the projective model structure.
	
	Since $\lambda(\overline{q}^{CW}) \subseteq U[\{\lambda(q)\}]$, the set $\lambda(\overline{q}^{CW})$ is $U^2$-bounded. It follows that $\colim_{i/Y} Q_Y \cong Q(Y)$.
	Hence the second map is also an equivalence.
\end{proof}

Recall the functor $\Sh$ and its localisation $\widehat \bV \colon \Coarse\times \CLLL\to \CLL$, see \eqref{wergwregvfvfdsv}.
By \cref{eporgjwepogrwegregwer},
 the  natural transformation  $\widehat\bV \to \wt\bV\circ (\id\times i)$ realises $\widehat\bV$ as a full subfunctor.

\begin{kor}\label{qiuerhgiwergerwferfefc}
 The realisation transformation factors over a transformation
 \[ r \colon \CW^{\op}\to \widehat \bV_{\Spc_{*}^{\op}}\ .\]
\end{kor}
\begin{proof}
 By  \cref{eporgjwepogrwegregwer}, $\widehat \bV_{\bC}(X)$ is the full subcategory of $\wt\bV_{\bC}(X)$ spanned by those objects which are sheaves on $X$.
 Then the corollary follows from \cref{goiwjegoergregergwreg}.
\end{proof}

There is a canonical notion of weak equivalence between objects in $\CW(X)$, namely that of a controlled homotopy equivalence; these are the weak equivalences of the Waldhausen structure on $\CW(X)$ considered in \cite{BKW:coarseA}.
We close this subsection by recalling the definition and showing that $r$ inverts controlled homotopy equivalences.

Let $X$ be a set and $U$ be an entourage of $X$. Let $(Q,\lambda)$ be an $X$-labelled CW-complex which is $U$-controlled.
Consider the unit interval $[0,1]$ as a CW-complex with exactly two $0$-cells and one $1$-cell.
Then the cylinder $[0,1]_+ \wedge Q$ carries an induced CW-structure and acquires a labelling via 
\[ z([0,1]_+ \wedge Q) \xrightarrow{\pr} z(Q) \xrightarrow{\lambda} X\ .\]
The cylinder on $(Q,\lambda)$ is the $X$-labelled CW-complex
\begin{equation}\label{eq:cylinder-cw}
 I(Q,\lambda) := ([0,1]_+ \wedge Q, \lambda \circ \pr)\ .
\end{equation}
This complex is also $U$-controlled.
The inclusions $\{0\} \to [0,1]$ and $\{1\} \to [0,1]$ determine $\diag(X)$-controlled morphisms $i_{0},i_{1} \colon (Q,\lambda) \to I(Q,\lambda)$.
Moreover, the projection $[0,1]_{+} \wedge Q\to Q$ is a $\diag(X)$-controlled morphism $I(Q,\lambda)\to (Q,\lambda)$.

Let $(Q',\lambda')$ be a second $U$-controlled CW-complex on $X$ and consider two morphisms $\phi_0, \phi_1 \colon (Q,\lambda)\to (Q',\lambda')$.

\begin{ddd}
 A $U$-controlled homotopy between $\phi_{0}$ and $\phi_{1}$ is a $U$-controlled morphism $H \colon I(Q,\lambda) \to (Q',\lambda')$ such that $\phi_{0} = H \circ i_{0}$ and $\phi_{1} = H \circ i_{1}$.
\end{ddd}

Let $\phi \colon (Q,\lambda) \to (Q',\lambda')$ be a $U$-controlled morphism between controlled CW-complexes. 
\begin{ddd} \label{ergbiojeorfqewfwfewffqf}
	The morphism $\phi$ is a $U$-controlled homotopy equivalence if there exist a $U$-controlled map $\phi' \colon (Q',\lambda') \to (Q,\lambda)$ and $U$-controlled homotopies $\phi' \circ \phi \sim \id_{Q}$ and $\phi \circ \phi' \sim \id_{Q'}$.
	
	If $X$ is a coarse space, we call $\phi$ a controlled homotopy equivalence if it is a $U$-controlled homotopy equivalence for some coarse entourage $U$ of $X$.
\end{ddd}

\begin{rem}\label{rem:controlledequiv-functorial}
 Since $I(f_*(Q,\lambda)) \cong f_*I(Q,\lambda)$ for every morphism $f \colon X \to X'$ of coarse spaces, the induced functor $f_* \colon \CW(X) \to \CW(X')$ preserves controlled homotopy equivalences. 
\end{rem}

\begin{rem}
 Controlled homotopy is an equivalence relation and is preserved under composition with controlled maps.
 It follows that controlled homotopy equivalences satisfy the two-out-of-six property.
\end{rem}

\begin{lem}\label{gfioqjgioewfewfewfqwefqwefqew}
	The realisation transformation $r$ sends controlled homotopy equivalences to equivalences.
\end{lem}
\begin{proof}
 Let $X$ be in $\Coarse$ and 
	$(Q,\lambda)$ be in $\CW(X)$.
	It suffices to show  that $r$ sends the projection
	$I (Q,\lambda)\to (Q,\lambda)$ to an equivalence.
	We observe that $r_{0}$ sends this map to a morphism represented by a map $r_{0,X}(Q,\lambda) \to [0,1]_{+} \wedge r_{0,X}(Q,\lambda)$ in $\PSh_{\Top_{*}^{\op}}(X)$.
	Since this map becomes an equivalence after postcomposition with $\psi$, the map $r(Q,\lambda) \to r(I(Q,\lambda))$ is an equivalence in $\widehat\bV_{\Spc_{*}^{op}}(X)$
\end{proof}

\subsection{Finiteness conditions}\label{sec:finiteness}

We consider a bornological coarse space $X$.
Recall the functor $\bV^{c,\perf}$ in \eqref{eoigjeoefvdsfsvdfv}.
We have a full subcategory $\bV^{c,\perf}_{\Spc_{*}^{\op}}(X)$ of 
$\widehat \bV_{\Spc_{*}^{\op}}(X)$.  If $ (Q,\lambda)$ is in $\CW(X)^{\op}$,
then we have  $r(Q,\lambda)$ in $\widehat \bV_{\Spc_{*}^{\op}}(X)$ by \cref{qiuerhgiwergerwferfefc}.
In this subsection, we introduce the notion of finite domination
which ensures that $r(Q,\lambda)$ belongs to the subcategory $\bV^{c,\perf}_{\Spc_{*}^{\op}}(X)$.

By construction, the functors $\widehat \bV$ from \eqref{wergwregvfvfdsv}, $ \bV$ from \eqref{afsaojqoifgafvasdva} and 
   $\bV^{c}$  from \eqref{fdsvkvsdfpvfvsfdvsvfdvsv}, all considered in the case of trivial $G$, give rise to a sequence of
 fully faithful functors
 \begin{equation}\label{qweoifhjqoiwefqwefqewf}
 \bV^{c}_{\Spc_{*}^{\op}}(X) \to \bV_{\Spc^\op_*}(X) \to \widehat\bV_{\Spc^\op_*}(X)\ .
\end{equation}
We consider $X$-controlled CW-complexes $(K,\kappa)$ and $(Q,\lambda)$ in $\CW(X)$.
Recall that $\cB_{X}$ denotes the bornology of $X$.

\begin{ddd}\label{wfiqjfioewfewfqwefqewfewf}\ 
 \begin{enumerate}
  \item The $X$-controlled CW-complex $(K,\kappa)$ is locally finite if the subcomplex $K(B)$ (see \eqref{qrfiojeroifefveffs} for notation) contains only finitely many cells for every $B$ in $\cB_X$.
  \item  The $X$-controlled CW-complex $(Q,\lambda)$ is finitely dominated if there exists a diagram
  \[ (Q',\lambda') \xrightarrow{i} (K,\kappa) \xrightarrow{p} (Q,\lambda) \]
  in $\CW(X)$ such that 
  $(K,\kappa)$ is  locally finite
   and the composition $p \circ i$ is a
  controlled homotopy equivalence.
  \item  \label{weriogjowegwergefww} We denote by $\CW^\hfd(X)$ the full subcategory of $\CW(X)$ consisting of the finitely dominated objects.\qedhere
 \end{enumerate}
\end{ddd}

Recall that $u \colon \BC\to \Coarse$ is the forgetful functor.
\begin{lem} \label{qerogijweroigjwerogegwerg}
 The collection of subcategories $\CW^\hfd(X)$ for all $X$ in $\BC$  forms a full subfunctor $\CW^{\hfd} \colon \BC \to \CAT$ of $\CW \circ u$.
\end{lem}
\begin{proof}
 If $f \colon X\to X'$ is a morphism in $\BC$ and $(Q,\lambda)$ is in $\CW^{\hfd}(X)$, then we must show that $f_{*}(Q,\lambda) \in  \CW^{\hfd}(X')$. 
 Consider a diagram  $(Q',\lambda') \to (K,\kappa) \to (Q,\lambda)$ in  $\CW(X)$      witnessing that $(Q,\lambda)$ is finitely dominated.
 Then the diagram $f_{*}(Q',\lambda')\to  f_{*}(K,\kappa)  \to f_{*} (Q,\lambda)$   witnesses that $f_{*}(Q,\lambda)$ is finitely dominated since
$f_{*}$ preserves controlled homotopies by \cref{rem:controlledequiv-functorial} and
$f_{*} (K,\kappa)$ is locally finite since $f$ is proper and $(K,\kappa)$ is locally finite.
\end{proof}

We consider the composition of the realisation transformation with the canonical inclusion into the idempotent completion
\begin{equation}\label{asdflkjoiadsvfasdf}
 \rp \colon \CW^\op \xrightarrow{r} \widehat\bV_{\Spc^\op_*} \to \Idem(\widehat\bV_{\Spc^\op_*})\ .
\end{equation}
We apply $\Idem$ to the fully faithful functor $\bV^{c}_{\Spc_{*}^{\op}}(X) \to \widehat\bV_{\Spc^\op_*}(X)$ from \eqref{qweoifhjqoiwefqwefqewf}.
Since $\Idem$ preserves fully faithfulness, the functor
\[ \bV^{c,\perf}_{\Spc^\op_*}(X) = \Idem(\bV^{c}_{\Spc_{*}^{\op}}(X)) \to \Idem(\widehat\bV_{\Spc^\op_*}(X)) \]
identifies $\bV^{c,\perf}_{\Spc^\op_*}(X)$ with a full subcategory of $\Idem(\widehat\bV_{\Spc^\op_*}(X))$.

Let $(Q,\lambda)$ be in $\CW(X)^\op$.
\begin{prop}\label{prop:locfin-diag}
 If $(Q,\lambda)$ is finitely dominated, then $\rp_{X}(Q,\lambda)$ belongs to $\bV^{c,\perf}_{\Spc_{*}^{\op}}(X)$.
\end{prop}
\begin{proof}
 Suppose for the beginning that $(K,\kappa)$ in $\CW(X)^\op$ is locally finite.
 We will first show that $r_X(K,\kappa)$ is a small sheaf.
 In view of \cref{def:realisation} and the definition of smallness given in  \cref{sec:controlled-objects},  
 we must show that $\psi(K(B))$ belongs to $\Spc_*^{\op,\omega}$
 for every $B$ in $\cB_{X}$.
 This is the case since
 $\psi$ sends finite CW-complexes to cocompact objects, and $(K,\kappa)$ is a finite CW-complex $K(B)$ by the assumptions on $(K,\kappa)$. 
 
 The subset $\kappa(K)$ of $X$ is by assumption a locally finite subset of $X$.
 In particular, $r_X(K,\kappa)$  belongs to the image of $i_{*} \colon \bV_{\Spc^{\op}_*}(\kappa(K))\to \bV_{\Spc^{\op}_*}(X)$, where $i \colon \kappa(K) \to X$ is the inclusion.
 By \cref{weigjwoerfewv}, this implies that $p_X(K,\kappa)$ belongs to $\bV^{c}_{\Spc^{\op}_*}(X)$.
 
 Suppose now that $(Q,\lambda)$ is finitely dominated. Then \cref{gfioqjgioewfewfewfqwefqwefqew} implies that $\rp_X(Q,\lambda)$ is a retract of $\rp_X(K,\kappa)$ for some locally finite $(K,\kappa)$ in $\CW(X)^\op$.
 So $\rp_X(Q,\lambda)$ belongs to $\Idem(\bV^{c}_{\Spc^{\op}_*}(X))$ as claimed.
\end{proof}

In view of  \cref{prop:locfin-diag} we can make the following definition.

\begin{ddd}\label{def:rfd}
 We let
 \[ r^\hfd \colon \CW^{\hfd,\op} \to {\bV^{c,\perf}_{\Spc^{\op}_*}}\ .\]
 be the natural transformation obtained by restricting the natural transformation \eqref{asdflkjoiadsvfasdf}.
\end{ddd}

\begin{rem}\label{qeirughioqergegwergweg}
 {Let $X$ be a bornological coarse space and let $Y$ be a union of coarse components.
 Consider an object $(Q,\lambda)$ in $\CW(X)$.
 Since the subcomplex generated by a single cell of $Q$ is necessarily supported on a coarse component of $X$, the subcomplex $Q(Y)$ coincides with the subcomplex spanned by all cells whose label lies in $Y$.
 Similarly, if $\phi \colon (Q,\lambda) \to (Q',\lambda')$ is a morphism in $\CW(X)$, the image $\phi(q)$ of a cell $q$ with $\lambda(q) \in Y$ can only have non-trivial intersections with cells in $Q(Y)$.
 Consequently, there exists a restriction functor
 \[ \CW(X) \to \CW(Y),\quad (Q,\lambda) \mapsto (Q(Y),\lambda_{|Q(Y)})\ .\]
 This functor preserves finitely dominated objects, so we also have a restriction functor $\CW^{\hfd}(X)\to \CW^{\hfd}(Y)$.}

 {Unwinding definitions, one checks that there exists a commutative diagram
 \[\xymatrix{
  \CW^\hfd(X)\ar[r]\ar[d]_-{r^\hfd} & \CW^\hfd(Y)\ar[d]^-{r^\hfd} \\
  \bV^{c,\perf}_{\Spc^\op_*}(X)\ar[r]^-{(-)_{|Y}} & \bV^{c,\perf}_{\Spc^\op_*}(Y)
 }\]
 in which the functor $(-)_{|Y}$ is the restriction functor from \cref{qeorighqoergergwregr}.}
\end{rem}

In the remainder of this section, we formulate a criterion to recognise finitely dominated objects in $\CW(X)$.

Let $f \colon S \to T$ be a map of topological spaces and let $U$ be an entourage of $T$.
\begin{ddd}
 The map $f$ is $U$-bounded if $f(S)$ is a $U$-bounded subset of $T$.
\end{ddd}

Let $T$ be a topological space, and let $U$ be an open entourage of $T$ containing the diagonal.
We let $T_{U}$ denote the coarse space obtained by equipping $T$ with the coarse structure generated by $U$.
Define $\Sing^U(T)$ as the sub-simplicial set of the singular complex $\Sing(T)$ consisting of the $U$-bounded singular simplices.
Taking the geometric realisation and adjoining a base point, we obtain the based CW-complex $\abs{\Sing^{U}(T)}_+$. Note that its set of cells $z(\abs{\Sing^{U}(T)})$ is the set of non-degenerate singular simplices.
We equip $\abs{\Sing^U(T)}_+$ with the labelling $\lambda \colon z(\abs{\Sing^{U}(T)}) \to T$ which sends a singular simplex $\sigma \colon \Delta^n \to T$ to $\sigma(b)$, where $b$ is the barycentre in $\Delta^{n}$.
The pair $(\abs{\Sing^U(T)}_+,\lambda)$ is a controlled CW-complex over $T$.

\begin{ddd}\label{def:small-simplices}
 We set
 \[ C^U(T) := (\abs{\Sing^U(T)}_+,\lambda) \text{ in } \CW(T_{U})\ .\qedhere\]
\end{ddd}

Let $T$ and $T'$ be topological spaces and $X$ be a set with an entourage $U$.
Suppose that $\ell \colon T \to X$ and $\ell' \colon T' \to X$ are functions between the underlying sets.
\begin{ddd}\label{def:Ucontrol-top}
 A continuous map $f \colon T' \to T$ is $U$-controlled if
 \[ \{ (\ell'(t),\ell(f(t))) \mid t \in T' \} \subseteq U\ .\qedhere\]
\end{ddd}

We define the cylinder on $(T,\ell)$ as the pair
\begin{equation}\label{eq:cylinder-top}
 I(T,\ell) := (T \times [0,1], T \times [0,1] \xrightarrow{\pr} T \xrightarrow{\ell} X)\ .\qedhere
\end{equation}

\begin{ddd}
 A $U$-controlled homotopy is a $U$-controlled map $I(T',\ell') \to (T,\ell)$.
\end{ddd}

\begin{rem}\label{rem:Ucontrol-top}
 The notion of a $U$-controlled homotopy is usually phrased in terms of open covers of $T$, see for example \cite[Ch.~IV.1]{hu}. If $\cU$ is an open cover of $T$, then $U := \bigcup_{V \in \cU} (V \times V)$ is an open entourage of $T$, and the notion of $U$-homotopy reduces to the definition in \cite{hu}.
  Similarly, every open entourage $U$ containing the diagonal induces an open cover $\cU := \{ V \subseteq T \mid V \text{ open and } V \times V \subseteq U \}$ such that a $\cU$-homotopy in the sense of \cite[Ch.~IV.1]{hu} is the same as a $U$-homotopy in the above sense.
  We will use this translation in the proof of \cref{prop:anrhfd}.
\end{rem}

If $U'$ is a second open entourage on the topological space $T$ such that $U' \subseteq U$, then there is a natural inclusion $C^{U'}(T) \to C^U(T)$ in $\CW(T_{U})$.
\begin{lem}\label{lem:small-simplices}
 The inclusion $C^{U'}(T) \to C^U(T)$ is a $U$-controlled homotopy equivalence.
\end{lem}
\begin{proof}
 Note that the inclusion is a homotopy equivalence by the excision property of singular homology with respect to open coverings.
 The difficulty lies in showing that the map is a controlled homotopy equivalence. 
 For metric spaces this is done in \cite[Lem.~7.21(2)]{Enkelmann:2018aa}.
 The same argument applies to spaces equipped with an open entourage: whenever one speaks of $\delta$-control in the metric world, one replaces this by $U$-control.
\end{proof}

\begin{construction}\label{const:small-simplices-induced-maps}
 Let $T$ be a topological space and $X$ be a coarse space.
 Let $V$ be an open entourage of $T$. 
 Then $C^V(T)$ is an object of $\CW(T_{V})$.
 If $\ell:T\to X$ is a function such that $\ell(V)$ is a coarse entourage of $X$, then $\ell \colon T_{V}\to X$ is a morphism of coarse spaces.
 
 Consider another topological space $T'$.
 If $f \colon T' \to T$ is a continuous map and $V'$ is an open entourage of $T'$ such that $f(V') \subseteq V$, then $f$ induces a map $\Sing^{V'}(T') \to \Sing^V(T)$.
 If $\ell' \colon T' \to X$ is a second  function such that $f$ is $U$-controlled for some coarse entourage $U$ of $X$, it follows that $\ell'(V')$ is a coarse entourage of $X$ and the map $\Sing^{V'}(T') \to \Sing^V(T)$ induces a controlled morphism
 \[ f_\sharp \colon \ell'_*C^{V'}(T') \to \ell_*C^V(T) \]
 in $\CW(X)$. We use this construction freely in the sequel.
\end{construction}

Let $T$ and $T'$ be topological spaces and $X$ be a coarse space.
Let $V$ be an open entourage of $T$.
Suppose that $\ell \colon T \to X$ and $\ell' \colon T' \to X$ are functions between the underlying sets such that $\ell(V)$ is a coarse entourage of $X$.
Consider {an $\ell(V)$}-controlled homotopy $h \colon I(T',\ell') \to (T,\ell)$.

\begin{lem}\label{lem:Uhomotopy}
 There exist an open entourage $V'$ of $T'$ such that $h_i(V') \subseteq V$ for $i=0,1$ and a controlled homotopy between the induced morphisms $h_{i,\sharp} \colon \ell'_*C^{V'}(T') \to \ell_*C^{V}(T)$ in $\CW(X)$.
\end{lem}
\begin{proof}
 For an arbitrary entourage $V'$ of $T'$, denote by $I(V')$ the entourage $V' \times [0,1]^2$ of $T' \times [0,1]$.
 Since $h$ is continuous and $[0,1]$ is compact, there exists an open entourage $V'$ of $T'$ such that $I(V') \subseteq h^{-1}(V)$. In particular, $h_i(V') \subseteq V$ for $i=0,1$.
 
 Consider the cylinder $I(C^{V'}(T'))$ in $\CW(T'_{V'})$, which was defined in \eqref{eq:cylinder-cw}.
 Since $[0,1] \cong \abs{\Delta^1}$ and both $\Sing$ and geometric realisation commute with finite products, the unit map $\Delta^1 \to \Sing([0,1])$ induces a map
 \[ \abs{\Sing^{V'}(T')} \times [0,1] \to \abs{\Sing^{V'}(T') \times \Sing([0,1])} \cong \abs{\Sing^{I(V')}(T' \times [0,1])} \]
  which is a controlled morphism $I(C^{V'}(T' )) \to \pr_*C^{I(V')}(T' \times [0,1])$ in $\CW(T'_{V'})$, where $\pr \colon T' \times [0,1] \to T'$ denotes the projection.
  Using this morphism, we obtain an induced controlled homotopy
 \[ I(\ell'_*C^{V'}(T')) \cong \ell'_*I(C^{V'}(T')) \to \ell'_*\pr_*C^{I(V')}(T' \times [0,1]) \xrightarrow{h_\sharp} \ell_*C^V(T) \]
 in $\CW(X)$ which restricts to $h_{i,\sharp} \colon \ell'_*C^{V'}(T') \to \ell_*C^V(T)$ at its endpoints.
\end{proof}

Let $T$ be a locally compact topological space  and $U$ be an open entourage of $T$.
We assume that
for every relatively compact subset $B$ of $T$ also the $U$-thickening $U[B]$ (see \eqref{qwefoiheiuohfiqwefewfewfqeefedq}) is  relatively compact.
Under this condition, we can consider the  bornological coarse space 
 $T_{U,\rc}$ obtained by equipping $T$ with the coarse structure generated by $U$ and the bornology of relatively compact subsets.
 
\begin{prop}\label{prop:anrhfd}\label{rgoirjgorgrwrvcwecr}
 If $T$ is a locally compact ANR,
 then $C^U(T)$ belongs to $\CW^\hfd(T_{U,\rc})$. 
\end{prop}
\begin{proof}
 Using \cref{rem:Ucontrol-top} to translate into our terminology, there exist by \cite[Cor.~IV.6.2]{hu} a locally finite simplicial complex $K$, continuous maps $\alpha \colon T \to K$ and $\omega \colon K \to T$ and a $U$-homotopy $h \colon \omega \circ \alpha \sim \id_{T}$ of controlled maps $(T,\id_T) \to (T,\id_T)$.
 By \cref{lem:Uhomotopy}, there exist an open entourage $V$ of $T$ such that $V$ and $\alpha^{-1}\omega^{-1}(V)$ are contained in $U$ and a controlled homotopy $k \colon I(C^V(T)) \to C^U(T)$ in $\CW(T_{U})$ from the inclusion $C^V(T) \to C^U(T)$ to $(\omega \circ \alpha)_\sharp \colon C^V(T) \to C^U(T)$.
  
 Note that $(\omega \circ \alpha)_\sharp$ factors as
 \[ C^V(T) \xrightarrow{\alpha_\sharp} \omega_*C^{\omega^{-1}(U)}(K) \xrightarrow{\omega_\sharp} C^U(T)\ .\]
 Replacing $K$ by an appropriate iterated barycentric subdivision, we may assume that all simplices in $K$ are $\omega^{-1}(U)$-bounded; note that the number of necessary subdivisions is locally bounded, but may not be globally bounded.
 Denote by $D(K)$ the tautological object in $\CW(K_{\omega^{-1}(U)})$ given by $K$ (labelling each simplex by its barycentre).
 By \cite[Lem.~7.21(3)]{Enkelmann:2018aa}, the canonical inclusion $j \colon D(K) \to C^{\omega^{-1}(U)}(K)$ admits an $\omega^{-1}(U)$-controlled homotopy inverse $t$.
 Consequently, $\omega_\sharp \circ \alpha_\sharp$ is controlled homotopic to the composition
 \[ C^V(T) \xrightarrow{\alpha_\sharp} \omega_*C^{\omega^{-1}(U)}(K) \xrightarrow{\omega_*t} \omega_*D(K) \xrightarrow{\omega_*j} \omega_*C^{\omega^{-1}(U)}(K) \xrightarrow{\omega_\sharp} C^U(T)\ .\]
 Since $\omega_*D(K)$ is locally finite and the entire composition is controlled homotopic to the inclusion $C^V(T) \to C^U(T)$, it is a controlled homotopy equivalence by \cref{lem:small-simplices}.
 Hence $C^U(T)$ lies in $\CW^{\hfd}(T_{U,\rc})$.
\end{proof}

\section{Finitely \texorpdfstring{$\cF$}{F}-amenable groups}\label{sec:famenablegroups}

Recall the notion of phantom equivalence from \cref{ethigoewggergwgerg}.
The notion of finite homotopy $\cF$-amenability will be introduced in \cref{def:f-amenable} below.
This section is dedicated to the proof of the following theorem.

Let $G$ be a finitely generated group and $\cF$ be a family of subgroups.
Let $\bC$ be a left-exact $\infty$-category with $G$-action and let $\Homol \colon \Cle \to \bM$ be a functor to a stably monoidal and cocomplete stable $\infty$-category which admits countable products.
We consider {the functor $\Homol \bC_G \colon G\Orb \to \Clep$ introduced in \cref{weoirgjwegwergwerg9}}.

\begin{theorem}
	\label{thm:famenable}
	Assume that
	\begin{enumerate}
	 \item $G$ is finitely homotopy $\cF$-amenable;
	 \item $\Homol$ is a lax monoidal, finitary localising invariant.
	\end{enumerate}
	Then the assembly map
	\[ \As_{\cF,\Homol\bC_{G}} \colon \mathop{\colim}\limits_{G_\cF\Orb} \Homol\bC_G \to \Homol\bC_G(*) \]
 is a phantom equivalence.
\end{theorem}

\subsection{Finitely homotopy \texorpdfstring{$\cF$}{F}-amenable groups}
\label{sec:famenable}

The notion of finite (homotopy) $\cF$-amenability goes back to \cite{blr,BL-borel,wegner:cat0}, where it was used to prove instances of the $K$-theoretic Farrell--Jones conjecture with coefficients in additive categories.
The formulation in \cref{def:f-amenable} below was given in \cite[Def.~2.11 \& Thm.~2.12]{Bartels:icm}.

Let $G$ be a group and $Z$ be a topological space. 

\begin{ddd}\label{qergqergreqfewfqef}
	A homotopy coherent $G$--action $(\Gamma,Z)$ is a continuous map
	\[
	\Gamma \colon \coprod_{k=0}^\infty\left(\prod_{j=1}^{k}\left(G \times [0,1]\right) \times G \times Z\right)\to Z
	\]
	with the following properties:
	\begin{equation}\label{qergojrgopregqfwefwefqw}
	\Gamma(g_k,t_k,\dots,g_1,t_1,g_0,z) = \begin{cases}
	\Gamma(g_{k},t_{k}\dots,g_j, \Gamma(g_{j-1},t_{j-1},\dots,g_{0},z)) & t_j = 0, 1\le j\le k \\
	\Gamma(g_{k},t_{k},\dots,t_{j+1},g_jg_{j-1},t_{j-1},\dots,g_{0},z) & t_j = 1, 1\le j\le k \\
	\Gamma(g_k,t_{k},\dots,g_2,t_2,g_1,z) & g_0 = e \\
	\Gamma(g_k,t_{k},\dots,g_{j+1},t_{j+1}t_j,g_{j-1},\dots,g_0,z) & g_j = e, 1 \leq j < k-1 \\
	\Gamma(g_{k-1},t_{k-1},\dots,g_0,z) & g_k = e
	\\
	x & g_0 = e, k=0
	\end{cases}
	\end{equation}
	Here we use the convention that non-existing entries are dropped,
	e.g.\ $g_{k}$, $t_{k}$ in the first line if $j=k$ or the entry $t_{j-1}$ in the second line if $j=1$.
\end{ddd}

\begin{rem}
 \cref{qergqergreqfewfqef} is a special case of the notion of homotopy coherent diagram introduced by Vogt \cite{vogt}. It is an explicit model (in the topologically enriched context) for a functor $BG \to \Spc$ whose underlying object is $\ell(X)$, where $\ell \colon \Top\to \Spc$ is the canonical functor. See \cite{cordier} for further discussion.
\end{rem}

Let $G$ be a finitely generated group and $\cF$ be a family of subgroups.
\begin{ddd}\label{def:f-amenable}
	The group $G$ is finitely homotopy $\cF$-amenable if there exist
	\begin{enumerate}
	 \item a family $(\Gamma_n,Z_n)_{n\in \nat}$ of homotopy coherent $G$-actions,
	 \item a family $(W_n)_{n \in \nat}$ of $G$-simplicial complexes,
	 \item a family $(f_n)_{n \in \nat}$ of continuous maps $f_n \colon Z_n \to W_n$
	\end{enumerate}
	such that the following holds:
	\begin{enumerate}
		\item\label{thiowergergwergwreg} for every $n$ in $\nat$ the topological space $Z_n$ is a compact AR;\footnote{AR stands for absolute retract (with respect to the class of metrisable spaces). See \cite[Sec.~III.6]{hu}. An ANR (absolute neighbourhood retract) is an AR if and only if it is contractible \cite[Thm.~7.1 \& Prop.~7.2]{hu}.}
		\item\label{it:f-amenablecond2} for every $n$ in $\nat$ the stabilisers of $W_n$ belong to $\cF$;
		\item\label{it:f-amenablecond3} $\sup\limits_{n \in \nat} \dim W_n < \infty$;
		\item \label{erguiergqwefeqwfqefew} for all $k$ in $\nat$ and all collections $g_0, \ldots, g_k$ in $G$ we have 
		\[ \sup_{\substack{(t_1,\ldots,t_k) \in [0,1]^{k} \\ z \in Z_n}} d(f_n(\Gamma_n(g_k,t_k,\ldots,t_1,g_0,z)), g_k\ldots g_0 f_n(z)) \xrightarrow{n \to \infty} 0\ .\qedhere\]
	\end{enumerate}
\end{ddd}
In Condition~\ref{def:f-amenable}.\ref{erguiergqwefeqwfqefew}, we equip the simplicial complexes $W_{n}$ with their spherical path metrics (or alternatively with the $\ell^1$-metric, see the discussion in \cref{beforergiorgergegergergerg3232424}).

\begin{rem}
	The condition formulated in \cref{def:f-amenable} is slightly weaker than the assumptions in \cite[Thm.~2.12]{Bartels:icm} since we do not require a uniform bound on the dimension of the ARs $Z_n$. In practice, however, the dimensions of the simplicial complexes $W_n$ are usually bounded in terms of the dimensions of the spaces $Z_n$. In this case $(Z_n)_n$ is a sequence of ERs\footnote{ER stands for Euclidean retract. These are precisely the finite-dimensional ARs.} with uniformly bounded covering dimension. 
\end{rem}

The proof of \cref{thm:famenable} relies on \cref{ergiooegergergwergergergw}.
We will apply this proposition to the functors
\begin{equation}\label{eq:V-and-U}
 V := \bV^{c,\perf}_{\bC,G} \colon G\BC \to \Clep\quad \text{and}\quad U := \bV^{c,\perf,G}_{\Spc^\op_*} \colon G\BC \to \Clep\ ,
\end{equation}
where $\bV^{c,\perf}_{\bC,G}$ is the evaluation of $\bV^{c,\perf}_G$ from \cref{def:VcperfG} at $\bC$ and $\bV^{c,\perf,G}_{\Spc_*^\op}$ is the evaluation of $\bV^{c,\perf,G}$ from \cref{def:VcperfG} at $\Spc_*^\op$.
By \cref{thm:fix-orbit-wmodule}, $V$ admits a weak module structure $(\eta,\mu)$ over the $\pi_0$-excisive functor $U$, see \cref{giooergrefwerfwrevwerfv}.

Let $\Homol \colon \Cle \to \bM$ be a lax monoidal, finitary localising invariant.
By \cite[Prop.~5.2]{unik}, the functor $\Homol V \colon G\BC \to \bM$ extends the functor  $\Homol \bC_G$ from {\cref{weoirgjwegwergwerg9}}
in the sense of \cref{wtgklpwergrewfwref}.
We are going to construct a transfer class $(X,t)$ for the tuple $(U,\eta,V,\Homol,\cF)$, see \cref{def:transfer-class}.
For an appropriate choice of $X$, the morphism $t$ will be determined by an object $t_0$ in $U(X)$ as explained in \cref{faklolergergwergwegwergwreg}.
This means that we have to produce a homotopy fixed point in $\bV^{c,\perf}_{\Spc^\op_*}(X)$ from the point-set data provided by the assumption that $G$ is finitely homotopy $\cF$-amenable.
The next section describes an auxiliary construction which allows us to do this.

\subsection{Shift categories}\label{qroihqiurgerggegwegweg}

In this subsection, it is useful not to drop the nerve functor $\Nerve \colon \Cat \to \Cati$ from notation.
By abuse of notation, we write $\Cat[\Equiv^{-1}]$ for the essential image of the nerve functor.
This notation is motivated by the fact that the factorisation of the nerve through its essential image $\Cat \to\Cat[\Equiv^{-1}]$
 presents its target as the Dwyer--Kan localisation of $\Cat$ at the equivalences of categories.

We denote the category of relative categories by $\RelCat$.
Its objects are pairs $(\cC,W)$ of $\cC$ in $\Cat$ and a wide subcategory $W$ of $\cC$. Morphisms $(\cC,W)\to (\cC',W')$  in $\RelCat$ are functors $\cC
\to \cC'$ sending $W$ to $W'$. Given $(\cC,W)$ in $\RelCat$ we can consider the Dwyer--Kan localisation
\begin{equation}\label{vsavfasvasdvsdavsac}
\ell \colon \Nerve(\cC)\to \Nerve(\cC)[W^{-1}]
\end{equation}
in $\Cati$.
The formation of these localisations is functorial with respect to morphisms in $\RelCat$  and gives rise to a localisation functor
\[ \RelCat \to \Fun(\Delta^1, \Cati),\quad (\cC,W) \mapsto \big( \ell \colon \Nerve(\cC) \to \Nerve(\cC)[W^{-1}] \big)\ . \]
Let $(\cC,W)$ be in $\Fun(BG,\RelCat)$.
By functoriality, $\ell$ induces a canonical morphism
\begin{equation}\label{eq:limell}
 \lim_{BG} \Nerve(\cC) \to \lim_{BG} \Nerve(\cC)[W^{-1}]\ .
\end{equation}
Hence we can produce objects on the right hand by providing objects on the left hand side.
To describe objects in the latter, it is useful to have an explicit model of $\lim_{BG} \Nerve(\cC)$.

\begin{construction}\label{rem:fixedpoints}
 For a category $\cC$ with strict $G$-action we define a new category $\cC^{\hG}$.
 In \cref{wroggrdfg} we will show that its nerve  models $\lim_{BG} \Nerve(\cC)$.
\begin{enumerate}
 \item The objects of $\cC^{\hG}$ are pairs $(C,\rho)$ consisting of an object $C$ of $\cC$ together with a family $\rho = (\rho_{g})_{g\in G}$ of morphisms  $\rho_{g} \colon C \to g(C)$ satisfying the cocycle condition
\begin{equation}\label{3iojoigergwergerg}
	g(\rho_{g^{\prime}}) \rho_{g}=\rho_{gg'} \, \quad \mbox{for all}\ g,g^{\prime}\ \text{in}\ G\ .
\end{equation}
 \item A morphism $(C,\rho) \to (C',\rho')$ is given by a morphism $f \colon C \to C'$ in $\cC$ such that the diagram
 \[\xymatrix{
  C\ar[r]^-{f}\ar[d]_{\rho(g)} & C'\ar[d]^-{\rho'(g)} \\
  g(C)\ar[r]^-{g(f)} & g(C')
 }\]
 commutes for all $g$ in $G$.\qedhere
\end{enumerate}\end{construction}
\begin{lem}\label{wroggrdfg}
 We have a natural equivalence
 \begin{equation}\label{qewfoihqiuwhfiuwewfwqewfwfq}
\iota:  \Nerve(\cC^{\hG}) \xrightarrow{\simeq} \lim_{BG} \Nerve(\cC)\ .  
\end{equation} 
\end{lem}
\begin{proof}
This can be shown similarly as \cite[Thm.~3.4.3]{Bunke:ab} which is an analogous result for  additive categories.
In the following, we sketch the argument. One considers the  model structure on the category of small categories given in \cite{rezk} which models the localisation $\Nerve \colon \Cat \to\Cat[\Equiv^{-1}]$ from above.
One then equips $\Fun(BG,\Cat)$ with the corresponding injective model category structure.   Given $\cC$ in $\Fun(BG,\Cat)$, by the general relation between homotopy limits in model categories and limits in the associated $\infty$-categories explained e.g.\ in \cite[Prop.~13.6]{bunke} or \cite[Sec.~4.2.4]{htt}, the $\infty$-category $\lim_{BG}\Nerve(\cC)$ is represented by
$\Nerve(\lim_{BG} R\cC)$, where $R\cC$ is some fibrant resolution of $\cC$.
An explicit choice of fibrant resolution is given by $\Fun(\wt G,\cC)$, where
$\wt G$ is the category with $G$-action whose underlying $G$-set of objects is $G$ with the left action by $G$, and whose morphism sets consist of single points.
The group $G$ acts on $\Fun(\wt G,\cC)$ by conjugation.
One checks by an explicit calculation that $\lim_{BG}\Fun(\wt G,\cC)\simeq \cC^{\hG}$.
\end{proof}
 
 We would like to apply the above construction to the relative $G$-category $(\cC,W) := (\CW^\hfd(X),\cW(X))$ for a certain $G$-bornological coarse space $X$, where $\cW(X)$ is the subcategory of  controlled homotopy equivalences.
However, $\CW^\hfd(X)$ itself  turns out to be too small 
to host the required homotopy invariants.
In the following, we describe
an enlargement of general objects $(\cC,W)$ in $\Fun(BG,\RelCat)$
which will increase the chance to find  homotopy invariants, and which indeed will work in our concrete application.
 
We start with introducing a category $\Fun_{\mathrm{shift}}^{W}(\nat,\cC)$ together with a factorisation of $\ell$ from \eqref{vsavfasvasdvsdavsac} as
\[ \Nerve(\cC) \to \Nerve(\Fun_{\mathrm{shift}}^{W}(\nat,\cC)) \to \Nerve(\cC)[W^{-1}]\ ,\]
see \cref{prop:shiftcat} below.
Since the category $\Fun_{\mathrm{shift}}^{W}(\nat,\cC)$ is in general bigger than $\cC$, it tends to be easier to construct objects in $\Fun_{\mathrm{shift}}^W(\nat,\cC)^{\hG}$.
In summary, our aim is to contruct homotopy fixed points in $\Nerve(\cC)[W^{-1}]$, but it is easier to do the construction in the 1-categorical setting. For this we work in the auxilliary category $\Fun_{\mathrm{shift}}^{W}(\nat,\cC)$ which is large enough to contain the needed fixed points but still explicit enough to do the construction.
The idea of enlarging $\cC$ to the category $\Fun^W_{\mathrm{shift}}(\nat,\cC)$ is originally due to Bartels and Reich \cite[Sec.~8.2]{br:fjc} and was also used in \cite{blr} for the construction of a transfer.

Let $\cC$ be in $\Cat$.
Considering $\nat$ as a poset, we have the functor category $\Fun(\nat,\cC)$ whose objects are given by sequences
\[ C_{0}\xrightarrow{f_{0}} C_{1} \xrightarrow{f_{1}} C_{2} \xrightarrow{f_{2}} \dots\ . \]
Let $T \colon \nat \to \nat$ denote the functor given by $T(n) := n+1$.
{We call the} restriction 
\begin{equation}\label{erfwefiweogwergref}
T^{*}\colon\Fun(\nat,\cC)\to \Fun(\nat,\cC)
\end{equation}
{along $T$ the shift functor.}
The canonical natural transformation $\id_\nat \to T$ induces a natural transformation\begin{equation}\label{qewfopkpfqwefqewqd}
 v \colon \id \to T^* \colon \Fun(\nat,\cC)\to \Fun(\nat,\cC)
\end{equation}
 between the corresponding restriction functors.

If $W$ is a set of morphisms in $\cC$, then we consider the full subcategory $\Fun^W(\nat,\cC)$ of $\Fun(\nat,\cC)$ consisting of the objects $(C_n,f_n)_{n \in \nat}$ with $f_n$ in $W$ for all $n$ in $\nat$.

\begin{ddd} \label{wjgriegjorgirewgwergwrg}
	We define the category $\Fun^{W}_{\mathrm{shift}}(\nat,\cC)$ as follows:
	\begin{enumerate}
		\item\label{it:wjgriegjorgirewgwergwrg1} The objects of  $\Fun^{W}_{\mathrm{shift}}(\nat,\cC)$ are  the objects of $\Fun^{W}(\nat,\cC)$.
		\item\label{it:wjgriegjorgirewgwergwrg2} A morphism $C \to D$ is an equivalence class of pairs $(k,\phi)$ of $k$ in $\nat$ and a morphism $\phi \colon C \to T^{k,*}D$ in $\Fun^{W}(\nat,\bC)$  subject to the equivalence relation generated by
		\begin{equation*}\label{fweiouuh2ui43rergerg}
		(k,\phi)\sim (k+1,T^{k,*}(v_{D})\circ \phi)\ ,
		\end{equation*}
		where $v_{D}$ is the natural transformation $v$ from \eqref{qewfopkpfqwefqewqd} evaluated at $D$.
		We denote the equivalence class of $(k,\phi)$ by $[k,\phi]$.
		\item\label{it:wjgriegjorgirewgwergwrg3} The composition of morphisms is given by
		\[ [k,\phi] \circ [k^{\prime},\phi^{\prime}] := [k+k',T^{k^{\prime},*}(\phi)\circ \phi^{\prime}]\ .\qedhere\]
	\end{enumerate}
\end{ddd}

One checks that composition in $\Fun^W_{\mathrm{shift}}(\nat,\cC)$ is well-defined and associative.

Let
\begin{equation}\label{eq:V}
 V := \{ v_C \mid C \in \Fun^W(\nat,\cC) \}
\end{equation}
be the set of morphisms in $\Fun^W(\nat,\cC)$ consisting of the components of $v$ in \eqref{qewfopkpfqwefqewqd}.
We have a canonical functor
\begin{equation}\label{fwefwefwfwefw}
 \ell_V \colon \Fun^{W}(\nat,\cC)\to \Fun_{\mathrm{shift}}^{W}(\nat,\cC)
\end{equation}
which is the identity on objects and sends a morphism $\phi$ to $[0,\phi]$. 

\begin{lem}\label{lem:fractions}
 The functor $\ell_V$ exhibits 
 \begin{enumerate}
  \item $\Fun_{\mathrm{shift}}^{W}(\nat,\cC)$ as the localisation of $\Fun^{W}(\nat,\cC)$ at $V$ in $\Cat$;
  \item $\Nerve(\Fun_{\mathrm{shift}}^{W}(\nat,\cC))$ as the localisation of $\Nerve(\Fun^{W}(\nat,\cC))$ at $V$ in $\Cati$.
 \end{enumerate}  
\end{lem}
\begin{proof}
 Both assertions will follow from \cref{lem:calculus}.
 Observe that the subcategory $\cV$ generated by $V$ consists precisely of all maps $v_C^k \colon C \to T^{k,*}C$ for $k$ in $\nat$ and $C$ in $\Fun^W(\nat,\cC)$.
 It is straightforward to check that $\cV$ satisfies a calculus of left fractions in the sense of Gabriel and Zisman \cite[Sec.~2]{GabrielZisman}.
 One checks that $\ell_V$ inverts $V$.
 Hence the universal property of Gabriel--Zisman's $\cV^{-1}\Fun^W(\nat,\cC)$ provides a functor $\cV^{-1}\Fun^W(\nat,\cC) \to \Fun^W_{\mathrm{shift}}(\nat,\cC)$.
 Comparing \eqref{eq:calculus-hom} with \cref{wjgriegjorgirewgwergwrg}.\ref{it:wjgriegjorgirewgwergwrg2}, one sees that this functor is fully faithful and hence an equivalence of categories. 
 So the {lemma follows} from \cref{lem:calculus}.
\end{proof}

Let $(\cC,W)$ be in $\Fun(BG,\RelCat)$. Then the projection $\nat \to *$ induces a morphism
\begin{equation}\label{fwefwefwfwefw2}
 j \colon \cC \to \Fun^W(\nat,\cC)\ .
\end{equation}
\begin{prop}\label{prop:shiftcat}
 There exists a commutative diagram
 \[\xymatrix{
  \Nerve(\cC)\ar[rr]^-{\ell}\ar[rd]_-{\ell_V \circ j} & & \Nerve(\cC)[W^{-1}] \\
  & \Nerve(\Fun^W_{\mathrm{shift}}(\nat,\cC))\ar[ru]_-{s} &
 }\]
 in $\Fun(BG,\Cati)$ which is functorial in the relative category $(\cC,W)$.
 
 If $\tau = (\tau_n)_{n \in \nat}$ is a morphism in $\Fun^W(\nat,\cC)$ such that $\tau_n$ belongs to $W$ for all $n$, then $s(\ell_V(\tau))$ is an equivalence in $\Nerve(\cC)[W^{-1}]$.
\end{prop}
\begin{proof}
 Let $\Fun^\iota(\nat,\cC[W^{-1}])$ be the full subcategory of functors $\nat \to \Nerve(\cC)[W^{-1}]$ which factor through the groupoid core of $\Nerve(\cC)[W^{-1}]$.
 The functor $\ell_* \colon \Nerve(\Fun^W(\nat,\cC)) \to \Fun(\nat,\Nerve(\cC)[W^{-1}])$ induced by $\ell$ factors through $\Fun^\iota(\nat,\Nerve(\cC)[W^{-1}])$.
 We obtain a commutative diagram
 \[\xymatrix{
  \Nerve(\cC)\ar[r]^-{\ell}\ar[d]_-{j} & \Nerve(\cC)[W^{-1}]\ar[d]^-{j_W} \\
  \Nerve(\Fun^W(\nat,\cC))\ar[r]^-{\ell_*}\ar[d]_-{\ell_V} & \Fun^\iota(\nat,\Nerve(\cC)[W^{-1}]) \\
  \Nerve(\Fun^W_{\mathrm{shift}}(\nat,\cC))\ar@{-->}[ru]_-{s'} &
 }\]
 in $\Fun(BG,\Cati)$, where $j_W$ is also induced by the projection $\nat \to *$.
 Since $\ell_*$ inverts all morphisms in $V$ from \eqref{eq:V}, the dashed arrow $s'$ exists by the universal property of $\ell_{V}$ shown in \cref{lem:fractions}.
 Since $j_W$ is an equivalence, we obtain the desired factorisation by setting
 \[ s := j_W^{-1} \circ s'\ .\] 
 The final assertion follows directly from the commutativity of the above diagram since $\ell_{*}$ sends $\tau$ to an equivalence.
\end{proof}

\subsection{The transfer class}\label{sec:famenable-transfer}

Recall the abbreviations
\[ U := \bV^{c,\perf,G}_{\Spc^\op_*} \colon G\BC \to \Clep \quad\text{and}\quad V := \bV^{c,\perf}_{\bC,G} \colon G\BC \to \Clep \]
from \eqref{eq:V-and-U}.
We assume that $G$ is a finitely homotopy $\cF$-amenable group (see \cref{def:f-amenable}) and that $\Homol$ is a lax monoidal, finitary localising invariant (see \cref{wtoigwgreerf}). In this subsection we will construct a transfer class $(\cX,t)$ for $(U,\eta,V,H,\cF)$ (see \cref{def:transfer-class}).

Suppose we are given a sequence of homotopy coherent $G$-actions $(\Gamma_n,Z_n)_{n\in \nat}$ and a sequence of maps $(f_n \colon Z_n \to W_n)_{n\in \nat}$, where $W_n$ is a $G$-simplicial complex for every $n$ in $\nat$.
For the time being, we do not assume that these data satisfy any of the conditions listed in \cref{def:f-amenable}.
Instead, we will gradually impose conditions as we develop our construction.
The goal is to provide some transparency where specific assumptions enter, and to state intermediate steps in a way which makes them easier to reuse in \cref{sec:dfhj}.

Let $\cX$ be a $G$-bornological coarse space.
As explained in \cref{faklolergergwergwegwergwreg}, one way of defining a morphism
\[ t \colon \beins_\bM \to \Homol U(\cX) \]
is to specify an object in the left-exact $\infty$-category $\lim_{BG} \bV^{c,\perf}_{\Spc_{*}^{\op}}(\cX)$.
 
In the following we explain how we will specify such an object.
It is again useful not to drop the nerve functor from the notation since we want to consider limits over $BG$ of $G$-categories, and for $\cC$ in $\Fun(BG,\Cat)$ the canonical  functor
 $\Nerve(\lim_{BG}\cC)\to \lim_{BG}\Nerve(\cC)$
 is not an equivalence in general.
 By \cref{rem:controlledequiv-functorial}   the functor $\CW^\hfd$ extends to a functor $(\CW^\hfd,\cW) \colon \BC \to \RelCat$, where $\cW(\cX)$ denotes the class of controlled homotopy equivalences in $\CW^\hfd(\cX)$. 
By \cref{prop:shiftcat}, the functor $\Fun^\cW_{\mathrm{shift}}(\nat,\CW^\hfd)$ from \cref{wjgriegjorgirewgwergwrg} fits into the following commutative diagram:
\begin{equation}\label{eq:rfd-factor1}
 \xymatrix{
 \Nerve(\CW^{\hfd})\ar[rr]^-{\ell_\cW}\ar[rd]_-{\ell_V \circ j} & & \Nerve(\CW^{\hfd})[\cW^{-1}] \\
 & \Nerve(\Fun^\cW_{\mathrm{shift}}(\nat,\CW^\hfd))\ar[ru]_-{s} & 
}\ .\end{equation}
The realisation transformation
\[ r^{\hfd} \colon \Nerve(\CW^{\hfd})^{\op} \to \bV^{c,\perf}_{\Spc_{*}^{\op}} \]
from \cref{def:rfd}  inverts all morphisms in $\cW$ by \cref{gfioqjgioewfewfewfqwefqwefqew}.
By the universal property of the localisation $\ell_{\cW}$, we have another factorisation
\begin{equation}\label{eq:rfd-factor2}
 \xymatrix{
 \Nerve(\CW^\hfd)^\op\ar[rr]^-{r^\hfd}\ar[rd]_-{\ell_\cW} & & \bV^{c,\perf}_{\Spc^\op_*} \\
 & \Nerve(\CW^\hfd)[\cW^{-1}]^\op\ar[ru]_-{\overline{r}^\hfd} & 
}\ .\end{equation}
We define the transformation
\begin{equation}\label{eq:buildfixedpoint}
 r_U \colon \lim_{BG} \Nerve(\Fun^\cW_{\mathrm{shift}}(\nat,\CW^\hfd))^\op \xrightarrow{\lim_{BG} (\overline{r}^\hfd \circ s)} \lim_{BG} \bV^{c,\perf}_{\Spc_{*}^{\op}}\ .
\end{equation}
Using the equivalence
\begin{equation}\label{wevevfsfdvewrvf}
 \iota\colon \Fun^{\cW(\cX)}_{\mathrm{shift}}(\nat,\CW^\hfd(\cX))^{\op, \hG} \xrightarrow{\simeq}\lim_{BG} \Nerve(\Fun^{\cW(\cX)}_{\mathrm{shift}}(\nat,\CW^\hfd(\cX)))^\op\ ,
\end{equation}
from \eqref{qewfoihqiuwhfiuwewfwqewfwfq},
it suffices to provide an object $(M,\rho)$ in $\Fun^{\cW(\cX)}_{\mathrm{shift}}(\nat,\CW^\hfd(\cX))^{\op, \hG}$ {instead of an object in $\lim_{BG} \Nerve(\Fun^{\cW(\cX)}_{\mathrm{shift}}(\nat,\CW^\hfd(\cX)))^\op$}.
As explained in  \cref{rem:fixedpoints},
{such an object $(M,\rho)$ consists of an object $M$
in $\Fun^{\cW(\cX)}_{\mathrm{shift}}(\nat,\CW^\hfd(\cX))$ and a family $\rho = (\rho(g))_{g \in G}$ of morphisms}  $\rho(g) \colon gM \to M$
satisfying the appropriate cocycle condition (note that the direction of $\rho(g)$ is reversed since we are taking the limit of the opposite category).
The object
$r_{U}(\iota(M,\rho))$ in $\lim_{BG} \bV^{c,\perf}_{\Spc_{*}^{\op}}(\cX)$ for an appropriate choice of $(M,\rho)$ will then represent the object determining the transfer class.

To prepare for the construction of $(M,\rho)$,
we recall a construction of strictifications of homotopy coherent $G$-actions (see \cref{qergqergreqfewfqef}.
 This construction has been previously employed for the proof of the A-theoretic Farrell--Jones conjecture \cite{Enkelmann:2018aa} and is similar to a construction appearing in \cite{wegner:cat0}. 

\begin{construction}\label{eqrigoqergewrgwergwergwegw}
 Following \cite[proof of Proposition~5.4]{vogt}   we associate to every homotopy coherent $G$-action $\Gamma$ on a topological space $Z$ a strictification given by a $G$-space $X$ containing $Z$ as a deformation retract.  As a byproduct we obtain  filtrations \eqref{qwefqewfqwfewfqwfqweffee} interpolating between $Z$ and $X$.

The space $X$ is given by
\begin{equation*}
	X:= \left( \coprod_{k \ge 0}   \prod_{i=1}^{k} \left(G\times [0,1]  \right)
	\times G \times Z \right)\big/ \sim,
	\end{equation*}
	where $\sim$ is the equivalence relation generated by
	\begin{equation*}
	(g_{k}, t_k, \dots,g_0,z)
	\sim 
	\begin{cases}
	(g_{k},t_k,\dots,g_{2},t_{2},g_1,z) & g_0 = e  \\
	(g_{k},t_{k},\dots,g_{j+1},t_{j+1}t_{j},g_{j-1},\dots,g_0,z) & g_j = e, 1 \leq j \leq k-1 \\
	(g_{k},t_{k},\dots,t_{j+1},g_{j}g_{j-1},t_{j-1},\dots,g_0,z)
	& t_j = 1, 1 \leq j \leq k \\
	(g_{k},t_{k},\dots,g_{j},
	\Gamma(g_{j-1},t_{j-1},\dots,g_0,z)) & t_j = 0, 1
	\leq j \leq k.
	\end{cases}
	\end{equation*}
	Let $ [g_k,t_k,\dots,g_0,z] $ denote the equivalence class of $ (g_k,t_k,\dots,g_0,z)$.
	The action  of $G$ on $X$ is defined  by 
	\begin{equation*}
	g \cdot [g_k,t_k,\dots,g_0,z] := [g g_k, t_k, \dots, g_0,z].
	\end{equation*}
	We identify $Z$ as a subspace of $X$ such that $z$ in $Z$ corresponds to  $[e,z]$ in $X$.
	We define a
	   retraction 
	\begin{equation}\label{ewrfvihiuqrfeqrfqwefewf}
	R \colon X \to Z \ , \quad 
	R([g_k,t_k,\dots,g_0,z]) :=
	\Gamma(g_k,t_k,\dots,g_0,z)\ .
	\end{equation} 
	The homotopy
	\begin{equation}\label{eq:strong-deformation-retraction-MGamma}
	H \colon X \times [0,1] \to X, \quad ([g_k,t_k,\dots,g_0,z],u) \mapsto [e,u,g_k,t_k,\dots,g_0,z]
	\end{equation}
	exhibits $Z$ as a strong deformation retract of $X$.
	In particular, the canonical embedding from $Z$ to $X$ is a homotopy equivalence
	\begin{equation}\label{eq:Z-to-X}
	 Z \xrightarrow{\simeq} X\ .
	\end{equation}
	This finishes the construction of the strictification.
	
	We now assume that $G$ is finitely generated.
	We regard $G$ as a $G$-coarse space $G_{can}$ by equipping with the $G$-coarse structure generated by sets of the form $F \times F$, where $F \subseteq G$ ranges through all finite subsets of $G$.
	For every generating entourage $S$ of $G_{can}$, there is an induced filtration
	\begin{equation}\label{qwefqewfqwfewfqwfqweffee}
  Z = Y_{0} \subseteq Y_{1} \subseteq \ldots \subseteq Y_{l} \subseteq \ldots \subseteq X  
\end{equation}
	of $X$, where $Y_l$ is the subspace of
	\[ \left( \coprod_{k=0}^{l} \prod_{i=1}^{k} \left(S^l[\{e\}] \times [0,1] \right) \times S^l[\{e\}] \times Z   \right) \big/ \sim \]
	consisting of the elements $[g_k, t_k, g_{k-1},\dots,g_0,z]$ with $g_k = e$.
	In other words, we have
	\begin{equation}\label{eq:filtration-strictification}
	Y_{l} = \{ [e, t_k, g_{k-1},\dots,g_0,z] \in X \mid k \leq l, g_0,\ldots,g_{k-1} \in S^{l}[\{e\}] \}\ .
	\end{equation}
	 For every $l$ in $\nat$ the homotoy  $H$ restricts to a deformation retraction $Y_{l} \times [0,1] \to Y_{l}$ of $Y_{l}$ onto $Z$.   In particular, the inclusion $Z\to Y_{l}$ is a homotopy equivalence.
\end{construction}

\begin{construction}\label{qergioheriougewergergwerg}
 We are going to define the transfer space.
 Recall that we are given a sequence of homotopy coherent $G$-actions $(\Gamma_{n}, Z_{n})$ and a sequence of maps $(f_n \colon Z_n \to W_n)_{n\in \nat}$.
 Let $X_n$ be the strictification of $(\Gamma_n,Z_n)$  provided by   \cref{eqrigoqergewrgwergwergwegw}.
 We define the $G$-sets \begin{equation}\label{dfsvoihqiejvoisvsdfvsdvdsfvf}
X := \coprod_{n\in\nat} X_n \quad\text{and}\quad \cX := \coprod_{n\in\nat} G \times X_n \ ,
\end{equation} 
where  $\cX$  has the diagonal $G$-action. 
The $G$-set $\cX$ will be the underlying $G$-set of the transfer space. Let $p \colon \cX \to \nat$ be the canonical projection. 

A part of \cref{def:transfer-class} of a transfer class is a morphism of the transfer space to an $(\Homol V,\cF)$-proper object.
 The latter will be given by the object $p \colon W_{h}\to  \nat_{min,min}$ in $G\BC_{/\nat_{min,min}}$ obtained from the $G$-simplicial complex
 \[ W := \coprod_{n \in \nat} W_{n} \]
 by application of \cref{beforergiorgergegergergerg3232424}. 
Define for $n$ in $\nat$ the $G$-equivariant map 
\begin{equation}\label{eq:famenable-f}
 f_n' \colon G \times X_n \to W_n, \quad (g,x) \mapsto gf_n(R_n(g^{-1}x))\ ,
\end{equation}
where $R_n$ is the retraction from \eqref{ewrfvihiuqrfeqrfqwefewf} associated to $(\Gamma_n,Z_n)$. 
The family of maps $(f_{n}')_{n\in \nat}$ gives rise to a $G$-equivariant map of $G$-sets
\begin{equation}\label{eq:famenable-f2}
 f := \coprod_{n \in \nat} f_n' \colon \cX \to W\ .
\end{equation}
Pulling back the bornological coarse structure on $W_{h}$ via $f$ turns $\cX$ into a $G$-bornological coarse space over $\nat_{min,min}$, which is our choice of the transfer space, and $f$ into a morphism $\cX \to W_h$ in this category.
\end{construction}

We will later use that the $G$-set $\cX$ also carries a topology since each $X_n$ is a topological space.
As such, the homotopy equivalence from \eqref{eq:Z-to-X} induces isomorphisms
\begin{equation}\label{eq:pi0ZX}
 \coprod_{n \in \nat} G \times \pi_0(Z_n) \cong \pi_0(\cX)\ .
\end{equation}
Denote by \begin{equation}\label{wefqrwfkpqwefqewfq}
 i \colon X \cong \{e\} \times X \to \cX
\end{equation}
 the inclusion map.
 We equip $X$ with the bornological coarse structure pulled back  from $\cX$ via $i$.
 Note  that in general   $X$ is not a $G$-bornological coarse space   since $i$ is not equivariant.

We now start with the construction of the object $M$ in $\Fun^{\cW(\cX)}_{\mathrm{shift}}(\nat,\CW^\hfd(X))$. 
By our standing assumption, $G$ is finitely generated. So we can fix a generating entourage $S$ of $G_{can}$. 
Without loss of generality, we assume that $S$ contains the diagonal of $G$.
Since $S$ contains the diagonal, the iterated compositions of $S$ form an increasing chain of entourages
\[ \diag(G) \subseteq S \subseteq \ldots \subseteq S^r \subseteq S^{r+1} \subseteq \ldots\ .\]
By  \cref{eqrigoqergewrgwergwergwegw} applied to $(\Gamma_n,Z_n)$, the choice of $S$ determines a filtration of $X_{n}$ by subsets 
\[ Y_{n,0} \subseteq Y_{n,1} \subseteq \ldots \subseteq Y_{n,l} \subseteq \ldots \subseteq X_{n} \ ,\]
see \eqref{qwefqewfqwfewfqwfqweffee}.
For every $l$ in $\nat$ we then set
\[ Y_l := \coprod_{n\in \nat} Y_{n,l} \quad\text{and}\quad \cY_{l} := \coprod_{n\in \nat} G \times Y_{n,l} \]
and get filtrations 
\[ Y_{0} \subseteq Y_{1} \subseteq \ldots \subseteq Y_{l} \subseteq \ldots \subseteq X \quad\text{and}\quad \cY_{0} \subseteq \cY_{1} \subseteq \ldots \subseteq \cY_{l} \subseteq \ldots \subseteq \cX \]
of the bornological coarse spaces $X$ and $\cX$.
The filtrations  are not $G$-invariant, but by an inspection of \eqref{eq:filtration-strictification} we have  for every $g$ in $S^r[\{e\}]$ and $l$ in $\nat$
\begin{equation}\label{vreqvioju3rvoirvqvvqvqwev}
gY_l  \subseteq Y_{l+r}\ .
\end{equation}

\begin{lem}\label{lem:Sdiag-entourage}
 If Condition~\ref{def:f-amenable}.\ref{erguiergqwefeqwfqefew} holds, then the set $S^r \times \diag(Y_l)$ is a coarse entourage of $\cX$ for every $l$ and $r$ in $\nat$.
\end{lem}
\begin{proof}
Since the coarse structure on $\cX$ is pulled back  from $W_{h}$ along $f$, we need to show that  for   every  $l$ and $r$ in $\nat $  the image $f(S^r \times \diag(Y_l))$ is a coarse entourage of $W_{h}$. 
Since $S^r$ consists  of finitely many pairs $(g,g')$, it suffices to check that $f(\{(g,g')\} \times \diag(Y_{l}))$ is a coarse entourage of $W_{h}$ for all $g,g'$ in $G$.
In view of \eqref{eq:vanishing-control}  we must show  that \[ \sup_{y \in Y_{n,l}} d(f(g,y),f(g',y)) \xrightarrow{n \to \infty} 0\ .\]
Since $S^{l}[\{e\}]$ is finite, the set of tuples $g_{k-1},\ldots,g_0$ which may appear as {components} of points $[e,t_k,g_{k-1},\ldots,g_0,z]$ of $Y_{l}$ is finite.
 Unwinding  the definition of $f$ from \eqref{eq:famenable-f}, expressing $R_n$ in terms of $\Gamma_n$ as in \eqref{ewrfvihiuqrfeqrfqwefewf} and using the triangle inequality, we have for every $k$ in $\nat$ and all $g_0,\ldots,g_{k-1}$ in $G$ \begin{align}
\sup_{\substack{(t_1,\ldots,t_k) \in [0,1]^{k} \\ z \in Z_n}} &d(f(g,[e,t_k,g_{k-1},\ldots,g_0,z]), f(g',[e,t_k,g_{k-1},\ldots,g_0,z]))\nonumber \\
&\leq \sup_{\substack{(t_1,\ldots,t_k) \in [0,1]^{k} \\ z \in Z_n}} d(gf_n(\Gamma_n(g^{-1},t_k,g_{k-1},\ldots,g_0,z)), g_{k-1}\ldots g_0f_n(z))\label{eq:Sdiag-entourage} \\
&\qquad+ \sup_{\substack{(t_1,\ldots,t_k) \in [0,1]^{k} \\ z \in Z_n}} d(g'f_n(\Gamma_n(g^{\prime,-1},t_k,g_{k-1},\ldots,g_0,z)), g_{k-1}\ldots g_0f_n(z))\nonumber\ .
\end{align}
Since $d$ is $G$-invariant, the term \eqref{eq:Sdiag-entourage} equals
\[ {\sup_{\substack{(t_1,\ldots,t_k) \in [0,1]^{k} \\ z \in Z_n}}}d(f_n(\Gamma_n(g^{-1},t_k,g_{k-1},\ldots,g_0,z)),g^{-1} g_{k-1}\ldots g_0f_n(z))\ ,\]
and similarly for the other term.
By Condition \ref{def:f-amenable}.\ref{erguiergqwefeqwfqefew}, both summands tend to zero as $n \to \infty$.
\end{proof}

We now consider $X=\coprod_{n\in \nat}X_{n}$ and $\cX=\coprod_{n\in \nat}G\times X_{n}$ as topological spaces. 
Let \begin{equation}\label{fvasvcsdscacsdcasdsc}
i_l \colon Y_l \to X \xrightarrow{i} \cX
\end{equation}   denote the inclusion map and recall the definition of $f$ from \eqref{eq:famenable-f2}.
 \begin{lem}\label{lem:Vl}
 If Condition~\ref{def:f-amenable}.\ref{erguiergqwefeqwfqefew} holds, then there exists a sequence \begin{equation}\label{sdfvefvsrwgefg}
V_0 \subseteq V_1 \subseteq \ldots \subseteq V_l \subseteq \ldots
\end{equation}  of coarse entourages of $X$ such that
 \begin{enumerate}
  \item\label{it:Vl1} $V_l$ is a subset of $Y_l \times Y_l$ for every $l$ in $\nat$;
  \item\label{it:Vl2} $V_l$ is an open entourage of $Y_l$ containing $i_l^{-1}f^{-1}\diag(W)$ for every $l$ in $\nat$;
  \item\label{it:Vl3} $gV_l \subseteq V_{l + r}$ for every $l, r$ in $\nat$ and every $g$ in $S^r[\{e\}]$.  
 \end{enumerate}  
\end{lem}
\begin{proof}
 The bornological coarse space $W_{h}$ admits open coarse entourages. For example, we could take the entourage
 \[ \bigcup_{n\in \nat} \{(w_{n},w_{n}')\in W_{n}\times W_{n} \mid d(w_{n},w_{n}')\le \frac{1}{n+1}\}\ .\]
 Since the $G$-coarse structure on $\cX$ is obtained by pulling back  the coarse structure on $W_{h}$ via the $G$-equivariant continuous map $f$ and $X$ is a union of components of $\cX$,  
 pulling back an open coarse entourage of $W_{h}$ by $f\circ i$ {yields} 
 an open coarse entourage $C$ on $X$ which contains $i^{-1}f^{-1}\diag(W)$.

 For every $l$ in $\nat$ we define the entourage
 \begin{equation*}\label{fijrfijfeoijqfefefqwefqwfwfqew}
  V_{l} := \bigcup_{g\in S^{l}[\{e\}]} (gC)_{|Y_{l}}\ .
 \end{equation*}
 of $Y_{l}$.
 Since $C$ is open and $G$ acts continuously, $gC$ is also  an open entourage of $X$ for every $g$ in $G$.
 Hence $V_l$ satisfies Condition \ref{it:Vl2}. 
 
 We now show that $V_{l}$ is a coarse entourage of $X$. We consider the composition of entourages
 \[ V_l' := ((\{e\} \times S^l[e]) \times \diag(Y_l)) \circ Gi(C) \circ \diag(G \times Y_l)\ .\]
 Since $(\{e\} \times S^l[\{e\}]) \times \diag(Y_l) \subseteq S^l \times \diag(Y_l)$, the first term is a coarse entourage of $\cX$ by \cref{lem:Sdiag-entourage}.
 The set $i(C)$ belongs to the $G$-coarse structure of $\cX$ by construction. Therefore $Gi(C)$ belongs to the $G$-coarse structure on $\cX$, too. 
 Hence the entire composition $V_l'$ is a coarse entourage of $\cX$.
 One checks that $V_l = i^{-1}(V_l')$.
 So $V_l$ is a coarse entourage of $X$ since the coarse structure on $X$ is induced via $i$ from $\cX$.

 We finally show  Condition~\ref{it:Vl3}.
 For every $h$ in $G$ we have
 \[ hV_l = \bigcup_{g\in S^{l}[\{e\}]} h(gC)_{|Y_{l}}\ .\]
 By \eqref{vreqvioju3rvoirvqvvqvqwev}, $h(gC)_{|Y_l} \subseteq (hgC)_{|Y_{l+r}}$ for $h \in S^r[\{e\}]$.
 As $hg \in S^{l+r}[\{e\}]$, it follows that $hV_l \subseteq V_{l+r}$ as desired.
\end{proof}

From now on, we assume that Condition~\ref{def:f-amenable}.\ref{erguiergqwefeqwfqefew} holds.
Then we  choose a sequence 
\[ V_0 \subseteq V_1 \subseteq \ldots \subseteq V_l \subseteq \ldots \]
  of entourages as in \cref{lem:Vl}. Let
\begin{equation}\label{sdfvsdfvwrgffsv}
j_l \colon Y_l \to Y_{l+1}
\end{equation}
  denote the  inclusion maps. 

\begin{ddd}\label{def:M} We define $M$ in $\Fun(\nat,\CW(\cX))$ as follows:
\begin{enumerate} \item $M$  
 sends the object $l$ in $\nat$ to the $\cX$-controlled CW-complex
 \begin{equation}\label{eq:M}
  M(l) := i_{l,*}C^{V_l}(Y_l)
 \end{equation}
 in $\CW(\cX)$, see \cref{def:small-simplices} and \eqref{fvasvcsdscacsdcasdsc} for notation.
 \item $M$ sends the morphism $l\to l+1$ in $\nat$ to  the morphism
  \[ M(l) = i_{l,*}C^{V_l}(Y_l)\xrightarrow{j_{l,\sharp}} i_{l+1,*}C^{V_{l+1}}(Y_{l+1}) = M(l+1)\ ,\]
 see \cref{const:small-simplices-induced-maps} and \eqref{sdfvsdfvwrgffsv} for notation. \qedhere
 \end{enumerate} \end{ddd}

\begin{lem}\label{lem:Mfd}\ 
\begin{enumerate}
 \item\label{it:Mfd1} The functor $M$ belongs to the subcategory $\Fun^{\cW(\cX)}(\nat,\CW(\cX))$ of $\Fun(\nat,\CW(\cX))$ (see \cref{wjgriegjorgirewgwergwrg}).
 \item\label{it:Mfd2}  If $Z_n$ is a compact ANR for every $n$ in $\nat$, then   $M$  
 belongs to the  subcategory  $\Fun^{\cW(\cX)}(\nat,\CW^{\hfd}(\cX))$ of $\Fun^{\cW(\cX)}(\nat,\CW(\cX))$ (see \cref{wfiqjfioewfewfqwefqewfewf}.\ref{weriogjowegwergefww}).
\end{enumerate}
\end{lem}
\begin{proof}
 To prove the first assertion, we have to show that $j_{l,\sharp}$ is a controlled homotopy equivalence.
 Since controlled homotopy equivalences satisfy the two-out-of-three property, it suffices to show that the structure map $M(0) \to M(l)$ is a controlled homotopy equivalence for every $l$ in $\nat$.

 As observed at the end of \cref{eqrigoqergewrgwergwergwegw}, the deformation retractions $H_n$ from \eqref{eq:strong-deformation-retraction-MGamma} associated to $(\Gamma_n,Z_n)$ restrict to deformation retractions $H_{n,l} \colon Y_{n,l} \times [0,1] \to Y_{n,l}$ for every $l$ in $\nat$.
 Their coproduct (over $n$ in $\nat$) is a deformation retraction
\[ K_{l} \colon Y_{l} \times [0,1] \to Y_{l} \]  of $Y_{l}$ onto $Y_{0}$.
 We first show that the diagram
\begin{equation}\label{wqefewfefeffefqwefewf}
\xymatrix@C=3em{
 Y_{l} \times [0,1]\ar[rr]^-{K_{l}}\ar[dr]_-{f \circ i_l \circ \pr} & & Y_l\ar[dl]^-{f \circ i_l} \\
  & W &
}
\end{equation} 
commutes, where $f$ was defined in \eqref{eq:famenable-f2}.
 Recall the explicit formulas  \eqref{eq:strong-deformation-retraction-MGamma} for $H_{n,l}$ and \eqref{ewrfvihiuqrfeqrfqwefewf} for the retraction $R_n$.
Using the fourth case of \eqref{qergojrgopregqfwefwefqw} for the second equality and the fifth case of \eqref{qergojrgopregqfwefwefqw} for the third equality, we see that
\begin{eqnarray*}
 f(i_{l}(H_{n,l}( [e,t_k,g_{k-1},\ldots,g_0,z],u))) &=&f_n'(e,H_{n,l}([e,t_k,g_{k-1},\ldots,g_0,z],u)) \\
 &=& R_{n}(H_{n,l}([e,t_k,g_{k-1},\ldots,g_0,z],u)) \\
 &=& R_n([e,u \cdot t_k,g_{k-1},\ldots,g_0,z]) \\
 &=& \Gamma_n( g_{k-1},t_{k-1},\ldots,g_0,z) ,
\end{eqnarray*}  
where $f'_{n}$ is as in \eqref{eq:famenable-f}.
Consequently,   \[ u \mapsto  f(i_{l}(H_{n,l}( [e,t_k,g_{k-1},\ldots,g_0,z],u))) \]
is constant for all $n$ in $\nat$ and $[e,t_k,g_{k-1},\ldots,g_0,z]$ in $Y_{n,l}$. Therefore,  diagram  \eqref{wqefewfefeffefqwefewf} commutes.

This shows that $K_l$ is an $i_l^{-1}f^{-1}(\diag(W))$-controlled homotopy over $Y_l$.
By \cref{lem:Vl}.\ref{it:Vl2}, $K_l$  is in particular a $V_l$-controlled homotopy.

By \cref{lem:Uhomotopy}, it follows that there exists some open entourage $V_l'$ of $Y_l$ such that $K_{l,0}(V_l') \subseteq V_l$, $V_l' = K_{l,1}(V_l') \subseteq V_l$ and $K_l$ induces a controlled homotopy
\begin{equation}\label{fqwefewffqf}
K_{l,0,\sharp} \sim K_{l,1,\sharp} \colon i_{l,*}C^{V_l'}(Y_l) \to i_{l,*}C^{V_l}(Y_l)\ .
\end{equation}  
Let the retraction  $R:X\to Y_{0} $ be induced by the family of retractions $(R_{n})_{n\in \nat}$
and write $R_{|Y_{l}}$ for the restriction of $R$ to $Y_{l}$.
Note that $R_{|Y_{l}}=K_{l,0}$ as maps to $Y_{l}$.
In particular, $R_{|Y_{l}} \colon (Y_{l},i_{l})\to (Y_{0},i_{0})$ is  $V_{l}$-controlled.
Hence the map $(R_{|Y_{l}})_{\sharp}$ below is defined by \cref{const:small-simplices-induced-maps}.
Similarly, we have the map $j_{0,l,\sharp}$ below.
The starting point of the homotopy in \eqref{fqwefewffqf} is given by the composition 
 \[ K_{l,0,\sharp}\colon i_{l,*}C^{V_l'}(Y_l) \xrightarrow{(R_{|Y_{l}})_{\sharp}} i_{0,*}C^{j_{0,l}^{-1}(V_l)}(Y_0) \xrightarrow{j_{0,l,\sharp}} i_{l,*}C^{V_l}(Y_l)\ .\]
Furthermore, its end point is
\[ K_{l,1,\sharp}=i_{l,*} ( C^{V_l'}(Y_l)  \xrightarrow{\incl} C^{V_l}(Y_l))\ .\]
The map $K_{l,0,\sharp}$ is a controlled homotopy equivalence since $K_{l,1,\sharp}$  is a controlled homotopy equivalence by \cref{lem:small-simplices} and \cref{rem:controlledequiv-functorial}, and these two maps are related by the controlled  homotopy \eqref{fqwefewffqf}.

Consider the composition
\[ i_{0,*}C^{j_{0,l}^{-1}(V_l')}(Y_0) \xrightarrow{j_{0,l,\sharp}} i_{l,*}C^{V_l'}(Y_l) \xrightarrow{(R_{|Y_{l}})_{\sharp}} i_{0,*}C^{j_{0,l}^{-1}(V_l)}(Y_0) \xrightarrow{j_{0,l,\sharp}} i_{l,*}C^{V_l}(Y_l) \ .\]
Since $R_{|Y_{l}}\circ j_{0,l}=\id_{Y_{0}}$, the composition of the first two arrows is a $j_{0,l}^{-1}(V_l)$-controlled homotopy equivalence, again by \cref{lem:small-simplices}. Furthermore,  
$j_{0,l,\sharp}\circ (R_{|Y_{l}})_{\sharp}=K_{l,0,\sharp}$ is also a controlled homotopy equivalence as shown above.
So $j_{0,l,\sharp}$ is a controlled homotopy equivalence since controlled homotopy equivalences satisfy the two-out-of-six property.

 The structure map $M(0) \to M(l)$ is given by the composition
\[ i_{0,*}C^{V_0}(Y_0) \to i_{0,*}C^{j_{0,l}^{-1}(V_l)}(Y_0) \xrightarrow{j_{0,l,\sharp}}  i_{l,*}C^{V_l}(Y_l)\ .\]
The first map exists since $V_0 \subseteq j_{0,l}^{-1}(V_l)$. It is a controlled homotopy equivalence by
another application of \cref{lem:small-simplices}. Since $j_{l,0,\sharp}$ is also a controlled homotopy equivalence as seen above, the structure map is a controlled homotopy equivalence. As argued at the beginning of this proof, this
 implies Assertion~\ref{it:Mfd1}.
 
 We now show Assertion \ref{it:Mfd2}. 
Suppose that $Z_n$ is a compact ANR for all $n$ in $\nat$.
 We have to show that $M(l)$ lies in $\CW^\hfd(\cX)$ for every $l$.
 \cref{wfiqjfioewfewfqwefqewfewf} is set up such that any object which is controlled homotopy equivalent to an object in $\CW^\hfd(\cX)$ also belongs to $\CW^\hfd(\cX)$.
 In view of Assertion \ref{it:Mfd1}, which is already shown, we know that $M(l)$ is controlled homotopy equivalent to $M(0)$ for every $l$ in $\nat$.
 Therefore, it suffices to show that $M(0)$ lies in $\CW^\hfd(\cX)$. 
 
 Since $Y_0 \cong \coprod_{n \in \nat} Z_n$ and each $Z_n$ is a compact ANR, $Y_0$ is a locally compact ANR.
 The entourage $V_0$ is open in $Y_0$ by \cref{lem:Vl}.\ref{it:Vl2}.  
 If $B$ is a relatively compact subset of $Y_{0}$, then $N(B):=\{n\in \nat \mid B\cap Z_{n}\not=\emptyset\}$ is finite.
 Then $N(V_{0}[B])=N(B)$ is also finite and $V_{0}[B]$ is relatively compact since it is contained in the compact subset $\bigcup_{n\in N(B)} Z_{n}$ of $V_{0}$.
 We conclude that \cref{rgoirjgorgrwrvcwecr} is applicable to $Y_{0}$ and $V_{0}$.
 It asserts that $C^{V_{0}}(Y_{0})$ belongs to $\CW^{\hfd}(Y_{0,V_{0},\rc})$.
 The map $Y_{0,V_{0},\rc}\to  Y_{0} $ induced by the identity of the underlying sets is a morphism in $\BC$.
 Hence, by 
  \cref{qerogijweroigjwerogegwerg}, we have $C^{V_{0}}(Y_{0})\in \CW^{\hfd}(Y_{0})$.
  Again by \cref{qerogijweroigjwerogegwerg}, 
  the object $M(0)=i_{0,*}C^{V_{0}}(Y_{0})$ belongs to $\CW^\hfd(\cX)$. 
\end{proof}

\begin{construction}\label{const:fixed-point}
 In this construction we define a family of isomorphisms
 \begin{equation}\label{fwerfwerfwerferwf}
\rho = (\rho(g) \colon gM \to M)_{g \in G}
\end{equation}
 which   promotes $M$ to an object $(M,\rho)$  in $ \Fun_{\mathrm{shift}}^{\cW(\cX)}(\nat,\CW^\hfd(\cX))^{\op,\hG}$, see \cref{rem:fixedpoints}.
Recall that we have fixed {a generating} entourage $S$ of $G_{can,min}$, and that $S^r$ denotes the $r$-fold iterated composition of $S$ with itself.
Let $g$ be in $S^r[\{e\}]$.
By \cref{lem:Vl}.\ref{it:Vl3} and using \eqref{vreqvioju3rvoirvqvvqvqwev}, the  multiplication map  $g\cdot -\colon Y_l \to Y_{l+r}$  induces a map
\[ \Sing^{V_l}(Y_l) \to \Sing^{V_{l+r}}(Y_{l+r}), \quad \sigma \mapsto g \cdot \sigma\ .\]
We want to show that the induced map
\begin{equation}\label{eq:G-action}
 g_\sharp \colon g_* i_{l,*}C^{V_{l}}(Y_{l}) \to i_{l+r,*}C^{V_{l+r}}(Y_{l+r})
\end{equation}
is a morphism of $\cX$-labelled CW-complexes in $\CW^{\hfd}(\cX)$, i.e,
that $g\cdot-$ is an $\cX$-controlled morphism, see \cref{wergoijweogwegferfwrfr}.
A cell in $g_{*}i_{l,*}C^{V_{l}}(Y_{l})$ is given by a  singular simplex $\sigma \colon \Delta^k \to Y_l$ with the $\cX$-labelling $(g,g\sigma(b))$,  where $b$ denotes the barycentre of $\Delta^k$.
The image of this cell under $g_{\sharp}$ is the  singular simplex $g\cdot \sigma:\Delta^{k}\to Y_{l+r}$  with the labelling $(e,g\sigma(b))$.
Hence $g_\sharp$ is a controlled morphism if $i_{l+r}((S^r[\{e\}] \times \{e\}) \times \diag(Y_{l+r}))$ is a coarse entourage of $\cX$.
In particular, $g_\sharp$ is controlled if $i_{l+r}(S^r \times \diag(Y_{l+r}))$ is a coarse entourage of $\cX$.  By virtue of \cref{lem:Sdiag-entourage}, which can be applied {since we assume} Condition~\ref{def:f-amenable}.\ref{erguiergqwefeqwfqefew}, this condition is satisfied.  

Consequently, the morphism $g_\sharp$ from \eqref{eq:G-action} defines a morphism
\[ \rho_l(g) \colon g M(l)= g_* i_{l,*}C^{V_{l}}(Y_{l})\xrightarrow{g_{\sharp}} i_{l+r,*}C^{V_{l+r}}(Y_{l+r})=M(l+r) = (T^{r,*}M)(l) \] 
in $\CW^\hfd(\cX)$, where $T^{r,*}$ is the $r$-fold iterate of the shift functor from 
\eqref{erfwefiweogwergref}.
  The family of these maps is a natural transformation  \[ (\rho_{l}(g))_{l\in \nat}\colon g M \to T^{r,*} M \]
  of functors from $\nat$ to $\CW^{\hfd}(\cX)$.
The  morphism
\[ \rho(g):=([r,\rho_{l}(g)])_{l\in \nat} \colon gM \to M \]
in $\Fun^{\cW(\cX)}_{\mathrm{shift}}(\nat,\CW(\cX))$ (see \cref{wjgriegjorgirewgwergwrg}.\ref{it:wjgriegjorgirewgwergwrg2} for notation) is independent of the choice of $r$.
Since $S$ is a generating entourage, every $g$ in $G$ lies in $S^r[\{e\}]$ for some $r$ in $\nat$.
So $\rho(g)$ is defined for every $g$ in $G$. 
\end{construction}

Recall that we assume Condition~\ref{def:f-amenable}.\ref{erguiergqwefeqwfqefew}.

\begin{lem}\label{lem:mrho-defined}
 If $Z_n$ is a compact ANR for every $n$ in $\nat$, then the pair $(M,\rho)$ is an object in $\Fun^{\cW(\cX)}_{\mathrm{shift}}(\nat,\CW^\hfd(\cX))^{\op,\hG}$.
\end{lem}
\begin{proof}
It remains to check the cocycle condition for $\rho$ and that $M$ takes values in $\CW^\hfd(\cX)$.
For the cocycle condition note that for $l$ in $\nat$, $g$ in $S^{r}[\{e\}]$ and $h$ in $S^{s}[\{e\}]$  the triangle
\begin{equation*}\label{}
\xymatrix{h_{*}g_* i_{l,*}C^{V_{l}}(Y_{l})
\ar[rr]^{h_{*}(g_{\sharp})}\ar[dr]_{(hg)_{\sharp}}&&h_{*}i_{l+r,*}C^{V_{l+r}}(Y_{l+r})\ar[dl]^{h_{\sharp}}\\&i_{l+r+s,*}C^{V_{l+r+s}}(Y_{l+r+s})&}
\end{equation*}
obviously commutes. This implies $\rho(h) h_{*}\rho(g) =\rho(hg)$. This is the required relation \eqref{3iojoigergwergerg} since we work in the opposite category of $\Fun^{\cW(\cX)}_{\mathrm{shift}}(\nat,\CW^\hfd(\cX))$.

The functor $M$ takes values in $\CW^\hfd(\cX)$ by \cref{lem:Mfd}
since  $Z_n$ is a compact ANR for every $n$ in $\nat$.
\end{proof}

\begin{construction}\label{const:X0}
 We construct another $G$-bornological coarse space $\cX_0$ which fits into a commutative diagram \begin{equation}\label{ewqflknkqweflqwefewqfqf}
 \xymatrix{\cX\ar[dr]_-{p}\ar[rr]^-{c}&&\cX_{0}\ar[dl]^-{p_{0}}\\&\nat_{min,min}&}\  ,
 \end{equation}
 where $c$ is the canonical map sending a point to its connected component.
 The underlying $G$-set is given by $\cX_{0} := G\times \pi_{0}(X)$.
 We equip $\cX_{0}$ with the smallest bornology such that $p_{0} \colon \cX_{0}\to \nat_{min,min}$ is proper, and the smallest coarse structure such that $c$ is controlled.
\end{construction}
Set
\begin{equation}\label{qewfwfewfwqffewf}
S_{\pi_0} := S \times \diag(\pi_0(X)) \ .
\end{equation}
Since $\pi_0(X) \cong \pi_0(Y_0)$ (as $Y_{0}$ is a deformation retract of $X$), we have 
$S_{\pi_{0}}=c( S \times \diag(Y_{0}))$.
Hence by \cref{lem:Sdiag-entourage} applied in the case $l=0$ and $r=1$ we see that $S_{\pi_{0}}$  is a coarse entourage of $\cX_{0}$.

We now construct a discrete version $(D, \delta)$ in $ \CW^\hfd(\cX_{0})^{\op,\hG}$  of the object $(M,\rho)$ from \cref{lem:mrho-defined}. 
\begin{construction}\label{const:Ddelta}
 Consider the  object 
 \begin{equation}\label{eq:Ddelta}
  D := \pi_0(i)_*C^{\diag(\pi_0(X))}(\pi_0(X))
 \end{equation}
 in $\CW(\cX_{0})$, where $\pi_{0}(i):\pi_{0}(X)\to \cX_{0}$ is induced by  $i$   from \eqref{wefqrwfkpqwefqewfq}.
 If $Z_n$ is a compact ANR for every $n$ in $\nat$, then $\pi_0(X_n)$ is a finite set since $Z_n \to X_n$ is one instance of the homotopy equivalence \eqref{eq:Z-to-X}, and $D$ belongs to $\CW^\hfd(\cX_{0})$ by \cref{wfiqjfioewfewfqwefqewfewf}.
 Multiplication by a group element $g$ induces an isomorphism
 \begin{equation}\label{eq:Ddelta2}
  g_\sharp \colon g_*\pi_0(i)_*C^{\diag(\pi_0(X))}(\pi_0(X)) \to \pi_0(i)_*C^{\diag(\pi_0(X))}(\pi_0(X))
 \end{equation}
 of $\cX_{0}$-labelled CW-complexes.
 By a similar computation as in \cref{const:fixed-point}, this morphism is $S_{\pi_0}$-controlled (see \eqref{qewfwfewfwqffewf})  if $g$ is in $S$. So $g_\sharp$ is an isomorphism in $\CW^\hfd(\cX_{0})$ for every $g$ in $G$. The family $\delta:= (g_\sharp)_{g \in G}$
 satisfies the cocycle condition.
 Therefore, we get the object $(D,\delta) $ in $ \CW^\hfd( \cX_{0})^{\op,\hG}$. \end{construction}

We continue to assume that each $Z_n$ is a compact ANR and that Condition~\ref{def:f-amenable}.\ref{erguiergqwefeqwfqefew} holds so that $(M,\rho)$ and $(D,\delta)$ are well-defined objects by \cref{lem:mrho-defined} and \cref{const:Ddelta}. Consider the composition $k$ in 
\begin{equation}\label{eq:const-seq} 
\xymatrix@C=.5em{ \CW^\hfd(\cX_{0})^{\op,\hG}\ar[rr]^{k}\ar[dr]_{j^{\hG}}&& \Fun^{\cW(\cX_{0})}_\mathrm{shift}(\nat,\CW^\hfd( \cX_{0}))^{\op,\hG}\\&\ar[ur]_{\ell_{V}^{\hG}} \Fun^{\cW(\cX_{0})} (\nat,\CW^\hfd( \cX_{0}))^{\op,\hG} & }\ ,\end{equation}
where $\ell_{V}^{\hG}$ and $j^{\hG}$ are obtained from 
 $\ell_V$  in \eqref{fwefwefwfwefw} and $j$ in \eqref{fwefwefwfwefw2} by applying the $(-)^{\hG}$-construction from \cref{rem:fixedpoints}. 
Recall the natural transformations
\[  r_U \colon \lim_{BG} \Nerve(\Fun^\cW_{\mathrm{shift}}(\nat,\CW^\hfd))^\op \to \lim_{BG} \bV^{c,\perf}_{\Spc_*^\op} \overset{\eqref{eq:V-and-U}}{=} U \ .\]
from \eqref{eq:buildfixedpoint} and $r^{\hfd}:\Nerve(\CW^{\hfd})^{\op}\to \bV^{c,\perf}_{\Spc^{\op}_{*}}$ from \cref{def:rfd}.

We have the solid part of the following diagram:
\begin{equation}\label{qergpoijqopfewfqweewfwefq}
 \xymatrix@C=-3em{
  \lim\limits_{BG} \Nerve(\CW^\hfd( \cX_{0}))^\op\ar[rr]^-{\lim\limits_{BG} r^\hfd}\ar@{-->}[dr]_-{m_{(M,\rho)}} & & U( \cX_{0}) \\
  & \lim\limits_{BG} \Nerve(\Fun^{\cW(\cX)}_{\mathrm{shift}}(\nat,\CW^\hfd(\cX)))^\op\ar[ur]_-{r_U \circ c_*}\ar@{==>}[u]
 }
\end{equation}

\begin{lem}\label{lem:product-mrho}
 Assume:
 \begin{enumerate}
  \item\label{it:product-mrho1} $Z_n$ is a compact ANR with contractible components for every $n$ in $\nat$;
  \item\label{it:product-mrho2} the projection $\pr \colon \cX_{0} \to \pi_0(X)_{min,min}$  is a morphism of bornological coarse spaces and    $\pr \circ c:\cX\to\pi_0(X)_{min,min} $ is bornological.
 \end{enumerate}
 Then there exist a functor $m_{(M,\rho)}$ and a filler completing \eqref{qergpoijqopfewfqweewfwefq} to a commutative diagram in $\CATi$.
\end{lem}
\begin{proof}
 To construct the functor $m_{(M,\rho)}$, we make use of the following auxiliary construction.
 Let $S$ be a $G$-set and  \begin{equation}\label{eq:product-mrho1}
  A \xrightarrow{\alpha} A_{0} \xrightarrow{p} S_{min,min}
 \end{equation}
 be morphisms in $G\BC$ such that $p \circ \alpha$ is bornological. 
 Consider $(Q,\lambda)$ in $\CW(A)$ and $(Q_{0},\lambda_{0})$ in $\CW(A_{0})$.
 Cells of $Q \wedge Q_{0}$ will be  denoted in the form $(q,q_{0})$.  We consider the  subset of these cells  \begin{equation}\label{vfasdvsdvsdvadsvv}
\{(q,q_{0})\in Q \wedge Q_{0}\:|\: (p \circ \alpha \circ \lambda)(q) = (p \circ \lambda_{0})(q_{0})\}\ . 
\end{equation}
  If $(q,q_{0})$ is in this subset and $(q',q_{0}')$ is a cell in $\overline{(q,q_{0})}^{CW}$,
 then $(p \circ \alpha \circ \lambda)(q') = (p \circ \lambda_{0})(q'_{0})$ since $p_{*}\alpha_{*}(Q,\lambda)$ and $p_{*}(Q_{0},\lambda_{0})$ are $S_{min,min}$-controlled CW-complexes. 
 Hence the subset \eqref{vfasdvsdvsdvadsvv}
 determines a subcomplex $Q \otimes_S Q_{0}$ of $Q \wedge Q_{0}$.
 We equip the subcomplex  $Q \otimes_S Q_{0}$  with the   $A$-labelling  given by
 \[ (\lambda \otimes_{S} \lambda_{0})(q,q_{0}) := \lambda(q)\ .\]
 We define a  $G$-equivariant functor
 \[ \otimes_{S} \colon  \CW(A) \times \CW(A_{0}) \to \CW(A)\]
 which sends  the pair of controlled CW-complexes 
  $((Q,\lambda),(Q_{0},\lambda_{0}))$ to
  \[ (Q,\lambda) \otimes_{S} (Q_{0},\lambda_{0}) := (Q \otimes_{S} Q_{0}, \lambda \otimes_{S} \lambda_{0})\ .\]
 This functor preserves controlled homotopies in each variable.
 Note that we  have an isomorphism \begin{equation}\label{wefeqwfwefwefwqef}
Q \otimes_{S} Q_{0} \cong \bigvee_{s \in S} Q((p \circ \alpha)^{-1}(s)) \wedge Q_{0}(p^{-1}(s))\ ,
\end{equation}
 where $Q(A')$ for a subset $A'$ of $A$  denotes the largest subcomplex of $Q$ such that $\lambda(Q(A'))\subseteq A'$,
 and similarly for $Q_{0}$. 
 We claim that  if $(Q,\lambda)$ and $(Q_{0},\lambda_{0})$ are locally finite, then also 
  $(Q,\lambda) \otimes_{S} (Q_{0},\lambda_{0})$ is locally finite.
  To this end, let $A'$ be a bounded subset of $A$. Since $p \circ \alpha$ is bornological, the  set $p(\alpha(A))$ is a finite subset of $S$. 
  Using the isomorphism  \eqref{wefeqwfwefwefwqef}, we  see that $ ((Q,\lambda) \otimes_{S} (Q_{0},\lambda_{0}))
 (A')$ consists of finitely many cells. This shows the claim.
 
 It follows that $\otimes_{S}$ restricts to a $G$-equivariant functor
 \[ \otimes_{S} \colon  \CW^\hfd(A) \times \CW^\hfd(A_{0})  \to \CW^\hfd(A)\ .\]
 Moreover, this functor respects controlled homotopy equivalences in both variables.
 
 This construction is natural in the following sense: for every commutative diagram
 \begin{equation}\label{eq:product-mrho2}\xymatrix@R=.5em{
  A\ar[r]^-{\alpha}\ar[dd]_-{\phi} & A_{0}\ar[rd]^-{p}\ar[dd]^-{\phi_{0}} & \\
  & & S_{min,min} \\
  B\ar[r]^-{\beta} & B_{0}\ar[ru]_-{q} & 
 }\end{equation}
 in $G\BC$ such that $p \circ \alpha$ and $q \circ \beta$ are bornological, we have $\phi_*(- \otimes_S -) = \phi_*(-) \otimes_S \phi_{0,*}(-)$.
 By objectwise application, $\otimes_S$ induces a product functor
 \[ \otimes_S \colon \Fun^{\cW(A)}(\nat, \CW^\hfd(A)) \times \Fun^{\cW(A_{0})}(\nat,\CW^\hfd(A_{0})) \to \Fun^{\cW(A)}(\nat,\CW^\hfd(A))\ .\] 
 The product induces a product (denoted by the same symbol)
 \[  \otimes_S \colon  \Fun^{\cW(A)}_{\mathrm{shift}}(\nat,\CW^\hfd(A)) \times \Fun^{\cW(A_{0})}_{\mathrm{shift}}(\nat,\CW^\hfd(A_{0})) \to \Fun^{\cW(A)}_{\mathrm{shift}}(\nat,\CW^\hfd(A)) \]
 which on morphisms  is given by 
 $[l,\tau]\otimes_{S} [l,\tau_{0}]:=[l,\tau \otimes_S \tau_{0}]$ (see \cref{wjgriegjorgirewgwergwrg}, note that we use representatives with the same first entry).
 
  Applying the $(-)^{\hG}$-construction from  \cref{rem:fixedpoints} we get the functor 
   \begin{equation}\label{eq:product-mrho3}
 \hspace{-0.5cm} m \colon  \Fun^{\cW(A)}_{\mathrm{shift}}(\nat,\CW^\hfd(A))^{\op,\hG} \times \Fun^{\cW(A_{0})}_{\mathrm{shift}}(\nat,\CW^\hfd(A_{0}))^{\op,\hG} \to \Fun^{\cW(A)}_{\mathrm{shift}}(\nat,\CW^\hfd(A))^{\op,\hG}  \end{equation}
 Like $\otimes_S$, this functor is natural in the sense that $\phi_* \circ m = m \circ (\phi_* \times \phi_{0,*})$, retaining the notation from \eqref{eq:product-mrho2}.
 
 We now specialise \eqref{eq:product-mrho1} to 
 $ \cX \xrightarrow{c}  \cX_{0} \xrightarrow{\pr} \pi_0(X)_{min,min}$.
 Note that $\pr$ is a morphism in $G\BC$ and $c\circ \pr$ is bornological by Assumption~\ref{it:product-mrho2}.
 Since $Z_n$ is a compact ANR for every $n$, \cref{lem:mrho-defined} gives  an object $(M,\rho)$ in $ \Fun^{\cW(\cX)}_{\mathrm{shift}}(\nat,\CW^\hfd(\cX))^{\op,\hG} $. 
 We define  the desired functor
 \[ m_{(M,\rho)}:= 
\iota  \circ \Nerve(m((M,\rho),k(-)))\circ \iota^{-1}\colon \lim_{BG} \Nerve(\CW^\hfd(\cX_{0}))^\op\to 
 \lim\limits_{BG} \Nerve(\Fun^{\cW(\cX)}_{\mathrm{shift}}(\nat,\CW^\hfd(\cX)))^\op \ , \] 
   where  $k$ is from \eqref{eq:const-seq}, and $\iota$ is the natural transformation from \eqref{qewfoihqiuwhfiuwewfwqewfwfq}.
 
 It remains to provide the filler of the triangle in \eqref{qergpoijqopfewfqweewfwefq}. In view of the chain of equivalences 
 \begin{eqnarray*}
r_{U}\circ \iota\circ \Nerve(k)\circ \iota^{-1}&\stackrel{\eqref{eq:buildfixedpoint}}{\simeq}& \lim_{BG} \bar r^{\hfd}\circ  \lim_{BG} s\circ \iota\circ \Nerve(k)\circ \iota^{-1}\\&\stackrel{\eqref{eq:const-seq}}{\simeq} &
   \lim_{BG} \bar r^{\hfd}\circ  \lim_{BG} s\circ \iota\circ   \Nerve(\ell_{V}^{\hG})\circ \Nerve(j^{\hG})\circ \iota^{-1}\\&\stackrel{!}{\simeq}&
     \lim_{BG} \bar r^{\hfd}\circ      \lim_{BG} s \circ     \lim_{BG}   \ell_{V} \circ    \lim_{BG}   j \\&\stackrel{\eqref{eq:rfd-factor1}}{\simeq} &
        \lim_{BG} \bar r^{\hfd}\circ   \lim_{BG} \ell_{\cW}   \\&\stackrel{\eqref{eq:rfd-factor2}}{\simeq}&
       \lim_{BG}   r^{\hfd}\ ,  
\end{eqnarray*}
 {where $!$  uses the naturality of $\iota$}, we must provide a natural equivalence 
 \[ r_U \circ c_* \circ m_{(M,\rho)} \simeq r_U \circ \iota\circ \Nerve(k)\circ \iota^{-1}\ .\]  
 By the naturality of $m$, the commutative diagram
 \[\xymatrix@R=.5em{
  \cX\ar[r]^-{c}\ar[dd]_-{c} & \cX_{0}\ar[rd]^-{\pr}\ar[dd]^-{\id} & \\
  & & \pi_0(X)_{min,min} \\
  \cX_{0}\ar[r]^-{\id} & \cX_{0}\ar[ru]_-{\pr} &  
 }\]
 gives rise to the equivalence
 \begin{equation}\label{eq:product-mrho4}
  c_* \circ m_{(M,\rho)} \simeq c_* \circ \iota \circ \Nerve(m((M,\rho),k(-)))\circ \iota^{-1}  \simeq  \iota \circ \Nerve(m(c_*(M,\rho),k(-)))\circ \iota^{-1}\ .
 \end{equation} 
 For every $n$ in $\nat$ the canonical map $Y_n \to \pi_0(Y_n) \to \pi_0(X)$  induces a morphism
 \[ \tau_{n} \colon c_*M(n) \overset{\eqref{eq:M}}{=} c_*i_{n,*}C^{V_n}(Y_n) \to \pi_0(i)_*C^{\diag(\pi_0(X))}(\pi_0(X)) \overset{\eqref{eq:Ddelta}}{=} D\ . \]  
 One checks that  the family $\tau=(\tau_{n})_{n\in \nat}$ is a natural transformation $\tau \colon c_*M \to j(D)$  in the category $\Fun^{\cW(\cX_{0})}(\nat,\CW^\hfd(\cX_{0}))$, where $j(D):\nat\to\CW^\hfd(\cX_{0}) $ is the constant functor with value $D$.  
 
 Using \eqref{eq:G-action} for the definition of $\rho(g)$ and \eqref{eq:Ddelta2} for the definition of $\delta(g)$, one checks explicitly that $[0,\tau] \circ c_*\rho(g) = [0, \delta(g) \circ \tau]$ in 
 $\Fun^{\cW(\cX_{0})}_\mathrm{shift}(\nat,\CW^\hfd( \cX_{0}))$
 for all $g$ in $G$.
 So $\tau$ is  equivariant and  therefore  defines a morphism $\tau^G \colon c_*(M,\rho) \to k(D,\delta)$ in the category $\Fun^{\cW(\cX_{0})}_\mathrm{shift}(\nat,\CW^\hfd( \cX_{0}))^{\op,\hG}$, where $k$ is as in \eqref{eq:const-seq}.
  
 By inspection, we check that there is an equivalence
 \[ m(k(D,\delta),k(-)) \simeq k(-) \]  of functors from 
 $\CW^{\hfd}(\cX_{0})^{\op,\hG}$ to 
  $\Fun^{\cW(\cX)}_\mathrm{shift}(\nat,\CW^\hfd( \cX_{0}))^{\op,\hG}$.
 The morphism $\tau^G$ therefore induces a natural transformation
 \[ \sigma^G \colon   m(c_*(M,\rho),k(-)) \xrightarrow{m(\tau^{G},k(-))} m(k(D,\delta),k(-))\simeq k(-) \]
 The composition
 \[r_{U}\circ  c_* \circ m_{(M,\rho)} \overset{\eqref{eq:product-mrho4}}{=} r_U \circ  
 \iota \circ \Nerve(m(c_{*}(M,\rho),k))\circ \iota^{-1} 
\xrightarrow{r_{U}(\iota\Nerve(\sigma^G)\iota^{-1})} r_U \circ \iota \circ \Nerve(k) \circ \iota^{-1}
\] 
 is the desired natural transformation. 
  This finishes the construction of a natural transformation filling the triangle in \eqref{qergpoijqopfewfqweewfwefq}. 
  
  It remains to show that {this transformation} is an equivalence.
 Since the forgetful functor
 \[ \lim_{BG} \bV^{c,\perf}_{\Spc^\op_*}( \cX_{0}) \to \bV^{c,\perf}_{\Spc^\op_*}( \cX_{0}) \]
 is conservative, it is enough  to show that $r_U$ sends the natural transformation $\sigma$  
between functors from $\CW^{\hfd}(\cX_{0}) $ to  $\Fun^{\cW(\cX)}_\mathrm{shift}(\nat,\CW^\hfd( \cX))^{\op}$
  underlying $\sigma^G$ to an equivalence.
 By the definition of $r_U$ in \eqref{eq:buildfixedpoint}, it is enough to show $s(\sigma)$  is an equivalence, where $s$ is as in \eqref{eq:rfd-factor1}. 
 Since each component of $Z_n$ is contractible and both maps $Z_n \to Y_n \to X_n$ are homotopy equivalences, the morphism $\tau_{n}$  going into the construction of $\sigma$ is a controlled homotopy equivalence for every $n$ {by \cref{lem:small-simplices}}. 
Hence $s(\sigma)$ is an equivalence by  the last assertion of \cref{prop:shiftcat}.
\end{proof}

Recall the definition of $U$ from \eqref{eq:V-and-U}.
Let $\Homol \colon \Clep \to \bM$ be a lax monoidal functor.
In the following, we use the notation introduced in \cref{sec:phantom-transfer}.
The morphism $\eta_H$ is defined in \eqref{eq:etaH}.
Since $U$ is $\pi_0$-excisive by \cref{thm:fix-orbit-wmodule}.\ref{it:fix-orbit-wmodule1}, \cref{regiouweroigegergewgreg} applies to define the functors $q_n^U \colon U(\nat_{min,min}) \to U(*)$.  
 Recall that we are given a sequence of homotopy coherent $G$-actions $(\Gamma_{n}, Z_{n})$ and a sequence of maps $(f_n \colon Z_n \to W_n)_{n\in \nat}$. The transfer space $\cX$ in $G\BC_{/\nat_{min,min}}$ is the $G$-set defined in \eqref{dfsvoihqiejvoisvsdfvsdvdsfvf} with the $G$-bornological coarse structure described in \cref{qergioheriougewergergwerg}.
 
\begin{prop}\label{prop:transfer-famenable}
 If Conditions~\ref{def:f-amenable}.\ref{thiowergergwergwreg} and \ref{def:f-amenable}.\ref{erguiergqwefeqwfqefew} hold, then there exists a commutative diagram
 \begin{equation}\label{ewrgergiwjgioergwregwergwe}
\xymatrix@C=3em{
	\beins_\bM\ar[rr]^-{\eta_\Homol}\ar[d]_-{t} & & \Homol U(*)\ar[d]^-{\diag} \\
	 \Homol U(\cX)\ar[r]^-{HU(p)} & \Homol U(\nat_{min,min})\ar[r]^-{(\Homol(q^U_n))_n} & \prod_{\nat} \Homol U(*)
 }
\end{equation} 
 in $\bM$.
\end{prop}
\begin{proof}
 We first construct the morphism $t$.
 Since $Z_n$ is contractible and $Z_n \to X_n$ is a homotopy equivalence for every $n$ in $\nat$, and $S_{\pi_{0}}$ in \eqref{qewfwfewfwqffewf} is a coarse entourage of $\cX_{0}$, we have a factorisation
 \[ p \colon \cX \xrightarrow{c} \cX_0 \cong G_{can,max} \otimes  \nat_{min,min} \xrightarrow{p_0} \nat_{min,min}\ ,\]
 see \cref{const:X0} for notation.
 Let $m_{(M,\rho)}$ be the functor provided by \cref{lem:product-mrho} and recall the object $(D,\delta)$ in $\CW^{\hfd}(\cX_{0})^{\op,\hG} $ given by \cref{const:Ddelta}.
 We interpret the object $r_U(m_{(M,\rho)}(\iota(D,\delta)))$ in $U(\cX)$ (with $r_{U}$ as in  \eqref{eq:buildfixedpoint}) as a left-exact functor
 \[ t_{\cX} \colon \Spc^{\op,\omega}_* \to U(\cX)\ .\]
 Then we define $t$ as the composition
 \[ t \colon \beins_\bM \to \Homol(\Spc_*^{\op,\omega}) \xrightarrow{\Homol(t_{\cX})} \Homol U(\cX)\ ,\]
 where the first morphism is the unit constraint of the lax monoidal functor $\Homol$.

 It remains to construct a filler of \eqref{ewrgergiwjgioergwregwergwe}.
 The object $(\lim_{BG}  r^\hfd)(\iota(D,\delta))$ in $\lim_{BG}\bV_{\Spc_{*}^{\op}}^{c,\perf}(\cX_{0})$
 (see   \cref{def:rfd} and \eqref{qewfoihqiuwhfiuwewfwqewfwfq}) determines a functor
 \[ t_{\cX_{0}} \colon \Spc_*^{\op,\omega} \to U( \cX_{0})\ .\]
 Consider the diagram
\begin{equation}\label{wergreregwergregergewrg}
\xymatrix{
  \Spc^{\op,\omega}_*\ar[rrr]^-{\eta}\ar[rd]^-{t_{\cX_{0}}}\ar[d]_-{t_{\cX}} & & & U(*)\ar[d]^-{\diag} \\
  U(\cX)\ar[r]^-{U(c)} & U(\cX_0)\ar[r]^-{U(p_0)} & U(\nat_{min,min})\ar[r]^-{(q^U_n)_n} & \prod_\nat U(*)
 }\ ,
\end{equation}  
 where $p_{0}$ and $c$ are as in \eqref{ewqflknkqweflqwefewqfqf}. 
 By \cref{lem:product-mrho} and naturality of $r_U$, we have an equivalence
 \[ U(c)(r_U(m_{(M,\rho)}(D,\delta))) \simeq \lim_{BG} r^\hfd(\iota(D,\delta))\ .\]
 Consequently, the triangle at the bottom left of diagram~\eqref{wergreregwergregergewrg} commutes.
 
 Recall from \cref{thm:fix-orbit-wmodule}.\ref{it:fix-orbit-wmodule4} that under the identification $U(*) \simeq \Fun(BG,\Spc^{\op,\omega}_*)$ the unit morphism $\eta \colon \Spc^{\op,\omega}_* \to U(*)$ is the unique left-exact functor which sends $S^0$ to $\underline{S^0}$, the space $S^0$ equipped with the trivial $G$-action. 
 We must therefore show that 
 $q_{n}^{U}(U(p_{0})(\lim_{BG}  r^\hfd(\iota( D,\delta)))$ is also given by $\underline{S^0}$.
 By naturality of $\lim_{BG}  r^\hfd$  and using \cref{qeirughioqergegwergweg}, we have an equivalence  \[ q_{n}^{U}(U(p_{0})(\lim_{BG}  r^\hfd(\iota( D,\delta)))\simeq
 q_{n}^{U}\lim_{BG}  r^\hfd \iota(p_{0,*}(D,\delta)) \simeq  \lim_{BG}  r^\hfd \iota(p_{0,*}(D(n),\delta_{|D(n)}))\ .\]
Unwinding \cref{const:Ddelta} of $(D,\delta)$ and using $\pi_0(X_n) \cong *$, the restriction of $p_{0,*}(D,\delta)$ in $\CW^{\hfd}(\nat_{min,min})^{\op,\hG}$ to the coarse component $\{n\}$ is given by  $(S^{0},(\id)_{g\in G})$. Finally we have 
 $  \lim_{BG}  r^\hfd \iota (S^{0},(\id)_{g\in G})\simeq 
\underline{S^0}$ by definition of the right-hand side. 
Hence the trapezoid in \eqref{wergreregwergregergewrg} also commutes.

We now apply $\Homol$ to the outer square in \eqref{wergreregwergregergewrg} and precompose with the unit $\beins_{\bM}\to \Homol(\Spc^{\op,\omega}_{*})$ in order to get the desired filler of  \eqref{ewrgergiwjgioergwregwergwe}.
\end{proof}

\begin{rem}
 The proof of \cref{prop:transfer-famenable} using \cref{lem:product-mrho} is more complicated than it needs to be.
 The object $r_U(\iota(M,\rho))$ determines a left-exact functor
 \[ t_{\cX}\colon\Spc^{\op,\omega}_* \to U(\cX)\ . \]
 One can check directly that
 \[ t\colon\beins_{\bM}\to \Homol(\Spc^{\op,\omega}_*)\xrightarrow{\Homol(t_{\cX})} \Homol U(\cX) \]
 satisfies the conclusion of the proposition.  Our presentation of the proof is written with an eye towards \cref{sec:dfhj} where such a direct argument is not possible and we need \cref{lem:product-mrho}.
\end{rem}

\begin{proof}[Proof of \cref{thm:famenable}]
We will deduce  \cref{thm:famenable} from \cref{ergiooegergergwergergergw}.
In the following  we verify its assumptions. 

By the assumptions for \cref{thm:famenable}, we have a family of homotopy coherent $G$-actions $(\Gamma_n,Z_n)_{n\in \nat}$ and a sequence of maps $(f_n \colon Z_n \to W_n)_{n\in \nat}$ satisfying all conditions listed in \cref{def:f-amenable}. We can then  construct a  transfer class $(\cX,t)$ for $(U,\eta,V,H,\cF)$, see  \cref{def:transfer-class}. 
  The transfer space $\cX$ is the $G$-bornological coarse space  from  \cref{qergioheriougewergergwerg}. Furthermore, the morphism $t \colon \beins_\bM \to \Homol U(\cX)$ is given by  \cref{prop:transfer-famenable}. 
  
Since $\Homol$ is a finitary localising invariant, the composite $\Homol V$ is a hyperexcisive equivariant coarse homology theory by \cite[Cor.~6.18]{unik}.
Due to Conditions~\ref{def:f-amenable}.\ref{it:f-amenablecond2} and  \ref{def:f-amenable}.\ref{it:f-amenablecond3}, we can apply \cref{rgiorgergegergergerg3232424} to deduce that $W_{h}$ is $(\Homol V,\cF)$-proper.
The morphism from $ \cX$ to an $(\Homol V,\cF)$-proper object required by \cref{def:transfer-class}.\ref{qwoifgoqergqfeqewfq} is the morphism
$f$  from \eqref{eq:famenable-f2}.    
Finally, diagram~\eqref{wefiewofefwefewfewfewfwe111} in \cref{def:transfer-class}.\ref{qwoifgoqergqfeqewfq1} commutes by \cref{prop:transfer-famenable}. 

The equivariant coarse homology theory $\Homol V$ extends $\Homol \bC_G$ by \cite[Prop.~5.2]{unik}. 
Like every coarse homology theory, it is $\pi_{0}$-excisive.
This completes the verification of the assumptions of \cref{ergiooegergergwergergergw}, so \cref{thm:famenable} follows.
\end{proof}

\section{Dress--Farrell--Hsiang groups}\label{qrgqiorgegegergwegre}

In this section, we introduce the Dress--Farrell--Hsiang condition (\cref{def:F-DFH}) and prove the following analogue of \cref{thm:famenable}.

Let $G$ be a group and $\cF$ be a family of subgroups of $G$.
Let $\bC$ be a left-exact $\infty$-category with $G$-action and let $\Homol \colon \Cle \to \bM$ be a functor to a stably monoidal and cocomplete stable $\infty$-category which admits countable products.
Recall \cref{weoirgjwegwergwerg9} and the assembly map \eqref{fevuihiuqhvvsad}.

\begin{theorem}
	\label{thm:dfh}
	Assume that
	\begin{enumerate}
	 \item $G$ is a Dress--Farrell--Hsiang group with respect to $\cF$;
	 \item $\Homol$ is a lax monoidal, finitary localising invariant.
	\end{enumerate}
	Then the assembly map
	\[ \As_{\cF,\Homol\bC_{G}} \colon \mathop{\colim}\limits_{G_\cF\Orb} \Homol\bC_G \to \Homol\bC_G(*) \]
 is a phantom equivalence.
\end{theorem}
 
We will first introduce the Dress--Farrell--Hsiang condition.
The remainder of the section after \cref{rem:dfhexplain} is dedicated to the proof of \cref{thm:dfh}.

The Dress--Farrell--Hsiang condition relies on the notion of a Dress group.
We use the notation $K \unlhd H $ in order indicate that $K$ is a normal subgroup of $H$.
\begin{ddd}\label{def:dress-group}
	A finite group $D$ is a Dress group if there exist prime numbers $p$ and $q$ and subgroups $P \unlhd C \unlhd D$ such that $P$ is a $p$-group, $C/P$ is cyclic and $D/C$ is a $q$-group.
\end{ddd}

For $F$ a finite group we denote the family of Dress subgroups of $F$ by $\cD(F)$.

Let $G$ be a finitely generated group, and let $\cF$ be a family of subgroups of $G$.

\begin{ddd}\label{def:F-DFH}
	The group $G$ is a \emph{Dress--Farrell--Hsiang group with respect to $\cF$} if there exist
	\begin{enumerate}
		\item a family $(F_{n})_{n\in \nat}$ of  finite groups;
		\item a family $(\alpha_n)_{n\in \nat}$ of epimorphisms $\alpha_n \colon G \to F_n$;
		\item a family $(W_{n,D})_{n\in \nat, D\in \cD(F_{n})}$, where $W_{n,D}$ is a $\overline{D}$-simplicial complex for the subgroup $\overline{D} := \alpha_n^{-1}(D)$ of $G$;
		\item a family $(f_{n,D})_{n\in \nat, D\in \cD(F_{n})}$ of maps of sets $f_{n,D} \colon G \to W_{n,D}$
	\end{enumerate}
	such that the following holds:
	\begin{enumerate}
		\item\label{it:dfhcond1} for every $n$ in $\nat$ and $D$ in $\cD(F_{n})$ the stabilisers of $W_{n,D}$ belong to $\cF\cap \overline{D}$;
		\item\label{it:dfhcond2} for every $n$ in $\nat$ and $D$ in $\cD(F_{n})$ the map $f_{n,D}$ is $\overline{D}$-equivariant, where the subgroup $\overline{D}$ acts on $G$ by left multiplication;
		\item\label{it:dfhcond3} $\sup\limits_{n \in \nat, D \in \cD(F_n)} \dim W_{n,D} < \infty$;
		\item\label{it:dfhcond4} for every $g$ in $G$ we have
		 \[ \sup\limits_{D \in \cD(F_n), \gamma \in G} d(f_{n,D}(\gamma g), f_{n,D}(\gamma)) \xrightarrow{n \to \infty} 0\ .\qedhere\]
	\end{enumerate}
\end{ddd}

\begin{rem}\label{rem:dfhexplain}
 The simplicial complex $W_{n,D}$ in \cref{def:F-DFH} is equipped with the spherical path metric $d$ which is used to formulate Condition~\ref{def:F-DFH}.\ref{it:dfhcond4}.
Due to \cref{rem:hybrid and} Assumption~\ref{def:F-DFH}.\ref{it:dfhcond3}, the spherical path metric may be replaced by the $\ell^1$-metric.
\end{rem}

We now start with the proof of \cref{thm:dfh}.
Let $G$ be a Dress--Farrell--Hsiang group and choose families $(F_n)_{n \in \nat}$, $(\alpha_n)_{n \in \nat}$, $(W_{n,D})_{n \in \nat, D \in \cD(F_n)}$ and $(f_{n,D})_{n \in \nat, D \in \cD(F_n)}$ as in \cref{def:F-DFH}.

As in the case of finitely homotopy $\cF$-amenable groups, we rely on \cref{ergiooegergergwergergergw} to prove \cref{thm:dfh}.
Once more, we consider the functor
\[ V := \bV^{c,\perf}_{\bC,G} \colon G\BC \to \Clep \]
which admits a weak module structure $(\eta,\mu)$ over the $\pi_0$-excisive functor
\[ U := \bV^{c,\perf,G}_{\Spc^\op_*} \colon G\BC \to \Clep \]
by \cref{thm:fix-orbit-wmodule}.
The composite functor $\Homol V \colon G\BC \to \bM$ is a hyperexcisive equivariant coarse homology theory \cite[Cor.~6.18]{unik} extending the functor $\Homol \bC_G \colon G\Orb \to \bM$ \cite[Prop.~5.2]{unik} in the sense of \cref{wtgklpwergrewfwref}.

We now proceed to construct a transfer class $(\cX,t)$ for $(U,\eta,V,\Homol,\cF)$.
Choose families
\[ (F_{n})_{n \in \nat},\quad (\alpha_n)_{n\in \nat},\quad (W_{n,D})_{n \in \nat, D \in \cD(F_n)}\quad\text{and}\quad (f_{n,D})_{n \in \nat, D \in \cD(F_n)} \]
as in \cref{def:F-DFH}.
Recall that for $D$ in $\cF(F_{n})$ we let $\overline{D} := \alpha_n^{-1}(D)$ denote the corresponding preimage.
It is a subgroup of $G$.

\begin{construction}\label{const:dfhX}
We construct the transfer space $\cX$.
This construction uses only the family of epimorphisms $(\alpha_n)_{n \in \nat}$.
Define the $G$-sets
\[ \cX_n := \coprod_{D \in \cD(F_n)} G \times G/\overline{D} \quad\text{and}\quad  \cX := \coprod_n \cX_n\ ,\]
where we let $G$ act diagonally on each summand of $\cX_n$.
We denote by $p \colon \cX \to \nat$ the canonical projection map.
We equip $\cX$ with the maximal bornology such that $p$ becomes a proper map to $\nat_{min,min}$.
Finally, we choose a generating entourage $S$ of $G$ and equip $\cX$ with the coarse structure generated by the entourage
\begin{equation}\label{eq:dfhS}
 \coprod_{n \in \nat} \coprod_{D \in \cD(F_n)} S \times \diag(G/\overline{D})\ .
\end{equation}
Note that this turns the projection to $\nat$ into a morphism $p \colon \cX \to \nat_{min,min}$ of $G$-bornological coarse spaces.
\end{construction}

We now construct the morphism $t \colon \beins_{\bM} \to \Homol U(\cX)$ as in \cref{sec:famenable-transfer}.

\begin{construction}\label{const:transfer-dfh}
 We construct a candidate $(Q,(\rho(g))_{g \in G}$ for an object in $\CW^\hfd(\cX)^{\op,\hG}$.
 By \cite[Cor.~2.10]{winges}, there exists for every $n$ in $\nat$ a finite $F_n$-simplicial complex $K^{\prime}_n$ whose underlying space is contractible and whose stabiliser groups belong to $\cD(F_{n})$.
	We obtain  a sequence of $G$-simplicial complexes $(K_n)_{n\in \nat}$ by letting $G$ act on $K^{{\prime}}_{n}$ via the epimorphism $\alpha_n$.

	Let $z(K_{n})$ denote the set of cells of $K_{n}$, and let $\abs{-} \colon z(K_{n}) \to \nat$  denote the dimension function.
	For every $n$ in $\nat$, choose a section $s_n \colon z(K_n/G) \to z(K_n)$ of the function $z(K_n) \to z(K_n/G)$ induced by the projection map.
	For $k$ in $\nat$ we let $S^{k}_{+}$ denote the $k$-sphere, equipped with the CW-structure consisting of precisely one $0$- and one $k$-cell and with a disjoint basepoint adjoined. 
	
	Then we form the based $G$-CW-complexes 
\begin{equation}\label{werferfiuhniwefwerfwerfre}
  Q_n := \bigvee_{q \in z(K_n/G)} (G/G_{s_n(q)})_+ \wedge S_{+}^{{\abs{q}}} \quad\text{and}\quad Q := \bigvee_{n\in \nat} Q_n\ .
\end{equation}	 
	Since the stabiliser of each cell in $Q$ is the preimage under $\alpha_{n}$ of an element of $\cD(F_{n})$, we can define an $\cX$-labelling on $Q$ by
	\begin{align*}
	  \lambda \colon z(Q) \cong &\coprod_{n\in \nat} \coprod_{q \in z(K_n/G)} G/G_{s_n(q)} \times z(S^{{\abs{q}}}) \to \coprod_{n\in \nat} \coprod_{q \in z(K_n/G)} G/G_{s_n(q)} \\
	  &\to \coprod_{n\in \nat} \coprod_{D \in \cD(F_n)} \{e\} \times  G/\overline{D} \subseteq \cX\ ,
	\end{align*}
	where the latter map
	sends $gG_{s_{n}(q)}$ (in the component of $(n,q)$) to $(e, g G_{s_{n}(q)})$ (in the component of $(n,\alpha_{n}(G_{s_n(q)}))$).
	
	The canonical $G$-action on $Q$ induces for every $g$ in $G$ a map
	\[ \rho(g) \colon g_*(Q,\lambda) \to (Q,\lambda)\ .\qedhere\]   
\end{construction}

Recall the description of objects in $\CW^\hfd(\cX)^{\op,\hG}$ from \cref{rem:fixedpoints}.
\begin{prop}\label{prop:Qcontrolled}
 The $\cX$-labelled CW-complex $(Q,\lambda)$ together with the family of maps $\rho = (\rho(g))_{g \in G}$ defines an object in $\CW^\hfd(\cX)^{\op,\hG}$.
\end{prop}
\begin{proof}
 The $\cX$-labelled CW-complex $(Q,\lambda)$ is $\diag(\cX)$-controlled and locally finite.
 Hence it belongs to $\CW^\hfd(\cX)$ in view of \cref{wfiqjfioewfewfqwefqewfewf}.
 The cocycle condition for $\rho$ is straightforward to verify.
 
 It remains to show that $\rho(g)$ is a morphism in $\CW^\hfd(\cX)$. Unwinding definitions, we find for every point $(n,q,\gamma G_{s_n(q)},z)$ in $z(Q)$ that
 \[ g \cdot \lambda(n,q,\gamma G_{s_n(q)},z) = (n,\alpha_n(G_{s_n(q)}),g,g\gamma \alpha_n(G_{s_n(q)})) \]
 and
 \[ \lambda(g \cdot(n,q,\gamma G_{s_n(q)},z)) = (n,\alpha_n(G_{s_n(q)}),e,g\gamma \alpha_n(G_{s_n(q)}))\ .\]
 Note that the labels on the right hand sides differ only in their third entry.
 Since the coarse structure on $\cX$ is generated by the entourage in \eqref{eq:dfhS}, this shows that $\rho(g)$ is $\cX$-controlled.
\end{proof}

The realisation map $r^{\hfd} \colon \CW^{\hfd,\op} \to \bV^{c,\perf}_{\Spc^\op_*}$ induces a natural transformation
\begin{equation}\label{eq:rfdG}
 r^{\hfd,G} \colon \Nerve(\CW^{\hfd,\op,\hG}) \xrightarrow{\iota} \lim_{BG} \Nerve(\CW^{\hfd,\op}) \xrightarrow{\lim_{BG} \Nerve(r^\hfd)} \lim_{BG} \bV^{c,\perf}_{\Spc^\op_*} = U\ ,
\end{equation}
where $\iota$ is the equivalence from \eqref{qewfoihqiuwhfiuwewfwqewfwfq}.
Hence we have an object $r^{\hfd,G}((Q,\lambda),\rho)$ in $U(\cX)$.
We regard this object as a left-exact functor
\[ t_\cX \colon \Spc^{\op,\omega}_{*} \to  U(\cX)\ .\]
\begin{ddd}\label{def:dfh-transfer}
 Define the morphism $t$ as the composition
 \[ t \colon \beins_{\bM} \to \Homol(\Spc^{\op,\omega}_{*}) \xrightarrow{\Homol(t_\cX)} \Homol U(\cX)\ .\qedhere\]
\end{ddd}

\begin{lem}\label{lem:dfh-transfer-project}
There exists a commutative diagram
\begin{equation}\label{eq:dfh-transfer-project}
 \xymatrix@C=3.5em{
 \beins_{\bM}\ar[rr]^-{\eta_{\Homol}}\ar[d]_-{t} & & \Homol U(*)\ar[d]^-{\diag} \\
 \Homol U(\cX)\ar[r]^-{\Homol U(p)} & \Homol U(\nat_{min,min})\ar[r]^-{(\Homol(q^U_n))_n} & \prod_{n\in \nat} \Homol U(*)
 }\end{equation}
in $\bM$.
\end{lem}
\begin{proof}
	We only need to construct a filler.
	Since the lower right corner of the diagram is a product category, we can construct fillers one factor at a time.

	Every finite based $G$-CW-complex $A$ (with $G$-action $\alpha$) provides an object $r^{\hfd,G}(A,\alpha)$ of $U(*)$
	which we interpret as a left-exact functor
	$r^{\hfd,G}(A,\alpha) \colon \Spc^{\op,\omega}_{*} \to U(*)$.
	This functor gives rise to a morphism
	\begin{equation}\label{eq:ua}
	 u_A \colon \beins_{\bM} \to \Homol(\Spc^{\op,\omega}_{*}) \xrightarrow{\Homol(r^{\hfd,G}(A,\alpha))} \Homol U(*)\ ,
	\end{equation}
	where {the first morphism} is the unit constraint of the monoidal structure of $\Homol$.
	
	Let $(\eta,\mu)$ be the weak module structure provided by \cref{thm:fix-orbit-wmodule}.
	Define $\eta_\Homol \colon \beins_{\bM}\to \Homol U(*)$ as in \eqref{eq:etaH}.
	By \cref{thm:fix-orbit-wmodule}.\ref{it:fix-orbit-wmodule4}, we have $\eta_\Homol \simeq u_{\underline{S^0}}$.

	Fix $n$ in $\nat$.
	Since the underlying space of $K_n$ introduced in \cref{const:transfer-dfh} is contractible, the projection map $(K_n)_+ \to S^0$ induces an equivalence $r^{\hfd,G}((K_n)_+) \to r^{\hfd,G}(\underline{S^0})$ in $U(*)$ by \cref{gfioqjgioewfewfewfqwefqwefqew}.
	Applying $\Homol$ and precomposing with {the unit constraint}, we obtain an equivalence
	\begin{equation}\label{eq:dfh-transfer-project1}
	  u_{(K_{n})_{+}} \xrightarrow{\simeq} u_{\underline{S^{0}}} \simeq \eta_{\Homol} \colon \beins_{\bM} \to \Homol U(*)
	\end{equation}
	{by virtue of \cref{thm:fix-orbit-wmodule}.\ref{it:fix-orbit-wmodule4}.}
	Since $\Homol$ is a localising invariant, every fibre sequence of left-exact functors $F_0 \to F_1 \to F_2$ in $\Fun^\lex(\bC,\bD)$ induces an equivalence $\Homol F_0 + \Homol F_2 \simeq \Homol F_1$; note that the target of $\Homol$ is stable, so addition of morphisms is well-defined.
	We apply this to left-exact functors $\Spc^{\op,\omega}_* \to U(*)$.
	As such functors are determined by their evaluation at $S^0$, fibre sequences of such functors correspond to fibre sequences in $U(*)$. 
	
	Since $\Fun(BG,\Spc^{\op,\omega}_*) \simeq U(*)$ by \cref{thm:fix-orbit-wmodule}.\ref{it:fix-orbit-wmodule2}, every cofibre sequence of finitely dominated, based $G$-CW-complexes induces a fibre sequence in $\Fun(BG,\Spc^{\op,\omega}_*)$ by application of $r^{\hfd,G}$.
	Hence we obtain for each two consecutive steps $K_n^{(i)} \to K_n^{(i+1)}$ in the skeletal filtration of $K_n$ an equivalence
	\[ \Homol(r^{\hfd,G}((K_n^{(i)})_+)) + \Homol(r^{\hfd,G}((K_n^{(i+1)})_+/(K_n^{(i)})_+)) \simeq \Homol(r^{\hfd,G}((K_n^{(i+1)})_+))\ .\]
	The $G$-CW-complex $Q_n$ from \eqref{werferfiuhniwefwerfwerfre} has a filtration by the subcomplexes $Q_n^{[i]}$ which consist of the wedge summands containing a sphere of dimension at most $i$.
	Therefore, we have an equivalence
	\[ \Homol(r^{\hfd,G}(Q_n^{[i]})) + \Homol(r^{\hfd,G}(Q_n^{[i+1]}/Q_n^{[i]})) \simeq \Homol(r^{\hfd,G}(Q_n^{[i+1]})) \]
	for every $i$ in $\nat$.
	Since by construction of $Q_{n}$ we have isomorphisms
	\[ K_n^{(i+1)}/K_n^{(i)} \cong Q^{[i+1]}_n/Q^{[i]}_n\ ,\]
	it follows by a finite induction that
	\[ \Homol(r^{\hfd,G}((K_n)_+)) \simeq \Homol(r^{\hfd,G}(Q_n)) \colon \Homol(\Spc^{\op,\omega}_*) \to \Homol U(*)\ .\]
	Consequently, we have
	\[ u_{(K_n)_+} \simeq u_{Q_n} \colon \beins_\bM \to \Homol U(*)\ .\]
	Combining this equivalence with \eqref{eq:dfh-transfer-project1}, we obtain an equivalence $\eta_H \simeq u_{Q_n}$.
	
	Note that the composite $\Spc^{\op,\omega}_* \xrightarrow{t_\cX} U(\cX) \xrightarrow{U(p)} U(\nat_{min,min})$ corresponds by the naturality of $r^{\hfd,G}$ to the object $r^{\hfd,G}(p_*(Q,\lambda),\rho)$.
	Using \cref{qeirughioqergegwergweg}, it follows that $q^U_n \circ U(p) \circ t_\cX \colon \Spc^{\op,\omega}_* \to U(*)$ corresponds to the object $r^{\hfd,G}(Q_n)$ of $U(*)$.
	Hence we have
	\[ \Homol(q^U_n) \circ \Homol U(p) \circ t \simeq u_{Q_n} \simeq \eta_\Homol\ .\]	
	This equivalence is the desired filler of the square in \eqref{eq:dfh-transfer-project} for the factor indexed by $n$.
\end{proof}

\cref{lem:dfh-transfer-project} shows that $(\cX,t)$ satisfies all conditions listed in \cref{def:transfer-class} except \ref{def:transfer-class}.\ref{qwoifgoqergqfeqewfq}:
we still have to verify that $\cX$ admits a morphism to an $(\Homol V,\cF)$-proper object.

Form the $G$-simplicial complexes
\[ W_n := \coprod_{D \in \cD(F_n)} G \times_{\overline{D}} W_{n,D}\ ,\quad W:=\coprod_{n\in \nat} W_{n} \]
and let $\pi \colon W \to \nat$ be the canonical projection.
We let $\pi \colon W_{h}\to \nat_{min,min}$ denote the object of $G\BC_{/\nat_{min,min}}$    obtained by applying \cref{beforergiorgergegergergerg3232424} to $W$.
Since $\Homol V$ is a hyperexcisive equivariant coarse homology theory and $W$ is a finite-dimensional $G$-simplicial complex whose stabilisers belong to $\cF$ by Conditions \ref{def:F-DFH}.\ref{it:dfhcond1} and \ref{def:F-DFH}.\ref{it:dfhcond3}, \cref{rgiorgergegergergerg3232424} shows that $W_{h}$ is $(\Homol V,\cF)$-proper.

Recall \cref{const:dfhX} of the transfer space $\cX$.
Moreover, for every $n$ in $\nat$ and $D$ in $\cD(F_n)$ we have a $G$-equivariant map
\begin{equation*}\label{qregihqiuwefqwefqewfeqwfef}
 f'_{n,D} \colon G \times G/\overline{D} \to G \times_{\overline{D}} W_{n,D},\quad (g,\gamma\overline{D}) \mapsto (\gamma, f_{n,D}(\gamma^{-1}g))
\end{equation*}
since $f_{n,D}$ is assumed to be $\overline{D}$-equivariant (Condition~\ref{def:F-DFH}.\ref{it:dfhcond2}).
These maps give rise to the $G$-equivariant map $f_n := \coprod_{D \in \cD(F_n)} f_{n,D}' \colon \cX_n \to W_n$  and the map (of underlying sets, for the moment) $f := \coprod_{n \in \nat} f_n \colon \cX \to W_{h}$. 

\begin{lem}
 The map $f$ is a morphism of $G$-bornological coarse spaces over $\nat_{min,min}$.
\end{lem}
\begin{proof}
 The bornological coarse structure on $\cX$ is described in \cref{const:dfhX}. It is obvious  from the constructions
 that  $f$ is a proper map which is compatible with the projections to $\nat$.
 To see that $f$ is also controlled, it suffices to check that the image of
 \[ \coprod_{n \in \nat} \coprod_{D \in \cD(F_n)} \{(g,e)\} \times \diag(G/\overline{D}) \]
 under $f$ is an entourage of $W_{h}$ for every $g$ in $G$.
 Since $f(n,D,g,\gamma\overline{D}) = (\gamma,f_{n,D}(\gamma^{-1}g))$, this follows from the fact that
 \[ \sup_{D \in \cD(F_n), \gamma \in G} d(f_{n,D}(\gamma^{-1}g), f_{n,D}(\gamma^{-1})) \xrightarrow{n \to \infty} 0\]
 for every $g$ in $G$ by Condition~\ref{def:F-DFH}.\ref{it:dfhcond4}.
\end{proof}

This finishes the construction of the transfer class $(\cX,t)$ for $(U,\eta,V,\Homol,\cF)$. \cref{thm:dfh} follows now from \cref{ergiooegergergwergergergw}.

\section{Dress--Farrell--Hsiang--Jones groups}\label{sec:dfhj}

This section combines the arguments of \cref{sec:famenablegroups,qrgqiorgegegergwegre} to prove a result which unifies and generalises \cref{thm:famenable,thm:dfh} and is instrumental in proving the Farrell--Jones conjecture for virtually solvable groups (see \cref{ex-solvable}).

Recall the notion of Dress groups from \cref{def:dress-group}.
For a finite group $F$, we continue to denote the family of Dress subgroups of $F$ by $\cD(F)$.

Let $G$ be a finitely generated group and $\cF$ be a family of subgroups.

\begin{ddd}\label{def:dfhj}
 The group $G$ is a Dress--Farrell--Hsiang--Jones group with respect to $\cF$ (or DFHJ group ({with respect to} $\cF$) for short) if there exist
 \begin{enumerate}
  \item\label{it:dfhj1} a family $(F_n)_{n \in \nat}$ of finite groups;
  \item\label{it:dfhj2} a family $(\alpha_n)_{n \in \nat}$ of epimorphisms $\alpha_n \colon G \to F_n$;
  \item\label{it:dfhj3} a family $(\Gamma_{n,D},Z_{n,D})_{n \in \nat, D \in \cD(F_n)}$ of homotopy coherent $G$-actions;
  \item\label{it:dfhj4} a family $(W_{n,D})_{n \in \nat, D \in \cD(F_n)}$, where $W_{n,D}$ is a $\overline{D}$-simplicial complex for the subgroup $\overline{D} := \alpha_n^{-1}(D)$ of $G$;
  \item\label{it:dfhj5} a family $(f_{n,D})_{n \in \nat, D \in \cD(F_n)}$ of continuous maps $f_{n,D} \colon G \times Z_{n,D} \to W_{n,D}$
 \end{enumerate}
 such that the following holds:
 \begin{enumerate}
  \item\label{it:dfhjcond0} for every $n$ in $\nat$ and $D$ in $\cD(F_n)$, the topological space $Z_{n,D}$ is a compact AR;\footnote{See Condition~\ref{def:f-amenable}.\ref{thiowergergwergwreg} for an explanation of this notion.}
  \item\label{it:dfhjcond1} for every $n$ in $\nat$ and $D$ in $\cD(F_n)$, the stabilisers of $W_{n,D}$ belong to $\cF \cap \overline{D}$;
  \item\label{it:dfhjcond2} $\sup\limits_{n \in \nat, D \in \cD(F_n)} \dim W_{n,D} < \infty$;
  \item\label{it:dfhjcond3} for every $n$ in $\nat$ and $D$ in $\cD(F_n)$, the map $f_{n,D}$ is $\overline{D}$-equivariant, where $\overline{D}$ acts on $G \times Z_{n,D}$ by left multiplication on the left factor;
  \item\label{it:dfhjcond4} for all $k$ in $\nat$ and $g_0,\ldots,g_k$ in $G$ we have
  \[ \sup_{\substack{D \in \cD(F_n), \gamma \in G, \\ (t_1,\ldots,t_k) \in [0,1]^k, \\ z \in Z_{n,D}}} d(f_{n,D}(\gamma, \Gamma_{n,D}(g_k,t_k,\ldots,t_1,g_0,z)), f_{n,D}(\gamma g_k\ldots g_0, z) ) \xrightarrow{n \to \infty} 0\ .\qedhere\]
 \end{enumerate}
\end{ddd}

\begin{rem}\label{rem:dfhj-vs-amenable}
 \cref{def:dfhj} combines finite $\cF$-amenability (\cref{def:f-amenable}) and the Dress--Farrell--Hsiang condition (\cref{def:F-DFH}) into a single notion:
 \begin{enumerate}
  \item If we choose $F_n$ to be the trivial group for all $n$, we have for each $n$ in $\nat$ a single homotopy coherent $G$-action $(\Gamma_n,Z_n)$, a single $G$-simplicial complex $W_n$ and a $G$-equivariant map $f_n \colon G \times Z_{n} \to W_n$.
  Restricting $f_n$ to $\{e\} \times Z_{n}$ yields a map as in \cref{def:f-amenable} of finite homotopy $\cF$-amenability.
  Conversely, if $f_n \colon Z_n \to W_n$ is as in \cref{def:f-amenable}, then $(g,z) \mapsto gf_n(z)$ defines a map as in Condition~\ref{def:dfhj}.\ref{it:dfhjcond4}.
 \item If we choose all the spaces $Z_{n,D}$ to be points, the DFHJ condition reduces to being a Dress--Farrell--Hsiang group.
 \item As explained in \cref{rem:dfhexplain}, the metric $d$ in Condition~\ref{def:dfhj}.\ref{it:dfhjcond4} is the spherical path metric (or the $\ell^1$-metric).\qedhere
\end{enumerate}
\end{rem}

\begin{ex}\label{ex-solvable}
 Let $w$ be a non-zero algebraic number which is not a root of unity.
 Let $\IZ[w,w^{-1}]$ be the underlying abelian group of the subring of $\IC$ generated by $\IZ$, $w$ and $w^{-1}$.
 We let $\IZ$ act on $\IZ[w,w^{-1}]$ by group automorphisms via multiplication with $w$.
 Then $\IZ[w,w^{-1}] \rtimes \IZ$ is DFHJ with respect to the family of virtually abelian subgroups \cite[Prop.~3.3]{KUWW}.
 Using the inheritance properties of the Farrell--Jones conjecture, this implies by \cite[Prop. 3.3]{wegner-solv} that the conjecture holds for all virtually solvable groups.
 See the proof of \cref{thm:fullfjc} in \cref{hieruhqwefiwfewf} for further details. 
\end{ex}

Let $G$ be a finitely generated group and $\cF$ be a family of subgroups.
Let $\bC$ be a left-exact $\infty$-category with $G$-action and let $\Homol \colon \Cle \to \bM$ be a functor to a stably monoidal and cocomplete stable $\infty$-category which admits countable products.
Recall \cref{weoirgjwegwergwerg9} and the assembly map \eqref{fevuihiuqhvvsad}.

\begin{theorem}\label{thm:dfhj}
 Assume that
 \begin{enumerate}
  \item $G$ is a Dress--Farrell--Hsiang--Jones group with respect to $\cF$;
  \item $\Homol$ is a lax monoidal, finitary localising invariant.
 \end{enumerate}
 Then the assembly map
 \[ \As_{\cF,\Homol\bC_{G}} \colon \mathop{\colim}\limits_{G_\cF\Orb} \Homol\bC_G \to \Homol\bC_G(*) \]
  is a phantom equivalence.
\end{theorem}

The remainder of this section is dedicated to the proof of \cref{thm:dfhj}.
As in \cref{sec:famenablegroups,qrgqiorgegegergwegre}, \cref{ergiooegergergwergergergw} reduces the proof to the construction of a suitable transfer class, see \cref{def:transfer-class}.
As in \cref{sec:famenablegroups,qrgqiorgegegergwegre}, we consider the functor
\[ V := \bV^{c,\perf}_{\bC,G} \colon G\BC \to \Clep \]
which admits a weak module structure $(\eta,\mu)$ over the $\pi_0$-excisive functor
\[ U := \bV^{c,\perf,G}_{\Spc^\op_*} \colon G\BC \to \Clep \]
by \cref{thm:fix-orbit-wmodule}.
The composite $\Homol V$ is a hyperexcisive equivariant coarse homology theory by \cite[Cor.~6.18]{unik} extending $\Homol \bC_G$ \cite[Prop.~5.2]{unik}.

The transfer class associated to $(U,\eta,V,\Homol,\cF)$ will arise by combining the transfer class of \cref{qrgqiorgegegergwegre} with a fibrewise version of the transfer class from \cref{sec:famenable}.
Let $G$ be DFHJ group with respect to the family $\cF$.
Choose families
\[ (F_n)_{n \in \nat},\quad (\alpha_n)_{n \in \nat},\quad (\Gamma_{n,D},Z_{n,D})_{n \in \nat, D \in \cD(F_n)},\quad (W_{n,D})_{n \in \nat, D \in \cD(F_n)} \text{ and}\quad (f_{n,D})_{n \in \nat, D \in \cD(F_n)} \]
as in \cref{def:dfhj}.

Since $\cD(F_n)$ is a finite set and $\overline{D}$ has finite index in $G$ for every $D$ in $\cD(F_{n})$, the topological space
\begin{equation}\label{eq:dfhj-Z}
 Z_n := \coprod_{D \in \cD(F_n)} G/\overline{D} \times Z_{n,D}
\end{equation}
is a compact ANR.
Since $G/\overline{D}$ carries a $G$-action and $Z_{n,D}$ comes equipped with a homotopy coherent $G$-action, $G/\overline{D} \times Z_{n,D}$ inherits a homotopy coherent $G$-action $\Gamma'_{n,D}$ given by
\begin{equation}\label{eq:GammanD}
 \Gamma'_{n,D}(g_k,t_k,\ldots,t_1,g_0,(\gamma\overline{D},z)) := (g_k\ldots g_0 \gamma\overline{D}, \Gamma_{n,D}(g_k,t_k,\ldots,t_1,g_0,z))\ .
\end{equation}
Taking coproducts, these homotopy coherent $G$-actions induce a homotopy coherent $G$-action $\Gamma_n$ on $Z_n$.

For every $n$ in $\nat$,  we form the $G$-simplicial complex
\[ W_n := \coprod_{D \in \cD(F_n)} G \times_{\overline{D}} W_{n,D}\ .\]
For $D$ in $\cD(F_n)$, we define the map
\[ f_{n,D}' \colon G/\overline{D} \times Z_{n,D} \to G \times_{\overline{D}} W_{n,D},\quad (\gamma\overline{D},z) \mapsto [\gamma, f_{n,D}(\gamma^{-1},z)] \]
which is well-defined since $f_{n,D}$ is $\overline{D}$-equivariant by Condition~\ref{def:dfhj}.\ref{it:dfhjcond3}.
We define
\[ f_n := \coprod_{D \in \cD(F_n)} f_{n,D}' \colon Z_n \to W_n \]
and let $d$ denote the spherical path metric on the simplicial complex $W_n$.

\begin{lem}\label{lem:dfhj-almostequiv}
 For all $k$ in $\nat$ and $g_0,\ldots,g_k$ in $G$ we have
 \[ \sup_{\substack{(t_1,\ldots,t_k) \in [0,1]^k \\ (\gamma\overline{D},z) \in Z_n}} d(f_n(\Gamma_n(g_k,t_k,\ldots,t_1,g_0,(\gamma\overline{D},z))), g_k\ldots g_0 f_n(\gamma\overline{D},z)) \xrightarrow{n \to \infty} 0\ .\]
\end{lem}
\begin{proof}
 Using \eqref{eq:GammanD}, we find that
 \begin{align*}
  f_{n,D}'(\Gamma_{n,D}'(g_k,t_k,\ldots,t_1,g_0,&(\gamma\overline{D},z))) \\
  &= [g_k\ldots g_0 \gamma, f_{n,D}(\gamma^{-1}g_0^{-1}\ldots g_k^{-1}, \Gamma_{n,D}(g_k,t_k,\ldots,t_1,g_0,z))]
 \end{align*}
 and
 \[ g_k\ldots g_0 f_{n,D}'(\gamma\overline{D},z) = [g_k\ldots g_0 \gamma, f_{n,D}(\gamma^{-1}, z)]\ .\]
 So the assertion of the lemma is simply a rephrasing of Condition~\ref{def:dfhj}.\ref{it:dfhjcond4}.
\end{proof}

Summing up, we have a sequence $(\Gamma_n,Z_n)_{n \in \nat}$ of topological spaces with homotopy coherent $G$-actions, a sequence of $G$-simplicial complexes $(W_n)_{n \in \nat}$ and a sequence $(f_n)_{n \in \nat}$ of $G$-equivariant maps $f_n \colon Z_n \to W_n$ such that
\begin{enumerate}
 \item $Z_n$ is a compact ANR with contractible components for every $n$ in $ \nat$ (by Condition~\ref{def:dfhj}.\ref{it:dfhjcond0});
 \item the stabilisers of $W_n$ belong to $\cF$ (by Condition~\ref{def:dfhj}.\ref{it:dfhjcond1});
 \item $\sup\limits_{n \in \nat} \dim W_n < \infty$ (by Condition~~\ref{def:dfhj}.\ref{it:dfhjcond2});
 \item for all $k$ in $\nat$ and $g_0,\ldots,g_k$ in $G$ we have
 \[ \sup_{\substack{(t_1,\ldots,t_k) \in [0,1]^k \\ (\gamma\overline{D},z) \in Z_n}} d(f_n(\Gamma_n(g_k,t_k,\ldots,t_1,g_0,\gamma\overline{D},z)), g_k\ldots g_0 f_n(\gamma\overline{D},z)) \xrightarrow{n \to \infty} 0 \]
 by \cref{lem:dfhj-almostequiv}.
\end{enumerate}
Note that this list of conditions is identical to the conditions of \cref{def:f-amenable} except that $Z_n$ is not necessarily contractible.
This allows us to use all statements from \cref{sec:famenable-transfer} except \cref{prop:transfer-famenable}.

In particular, an application of \cref{qergioheriougewergergwerg} provides $G$-bornological coarse spaces $\cX$ and $W_{h}$ over $\nat_{min,min}$ and a morphism $f \colon \cX \to W$ of $G$-bornological coarse spaces over $\nat_{min,min}$.
Let $p_\cX \colon \cX \to \nat_{min,min}$ denote the projection.
 
\begin{prop}\label{prop:transfer-dfhj}
 There exists a commutative diagram
 \[\xymatrix@C=3.5em{
  \beins_{\bM}\ar[rr]^-{\eta_{\Homol}}\ar[d]_-{t} & & \Homol U(*)\ar[d]^-{\diag} \\
  \Homol U(\cX)\ar[r]^-{\Homol(p_\cX)} & \Homol U(\nat_{min,min})\ar[r]^-{(\Homol(q^U_n))_n} & \prod_{n\in \nat} \Homol U(*)
 }\]
 in $\bM$.
\end{prop}
\begin{proof}
 By \cref{const:dfhX}, the epimorphisms $(\alpha_n \colon G \to F_n)_{n \in \nat}$ give rise to a $G$-bornological coarse space $\cT$ whose underlying set is $\coprod_{n \in \nat} \coprod_{D \in \cD(F_n)} G \times G/\overline{D}$ together with a morphism $p_\cT \colon \cT \to \nat_{min,min}$ (in \cref{const:dfhX}, this space is called $\cX$).
 The isomorphism from \eqref{eq:pi0ZX}, the definition of the spaces $Z_{n}$ in \eqref{eq:dfhj-Z} and that fact that each $Z_{n,D}$ is contractible provides a canonical bijection $\cT \cong \pi_0(\cX)$.
 
 Let the $G$-bornological coarse space $\cX_0$ be given by \cref{const:X0}.
 By \cref{lem:Sdiag-entourage}, the generating entourage \eqref{eq:dfhS} of $\cT$ is also an entourage of $\cX_0$.
 Hence we have an induced map $i \colon \cT \to \cX_0$ of $G$-bornological coarse spaces over $\nat_{min,min}$.

 By \cref{const:transfer-dfh}, we obtain an object $((Q,\lambda),\rho^Q)$ in $\CW^{\hfd}(\cT)^{\op,\hG}$.
 Recall the transformation $r^{\hfd,G}$ from \eqref{eq:rfdG}.
 Then the object $r^{\hfd,G}((Q,\lambda),\rho^Q)$ in $U(\cT)$ induces the second morphism in the composition
 \[ t_\cT \colon \beins_\bM \to \Homol U(\Spc^{\op,\omega}_*) \to \Homol U(\cT)\ .\]
 Similarly, the object $r^{\hfd,G}(i_*((Q,\lambda),\rho^Q))$ in $U(\cX_0)$ induces the morphism
  \[ t_{\cX_0} \colon \beins_\bM \to \Homol U(\Spc^{\op,\omega}_*) \to \Homol U(\cX_0)\ .\]
 \cref{lem:product-mrho} yields a commutative diagram
\begin{equation}\label{eq:m-rfd}
 \xymatrix@C=-3em{
  \lim\limits_{BG} \Nerve(\CW^\hfd( \cX_{0}))^\op\ar[rr]^-{\lim_{BG} r^\hfd}\ar[dr]_-{m_{(M,\rho)}} & & U( \cX_{0}) \\
  & \lim\limits_{BG} \Nerve(\Fun^{\cW(\cX)}_{\mathrm{shift}}(\nat,\CW^\hfd(\cX)))^\op\ar[ur]_-{r_U \circ c_*}
  }\end{equation}
  in which $r_U$ is the natural transformation from \eqref{eq:buildfixedpoint}.
 
 Then $\iota i_*((Q,\lambda),\rho^Q)$ is an object in $\lim\limits_{BG} \Nerve(\CW^\hfd( \cX_{0}))^\op$, where $\iota$ is the natural equivalence from \eqref{qewfoihqiuwhfiuwewfwqewfwfq}.
 Hence the object $r_U(m_{(M,\rho)}(\iota i_*((Q,\lambda),\rho^Q)))$ in $U(\cX)$ induces the morphism
 \[ t_\cX \colon \beins_\bM \to \Homol(\Spc^{\op,\omega}_*) \to \Homol U(\cX)\ .\]
 Consider the following diagram:
\begin{equation}\label{sdfvsvnsdvosdfvsdfv}
\xymatrix@C=3.5em{
  \beins_\bM\ar[dd]_-{t_\cX}\ar[rdd]_-{t_{\cX_0}}\ar[rd]^-{t_\cT}\ar[rrr]^-{\eta_\Homol} & & & \Homol U(*)\ar[dd]^-{\diag} \\
  & \Homol U(\cT)\ar[d]^-{\Homol U(i)}\ar[rd]^-{\Homol U(p_\cT)} & & \\
  \Homol U(\cX)\ar[r]^-{\Homol U(c)} & \Homol U(\cX_0)\ar[r]^-{\Homol U(p_0)} & \Homol U(\nat_{min,min})\ar[r]^-{(\Homol(q^U_n))_n} & \prod_{n\in \nat} \Homol U(*)
 }
\end{equation}  
 We will first show that the triangle in the bottom left corner commutes.
 We have equivalences 
 \begin{align*}
  c_* r_U m_{(M,\rho)}(\iota i_*((Q,\lambda),\rho^Q))
  &\simeq r_U c_*m_{(M,\rho)}(\iota i_*((Q,\lambda),\rho^Q)) \\
  &\simeq \lim_{BG} r^\hfd(\iota i_*((Q,\lambda),\rho^Q)) \\
  &\simeq r^{\hfd,G}(i_*((Q,\lambda),\rho^Q))\ ,
 \end{align*}
  where the first equivalence uses the naturality of $r_U$, the second equivalence follows from the commutativity of \eqref{eq:m-rfd}, and the third equivalence uses the definition of $r^{\hfd,G}$ in \eqref{eq:rfdG}.
 In view of the definitions of $t_\cX$ and $t_{\cX_0}$, this equivalence of objects in $U(\cX_0)$ induces an equivalence
 \[ \Homol U(c) \circ t_\cX \simeq t_{\cX_0}\ .\]
 Both $\Homol U(i) \circ t_\cT$ and $t_{\cX_0}$ are defined by the same object of $U(\cX_0)$, so $\Homol U(i) \circ t_\cT \simeq t_{\cX_0}$.
 Since $i \colon \cT \to \cX_0$ is a morphism over $\nat_{min,min}$, we have $\Homol U(p_\cT) \simeq \Homol U(p_0) \circ \Homol U(i)$.
 
 Finally, the remaining part of the diagram commutes by \cref{lem:dfh-transfer-project} (recall that $\cT$ corresponds to the $G$-bornological coarse space $\cX$ in \cref{lem:dfh-transfer-project}). 
\end{proof}

Since the stabilisers of $W$ belong to $\cF$ and $\dim W < \infty$, \cref{rgiorgergegergergerg3232424} shows that $W_{h}$ is $(\Homol V,\cF)$-proper.
Therefore, taking $t_\cX \colon \beins_\bM \to \Homol U(\cX)$ as in \cref{prop:transfer-dfhj} yields a transfer class $(X,t_\cX)$ for $(U,\eta,V,\Homol,\cF)$ (see \cref{def:transfer-class}).
\cref{thm:dfhj} now follows from \cref{ergiooegergergwergergergw}.

\section{Inheritance properties of the isomorphism conjecture}\label{sec:inheritance}

 The class of groups which are DFHJ relative to the family of virtually cyclic subgroups  includes hyperbolic groups, $\mathrm{CAT}(0)$-groups and the groups described in \cref{ex-solvable}. 
 Due to various inheritance properties,
one can show that the assembly map in \cref{wriohgjrotggerergw} is an equivalence for a much larger class of groups.
The proofs of these inheritance properties use little more than the construction of the assembly map described in \cref{sec:intro}.

We consider a cocomplete $\infty$-category $\cK$ and a functor $F \colon \cK \to \bM$.
Examples to keep in mind are $\cK = \Cle$ or $\cK = \Catex$ and $F$ being a finitary localising invariant.
Note, however, that we are making no assumptions about $F$ at the moment.
We consider a functor $\bC \colon BG \to \cK$.
Then, as in \eqref{vfoijiorvefvfdsv}, we let
\[ \bC_G := j^{G}_{!}(\bC)\colon G\Orb \to \cK \]
denote the left Kan extension of $\bC$ along $j^{G}$.
As in \cref{weoirgjwegwergwerg9}, we set 
\[ F\bC_G := F \circ \bC_G \colon G\Orb \to \bM \]
and consider the assembly map
\[ \As_{\cF,F\bC_G} \colon \mathop{\colim}\limits_{G_\cF\Orb} F\bC_G \to F\bC_G(*) \]
introduced in \cref{wriohgjrotggerergw}.

Since we assume that $\bM$ is cocomplete, the functor $F\bC_G$ has an essentially unique extension to a colimit-preserving functor
\[ F_\bC \colon \PSh(G\Orb) \to \bM \]
such that $F_\bC \circ \yo_G \simeq F\bC_G$, where $\yo_G \colon G\Orb \to \PSh(G\Orb)$ is the Yoneda embedding.
Note that, in contrast to the convention in \cref{sec:intro}, we use the symbol  $F_{\bC}$ for this extension. 

Let $\phi \colon H \to G$ be a group homomorphism.
Then there exists an adjunction
\begin{equation}\label{vsdfvsvwvdsdfvsdfv}
 \ind_\phi \colon \PSh(H\Orb) \rightleftarrows \PSh(G\Orb) \cocolon \res_\phi
\end{equation}
 in which $\res_\phi$ is given by precomposition with the induction functor
\[ G \times_\phi - \colon H\Orb \to G\Orb\ , \]
while $\ind_\phi$ is the unique colimit-preserving extension of the composition
\[ H\Orb \xrightarrow{G \times_\phi -} G\Orb \xrightarrow{\yo_G} \PSh(G\Orb)\ .\]
Moreover, $\phi$ induces a functor $B\phi \colon BH \to BG$.

We will {also} make use of the auxiliary functors appearing in the following diagram:
 \begin{equation}\label{advhivuoasvsdvcdsvadv}
\xymatrix{
  BH\ar[r]^-{j^{H}}\ar[d]_-{B\phi} & H\Orb\ar@/^.75cm/[rr]^{\yo_{H}}\ar[r]^-{s^{H}}\ar[d]_-{G \times_\phi -} & H\Set\ar[r]^-{\ell^H}\ar@<-3pt>[d]_-{\ind_\phi} & \PSh(H\Orb)\ar@<-3pt>[d]_-{\ind_\phi} \\
  BG\ar[r]^-{j^{G}} & G\Orb\ar[r]^-{s^{G}} \ar@/^-.75cm/[rr]_{\yo_{G}}& G\Set\ar[r]^-{\ell^G}\ar@<-3pt>[u]_-{\res_\phi} & \PSh(G\Orb)\ar@<-3pt>[u]_-{\res_\phi}
 }
\end{equation}
 The functor $s^{H}$ regards a transitive $H$-set as an $H$-set and $\ell^H$ regards $H$-sets as discrete $H$-spaces to obtain objects in $\PSh(H\Orb)$.
 In particular, $\ell^H$ is fully faithful and the composition $\ell^H \circ s^{H}$ is equivalent to the Yoneda embedding $\yo_H$.
 If we drop the restriction functors, the above diagram commutes.
 Moreover, we have an equivalence $\ell^H \circ \res_\phi \simeq \res_\phi \circ \ell^G$.

{For $\bD$ in $\Fun(BH,\cK)$, we let} $B\phi_{!}\bD \colon BG \to \cK$ denote the left Kan extension of $\bD$ along $B\phi \colon BH \to BG$.

\begin{lem}[]\label{lem:res-ind-asm}
 If $F$ preserves arbitrary coproducts, then there exists a natural equivalence
 \[ F_\bD \circ \res_\phi \simeq F_{B\phi_{!}\bD} \]
 of functors $\PSh(G\Orb) \to \bM$.
\end{lem}
\begin{proof}
 Since both $F_\bD \circ \res_\phi$ and $F_{B\phi_{!}\bD}$ are colimit-preserving functors, it suffices to construct an equivalence between the restrictions of these functors along the Yoneda embedding $\yo_{{G}} \colon {G}\Orb \to \PSh({G}\Orb)$.
 Providing such an equivalence is an exercise in the formal properties of left Kan extensions which we include for completeness.
 
 Recall that $F_{B\phi_{!}\bD} \circ \yo_G \simeq F(B\phi_{!}\bD)_G$ by definition.
 Since $s^{G}$ is fully faithful, the equivalence $\id \simeq (s^{G})^{*}s^{G}_{!}$ yields the second   equivalence in the chain
 \[ (B\phi_{!}\bD)_G \simeq j^{G}_!B\phi_{!}\bD \simeq s^{G}_!j^{G}_!B\phi_{!}\bD \circ s^{G} \simeq (\ind_\phi)_! s^{H}_! j^{H}_! \bD \circ s^{G}\ .\]
 The last equivalence is given by the  left and middle commutative squares in  \eqref{advhivuoasvsdvcdsvadv}.
 Since $\ind_\phi$ is a left adjoint of $\res_\phi$, the left Kan extension functor $(\ind_{\phi})_{!}$
 is given by restriction along $\res_\phi$.
 Hence
 \[ (\ind_\phi)_! s^{H}_! j^{H}_! \bD \circ s^{G} \simeq s^{H}_! j^{H}_! \bD \circ \res_\phi \circ s^{G}\ . \]
 For every $H$-set $X$, the discrete subcategory of the slice category $s^{H}_{{/X}}$ formed by the inclusions of $H$-orbits is cofinal.
 Thus, if $E \colon H\Orb \to \cN$ is a functor such that $\cN$ admits arbitrary coproducts, the left Kan extension $s^{H}_!E$ of $E$ along $s^{H}$ exists and the canonical morphism
 \[ \coprod_{S \in H \backslash X} E(S) \to s^{H}_!E(X) \]
 is an equivalence.
 Since $F$ preserves coproducts, it follows from this identification that
 \[ F \circ s^{H}_!j^{H}_! \bD \circ \res_\phi \circ s^{G} \simeq s^{H}_!(F\bD_H) \circ \res_\phi \circ s^{G}\ . \]
 We conclude from this discussion that 
 \begin{equation}\label{eq:res-ind-1}
  F(B\phi_{!}\bD)_G \simeq s^{H}_!(F\bD_H) \circ \res_\phi \circ s^{G}\ .
 \end{equation} 
 Since $\ell^{H}$ is fully faithful and thus $\id\simeq (\ell^{H})^{*} \ell^{H}_{!}$, we get the first equivalence in the chain
 \[ s^{H}_!(F\bD_H) \simeq \ell^{H}_!s^{H}_!(F\bD_H) \circ \ell^H \simeq F_\bD \circ \ell^{H} \]
 The second equivalence follows from $\ell^{H} \circ s^{H} \simeq \yo_H$ and
 $F_{\bD}\simeq (\yo_{H})_{!} F\bD_{H}$.
 Therefore, we have
 \begin{align}\label{eq:res-ind-2}
  s^{H}_!(F\bD_H) \circ \res_\phi \circ s^{G}
  &\simeq F_\bD \circ \ell^H \circ \res_\phi \circ s^{G} \nonumber\\
  &\simeq F_\bD \circ \res_\phi \circ \ell^{G} \circ s^{G}
  \simeq F_\bD \circ \res_\phi \circ \yo_G\ .
 \end{align}
 Combining \eqref{eq:res-ind-1} and \eqref{eq:res-ind-2} gives the desired identification.
\end{proof}

As before, let $\phi \colon H \to G$ be a group homomorphism.
Let $\cF$ be a family of subgroups of $G$. 
By
\[ \phi^*\cF := \{ K \leq H \mid  \phi(K)\in \cF \} \]
we denote the induced family on $H$.

\begin{kor}\label{cor:res-ind-asm}
 If $F$ preserves arbitrary coproducts, then the assembly map
 \[ \As_{\phi^*\cF,\bD_H} \colon \mathop{\colim}\limits_{H_{\phi^*\cF}\Orb} F\bD_H \to F\bD_H(*) \]
 is equivalent to the assembly map
 \[ \As_{\cF,F(B\phi_{!}\bD)_G} \colon \mathop{\colim}\limits_{G_\cF\Orb} F(B\phi_{!}\bD)_G \to F(B\phi_{!}\bD)_G(*). \]
\end{kor}
\begin{proof}
 We let $E_\cF G$ in $\PSh(G\Orb)$ be the classifying space of $G$ for the family $\cF$ as described in \cref{wrtoihgjwtgergwerfew}.
 By \cref{qgoiegjqefewfqfewfqewf}, the assembly map $\As_{\cF,F(B\phi_!\bD)_G}$ is equivalent to the map
 \begin{equation}\label{sfdvoijvfsiodvsdfvsdfvfdv}
  F_{B\phi_!\bD}(E_\cF G) \to F_{B\phi_!\bD}(*)
 \end{equation}
 induced by the projection $E_\cF G \to *$.
 Using \cref{lem:res-ind-asm} in order to replace
  $F_{B\phi_!\bD}$ by $F_{\bD}\circ \res_{\phi}$ and the equivalences 
 $\res_\phi(E_\cF G) \simeq E_{\phi^*\cF} H$ and $\res_\phi(*) \simeq *$,
 we see that the map in \eqref{sfdvoijvfsiodvsdfvsdfvfdv} is equivalent to
 \[ F_{\bD}(E_{\phi^{*}\cF}H)\to F_{\bD}(*) \]
 induced by the projection $E_{\phi^{*}\cF}H\to *$.
 By \cref{qgoiegjqefewfqfewfqewf} again, this map is equivalent to $\As_{\phi^{*}\cF,\bD_{H}}$.
\end{proof}

Let $H$ be a subgroup of $G$ and let $\bD$ be an object in $\Fun(BH,\cK)$.
We denote the inclusion of $H$ into $G$ by $\iota \colon H \to G$ (we use $\iota$ instead of $\phi$ to emphasise injectivity) and use the notation $\cF|_H:=\iota^*\cF$  for the restriction of the family $\cF$ to $H$.

\begin{kor}[Passage to subgroups]\label{cor:subgroups}
 Assume that $F$ preserves arbitrary coproducts.
 If the assembly map $\As_{\cF,F(B\iota_!\bD)_G}$ is a (phantom) equivalence, 
 then the assembly map $\As_{\cF|_H,F\bD_H}$ is also a (phantom) equivalence.
\end{kor}
\begin{proof}
 In fact, \cref{cor:res-ind-asm} implies that $\As_{\cF,F(B\iota_!\bD)_G}$ is equivalent to $\As_{\cF|_H,F\bD_H}$.
\end{proof}

Let now $\bC$ be in $\Fun(BG,\cK)$ and use the notation $\res^G_H\bC := (B\iota)^*\bC$ for the restriction of $\bC$ to an object in $\Fun(BH,\cK)$.
 Recall the induction functor $\ind_{\iota}\colon\PSh(H\Orb)\to \PSh(G\Orb)$ from  \eqref{vsdfvsvwvdsdfvsdfv}.

\begin{lem}\label{lem:ind-asm}
 There is a natural equivalence of functors
 \[ F_{\res^G_H\bC} \simeq F_\bC \circ \ind_\iota \]
 of functors $\PSh(H\Orb) \to \bM$.
\end{lem}
\begin{proof}
 Since  both functors are colimit-preserving, it suffices to find an equivalence between the restrictions of these functors along the Yoneda embedding $\yo_H$, i.e., an equivalence 
\begin{equation}\label{qdwewdwedqed}
 F(\res^G_H\bC)_{H}\simeq F\bC_{G}\circ  (G\times_{H}-)
\end{equation}
of functors $H\Orb\to \bM$.
We have the chain of equivalences 
\begin{equation}\label{fqfwefewfqe}
 \res^G_H\bC \simeq \bC \circ B\iota \simeq j^{G}_!\bC \circ  j^{G}   \circ B\iota  \simeq j^{G}_!\bC \circ (G \times_{H} -) \circ j^{H} 
\end{equation} 
of functors $BH\to \cK$,
where the first equivalence is the definition of $\res^{G}_{H}$, the second follows from $\id\simeq (j^{G})^{*}j^{G}_{!}$ since $j^{G}$ is fully faithful, and the last uses the left square in \eqref{advhivuoasvsdvcdsvadv}.  We now apply $j^{H}_{!}$ to \eqref{fqfwefewfqe} in order to get the first equivalence in \eqref{asdcadsccasdcadsc} below.  The unit $j^{H}_{!}(j^{H})^{*}\to \id$ induces the second arrow in the  natural transformation  \begin{equation}\label{asdcadsccasdcadsc}
(\res^G_H\bC)_{H}\simeq  j^{H}_{!}( j^{G}_!\bC \circ (G \times_{H} -) \circ j^{H})\to   j^{G}_!\bC \circ (G \times_{H} -)
\end{equation}
of functors $H\Orb\to \cK$.
We claim that  the natural transformation \eqref{asdcadsccasdcadsc} is an equivalence.
By the pointwise formula for the left Kan extension $j^{H}_{!}$, its evaluation at $S$ in $H\Orb$ is the map
\[ \mathop{\colim}\limits_{j^{H}_{{/S}}} j^{G}_{!} \bC \circ   (G \times_{H} (j^{H}\circ \ev^{H}_{S}))  \to  j_{!}^{G}\bC(G\times_{H}S) \]
in $\cK$, where $\ev^{H}_{S} \colon j^{H}_{{/S}} \to BH$ is the canonical functor.
 Since the induction functor $G \times_{H} -$ induces an equivalence of categories $j^{H}_{{/S}} \to j^{G}_{{/(G}} \times_H S)$,
 the left-hand side is equivalent to 
 \[ \mathop{\colim}\limits_{j^{G}_{{/G \times_{H} S}}}  j^{G}_{!}\bC \circ (j^{G}\circ \ev^{G}_{G\times_{H}S}) \simeq
 \mathop{\colim}\limits_{j^{G}_{{/G \times_{H} S}}} \bC \circ  \ev^{G}_{G\times_{H}S}  \simeq
  (j^{G}_{!}\bC)(G\times_{H}S)\ , \]
 where the second equivalence is again a consequence of the pointwise formula for the left Kan extension, this time for $j^{G}_{!}$. This finishes the proof of the claim. 
 
 We get the desired equivalence \eqref{qdwewdwedqed} by applying $F$ to the equivalence \eqref{asdcadsccasdcadsc}.
\end{proof}

 Each of the following statements actually contains two statements, one for equivalences, and another one for phantom equivalences under the additional assumption that the target $\bM$ of $F$ is stable.
 We will provide arguments for the case of equivalences.
 The case of phantom equivalences is similar, using the observation that stability of $\bM$ implies that colimits of phantom objects are again phantom objects.
  
Let $ \cF',\cF$ be families of subgroups of $G$ such that $\cF' \subseteq \cF$.
The relative assembly map in the following statement is induced by the inclusion of index categories $G_{\cF'}\Orb \to G_{\cF}\Orb$.

\begin{prop}\label{prop:transitivity}
 If the assembly map $\As_{\cF'|_H,F(\res^{G}_{H}\bC)_H}$ is an equivalence for every $H$ in $\cF$, then the relative assembly map
 \[\As_{\cF',F\bC_{G}}^{\cF} \colon \mathop{\colim}\limits_{G_{\cF'}\Orb} F\bC_G \to \mathop{\colim}\limits_{G_\cF\Orb} F\bC_G \]
 is an   equivalence.
 If $\bM$ is stable, the same assertion holds with ``equivalence'' replaced by ``phantom equivalence''. 
\end{prop}
\begin{proof}
 It suffices to show that $F\bC_G|_{G_\cF\Orb}$ is a left Kan extension of $F\bC_G|_{G_{\cF'}\Orb}$ along the inclusion functor $G_{\cF'}\Orb \to G_{\cF}\Orb$.
 By the pointwise formula for the left Kan extension, it is enough to check that the canonical map
\begin{equation}\label{fdsvdfvfdvvvs}
 \mathop{\colim}\limits_{G_{\cF'}\Orb_{/(G/H)}} F\bC_G \to F\bC_G(G/H)  
\end{equation} 
 is an equivalence for every $H$ in $\cF$.
 The functor $G \times_H - \colon H\Orb \to G\Orb$ induces an equivalence
 \[ H_{\cF'|_H}\Orb \simeq H_{\cF'|_H}\Orb_{/*} \xrightarrow{\simeq} G_{\cF'}\Orb_{/(G/H)} \]
 which allows us to identify the map in \eqref{fdsvdfvfdvvvs} with the map
 \[ \mathop{\colim}\limits_{H_{\cF'|_H}\Orb} F\bC_G \circ (G \times_H -) \to F\bC_G \circ (G \times_H -)(*). \]
 By \cref{lem:ind-asm}, this map is equivalent to the assembly map
 \[ \As_{\cF'|_H,F(\res^{G}_{H}\bC)_H} \colon \mathop{\colim}\limits_{H_{\cF'|_H}\Orb} F(\res^G_H\bC)_H \to F(\res^G_H\bC)_H(*)\ ,\]
 which is an equivalence by assumption.
\end{proof}

\begin{kor}[Transitivity Principle]\label{cor:transitivity}
 {Assume:}
 \begin{enumerate}
  \item the assembly map $\As_{\cF,F\bC_G}$ is an  equivalence;
  \item for every $H$ in $\cF$, the assembly map $\As_{\cF'|_H,F(\res^G_H\bC)_H}$ is an equivalence.
 \end{enumerate} 
 Then the assembly map $\As_{\cF',F\bC_G}$ is an equivalence.
 If $\bM$ is stable, the same assertion holds with ``equivalence'' replaced by ``phantom equivalence''.
\end{kor}
\begin{proof}
The composition
\[ \mathop{\colim}\limits_{G_{\cF'}\Orb} F\bC_G \to \mathop{\colim}\limits_{G_\cF\Orb} F\bC_G \to F\bC_G(*) \]
is equivalent to $\As_{\cF',F\bC_G}$. The first map in this composition is the relative assembly map $\As^{\cF}_{\cF',F\bC_{G}}$, which is an equivalence by \cref{prop:transitivity}, and the second map is the assembly map $\As_{\cF,F\bC_{G}}$, which is an equivalence by assumption.
\end{proof}
 
 Let $\pi \colon G \to Q$ be an epimorphism of groups and $\cH$ be a family of subgroups of $Q$ such that  $\pi^{*}\cH\subseteq \cF$.
 
\begin{kor}\label{cor:epis}
Assume that $F$ preserves arbitrary coproducts.
If
 \begin{enumerate}
  \item the assembly map $\As_{\cH,F(\pi_!\bC)_Q}$ is an equivalence,
  \item the assembly map $\As_{\cF|_{H},F(\res^G_{H}\bC)_{H}}$ is an equivalence for every $H$ in $\pi^{*}\cH$,
 \end{enumerate}
 then the assembly map $\As_{\cF, F\bC_G}$ is an equivalence.
\end{kor}
\begin{proof}
 \cref{cor:res-ind-asm} {together with the first assumption} implies that $\As_{\pi^*\cH,F\bC_G}$ is an equivalence.
 Now apply \cref{cor:transitivity} to finish the proof.
\end{proof}

The final abstract inheritance property concerns filtered colimits.
Let $\Gamma \colon I \to \Grp$ be a filtered diagram of groups (without any assumptions on the structure maps).
We then set
\[ G:=\colim_{I} \Gamma \]
and let $\phi_i \colon \Gamma_i \to G$ denote the structure maps of the colimit.
Let $\cF$ be a family of subgroups of $G$ and let $\bC$ be an object in $\Fun(BG,\cK)$.

\begin{prop}[Passage to filtered colimits]\label{prop:filteredcolims}
 Suppose for every $i$ in $I$ that the assembly map $\As_{\phi_i^*\cF,F(B\phi_i^*\bC)_{\Gamma_i}}$ is {an} equivalence. 
 If $F$ preserves filtered colimits, then the assembly map $\As_{\cF,F\bC_G}$ is {an} equivalence.
 {If $\bM$ is stable, the same assertion holds with ``equivalence'' replaced by ``phantom equivalence''.}
\end{prop}
\begin{proof}
 Using \cref{qgoiegjqefewfqfewfqewf} and the observation that $\res_{\phi_i}E_\cF G \simeq E_{ \phi_i^*\cF}\Gamma_i$ (see \eqref{vsdfvsvwvdsdfvsdfv} for $\res_{\phi_{i}}$), we see that $\As_{\cF,F\bC_G}$ is equivalent to the map $F_\bC(E_\cF G) \to F_\bC(*)$ induced by the projection $E_\cF G \to *$, while $\As_{\phi_i^*\cF,F(B\phi_i^*\bC)_{\Gamma_i}}$ is equivalent to $F_{B\phi_i^*\bC}(\res_{\phi_i} E_\cF G) \to F_{B\phi_i^*\bC}(*)$.
 
 We claim that the assignment $i \mapsto F_{B\phi_i^*\bC} \circ \res_{\phi_i}$ extends to a diagram 
 \[ I \to \Fun^{\colim}(\PSh(G\Orb),\bM) \]
 such that
 \[ \colim_{i \in I} F_{B\phi_i^*\bC} \circ \res_{\phi_i} \simeq F_\bC\ .\]
 Since filtered colimits of (phantom) equivalences are (phantom) equivalences, this claim implies the proposition.
 
 We have functors
 \[ B\Gamma  \colon I \to \Cati\ , \quad  \Gamma\Orb \colon I\to \Cati\ ,\]
   where $\Gamma\Orb$ sends an object $i$ in $I$ to $\Gamma_{i}\Orb$, and a morphism $f \colon i\to i'$ in $I$ to the induction  functor 
 $\Gamma_{i'} \times_{\Gamma(f)} - \colon \Gamma_{i}\Orb\to \Gamma_{i'}\Orb$. 
 Applying $\PSh$ to $\Gamma\Orb$, this first gives a diagram $I^{\op}\to \Cat_{\infty}$, but taking the left adjoints of the structure maps we get a diagram
 \[ \PSh(\Gamma\Orb)\colon I \to \Prl\to {\CATi}\ . \]
 The morphisms $j^{\Gamma_i} \colon B\Gamma_i \to \Gamma_i\Orb$ assemble to a natural transformation $j^{\Gamma} \colon B\Gamma \to \Gamma\Orb$, while the Yoneda embeddings $\yo_{\Gamma_i}$ assemble to a natural transformation $\yo_{\Gamma} \colon \Gamma\Orb \to \PSh\Gamma\Orb)$.
 
 By passing to the associated cocartesian fibrations, we obtain morphisms
 \[ \wt{B\Gamma}\xrightarrow{\wt{j^{\Gamma}}} \wt{\Gamma\Orb}\xrightarrow{\wt{\yo_\Gamma}} \wt{\PSh(\Gamma\Orb)} \]
 of categories over $I$, where both arrows preserves cocartesian morphisms.
 Moreover, the diagrams $B\Gamma$, $\Gamma\Orb$ and $\PSh(\Gamma\Orb)$ admit cocones with vertices $BG$, $G\Orb$ and $\PSh(G\Orb)$, respectively.
 Hence we have a commutative diagram
 \[\xymatrix@C=4em{
  \wt{B\Gamma}\ar[r]^-{\wt{j^{\Gamma}}}\ar[d]_-{B\phi} & \wt{\Gamma\Orb}\ar[r]^-{\wt{\yo_\Gamma}}\ar[d]^-{\phi\Orb} & \wt{\PSh(\Gamma\Orb)}\ar[d]^-{\PSh(\phi\Orb)} \\
  BG\ar[r]^-{j^{G}} & G\Orb\ar[r]^-{\yo_G} & \PSh(G\Orb)
 }\]
 whose left half lies in the image of the nerve functor $\Nerve \colon \Cat \to \Cati$.
 The claim will follow by comparing the left Kan extension functors arising from this diagram.
 
 Since $B$ sends filtered colimits of groups to filtered colimits in $\Cati$, we have $\colim_I B\Gamma \simeq BG$.
 As colimits in $\Cati$ can always be computed by localising the (total space of the) associated cocartesian fibration at the collection of cocartesian morphisms \cite[Cor.~3.3.4.3]{htt}, it follows that $B\phi$ is a localisation.
 In particular, the restriction functor $B\phi^* \colon \Fun(BG,\cK) \to \Fun(\wt{B\Gamma},\cK)$ is fully faithful. 
 So the counit of the adjunction $(B\phi_!,B\phi^*)$ induces an equivalence
 \begin{equation}\label{eq:filteredcolims1}
  B\phi_!(\bC \circ B\phi) \simeq \bC\ .
 \end{equation}
 Let $E \colon \wt{\Gamma\Orb} \to \bN$ be a functor.
 If $\bN$ has sufficiently many colimits, the values of the left Kan extension $\phi\Orb_!E$ of $E$ along $\phi\Orb$ can be computed using the pointwise formula:
 \[ \phi\Orb_!E(S) \simeq \mathop{\colim}\limits_{(\Gamma_i \times_{\phi_i} T) \in \phi\Orb_{{/S}}} E(T)\ .\]
 It is straightforward to check that the indexing category $\phi\Orb_{{/S}}$ is filtered since $I$ is filtered.
 Since $F$ preserves filtered colimits, it follows that
 \begin{equation}\label{eq:filteredcolims2}
  \phi\Orb_!(F \circ -) \simeq F \circ \phi\Orb_!(-)\ .
 \end{equation}
 We combine these observations to obtain the following chain of equivalences:
 \begin{eqnarray*}
  F_\bC
  &\simeq &(\yo_G)_!(F \circ j^{G}_!\bC) \\
  &\overset{\eqref{eq:filteredcolims1}}{\simeq} &(\yo_G)_!(F \circ j^{G}_!B\phi_!(\bC \circ B\phi)) \\
  &\simeq &(\yo_G)_!(F \circ \phi\Orb_!\wt{j^{\Gamma}_!}(\bC \circ B\phi)) \\
  &\overset{\eqref{eq:filteredcolims2}}{\simeq} &(\yo_G)_!\phi\Orb_!(F \circ \wt{j^{\Gamma}_!}(\bC \circ B\phi)) \\
  &\simeq &\PSh(\phi\Orb)_!(\wt{\yo_\Gamma})_!(F \circ \wt{j^{\Gamma}_!}(\bC \circ B\phi))
 \end{eqnarray*}
 We have to identify the last term.
 Using \cite[Prop.~4.3.3.10]{htt}, we see that
 \begin{equation}\label{eq:filteredcolims3}
  \wt{j^{\Gamma}}_!(\bC \circ B\phi)|_{\Gamma_i\Orb} \simeq j^{\Gamma_i}_!(\bC \circ B\phi_i)
 \end{equation}
 for all $i$ in $I$.
 Then a second application of \cite[Prop.~4.3.3.10]{htt} implies that
 \begin{eqnarray}\label{eq:filteredcolims4}
  (\wt{\yo_\Gamma})_!(F \circ \wt{j^{\Gamma}_!}(\bC \circ B\phi))|_{\PSh(\Gamma_i\Orb)}
  &\simeq &(\yo_{\Gamma_i})_!(F \circ \wt{j^{\Gamma}_!}(\bC \circ B\phi))|_{\Gamma_i\Orb} \nonumber\\
  &\overset{\eqref{eq:filteredcolims3}}{\simeq} &(\yo_{\Gamma_i})_!(F \circ j^{\Gamma_i}_!(\bC \circ B\phi_i)) \nonumber\\
  &\simeq &F_{B\phi_i^*\bC}\ .
 \end{eqnarray}
 Now note that $\PSh(\phi\Orb)$ factors as
 \[ \wt{\PSh(\Gamma\Orb)} \xrightarrow{\ind_\phi} I \times \PSh(G\Orb) \xrightarrow{\pr} \PSh(G\Orb)\ ,\]
 where $\ind_\phi$ is a morphism over $I$ which preserves cocartesian morphisms and satisfies $\ind_\phi|_{\PSh(\Gamma_i\Orb)} \simeq \ind_{\phi_i}$.
 Therefore, we can apply \cite[Prop.~4.3.3.10]{htt} a third time to see that
 \begin{align}\label{eq:filteredcolims5}
  (\ind_\phi)_!&(\wt{\yo_\Gamma})_!(F \circ \wt{j^{\Gamma}}_!(\bC \circ B\phi))|_{\{i\} \times \PSh(\Gamma_{i}\Orb)}  \nonumber\\
  &\overset{\hphantom{\eqref{eq:filteredcolims4}}}{\simeq} (\ind_{\phi_i})_!({\yo_{\Gamma_{i}}})_!(F \circ j^{\Gamma_{i}}_!(\bC \circ B\phi)) \nonumber\\
  &\overset{\eqref{eq:filteredcolims4}}{\simeq} (\ind_{\phi_i})_!F_{\phi_i^*\bC} \nonumber\\
  &\overset{\hphantom{\eqref{eq:filteredcolims4}}}{\simeq} F_{B\phi_i^*\bC} \circ \res_{\phi_i}\ ,
 \end{align}
 where the last identification follows from the fact that $\res_{\phi_i}$ is right adjoint to $\ind_{\phi_i}$.
 Under the identification
 \[ \Fun(I \times \PSh(G\Orb),\bM) \simeq \Fun(I,\Fun(\PSh(G\Orb),\bM))\ , \]
 taking the left Kan extension along $\pr$ corresponds to taking the colimit over $I$.
 In conjunction with \eqref{eq:filteredcolims5}, this allows us to interpret
 \begin{align*}
  \PSh(\phi\Orb)_!(\wt{\yo_\Gamma})_!(F \circ \wt{j^{\Gamma}_!}(\bC \circ B\phi))
  &\simeq \pr_!(\ind_\phi)_!(F \circ \wt{j^{\Gamma}_!}(\bC \circ B\phi)) \\
  &\simeq \colim_{i \in I} F_{B\phi_i^*\bC} \circ \res_{\phi_i}\ ,
 \end{align*}
 which finishes the proof of the claim.
\end{proof}

\begin{proof}[Proof of \cref{thm:fullfjc}]
\phantomsection{}\label{hieruhqwefiwfewf}
  If $K$ is a subgroup of $G$, then $K \wr F$ is a subgroup of $G \wr F$ for every finite group $F$. 
  Hence \eqref{it:fullfjc13} follows from \cref{cor:subgroups}.
  This shows that for any group $G$ the set $\cF\cJ_{H}(G)$ of subgroups of $G$ 
which belong to $\cF\cJ_{H}$ is a family  of subgroups.
 
 In several cases, we will invoke \cref{thm:dfhj}, respectively its special cases \cref{thm:famenable,thm:dfh}.
 In contrast to \cref{def:f-amenable,def:F-DFH,def:dfhj}, the literature often establishes the existence of certain open covers instead of almost equivariant maps, and uses the term ``(strongly) transfer reducible'' or a variant thereof instead of ``finitely homotopy $\cF$-amenable''.
 We refer to \cite[Prop.~5.3]{blr} and \cite[Thm.6.19]{Enkelmann:2018aa} for an explanation how the nerve construction produces almost equivariant maps from appropriate open covers.
 We will tacitly use this translation in the sequel.
 
 If $G$ acts isometrically, properly and cocompactly on a finite-dimensional $\mathrm{CAT}(0)$-space, then the same is true for the wreath product $G \wr F$ with any finite group $F$.
 Building on \cite{BL-borel,BL-geodesicCAT0}, it was shown in \cite[Thm.~3.4]{wegner:cat0} that $G \wr F$ is therefore finitely homotopy $\cV\cC yc$-amenable in the sense of \cref{def:f-amenable}.
 So \cref{thm:famenable} implies \eqref{it:fullfjc1}. 
 
 In the sequel, we will make frequent use of the following argument.
 Suppose that $G$ is a DFHJ-group with respect to a family of subgroups contained in $\cF\cJ_{H}(G)$.
 Then \cite[Lem.~3.4]{KUWW} implies for every finite group $F$ that $G \wr F$ is a DFHJ-group with respect to the family $\cF\cJ_{H}(G \wr F)$.
 By \cref{thm:dfhj}, the assembly map $\As_{\cF,\Homol\bC_{G \wr F}}$ is an equivalence for every object $\bC$ in $\Fun(B(G \wr F),\Cle)$.
 Then \cref{cor:transitivity} shows that $G$ belongs to $\cF\cJ_\Homol$.

 We now show \eqref{it:fullfjc2}, so we assume that $G$ is a hyperbolic group. 
 As a special case of \eqref{it:fullfjc1}, every virtually cyclic group belongs to $\cF\cJ_\Homol$.
 By \cite[Lem.~2.1]{blr}, the  hyperbolic group $G$ is finitely homotopy $\cV\cC yc$-amenable.
 So it is in particular a DFHJ-group with respect to the family $\cV\cC yc$, and the latter is contained in $\cF\cJ_{H}(G)$.
 The preceding argument therefore shows that $G$ belongs to $\cF\cJ_{H}$.
 
 Before continuing with the concrete classes of groups listed in the theorem, we turn to the closure properties of the class $\cF\cJ_\Homol$.

 If the group $G'$ contains a group $G$ as a subgroup of finite index, then $G'$ is a subgroup of $G \wr F'$, where we set $F' := G'/\bigcap_{g \in G'} gGg^{-1}$.
 Hence $G' \wr F$ embeds into $(G \wr F') \wr F$, which in turn is a subgroup of $G \wr (F' \wr F)$.
 Therefore, if $G$ belongs to $\cF\cJ_\Homol$, then $G'$ also belongs to $\cF\cJ_\Homol$ by \cref{cor:subgroups}.
 This proves \eqref{it:fullfjc14}.
 
 Suppose that $G_1$ and $G_2$ belong to $\cF\cJ_\Homol$ and let $F$ be a finite group.
 Since $(G_1 \times G_2) \wr F$ is a subgroup of $(G_1 \wr F) \times (G_2 \wr F)$, it suffices to show that $\As_{\cV\cC yc, \Homol\bC_{G_1 \times G_2}}$ is an equivalence if $\As_{\cV\cC yc,\Homol(\bC_i)_{G_i}}$ is an equivalence for $i=1,2$ and every $\bC_i$ in $\Fun(BG_i,\Cle)$.
 Applying \cref{cor:epis} to the projection $\pr_1 \colon G_1 \times G_2 \to G_1$ and using  \eqref{it:fullfjc13}, it suffices to show that $\As_{\cV\cC yc, \Homol\bD_{V_1 \times G_2}}$ is an equivalence for every virtually cyclic subgroup $V_1$ of $G_1$ and every $\bD$ in $\Fun(B(V_1 \times G_2)),\Cle)$.
 Another application of \cref{cor:epis} to the projection $\pr_2 \colon V_1 \times G_2 \to G_2$ shows,  using \eqref{it:fullfjc13}, that it is enough to check that $\As_{\cV\cC yc, \Homol\bE_{V_1 \times V_2}}$ is an equivalence for every pair of virtually cyclic subgroups $V_1$ and $V_2$ in $G_1$ and $G_2$, respectively, and for every $\bE$ in $\Fun(B(V_1 \times V_2),\Cle)$.
 Since $V_1 \times V_2$ is virtually finitely generated abelian, this is a consequence of \eqref{it:fullfjc1}.
 Hence \eqref{it:fullfjc15} holds.
 
 Let $\Gamma$ be a filtered diagram of groups with colimit $G$ and let $F$ be a finite group. It is straightforward to check that the colimit of the diagram $\Gamma \wr F$ obtained by taking the wreath product with $F$ in each component is given by $G \wr F$.
 Hence \cref{prop:filteredcolims} implies that $G$ belongs to $\cF\cJ_\Homol$ if $\Gamma_i$ belongs to $\cF\cJ_\Homol$ for all $i$, which is precisely Assertion \eqref{it:fullfjc16}.

 We prove \eqref{it:fullfjc17} next.
 Let $F$ be a finite group and consider the induced epimorphism $\pi_F \colon G \wr F \to Q \wr F$.
 By \cref{cor:epis} and using \eqref{it:fullfjc13}, it is enough to show that $\pi_F^{-1}(V)$ belongs to $\cF\cJ_\Homol$ for all virtually cyclic subgroups $V$ of $Q \wr F$.
 The subgroup $W := \pi_F^{-1}(V) \cap G^F$ (note that $G^{F}$ is a subgroup of $G\wr F$)
  has finite index in $\pi_F^{-1}(V)$.
 So \eqref{it:fullfjc14} implies that it is enough to see that $W$ belongs to $\cF\cJ_\Homol$.
 Let $V_f$ denote the image of $V \cap Q^F$ under the projection map $Q^F \to Q$ to the $f$-th component.
 Then the group $W$ in turn embeds into the group $\prod_{f \in F} \pi^{-1}(V_f)$.
 Note that $V_f$ is virtually cyclic for every $f$ in $F$.
 Hence $\pi^{-1}(V_f)$ contains the preimage $\pi^{-1}(C)$ of a cyclic subgroup $C$ of $V_f$ as a subgroup of finite index.
 Since $\pi^{-1}(C)$ belongs to $\cF\cJ_\Homol$ by assumption, it follows from \eqref{it:fullfjc14} that $\pi^{-1}(V_f)$ also belongs to $\cF\cJ_\Homol$.
 Now \eqref{it:fullfjc15} implies that $\prod_{f \in F} \pi^{-1}(V_f)$ belongs to $\cF\cJ_\Homol$.
 So $W$ belongs to $\cF\cJ_\Homol$ by \eqref{it:fullfjc13}.
 
 Consider the canonical map $p \colon G_1 \ast G_2 \to G_1 \times G_2$.
 Since \eqref{it:fullfjc15} tells us that $G_1 \times G_2$ belongs to $\cF\cJ_\Homol$, \eqref{it:fullfjc17} implies that it suffices to show that $p^{-1}(C)$ belongs to $\cF\cJ_\Homol$ for every cyclic subgroup $C$ of $G_1 \times G_2$.
 Let $T$ denote the Bass--Serre tree of $G_1 \ast G_2$.  
  We restrict the $G_1 \ast G_2$-action  on $T$ to the subgroup $p^{-1}(C)$. The 
  restriction of $p$ to the vertex stabilisers of this $p^{-1}(C)$-action is injective, so all vertex stabilisers are isomorphic to subgroups of $C$.
 By Kurosh's theorem \cite[Ch.~1, Thm.~14]{trees}, $p^{-1}(C)$ is isomorphic to the free product of a free group and a collection of such {cyclic} vertex stabilisers.
 Therefore, $p^{-1}(C)$ is a filtered colimit of hyperbolic groups, and \eqref{it:fullfjc18} follows from \eqref{it:fullfjc2} and \eqref{it:fullfjc16}.  
 
 We now continue with Assertion \eqref{it:fullfjc3}.
 As another special case of \eqref{it:fullfjc1}, finitely generated, virtually abelian groups belong to $\cF\cJ_\Homol$.
 Applying \eqref{it:fullfjc16}, every virtually abelian group belongs to $\cF\cJ_\Homol$.
 By \cite[Prop.~3.3]{KUWW}, the groups $\IZ[w,w^{-1}] \rtimes \IZ$ described in \cref{ex-solvable} are DFHJ groups with respect to the family of virtually abelian subgroups, so they belong to $\cF\cJ_\Homol$.  
 We now argue as in the proof of \cite[Prop.~3.3]{wegner-solv} to conclude that \eqref{it:fullfjc3} holds. Note that this proof only makes use of closure properties we have also established for $\cF\cJ_\Homol$.
 
 For Assertion \eqref{it:fullfjc4}, it is enough to show that $\mathrm{GL}_n(\IQ)$ and $\mathrm{GL}_n(k(t))$ belong to $\cF\cJ_\Homol$ due to \eqref{it:fullfjc13}.
 By \eqref{it:fullfjc16}, one only has to show that $\mathrm{GL}_n(\IZ[S^{-1}])$ and $\mathrm{GL}_n(k[t][S^{-1}])$ belong to $\cF\cJ_\Homol$, where $S$ denotes a finite set of primes in the respective ring.
 Let $G$ denote one of these groups.
 \cite[Prop.~2.2 \& 7.23]{rueping-sarithmetic} shows that there exists a certain family $\cF$ such that $G$ is a DFHJ group with respect to $\cF$.
 Hence it is enough to show that all groups in $\cF$ belong to $\cF\cJ_\Homol$. 
 The proof of \cite[Thm.~8.12]{rueping-sarithmetic} applies  also in our situation due to \eqref{it:fullfjc3} and the closure properties of $\cF\cJ_\Homol$.   
 
 For Assertions \eqref{it:fullfjc7}, \eqref{it:fullfjc8} and \eqref{it:fullfjc10}, we observe that the arguments in \cite{klr-lattices}, \cite[Sec.~7]{bfl-lattices}, \cite{wu-graphs} respectively \cite{gmr-graphs} use only properties of the class $\cF\cJ_\Homol$ that have been established in previous steps.
 
 Using \eqref{it:fullfjc3} and the closure properties of $\cF\cJ_\Homol$, Assertion \eqref{it:fullfjc9} follows by induction on the complexity of the surface from \cite[Cor.~9.1 \& Lem.~9.2]{BB}. See also the proof of \cite[Lem.~9.3]{BB}.
\end{proof}

\bibliographystyle{alpha}
\bibliography{farjo}

\newcommand{\etalchar}[1]{$^{#1}$}
\begin{thebibliography}{BEKW20b}

\bibitem[Bar17]{bartels-relhyp}
A.~Bartels.
\newblock Coarse flow spaces for relatively hyperbolic groups.
\newblock {\em Compos. Math.}, 153(4):745--779, 2017.

\bibitem[Bar18]{Bartels:icm}
A.~Bartels.
\newblock {\(K\)-theory and actions on Euclidean retracts}.
\newblock In {\em Proceedings of the international congress of mathematicians,
  ICM 2018, Rio de Janeiro, Brazil, August 1--9, 2018. Volume II. Invited
  lectures}, pages 1041--1062. Hackensack, NJ: World Scientific; Rio de
  Janeiro: Sociedade Brasileira de Matem\'atica (SBM), 2018.

\bibitem[BB19]{BB}
A.~Bartels and M.~Bestvina.
\newblock The {F}arrell--{J}ones {C}onjecture for mapping class groups.
\newblock {\em Invent. Math.}, 215(2):651--712, 2019.

\bibitem[BCKW]{unik}
U.~Bunke, D.-C. Cisinski, D.~Kasprowski, and C.~Winges.
\newblock Controlled objects in left-exact $\infty$-categories and the
  {N}ovikov conjecture.
\newblock \href{http://arxiv.org/abs/1911.02338}{arXiv:1911.02338}.

\bibitem[BE20a]{buen}
U.~Bunke and A.~Engel.
\newblock {\em {Homotopy theory with bornological coarse spaces}}, volume 2269.
\newblock Cham: Springer, 2020.

\bibitem[BE20b]{be:coarseassembly}
Ulrich Bunke and Alexander Engel.
\newblock Coarse assembly maps.
\newblock {\em J. Noncommut. Geom.}, 14(4):1245--1303, 2020.

\bibitem[BEKW20a]{equicoarse}
U.~Bunke, A.~Engel, D.~Kasprowski, and C.~Winges.
\newblock {Equivariant coarse homotopy theory and coarse algebraic
  $K$-homology}.
\newblock In Guillermo Corti\~nas and Charles~A. Weibel, editors, {\em
  {$K$-theory in Algebra, Analysis and Topology}}, volume 749 of {\em Contemp.
  Math.}, pages 13--104. American Mathematical Society (AMS), Providence, RI,
  2020.

\bibitem[BEKW20b]{Bunke:ab}
U.~Bunke, A.~Engel, D.~Kasprowski, and C.~Winges.
\newblock Homotopy theory with marked additive categories.
\newblock {\em Theory Appl. Categ.}, 35(13):371--416, 2020.

\bibitem[BEKW20c]{injectivity}
U.~Bunke, A.~Engel, D.~Kasprowski, and C.~Winges.
\newblock {Injectivity results for coarse homology theories}.
\newblock {\em {Proc. Lond. Math. Soc. (3)}}, 121(6):1619--1684, 2020.

\bibitem[BFL14]{bfl-lattices}
A.~Bartels, F.~T. Farrell, and W.~L{\"u}ck.
\newblock The {F}arrell-{J}ones {C}onjecture for cocompact lattices in
  virtually connected {L}ie groups.
\newblock {\em J. Amer. Math. Soc.}, 27(2):339--388, 2014.

\bibitem[BGT13]{MR3070515}
A.~J. Blumberg, D.~Gepner, and G.~Tabuada.
\newblock A universal characterization of higher algebraic {$K$}-theory.
\newblock {\em Geom. Topol.}, 17(2):733--838, 2013.

\bibitem[BGT14]{bgt-2}
A.~J. Blumberg, D.~Gepner, and G.~Tabuada.
\newblock Uniqueness of the multiplicative cyclotomic trace.
\newblock {\em Adv. Math.}, 260:191--232, 2014.

\bibitem[BKW19]{BKW:coarseA}
U.~Bunke, D.~Kasprowski, and C.~Winges.
\newblock {S}plit {I}njectivity of {A}-{T}heoretic {A}ssembly {M}aps.
\newblock {\em Int. Math. Res. Not. (IMRN)}, 2021(2):885--947, 10 2019.

\bibitem[BL12a]{BL-borel}
A.~Bartels and W.~L{\"u}ck.
\newblock The {B}orel conjecture for hyperbolic and {$\mathrm{CAT}(0)$}-groups.
\newblock {\em Ann. of Math. (2)}, 175(2):631--689, 2012.

\bibitem[BL12b]{BL-geodesicCAT0}
A.~Bartels and W.~L{\"uck}.
\newblock Geodesic flow for $\mathrm{CAT}(0)$--groups.
\newblock {\em Geom. Topol.}, 16(3):1345--1391, 2012.

\bibitem[BLR08]{blr}
A.~Bartels, W.~L{\"u}ck, and H.~Reich.
\newblock {The $K$-theoretic Farrell--Jones conjecture for hyperbolic groups}.
\newblock {\em Invent. math.}, 172:29--70, 2008.

\bibitem[BLRR14]{BLRR}
A.~Bartels, W.~L{\"u}ck, H.~Reich, and H.~R{\"u}ping.
\newblock K- and {L}-theory of group rings over {$GL_n({\mathbf Z})$}.
\newblock {\em Publ. Math. Inst. Hautes \'Etudes Sci.}, 119:97--125, 2014.

\bibitem[BR05]{br:fjc}
A.~Bartels and H.~Reich.
\newblock {On the Farrell-Jones conjecture for higher algebraic ${K}$-theory}.
\newblock {\em J. Amer. Math. Soc.}, 18(3):501--545, 2005.

\bibitem[BR07]{BR-coefficients}
A.~Bartels and H.~Reich.
\newblock Coefficients for the {F}arrell-{J}ones conjecture.
\newblock {\em Adv. Math.}, 209(1):337--362, 2007.

\bibitem[Bun19]{bunke}
U.~Bunke.
\newblock Homotopy theory with $*$-categories.
\newblock {\em Theory Appl. Categ.}, 34(27):781--853, 2019.

\bibitem[CDH{\etalchar{+}}23]{TheNineI}
B.~Calm{\`e}s, E.~Dotto, Y.~Harpaz, F.~Hebestreit, M.~Land, K.~Moi, D.~Nardin,
  T.~Nikolaus, and W.~Steimle.
\newblock Hermitian {K}-theory for stable {{\(\infty\)}}-categories. {I}:
  {Foundations}.
\newblock {\em Sel. Math., New Ser.}, 29(1):{Paper No. 10, 269 p.}, 2023.

\bibitem[Cis19]{Cisinski:2017}
D.-C. Cisinski.
\newblock {\em Higher categories and homotopical algebra}, volume 180 of {\em
  Cambridge studies in advanced mathematics}.
\newblock Cambridge University Press, 2019.
\newblock Available online under
  \url{http://www.mathematik.uni-regensburg.de/cisinski/CatLR.pdf}.

\bibitem[Cor82]{cordier}
J.-M. Cordier.
\newblock Sur la notion de diagramme homotopiquement coh\'erent.
\newblock {\em Cahiers de Topologie et G\'eom\'etrie Diff\'erentielle
  Cat\'egoriques}, 23(1):93--112, 1982.

\bibitem[DL98]{davis_lueck}
J.~F. Davis and W.~L{\"u}ck.
\newblock {Spaces over a Category and Assembly Maps in Isomorphism Conjectures
  in $K$- and $L$-Theory}.
\newblock {\em $K$-Theory}, 15:201--252, 1998.

\bibitem[ELP{\etalchar{+}}18]{Enkelmann:2018aa}
N.-E. Enkelmann, W.~L{\"u}ck, M.~Pieper, M.~Ullmann, and C.~Winges.
\newblock On the {F}arrell--{J}ones {C}onjecture for {W}aldhausen's
  {$A$}-theory.
\newblock {\em Geom. Topol.}, 22:3321--3394, 2018.

\bibitem[FJ93]{FJ}
F.~T. Farrell and L.~E. Jones.
\newblock {I}somorphism {C}onjectures in {A}lgebraic {K}-{T}heory.
\newblock {\em J. Amer. Math. Soc.}, 6(2):249, apr 1993.

\bibitem[GMR15]{gmr-graphs}
G.~Gandini, S.~Meinert, and H.~R{\"u}ping.
\newblock The {F}arrell--{J}ones conjecture for fundamental groups of graphs of
  abelian groups.
\newblock {\em Groups Geom. Dyn.}, 9:783---792, 2015.

\bibitem[GZ67]{GabrielZisman}
P.~Gabriel and M.~Zisman.
\newblock {\em Calculus of fractions and homotopy theory}.
\newblock Ergebnisse der Mathematik und ihrer Grenzgebiete, Band 35.
  Springer-Verlag New York, Inc., New York, 1967.

\bibitem[Hin16]{hinich}
V.~Hinich.
\newblock Dwyer--{K}an localization revisited.
\newblock {\em Homology Homotopy Appl.}, 18(1):27--48, 2016.

\bibitem[HMS20]{hms:coisotropic}
R.~Haugseng, V.~Melani, and P.~Safronov.
\newblock Shifted coisotropic correspondences.
\newblock {\em J. Inst. Math. Jussieu}, 2020.

\bibitem[Hu65]{hu}
S.-T. Hu.
\newblock {\em Theory of Retracts}.
\newblock Wayne State University Press, Detroit, 1965.

\bibitem[KLR16]{klr-lattices}
H.~Kammeyer, W.~L{\"u}ck, and H.~R{\"u}ping.
\newblock The {F}arrell--{J}ones conjecture for arbitrary lattices in virtually
  connected {L}ie groups.
\newblock {\em Geom. Topol.}, 20(3):1275--1287, 2016.

\bibitem[KUWW18]{KUWW}
D.~Kasprowski, M.~Ullmann, C.~Wegner, and C.~Winges.
\newblock The {$A$}-theoretic {F}arrell--{J}ones conjecture for virtually
  solvable groups.
\newblock {\em Bull. Lond. Math. Soc.}, 50(2):219--228, 2018.

\bibitem[LR05]{LR}
W.~L{\"u}ck and H.~Reich.
\newblock The {B}aum-{C}onnes and the {F}arrell-{J}ones conjectures in {$K$}-
  and {$L$}-theory.
\newblock In {\em Handbook of {$K$}-theory. {V}ol. 1, 2}, pages 703--842.
  Springer, Berlin, 2005.

\bibitem[LRRV17]{LRRV}
W.~L{\"u}ck, H.~Reich, J.~Rognes, and M.~Varisco.
\newblock Algebraic {K}-theory of group rings and the cyclotomic trace map.
\newblock {\em Adv. Math.}, 304:930--1020, 2017.

\bibitem[LRRV19]{LRRV2}
W.~L{\"u}ck, H.~Reich, J.~Rognes, and M.~Varisco.
\newblock Assembly maps for topological cyclic homology of group algebras.
\newblock {\em J. Reine Angew. Math.}, 755:247--277, 2019.

\bibitem[L{\"u}c]{ICbook}
W.~L{\"u}ck.
\newblock {I}somorphism {C}onjectures in {K}- and {L}-theory.
\newblock Ongoing book project, available
  \href{https://www.him.uni-bonn.de/lueck/}{on the author's homepage}.

\bibitem[L{\"u}c20]{surveyLueck}
W.~L{\"u}ck.
\newblock Assembly maps.
\newblock In {\em Handbook of {H}omotopy {T}heory}, pages 853--892. Chapman and
  Hall/CRC, New York, 2020.

\bibitem[Lur]{HA}
J.~Lurie.
\newblock Higher {A}lgebra.
\newblock Available at
  \href{https://www.math.ias.edu/~lurie/papers/HA.pdf}{https://www.math.ias.edu/{\textasciitilde}lurie/papers/HA.pdf}.

\bibitem[Lur09]{htt}
J.~Lurie.
\newblock {\em Higher topos theory}, volume 170 of {\em Annals of Mathematics
  Studies}.
\newblock Princeton University Press, Princeton, NJ, 2009.

\bibitem[Nik]{Nikolaus:2016aa}
Th. Nikolaus.
\newblock {Stable $\infty$-Operads and the multiplicative Yoneda lemma}.
\newblock \href{https://arxiv.org/abs/1608.02901}{arXiv:1608.02901}.

\bibitem[Rei]{reis}
J.~F. Reis.
\newblock {An improvement of the Farrell-Jones conjecture for localising
  invariants}.
\newblock \href{https://arxiv.org/abs/2211.15523}{arXiv:2211.15523}.

\bibitem[Rez]{rezk}
C.~Rezk.
\newblock A model category for categories.
\newblock Available at
  \href{https://faculty.math.illinois.edu/~rezk/cat-ho.dvi}{https://faculty.math.illinois.edu/{\textasciitilde}rezk/cat-ho.dvi}.

\bibitem[R{\"u}p16]{rueping-sarithmetic}
H.~R{\"u}ping.
\newblock The {F}arrell--{J}ones conjecture for {$S$}-arithmetic groups.
\newblock {\em J. Topol.}, 9(1):51--90, 2016.

\bibitem[RV18]{surveyRV}
H.~Reich and M.~Varisco.
\newblock Algebraic {$K$}-theory, assembly maps, controlled algebra, and trace
  methods.
\newblock In {\em Space---time---matter}, pages 1--50. De Gruyter, Berlin,
  2018.

\bibitem[Ser80]{trees}
J.-P. Serre.
\newblock {\em Trees}.
\newblock Springer-Verlag Berlin Heidelberg, 1980.

\bibitem[UW19]{UW}
M.~Ullmann and C.~Winges.
\newblock {O}n the {F}arrell--{J}ones {C}onjecture for algebraic {$K$}--theory
  of spaces: the {F}arrell-{H}siang method.
\newblock {\em Ann. $K$-Theory}, 4(1):57--138, 2019.

\bibitem[Vog73]{vogt}
R.~M. Vogt.
\newblock Homotopy limits and colimits.
\newblock {\em Math. Z.}, 134:11--52, 1973.

\bibitem[Weg12]{wegner:cat0}
C.~Wegner.
\newblock The {$K$}-theoretic {F}arrell--{J}ones conjecture for
  {CAT}(0)-groups.
\newblock {\em Proc. Amer. Math. Soc.}, 140(3):779--793, 2012.

\bibitem[Weg15]{wegner-solv}
C.~Wegner.
\newblock {The Farrell-Jones conjecture for virtually solvable groups.}
\newblock {\em {J. Topol.}}, 8(4):975--1016, 2015.

\bibitem[Wei02]{Weiss2002}
M.~Weiss.
\newblock Excision and restriction in controlled {$K$}-theory.
\newblock {\em Forum Math.}, 14(1):85--119, 2002.

\bibitem[Win15]{winges}
C.~Winges.
\newblock On the transfer reducibility of certain {F}arrell--{H}siang groups.
\newblock {\em Algebr. Geom. Topol.}, 15(5):2921--2948, 2015.

\bibitem[Wu16]{wu-graphs}
X.~Wu.
\newblock {F}arrell--{J}ones {C}onjecture for fundamental groups of graphs of
  virtually cyclic groups.
\newblock {\em Topology Appl.}, 206:185--189, 2016.

\end{thebibliography}

\end{document}